\documentclass[11pt]{amsart}
\usepackage{todonotes}

\usepackage{enumitem}
\makeatletter
\newcommand{\mylabel}[2]{#2\def\@currentlabel{#2}\label{#1}}
\makeatother

\usepackage{amsmath}
\usepackage{marginnote}
\usepackage[scaled]{helvet} % ss
\usepackage[T1]{fontenc}
\usepackage{textcomp}
\usepackage{soul}
\normalfont
\usepackage{amssymb}
\usepackage{cancel}
\usepackage{color}
\usepackage{graphicx}
\usepackage[utf8]{inputenc}
\usepackage[english]{babel}
\usepackage{csquotes}
\usepackage{amssymb}
\usepackage{tikz-cd}
\usepackage{csquotes}
\usepackage{amsmath, amsfonts,graphicx,fullpage,etex,tikz,hyperref}

\hypersetup{
  colorlinks   = true, %Colours links instead of ugly boxes
  urlcolor     = blue, %Colour for external hyperlinks
  linkcolor    = blue, %Colour of internal links
  citecolor   = red %Colour of citations
}

\usepackage{amssymb}

\usepackage{enumerate}
\usepackage{enumitem}
\newtheorem{theorem}{Theorem}[section]
\newtheorem{lemma}[theorem]{Lemma}

\newtheorem{proposition}[theorem]{Proposition}
\newtheorem{corollary}[theorem]{Corollary}
\newtheorem{conjecture}[theorem]{Conjecture}

\theoremstyle{definition}

\newtheorem{definition}[theorem]{Definition}
\newtheorem{question}[theorem]{Question}

\newtheorem{example}[theorem]{Example}

\theoremstyle{remark}
\newtheorem{remark}[theorem]{Remark}

\numberwithin{equation}{section}
\newcommand{\la}{\lambda}

\newcommand{\C}{\mathbb{C}}

\newcommand{\Aff}{\mathrm{Aff}}
\newcommand{\Aut}{\mathrm{Aut}}

\newcommand{\id}{\mathrm{Id}}
\newcommand{\Z}{\mathbb{Z}}

\newcommand{\Q}{\mathbb{Q}}
\newcommand{\R}{\mathbb{R}}

\newcommand{\N}{\mathbb{N}}

\newcommand{\mc}{\mathcal}
\newcommand{\set}[1]{\left\lbrace #1 \right\rbrace}
\newcommand{\mf}{\mathfrak}
\newcommand{\abs}[1]{\left| #1 \right|}
\newcommand{\norm}[1]{\left|\left| #1 \right|\right|}
\newcommand{\ve}{\varepsilon}
\newcommand{\of}{\circ}

\newcommand{\leftN}{{}^{\lambda}\hspace{-1.5pt}N}

\newcommand{\holder}{H\"{o}lder }

\definecolor{darkcyan}{rgb}{0. 0.65, 0.65}

\newlength{\wdth}

\newcommand{\G}{G}

\newtheorem{Structural Stability Theorem}[theorem]{Structural Stability Theorem}

\def\bt{\begin{theorem}}
\def\et{\end{theorem}}
\def\bd{\begin{definition}}
\def\ed{\end{definition}}
\def\bl{\begin{lemma}}
\def\el{\end{lemma}}
\def\bp{\begin{proof}}
\def\ep{\end{proof}}
\def\holder{ H\"{o}lder }

\def\Diff{\operatorname{Diff}}
\def\Homeo{\operatorname{Homeo}}
\def\Lie{\operatorname{Lie}}
\def\Stab{\operatorname{Stab}}
\def\loc{\operatorname{loc}}

\def\rank{\operatorname{rank}}
\def\ad{\operatorname{ad}}
\def\Ad{\operatorname{Ad}}
\def\diag{\operatorname{diag}}
\def\Diag{\operatorname{Diag}}

\def\Isom{\operatorname{Isom}}

\def\Lip{\operatorname{Lip}}

\newenvironment{dedication}
  {
   \vspace*{\stretch{1}}% some space at the top 
   \itshape             % the text is in italics
   \raggedleft          % flush to the right margin

  }

\begin{document}

\thanks{$^\dagger$ Supported in part by Swedish Research Council grant VR 2019-67250 and VR 2023-03596}
\thanks{$^\ast$ Supported in part by NSF grants DMS 1607260 and DMS 2003712}
\thanks{$^{\ddag}$ Supported VY in part by NSF grant DMS  1604796}
\thanks{$^\ast {^\ast}$ Supported by National Key R$\&$D Program of China No. 2024YFA1015100 and NSFC grant 1209001, 12090015}
\title{The Zimmer Program for partially Hyperbolic Actions}
\author{ Danijela Damjanovi\'c $^\dagger$}
\address{DEPARTMENT OF MATHEMATICS, KUNGLIGA TEKNISKA H\"{O}GSKOLAN, LINDSTEDTSV\"{A}GEN 25, SE-100 44 STOCKHOLM, SWEDEN.}
\email{ddam@kth.se}

\author{Ralf Spatzier $^\ast$}
\address{DEPARTMENT OF MATHEMATICS, UNIVERSITY OF MICHIGAN, ANN ARBOR, MI 48109, USA.}
\email{spatzier@umich.edu}

\author{Kurt Vinhage $^{\ddag}$}
\address{DEPARTMENT OF MATHEMATICS, UNIVERSITY OF UTAH,
 SALT LAKE CITY, UT 84112, USA.}
\email{vinhage@math.utah.edu}

\author{Disheng Xu $^\ast {^\ast}$}
\address{DEPARTMENT OF SCIENCE, GREAT BAY UNIVERSITY AND GREAT BAY INSTITUTE FOR ADVANCED STUDY, SONGSHAN LAKE INTERNATIONAL INNOVATION ENTREPRENEURSHIP XOMMUNITY A5, DONGGUAN, GUANGDONG, 523000, CHINA.}
\email{xudisheng@gbu.edu.cn}
\hspace{4cm}

%Consider a volume preserving  Anosov $C^{\infty}$ action $\alpha$ on a compact manifold $X$ by a semisimple Lie group with all simple factors of real rank at least 2.  More precisely we assume that some Cartan subgroup $A$ of $G$ {%\color{cyan}
%(or equivalently $G$)} contains a dense set of elements which act normally hyperbolically    on $M$ with respect to the orbit foliation of $A$. We show that $\alpha$ is $C^{\infty}$-conjugate to  an action by left translations of a bi-homogeneous space $M\backslash H/\Lambda$, where $M$ is a compact subgroup of a Lie group $H$ and $\Lambda$ is a uniform lattice in $H$.  {\color{red} We extend these results to partially hyperbolic accessible $G$-actions.}

\begin{abstract}

{\color{black} Zimmer's superrigidity theorems on higher rank Lie groups and their lattices launched a program of study aiming to classify actions of semisimple Lie groups and their lattices, known as the {\it Zimmer program}. When the group is too large relative to the dimension of the phase space, the Zimmer conjecture predicts that the actions are all virtually trivial. At the other extreme, when the actions exhibit enough regular behavior, the actions should all be of algebraic origin.

We make progress in the program by showing smooth conjugacy to a bi-homogeneous model (up to a finite cover) for volume-preserving actions of semisimple Lie groups without compact or rank one factors, which have two key assumptions: partial hyperbolicity for a large class of elements ({\it totally partial hyperbolicity}) and accessibility, a condition on the webs generated by dynamically-defined foliations. We also obtain classification for actions of higher-rank abelian groups satisfying stronger assumptions.}
\end{abstract}

\maketitle

\begin{dedication}
\begin{center}
{  In memory of Robert Zimmer who started these investigations}
\end{center}
\end{dedication}
\tableofcontents

%\newpage

%{\cb
\section{Introduction}
This paper is a contribution to the Zimmer Program of studying actions of higher rank semisimple Lie groups $G$ and their lattices $\Gamma$ on compact manifolds.  More precisely one assumes that $G$ has real rank at least 2, and that $\Gamma$ is an irreducible lattice in $G$.  This program was inspired by Margulis' superrigidity theorem which classifies finite dimensional representations of such $\Gamma$ \cite{M91}.   A classification of  such %actions
is impossible in general as one can construct such actions starting with any flow  via the induction procedure. However, classification or at least a detailed structural understanding  might be possible under suitable geometric or dynamical assumptions on the action or underlying manifold.  Zimmer formulated this program in his ICM address in 1986, and in other papers from the early 1980s, e.g. \cite{Z-ICM,Z-1983,Z-IHES}.
{ Margulis further cemented this   in his list of problems for the new century \cite[Problem 11]{Margulis-Problems}.}

Much progress has been made in recent years \cite{MR2807830,MR4186267}.  On the one hand, Brown, Fisher and  Hurtado made major progress on  one of the main conjectures of the Zimmer program, the Zimmer Conjecture,  that lattices in higher rank semisimple Lie groups cannot act (except via finite groups) on compact manifolds of  dimension at most $d(G)$ where $d(G)$ can be calculated explicitly in terms of $G$ and the structure of its roots \cite{BFH2020,BFH2022,BFH2021}.   For $G= SL(n,\R)$, this dimension is simply the optimal $n-1$.  For other groups though  much work is left to be done. {\color{black} For current progress, extending the range of groups for which the Zimmer conjecture holds, see \cite{AnBrZh}.} One key aspect of the proof is that such actions preserve a Riemannian metric, and hence are rather tame from point of view of dynamics.  

{\color{black}When the manifold is high dimensional, the situation is much more complicated.  Indeed, Katok and Lewis constructed some exotic examples of volume preserving $\Gamma$-actions using an algebraic blow-up procedure at a common fixed point of an action by automorphisms of a torus \cite{KL96}. {\color{black} Benveniste constructed blow-up  examples for $G$-actions
\cite{Benveniste-thesis, Benveniste-Fisher}. These constructions disturb} certain invariant foliations and distributions which always exist for algebraic actions. Therefore, to obtain a complete classification, one should make some dynamical or geometric assumptions about the action. A natural dynamical assumption is the existence of certain invariant distributions which exhibit uniform expansion and contraction.}
%On the other end of the dynamical spectrum lie hyperbolic actions,  especially actions with 

For instance, one may ask that the action exhibits an Anosov element.  By this one usually means that some element $g \in G$  acts normally hyperbolically with respect to the orbit foliation of an associated subgroup of $G$ (e.g. the centralizer of an $\R$-split Cartan subgroup in a semisimple Lie group). That is, $g$ uniformly expands and contracts complementary subbundles transverse to the orbit foliation of the associated subgroup. This is what we commonly see in algebraic examples.  For lattices $\Gamma$, an action is Anosov simply means  that some element $\gamma \in 
\Gamma$ acts via an Anosov diffeomorphism. 
%of $M$.  
Examples of such actions are known to exist on certain nilmanifolds $N/\Lambda$ and  finite quotients of such where $\Lambda $ is a lattice in a (simply connected) nilpotent group $N$. Conversely, if the underlying manifold is a nilmanifold, 
Brown, Rodriguez Hertz and Wang prove the beautiful global rigidity result that all actions of higher rank cocompact lattices  are $C^{\infty}$-conjugate to one by automorphisms \cite{BRHW2}. For non-uniform $\Gamma$, they need to assume additional conditions such as being able to lift the action to the universal cover of $N/\Lambda$. In addition,   a well-known 50 year old question of Anosov and Smale asks whether Anosov diffeomorphisms only exist on nilmanifolds and their finite quotients \cite{Smale66}.  A positive answer to this question would immediately imply the classification of Anosov actions on arbitrary compact manifolds, at least for irreducible higher rank uniform lattices.  However, at this point this question remains wide open.  
%\color{black} One important conclusion of this paper is the classification of actions of {\it connected higher rank semisimple Lie groups} with many Anosov elements. 
\color{black}

Understanding the structure  of Anosov actions of {\it connected} higher rank semisimple Lie groups $G$ on arbitrary compact manifolds is similarly intriguing, and global classification of such actions is one of the the main conclusions of this paper. {\it Anosov elements} of $G$ actions are actually never Anosov  diffeomorphisms.  Indeed, they will commute with their centralizers, typically non-discrete,  and thus cannot act hyperbolically in the orbit directions of their centralizer. Really, Anosov $G$ actions are ones which give the maximal amount of hyperbolicity possible for a $G$-action.  However, our techniques apply to a  much more general class of actions which we will now discuss.   

% We  relax the notion of an Anosov action by requiring invariant distributions with uniform expansion and contraction are the partially hyperbolic systems. 
\color{black}
We relax the notion of an Anosov action to a {\it partially hyperbolic action}. In this case, one does not ask that $g \in G$ is normally hyperbolic with respect to some orbit foliation of some subgroup of $G$, but only that there is a complementary $g$-invariant distribution,  which $g$ expands or contracts at a rate {\it weaker} than the uniformly expanding and contracting ones. The uniformly expanding and contracting distributions then  integrate to foliations called unstable and stable, respectively. Clearly, a product of Anosov dynamics and any relatively tame dynamics, gives partial hyperbolicity. The class of partially hyperbolic dynamics is however much more rich. For general partially hyperbolic systems the center distribution may not even be integrable. Even when it is integrable, the foliation may be wild and dynamics on it may be difficult to access. We restrict here to partially hyperbolic systems which have the {\it accessibility} property. This means that one can reach any point on the manifold by traveling along leaves of unstable and stable foliations.
In fact, partially hyperbolic {\it accessible} dynamics is typically what we see in the world of  {\it homogeneous} partially hyperbolic dynamics, where actions are given by left translations on homogeneous spaces. Simplest example is the time-1 map of the geodesic flow on a compact hyperbolic surface.

\color{black}

Some examples of  $G$-actions arise from Anosov and partially hyperbolic actions of uniform lattices $\Gamma$ in $G$ by automorphisms of tori and nilmanifolds by the suspension construction (see Example \ref{ex:suspension} and Section \ref{sec:lattice-q}). 
Other examples of Anosov actions %more specific to connected Lie groups
include the action by left translations   of $SO(p,q)$ on $SO(p,q+1)/ \Lambda$ for suitable $p,q$. These examples were introduced in \cite{Goetze-Spatzier}, and we give a detailed description in Example \ref{ex:SO22-on-SO2n}. These examples can be extended to {\it bi-homogeneous actions}, which occur when $G$ embeds in a group $H$, $K$ is a compact subgroup of the centralizer of $G$ in $H$, and $G$ acts by translations on $K \backslash H / \Lambda$. 
%\color{black} 
%We note that these actions preserve natural affine connections on these spaces, themselves of algebraic nature. bi-homogeneous .
\color{black} These bi-homogeneous actions will be Anosov exactly when $K$ is  the entire centralizer of $G$, which is an extremely rare property for an embedding of  $G$. Whenever $H$ is simple, and $G$ acts by translations, the action is {\it always} partially hyperbolic, and exhibits key dynamical properties (ergodicity and accessibility, which we discuss in Section \ref{sec:results}). These examples are all included in Example \ref{ex:general-headache}. This class of examples is much larger than those which are Anosov. 

\color{black}The main result of this paper   is that bi-homogeneous $G$-actions in fact smoothly model {\it any} smooth volume preserving $G$-action with sufficiently many partially hyperbolic elements and some accessibility.  Namely in Theorem \ref{G-action} we show that all such $G$-actions essentially are  bi-homogeneous actions as described above, up to a smooth conjugacy 
%(i.e. smooth change of coordinates) 
and up to a finite cover. Along the way, this result also answers the question which manifolds support such $G$-actions. 

The accessibility assumption is always satisfied by Anosov $G$-actions, which implies directly global classification for volume preserving Anosov $G$-actions with sufficiently many Anosov elements, in Corollary \ref{G-Anosov}.
\color{black}%Indeed, to preserve the Anosov property, if $G\subset H$ is acting on a compact $H$-homogeneous space by translations, where $H$ is another Lie group, the action will only be Anosov when $\Lie(H) = \Lie(G) \oplus E$, and a regular element of $G$ preserves and acts hyperbolically on $E$ as a vector space. This situation is extremely rare. One may instead build {\it bi}-homogeneous actions on $K \backslash H / \Lambda$ when there is a splitting $\Lie(H) = \Lie(G) \oplus \Lie(K) \oplus E$ for some compact Lie group $K$ commuting with $G$ and $g \in G$ preserves and acts hyperbolically on $E$. Which this broadens the examples slightly, it is still very restrictive since $K$ must be compact.

Prior works on Anosov (or sufficiently hyperbolic) actions in the Zimmer program have been deeply interwoven with  understanding  Anosov actions of higher rank abelian groups.  Indeed,   the  split Cartan subgroup of any higher rank semisimple Lie group $G$ (without compact factors) is isomorphic to  $\R^k$, $k \geq 2$. If it has sufficiently many Anosov elements, one can then  hope to use their dynamics, to classify such $\R^k$-actions, and from this  obtain classification of $G$-actions. This is precisely what we do. At the heart of our approach lies  a classification of certain partially hyperbolic actions of higher rank abelian groups, and building additional invariant structures which leads to classification.

 Connecting actions of semisimple groups and their lattices with those of higher rank {\it abelian} subgroups goes all the way back to Hurder's proof of local rigidity of the action of $SL(n,\Z)$ on the $n$-torus $\mathbb T^n$ \cite{Hurder}.  It was used again by Katok, Lewis and Zimmer in various works \cite{KLZ, KL91, KL96}, then by Katok and Spatzier in their work on local rigidity \cite{KatSpat}. As we will discuss below, these ideas were further developed by Goetze and Spatzier in \cite{Goetze-Spatzier}.

%This is precisely what we do.  Indeed, {\color{black} we formulate in a completely topological setting, a classification of topological  actions of $\R^k $  which \color{green} \todo{Comment (1), fixed the spelling}intertwine \color{black} certain homogeneous structures on certain dynamical foliations, via bi-homogeneous actions.  
%First we use Zimmer's cocycle superrigidity theorem and automatic accessibility properties of the  action of the Cartan subgroup to find equivariant geometric structures on dynamical foliations.   In a second, independent part of the paper, we analyze  $\R ^k$-actions with a dense set of Anosov elements which trans

%Thereby at the heart of our approach lies  a classification of certain partially hyperbolic actions of higher rank abelian groups, and building additional invariant structures which leads to classification. 

\color{black}
  Higher rank abelian subgroups inside (bi-)homogeneous models often have two distinctive properties: if they contain one partially hyperbolic element, then almost \color{black} all  elements are partially hyperbolic; and if they contain an accessible partially hyperbolic element, then almost \color{black} all the elements are accessible. A nice illustration for this is the action by any subgroup of the diagonal group in $SL(n, \mathbb R)$, $n\ge 3$, on  $SL(n, \mathbb R)/\Gamma$, where $\Gamma$ is a co-compact lattice in $SL(n, \mathbb R)$. These two {\it purely dynamical} properties are the crucial ones that we will require 
  %(in fact in a weaker form), 
  in order to obtain global classification of general partially hyperbolic actions of higher rank abelian groups.  

  \color{black}
 
 Crucial to our approach is a deep understanding of actions of higher rank abelian  groups with a {\it dense set} of partially hyperbolic elements (these we call {\it totally partially hyperbolic} actions).  While a classification of such is outstanding in general, and likely extremely difficult, we manage to do this here in Theorem \ref{abelian} assuming \color{black} the accessibility property for many action elements (we label this property {\it super accessiblity}) \color{black} and existence of certain measurable leafwise structures and measurable solutions to coboundary equations. %strong
%\color{black}  Super accessibility property is a natural assumption for our purpose since it typically occurs for (bi-)homogeneous abelian actions with at least one accessible partially hyperbolic element. \color{black} 
Super accessibility gives us two crucial properties for general abelian totally partially hyperbolic actions on any manifold: the first one is that the actions are {\it genuinely} higher rank, meaning that they do not reduce to products for example, and the second one is improved regularity of the measurable leafwise invariant structures (here, the crucial mechanism we employ is the {\it invariance principle} from partially hyperbolic dynamics). Once in this set-up,  we apply a refinement of  the powerful techniques  \color{black} developed in the recent work by Spatzier and Vinhage in  \cite{{Spatzier-Vinhage}}. 
 \color{black} The work in  \cite{{Spatzier-Vinhage}} gives a classification/structure theorem  of Anosov $\mathbb R^k$-actions with a dense set of Anosov elements and a special property that maximal nontrivial intersections of stable foliations of distinct elements are {\it one-dimensional}. ($\mathbb R^k$-actions which have one-dimensional intersections of various stable foliations are called {\it Cartan} $\mathbb R^k$-actions). %and the set of Anosov elements in $\R^k$ is dense.  
 In that case, one can solve the relevant cohomology problem directly by using, in an essential way, the one-dimensionality of these intersections (cf. \cite{KSp,Spatzier-Vinhage}). 
  \color{black} In our case here, intersections of stable foliations for various elements may be higher dimensional which makes the construction of the homogeneous space considerably  more intricate. Namely, we construct a lift to a principal bundle extension of the given action, whereby we actually loose accessibility property for the lift, while we gain the existence of (global) continuous group actions intertwining the lifted action. The fact that we obtain these {\it continuous} group actions on the extended space demonstrates the far reaching applications of the invariance principle in partially hyperbolic dynamics, which in usual applications gives continuity of certain measurable structures in presence of   accessibility.  We then formulate and prove a purely {\it topological} result which gives topological classification via homogeneous models for a class of  $\mathbb R^k$-actions that come together with a collection of continuous group actions intertwining the $\mathbb R^k$-dynamics. The proof of this topological result demonstrates the power of the method of building a homogeneous structure on a manifold,  developed in \cite{Spatzier-Vinhage}. Theorem \ref{abelian} actually provides a framework for classification of higher rank actions of abelian groups which hopefully will prove useful in other problems, see Section \ref{sec:PH-KS}.   
  \color{black}
 
 In our application of Theorem \ref{abelian} to actions of  higher rank semisimple Lie groups and Theorem \ref{G-action}, \color{black} 
 we first derive super accessibility property for the abelian partially hyperbolic subgroup action. By a simple argument, volume preservation extends from such a sub-action to the whole $G-$action. After that
 \color{black} we  invoke Zimmer's superrigidity theorem for cocycles to get the measurable solutions to cohomological equations which give us measurable leafwise structures.

 %The totally Cartan assumption is needed as was shown  by Vinhage in his  construction of $\R^k$ Cartan actions which are not homogeneous and don't have a rank 1 factor. {\color{cyan}Kurt:  is this correct?}
 
% For a general Anosov $\R ^k$ action, even with a dense set of Anosov elements,  we do not yet know how to solve the relevant cocycle problem.  

\color{black} Let us now describe this connection with Zimmer's work and what we do here in more detail. 
\color{black} Ultimately it is based on Zimmer's deep insight that a classification of actions of higher rank semsimple Lie groups and their lattices may be possible, at least if they preserve geometric structures or have strong dynamical properties.  This overarching vision  was certainly based on Zimmer's superrigidity theorem for cocycles \cite{Zimmer-notes}.  As already mentioned,  we use it very fruitfully in our work here. Zimmer's result was measurable.  He himself already realized in the early 1990s that versions with higher regularity could prove important, and formulated and proved a topological superrigidity theorem to that effect, in unpublished notes \cite{Zimmer-notes}. Later, Feres and Labourie pursued similar ideas in \cite{Feres95}, and used them to prove various rigidity statements.  

In Zimmer's approach to topological superrigidity,  he assumed existence of  a H\"older section of a suitable bundle (with a bundle action by $G$) invariant under a parabolic subgroup of $G$.  This fits well with   Anosov dynamics  as contracting bundles will furnish such objects.  Goetze and Spatzier developed these ideas in \cite{Goetze-Spatzier-Duke} and used them to classify Cartan actions of  semisimple Lie groups of real rank at least 3 \cite{Goetze-Spatzier}. 
Under various technical assumptions, %in particular real rank at least 3, 
they used this to prove existence of \holder metrics along suitable foliations conformally invariant under some Cartan subgroup of $G$. Then they get homogeneous structures along these one-dimensional foliations, which allowed them to prove smoothness of foliations and  metrics. To be clear, this approach required  that the acting group has real rank at least 3 and the superrigidity representation from Zimmer's cocycle rigidity theorem is multiplicity-free, rather strong conditions indeed.  

We overcome all these restrictions and more in our current work.  While we use a radically different approach, we  incorporate  some of the prior ideas. 
In particular, finding homogeneous structures along suitable foliations is key, for us and  for a variety of other rigidity problems, such as proving measure rigidity and local rigidity of higher rank abelian actions \cite{KatSpat}.  
%The key is to find subgroups which act isometrically along certain foliations and are still transitive. This comes from the action of a Cartan subgroup that has strong transitivity properties. Using returns to a given leaf by suitable hyperplanes (Lyapunov hyperplanes) then provides a homogeneous structure on the leaves of the foliations. One then  

%However, these constructions are only measurable.  

In the setting of actions of higher rank semisimple Lie groups, we get these {\it measurable} leafwise homogeneous structures from Zimmer's cocycle superrigidity theorem.  Moreover, the action of a Cartan subgroup of $G$ gives an $\mathbb R^k$ action which due to accessibility assumption and the higher-rank assumption on $G$, is {\it super accessible} (which amounts to having many more accessible elements of the action).

\color{black}

In the final stage, we apply  Theorem \ref{abelian} that $\R^k$ partially hyperbolic super accessible actions which preserve %\holder
measurable leafwise homogenous structures are smoothly modelled by  bi-homogeneous actions. We still face the problem of combining different conjugacies for different Cartan subgroups to get a conjugacy for all of the  $G$-action.  To resolve it we use work by Zeghib on centralizers of homogeneous flows \cite{Zeghib}. Then the conjugacy will extend to the centralizer of any Weyl chamber wall, which is good enough to control all of $G$.

%This allows us to use one and the same conjugacy for two Cartan subgroups which intersect in a hyperplane.As we show, such are modelled on bi-homogeneous actions. We stitch these models together for    different Cartan subgroups of $G$, to arrive at the desired we can stitch 

So in the core of our global classification result for $G$-actions lies the proof of smooth classification (via bi-homogeneous models) for partially hyperbolic $\mathbb R^k$-actions, i.e. Theorem \ref{abelian}. We comment now some more on some of the key features of the proof of this result.

While one might hope for a model by global homogeneous structures, as we already noted, natural examples (Section \ref{examples}) show that  one can only  get  bi-homogeneous  models.  We introduce a new construction to resolve this problem, by building  a suitable principal bundle extension of the $\R^k$ action, naturally built from the leafwise invariant conformal structures using frames.  To our knowledge, this is the first time global rigidity with bi-homogeneous models was achieved in either the Zimmer program or the classification of actions of higher rank abelian groups. Even in rank one, such results are extremely rare and require significant additional structure. The few examples in rank one include entropy or exponent rigidity results for geodesic or contact flows (e.g. \cite{BFL} or \cite{Butler}). Our new methods work without this additional geometric data to produce a principal bundle extension on which the action is actually homogeneous.

\color{black}
One of the key ingredients is to build  {\it continuous} leafwise homogeneous structures on the principal bundle extension from measurable stuctures of the base action. The initial step in this direction is to show that the measurable structures, due to super accessibility, are in fact continuous and moreover H\"older along dynamical foliations. To this end we appeal to the \textit{invariance Principle} and the Livsic theory from partially hyperbolic dynamics. It gives extra regularity  
%and the theory of cohomological equations over partially hyperbolic systems that guarantees  extra regularity 
of measurable invariant structures under certain accessibility assumptions \cite{Ledrappier1986,AV, ASV,KalSad,W}. The invariance principle had been applied before in the context of Anosov $\Z ^k$-group actions by Kalinin and Sadovskaya \cite{KalSad07} as well as Damjanovi\'c and Xu \cite{DX1}.  

Continuity of the structures allows us to use the second important consequence of super accessibility, which is that the action is {\it genuinely} higher rank. This, similarly as in \cite{Spatzier-Vinhage}, gives the existence of transitive Lie group actions on leaves of dynamical foliations. As explained before, and as the examples show, these actions need not give { global} actions. In order to construct global actions we lift this to a suitable principal fiber bundle over the given manifold. For the lifted action and for each   lifted foliation we get a Lie group which acts transitively on the leaves, and does so in a manner which preserves dynamical information within the leaves (which is  important  for the next step in the proof). The lifted action however is not accessible. So we actually need to adapt some of the key arguments in \cite{ASV} to the setting of principal bundle extensions of accessible group actions in order to obtain continuity properties of the transitive actions that intertwine the lifted dynamics.

%Our second main theorem, Theorem \ref{abelian},  then classifies \color{black} accessible \color{black} partially hyperbolic actions of $\R^k$ which preserve the leafwise conformal structures. 
%Such turn out to be homogeneous.  
Once we have continuous homogeneous leafwise structures in place (on the principal fiber bundle) we combine them in a global homogeneous structure. 
\color{black}
This is  inspired by Spatzier and Vinhage's classification of Cartan actions \cite{Spatzier-Vinhage} though is considerably more complicated here.  The main idea is to take the isometry groups of the leafwise conformal structures to  build a transitive  action   of a free product of the Lie groups. Then we show that this free product action actually factors through an actual Lie group, yielding our desired  global homogeneous structure. 

\color{black}
Lastly we remark that in parallel to Theorem \ref{abelian} we also obtain a global classification result for Anosov $\mathbb R^k$-actions in Theorem \ref{basic abelian}. This is not a consequence of Theorem \ref{abelian} because Anosov $\mathbb R^k$-actions in general need not be accessible. Dropping (super) accessibility has a cost in that we assume {\it continuity} of measurable invariant structures. However, the intersections of various stable manifolds may still be multidimensional, which makes this currently the most general global rigidity result for Anosov $\mathbb R^k$-actions on general smooth manifolds. 
\color{black}

%\medskip

%{\color{black} shouldthis medium skip be here?  Is there something to add or delete?}

\noindent {\bf Acknowledgements.} The authors would first and foremost like to thank \color{black}
the anonymous referee.  Thanks to the questions raised by the referee we vastly improved the initial version of this paper and extended results from the Anosov to the partially hyperbolic setting. We owe huge thanks to Clark Butler who
\color{black} contributed significantly to the early discussions on the topics of this paper. His ideas promise an approach which may allow the method to be applied more broadly in the future.

The authors would also like to thank Aaron Brown, David Fisher, Livio Flaminio, Boris Kalinin and Homin Lee for useful comments and discussions.    We also thank Ekaterina Shchetka for pointing out an error in \cite{Spatzier-Vinhage} which led to corrections in the initial version of this paper. \color{black} The idea for this paper originated from discussions and interactions at the 2018 Semi-annual Workshop in Dynamical Systems and Related Topics hosted at Penn State University.  

%\color{teal}More names, Homin Lee, Aaron Brown?\color{black}

%The key idea then is to define the action of a free product of the natural nilpotent groups 

%Another key issue is how to find the bi-quotient on which the bi-homogeneous action lives.  To this end, we consider suitable isometric extensions, essentially comng from sutiable framings, refine them and classify them.  

%This adapted an idea from Katok and Spatzier  \cite{MR1858547} where local rigidity of certain homogeneous  actions is proved. 
%

\section{Results}

\label{sec:results} 

\subsection{Totally partially hyperbolic actions\color{black}}
In this section we introduce the basic notions needed to state the main results. \color{black}Let $X$ denote a smooth compact connected manifold.  \color{black}
Let $a: X\to X$ be a diffeomorphism of a smooth compact manifold $X$. 
If \color{black}$E^c\subset TX$ is a continuous distribution invariant under the action of $a$, we say that $a$ is \emph{partially hyperbolic} with respect to \color{black}$E^c$ (or $E^c$-partially hyperbolic) %(\color{black}F a foliation?  Do we want to assume the center is integrable? Probably no need ) 
if there exists a continuous $a$-invariant splitting $TX= E_a^s\bigoplus E^c\bigoplus E_a^u$, where $E_a^s$  (resp. $E_a^u$) are non-trivial and uniformly contracted (resp. uniformly expanded) by the action of $a$.  Furthermore, we assume that 
the contractions on $E^s_a$ and $E^u_a$ dominate the dynamics on $E^c$ uniformly (for a precise definition see Section \ref{sec:PH-prelims}). 

%More precisely, one can choose a Riemannian metric on $X$ and continuous positive functions $\nu< 1,\hat \nu < 1, \gamma, \hat\gamma$ on $X$ such that for any $x\in X$ and unit vectors $v_s\in E^s(x), v_c\in  E^c(x), v_u\in E^u(x)$, \begin{equation}\label{def: ph dif def}
    % \|Da(v_s)\| < \nu(x) < \gamma(x) < \|Da(v_c)\| < \hat\gamma(x)^{-1} < \hat\nu(x)^{-1} < \|Da(v_u)\|
%\end{equation}

 The {\it stable} distribution $E^s_a$ (resp. the {\it unstable} distribution $E^u_a$) integrates to stable foliation $W^s_a$ (resp. unstable foliation $W^u_a$). The diffeomorphism $a$ is called {\it accessible} (with respect to $E^c$) if any two points in $X$ can be connected by a broken path whose legs lie in leaves of foliations $W^s_a, W^u_a$ (see Section \ref{sec:PH-prelims}).

%\color{black}\st{An element $a \in A$ will be called  partially hyperbolic if it acts partially hyperbolically   with respect to a fixed distribution $E^c$. }
Group actions  containing at least one partially hyperbolic  element are called {\it partially hyperbolic}. In what follows we will require more partially hyperbolic elements. We will work here with two cases: when the acting group is semisimple and when the acting group is $\mathbb R^k$. 
\color{black}
\begin{definition}\label{totally}
A partially hyperbolic $\mathbb R^k$-action is {\it totally partially hyperbolic} if there is a distribution $E^c$ and a dense set of elements in $\mathbb R^k$ which are all partially hyperbolic with respect to the same $E^c$. In particular, if $E^c$ is just the orbit direction of a totally partially hyperbolic $\R^k$ action, then we say the action is totally Anosov. \footnote{We remark that not every $\R^k$-action with a dense set of individually partially hyperbolic elements is partially hyperbolic in this sense, as the elements may not be partially hyperbolic with respect to the same $E^c$.} 
\end{definition}
%\color{olive}

\begin{definition}
\label{def: totallyPH} 
    {\color{black}Recall that if $G$ is a real semisimple Lie group, 
    then its Lie algebra  $\frak g$ admits a Cartan decomposition
$\frak g = \mf k\oplus  \frak p$, where $\frak k$ is a maximal ad-compact subalgebra of $\frak g$ and $\frak p$ is its orthogonal complement as determined by the %Killing form [?].
%$\mf g = \mf k \oplus \mf p$, where $\mf k$ is a maximal $\ad$-compact subalgebra of $\mf g$ and $\mf p$ is orthogonal complement as determined by the 
Killing form \cite{Hel}. The {\it real rank} of
$G$ is the maximal dimension of an abelian subalgebra
$\mf a \subset \mathfrak p$. Such subalgebras are called $\R$-split Cartan subalgebras, and their corresponding subgroups $A = \exp(\mf a)$ are called {\it $\R$-split Cartan subgroups}. % in the Cartan decomposition of its Lie algebra 
%$\mathfrak g: \mathfrak g=\mathfrak{k}\oplus \mathfrak{p}$ where $\mathfrak k$ is the Lie algebra of a maximal compact subgroup 
%$K\subset G$, and $\mathfrak p$ is the orthogonal complement of 
%$\mathfrak k$ with respect to the Killing form.
Equivalently, the real rank and $\R$-split Cartan subalgebra can be defined using algebras of maximal dimension among those  which are $\ad$-diagonalizable over $\R$.  }

If $\G$ is a real semisimple Lie group, and $A \subset \G$ an $\mathbb R$-split Cartan subgroup, an action $G \curvearrowright X$ is called {\it totally partially hyperbolic} if the $A$-action is totally partially hyperbolic (with respect to some $E^c$). Moreover, if there is a partially hyperbolic (with respect to the same $E^c$) element $a\in A$ which is accessible, then we say $A \curvearrowright X$ and $G \curvearrowright X$ are {\it accessible actions}. 
\color{black}
\end{definition}

An important example of (totally) partially hyperbolic $G$ action is the (totally) Anosov $G$-action defined as follows. 

\begin{definition}\label{def: tot Ans}Let $A\subset G$ be an $\R$-split Cartan subgroup, $C_{\G}(A)$ be the centralizer of $A$ in $\G$. The centralizer always decomposes as a direct product of a compact group $K$ with $A$, $C_\G(A) = K \cdot A$. Then a (totally) Anosov $\G$ action is a (totally) partially hyperbolic $\G$ action for which  $E^c$ is the distribution tangent to the $C_\G(A)$-orbit.
\end{definition}
%\begin{definition}\label{def: tot Ans} \color{black}ADD DEFINITION OF ANOSOV ACTION (WITH SINGLE ANOSOV ELEMENT)\color{black}
 %If $\G$ is a real semisimple Lie group, and $A \subset \G$ a split Cartan subgroup, a $C^\infty$ action $\G \curvearrowright X$ is called totally Anosov (with respect to $A$) if for a dense set of $a\in A$, $a$ is \st{ normally hyperbolic with respect to the homogeneous foliation given by  $C_{\G}(A)$, i.e. } partially hyperbolic with respect to the distribution tangent to the $C_\G(A)$-orbit.
%\end{definition}

\color{black} Note that for a (totally) Anosov action $G\curvearrowright X$, $G$ semisimple, the action's restriction to $A\curvearrowright X$  where  $A\cong \mathbb R^k$ is a maximal $\mathbb R$-split Cartan subgroup of $G$, is {\it not} a (totally) Anosov $\mathbb R^k$ action according to the definitions above when $K$ is non-trivial. Rather the $A$ action is an $\R^k$-totally partially hyperbolic action with center distribution tangent to the $K\cdot A$ -orbits. 
\color{black}

\subsection{Bi-homogeneous actions}

Let $G$ be an arbitrary Lie group. There is a large class of $\G$-actions, built from certain algebraic data, which we call {\it algebraic}. This class will be the base for our main classification result. Generally, algebraic actions may be defined via group automorphisms and group multiplication. Actions defined via group multiplications are {\it homogeneous actions}. \color{black}Here we define {\it bi-homogeneous } actions as follows: Let $H$ be a Lie group, and $q : G \to H$ be an embedding of $G$ into $H$. Suppose that $K \subset H$ is a compact subgroup commuting with $q(G)$ and $\Gamma 
\subset H$ is a cocompact lattice. Then the bi-homogeneous $\G$-action $(q,H,K,\Gamma)$ is the action on $X = K \backslash H /\Gamma$ defined by:

\[ g \cdot (Kh\Gamma) := K(q(g)h)\Gamma . \]

The fact that $K$ commutes with $q(G)$ ensures the action is well-defined, and the fact that $K$ is compact will ensure that any right-invariant metric on $H$ which is bi-invariant under $K$ is well-defined on the quotient.  Moreover,  Haar measure on $H$ will project to a well-defined $G$-invariant volume on $X$.

%\color{black}
%Suppose that every simple factor of a real semisimple group $\G$ has real rank at least 2, and let  $\G \curvearrowright X$ be a $C^\infty$ \color{teal} totally partially hyperbolic accessible action\color{black}.  Assume further that the restriction of the action to $A$ preserves an invariant volume. Then

\color{black} See Section \ref{examples} for a variety of examples featuring Anosov and partially hyperbolic algebraic actions.

%\color{teal}Similarly we could define bi-homogeneous $\R^k$ action by viewing $\R^k$ as a Lie group $G$. \color{black}

Now we define actions which are {\it smoothly} modeled on bi-homogeneous actions.
\begin{definition}\label{bi hom models} Let $G$ be a group. 
 We say an action $G \curvearrowright X$ is {\it finitely covered by a bi-homogeneous action} if there exists a finite cover of $X$ and a lift of the $G$ action to the finite cover, which is { $C^\infty-$conjugate} to a bi-homogeneous $G$-action.   
\end{definition}

 \color{black}

\subsection{Rigidity for semisimple group actions} \label{sec: G-action}

 %\color{olive} 

%{We make the same definitions for  real reductive Lie groups, and in particular for $\R^k$ actions.}
%\end{definition}

%If $a$ is normally hyperbolic with respect to the orbit foliation of the action we call it \emph{Anosov}.

\color{black} We have the following classification theorem. 
\begin{theorem}\label{G-action}
Suppose that every simple factor of a real semisimple group $G$ has real rank at least 2, and let  $G \curvearrowright X$ be a $C^\infty$ \color{black} totally partially hyperbolic accessible action\color{black}.  Assume further that the restriction of the action to a maximal $\mathbb R-$split Cartan subgroup $A$ preserves an invariant volume. Then the $G$-action
%\color{teal} there exists a finite cover of $X$ and a lift of the  $G$ action which is  $C^\infty$ 
%\color{black} conjugate to 
is \color{black} finitely covered by 
a bi-homogeneous $G$-action.
\end{theorem}

\color{black}

In particular, a totally Anosov  $G$-action is accessible so we have 

\begin{corollary}\label{G-Anosov}
    Suppose that every simple factor of a real semisimple group $\G$ has real rank at least 2, and let  $G \curvearrowright X$ be a $C^\infty$  totally Anosov $G$-action\color{black}. 
    Assume further that the restriction of the action to a maximal $\mathbb R-$split Cartan subgroup $A$ preserves an invariant volume. Then the action is finitely covered by 
    %smoothly conjugate to 
    a bi-homogeneous $G$-action.
    %\color{teal} there exists a finite cover of $X$ and a lift of the $\G$ \color{black} action which is smoothly conjugate to a bi-homogeneous $\G$-action.
\end{corollary}

\begin{remark}
These actions are automatically bi-homogeneous under natural assumptions on orientability of suitable distributions.  Thus we do not have to pass to finite covers  in such cases.  Same remark applies to all our results which involve passing to a finite cover. 
\end{remark}

\color{black}
\begin{remark}\label{one element volume}
The assumption that the action of $A$ preserves an invariant volume can be weakened to assuming that there is $a\in A$ which acts as a topologically transitive and volume preserving diffeomorphism. Then any diffeomorphism commuting with $a$  also preserves the volume \cite[Lemma 11]{DWX}, so the whole action of $A$ preserves the volume. Moreover, the {\it totally} partially hyperbolic accessible assumption for the action in Theorem \ref{G-action} implies that the accessible partially hyperbolic diffeomorphism $a\in A$ (assumed to exist by definition of accessible $G$-action)  is {\it center bunched} (see Section \ref{sec:PH-prelims}, Lemma \ref{Ec center bunching}). Then by \cite[Theorem 0.1]{BW}, if $a$ is additionally volume preserving, then   it is also ergodic with respect to volume, thus $a$ is topologicaly transitive.    
\end{remark}

\color{black}

\color{black}

%{\color{black}\color{black} Cor 2.6 works for partially hyperbolic case? \color{black}
The totally Anosov condition appearing in \color{black} Corollary \ref{G-Anosov} \color{black} relies on a distinguished abelian subgroup $A$ in which to find hyperbolic elements. We also get a formulation which is independent of such a subgroup. To do so, we make two important definitions:

 The first is that of a hyperbolic element of an action $G \curvearrowright X$. Recall that if $F : V \to V$ is a linear transformation, $V$ splits as a sum of {\it generalized eigenspaces}. Each such space corresponds to the sum of the blocks in the Jordan normal form of $F$ for a fixed eigenvalue. If $G$ is a Lie group with Lie algebra $\mf g$, and $g \in G$, consider the splitting $\mf g = \mf g^+_g \oplus \mf g^-_g \oplus \mf g^0_g$, where $\mf g^+_g$ denotes the {\color{black} the sum of the} generalized eigenspaces of $\Ad(g)$ whose eigenvalues have modulus greater than 1, $\mf g^-_g$ is the sum of the generalized eigenspaces whose eigenvalues modulus less than 1, and $\mf g^0_g$ is the {\color{black} the sum of the} generalized eigenspace for eigenvalues of modulus 1. Note that $\mf g^\pm_g$ and $\mf g^0_g$ are subalgebras and have corresponding connected Lie subgroups. 
 
 \begin{definition}\label{def: hyp elem}
 Let $G \curvearrowright X$ be a locally free $C^r$ group action, $r \ge 1$. We say that $g \in G$ is {\it hyperbolic for $G \curvearrowright X$} if there is a splitting $TX = E^s_g \oplus E^0_g \oplus E^u_g$, where $E^0_g = \mf g^0_g$ and $E^s_g, E^u_g$ are subbundles of $TX$ which contract uniformly under forward and backwards iterates of $g$, respectively.
 \end{definition}
 
 Notice that $g$ always exponentially contracts $\mf g^-_g$ and $\mf g^+_g$ under forward and backward iterates, respectively. That $g$ is hyperbolic asks that these bundles can be extended to bundles $TX$ with the same property. Let $\mathcal{H}$ be the set of hyperbolic elements of $G \curvearrowright X$. %That is, the set of regular elements of $G$ which act normally hyperbolically with respect to the orbit foliation of their centralizer in $G$. 
Note that $\mathcal{H}$ is invariant under conjugation in $G$.

 The other definition required is the Jordan-Chevalley projection. If $G$ is a semisimple Lie group, and $g \in G$, there exists a decomposition of $g$ called the {\it Jordan-Chevalley decomposition} as $g = kan$, where  $k$ is $\Ad$-compact, $a$ is $\R$-semisimple, $n$ is $\Ad$-unipotent, and $k$, $a$ and $n$ pairwise commute. \color{black} Given a split Cartan subgroup $A \subset \G$, the {\it Jordan projection} of $g$ is defined by $\mc J(g) = a' \pmod W$, where $a' \in A$ is conjugate to $a$ and $\mod W$ is taken to mean modulo the action of the Weyl group (here and in the rest of the paper by {\it the Weyl group} we will always mean the \emph{restricted} Weyl group, i.e.  the group generated by the reflections corresponding to restricted roots). 
 \color{black} The Jordan\color{black}-Chevalley \color{black} projection takes values in a Weyl chamber of $\G$, which can be chosen with respect to any $\R$-split Cartan subgroup. For a thorough treatment of this topic, see \cite[Section 4.2]{humphreys}.

\begin{corollary} \label{cor:semisimple}
Suppose that every simple factor of a real semisimple group $G$ has real rank at least 2, and let  $G \curvearrowright X$ be a $C^\infty$ \color{black} volume preserving \color{black} action.  Let $\mathcal{H}$ be the set of hyperbolic elements for $G \curvearrowright X$.  Assume either 
\begin{itemize}
\item that  $\mc J(\mathcal{H})$ has dense image, or%intersects some positive Weyl chamber (of the group $G$) in a dense set, or
\item that  $\mathcal{H}$ intersects the set of $\R$-semisimple elements in a dense set. %{\color {cyan}open} set.
\end{itemize}
Then the $\G$ action is finitely covered by
%\color{black}there exists a finite cover of $X$ and a lift of the $\G$ \color{black} action which is smoothly conjugate to 
a bi-homogeneous $\G$-action.
\end{corollary}

\begin{remark}
In fact, each of the conditions of Corollary \ref{cor:semisimple} are equivalent to the totally  Anosov assumption of  Corollary \ref{G-Anosov}  (see Section \ref{sec:cor-proof}).
\end{remark}

{\color{black}
We state one more theorem, which is our most general for a semisimple Lie group action.

\begin{definition}
\label{def:gen-HR}
    If $G$ is a semisimple Lie group, $A$ is an $\R$-split Cartan subgroup, and $B \subset A$ is a connected Lie subgroup of $A$, we say that $B$ is {\it genuinely higher rank} if whenever $\pi_i : G \to G_i$ is a projection onto a simple factor of $G$, $\rank_{\R}(\pi_i(B)) \ge 2$. Call an action $G \curvearrowright X$ {\it $B$-totally partially hyperbolic} if the restriction of the action to $B$ is totally partially hyperbolic (similarly for $B$-accessible, $B$-volume preserving).
\end{definition}

\begin{theorem}
\label{thm:genuinely-HR}
    Let $G$ be a higher-rank semisimple Lie group, $G \curvearrowright X$ be a $C^\infty$ action, and $a \in A$ be a volume-preserving, partially hyperbolic, accessible element with central distribution $E^c$. Let $\mc{PH}$ denote the set of elements of $A$ which are partially hyperbolic with respect to $E^c$. If $\mc{PH} \cap B$ is dense for some genuinely higher-rank subgroup $B$ of $A$, then $G\curvearrowright X$ is finitely covered by a bi-homogeneous action.
\end{theorem}

%\begin{theorem}
%    If $G$ is a semisimple Lie group, $B \subset G$ is genuinely higher-rank, and $G \curvearrowright X$ is a $B$-totally partially hyperbolic, $B$-accessible, $B$-volume preserving action, then the $G$-action is finitely covered by a bi-homogeneous $G$-action.
%\end{theorem}}

%\begin{theorem}
 %   Let $G$ be a semisimple Lie group, $G \curvearrowright X$ be a $C^\infty$ action, and $a \in A$ be an  volume-preserving, partially hyperbolic, accessible element with central distribution $E^c$. Let $\mc{PH}$ denote the elements of $A$ which are partially hyperbolic with respect to $E^c$, and assume that $\mc{PH} \cap B$ is dense inside of a higher-rank subgroup $B\subset A$. Furthermore, assume that for every root $\beta$ of $G$ with respect to $A$, there exists $b \in B$ such that $\beta(b) = 0$ and $b$ acts partially hyperbolically and accessibly with respect to a distribution $E^c + E_\beta + E_{-\beta}$. 
    
    %Then, the action $G \curvearrowright X$ is finitely covered by a bi-homogeneous action.
%\end{theorem}

\color{black}
%For $\Gamma<G$, $G$ as in the previous section, we take a $\Gamma$ action on $M$ such that the maximal free abelian subgroup $\mathbb Z^k$ has a set of Anosov elements  $A$, which projects to a dense set on the sphere $S^k$. Then the previous theorem applies to the suspension action. 

\color{black}

\subsection{Rigidity for  \ensuremath{\mathbb R^k}-actions}

Given a volume preserving $\mathbb R^k$-action, the Oseledets theorem for actions \cite[Theorem 2.4]{BRHW} implies the existence of finitely many linear functionals $\chi\colon \mathbb R^k\to \mathbb R$, called \emph{Lyapunov functionals}, and an action invariant measurable splitting \color{black}$TX=\oplus_\chi E^\chi$\color{black} \,  of the tangent bundle on a full volume set,  such that the Lyapunov exponent of $a$ in the direction of $v\in E^\chi(x)$ is $\chi(a)$. 
\color{black}All the Lyapunov functionals positively proportional to non-zero $\chi$ constitute a class which is commonly called a \emph{coarse Lyapunov functional}. \color{black}
The sum of Oseledets distributions \color{black} $E_\chi(x):=\bigoplus_{c_i>0}E^{c_i\chi}(x)$ \color{black} corresponding to %\color{black} non-zero 
\color{black} positively proportional {\it non-zero} Lyapunov functionals is called a \emph{ coarse Lyapunov distribution}.  Denote by $\Delta$ the set of finitely many coarse Lyapunov functionals 
%\todo{Comment (2) removed ?}
with respect to the volume. \color{black} To avoid confusion between Oseledets spaces and the coarse Lyapunov spaces we will usually use index $\lambda$ for the coarse Lyapunov spaces and denote them by $E_\lambda$, $\lambda\in \Delta$. For an $\mathbb R^k$ {\it totally} $E^c$-partially hyperbolic action the Lyapunov functionals in the $E^c$ direction are 0 (see Section \ref{PHactions}), so we have the splitting $TX= \bigoplus_{\lambda\in \Delta}E_{\lambda}\oplus E^c$ .\color{black}

%Two nonzero Lyapunov functionals $\chi_i$ and $\chi_j$ are \emph{coarsely equivalent} if they are positively proportional: there exists $c>0$ such that $\chi_i=c\cdot \chi_j$. This is an equivalence relation on the set of Lyapunov functionals, and a \emph{coarse Lyapunov functional} is an equivalence class under this relation.

\color{black}
Unlike $G$-actions, $\mathbb R^k$-actions can be quite non-rigid (an example is a product of two Anosov flows). So we surely need an additional assumption in order to characterize the rigid actions. Actions satisfying this assumption belong to the class of  {\it genuinely higher rank abelian actions}. The key genuinely higher rank assumption for an $\mathbb R^k$ totally partially hyperbolic action in this paper will be the following:   %(\color{black} Do we really need them for partially hyperbolic action or just Anosov action?)
%They follow from METHOD 1 or METHOD 2 and will be the only assumptions used in subsequent sections.

\begin{itemize}
%[label=\textnormal{(\arabic*)}]
\item [\mylabel{FA1}{\rm (GHR)}] 
 %\color{olive} $\R^k$ has a dense orbit AND???} 
\color{black} For every $\lambda \in \Delta$, $\ker \lambda$ %\times M
  has a dense orbit.
\end{itemize}

\color{black}
A connection between rank-one factors of an abelian action and condition \ref{FA1}, justifying the "genuinely higher rank" terminology we use here, has been established for a class of Cartan $\mathbb R^k$-actions in \cite[Section 2.2, Theorem 2.1]{Spatzier-Vinhage}, also see Section \ref{sec:questions}.

We introduce now additional assumptions we will need for the global rigidity results of  abelian actions.

%the Oseledets splitting is continuous, and for 
\begin{definition}
\label{def:hoc}
For an $\mathbb R^k$ partially hyperbolic action we say that the action is {\it (measurably) Oseledets conformal} if  there is an $\R^k$-invariant \emph{measurable} conformal structure on each Oseledets space.
By this we mean that there  exists a \emph{measurable} family of metrics $\norm{\cdot}_\chi$ on the corresponding Oseledets spaces $E^\chi$ such that $\norm{a_*v}_\chi = e^{\chi(a)}\norm{v}_\chi$. If in addition the Oseledets spaces and the metrics on them are continuous then we say that the action is {\it continuously Oseledets conformal}.
\end{definition}

\begin{remark}\label{Rem: 0 exp expl}
    We include the Oseledets space corresponding the Lyapunov exponent $\chi = 0$ in Definition \ref{def:hoc}. In particular, we require a measurable (resp. continuous) metric for which the dynamics is always an isometry along $E^c$. 
\end{remark}

%each Oseledets space $E^{\chi}$, there exists a measurable (resp. continuous) metric $\langle \, \cdot\, ,  \,\cdot \, \rangle_{\chi}$ on the bundle $E^{\chi}$ such that for every \color{black} $a \in \R^k$ \color{black} and $v \in E^{\chi}$, \[ \norm{a_*v}_{\chi} = e^{\chi(a)}\norm{v}_{\chi}. \]

\color{black}

\subsubsection{Rigidity for Anosov \ensuremath{\mathbb R^k} actions}
% \st{Let $M$ be a compact Lie group and Consider a locally free, volume preserving $C^\infty$ action of  $C=A {\times M}$ $A=\mathbb R^k$  on a smooth compact manifold $X$. where $A=\mathbb R^k$.
%If $\mathcal F$ is a foliation invariant under action of an element $a$, we say that $a$ is \emph{normally hyperbolic} with respect $\mathcal F$, if there exists an invariant splitting $TX= E_a^s\bigoplus T\mathcal {F}\bigoplus E_a^u$, where $E_a^s$  (resp. $E_a^u$) are uniformly contracted (resp. expanded) by the action of $a$.
%If $a\in A$ is   normally partially   hyperbolic with respect to  the distribution  $E^c=T\R^k$ tangent to the $A$-orbit foliation of the action we say that $a$ is Anosov. An  $\R^k$  action with a dense set of Anosov elements in $A$ is called totally Anosov.}
\color{black}
%Any action obtained by taking compositions of translations and automorphisms of a homogeneous space $G/ \Gamma$ where $G$ is a Lie group and $\Gamma$ a uniform lattice in $G$ is an \emph{algebraic action}. For examples see Section \ref{examples}.

%\emph{Generalities on $\mathbb R^k$ Anosov actions}.

\color{black} For $\mathbb R^k$-totally Anosov actions (see Definition \ref{totally}) we have the following classification result, via the bi-homogeneous models (see Definition \ref{bi hom models}).\color{black}

\begin{theorem}\label{basic abelian}
If  $\R^k% \times M
\curvearrowright X$\color{black} is a volume preserving totally Anosov $C^\infty$-action on a $C^\infty$-manifold $X$ satisfying  \ref{FA1} \color{black} and which is continuously Oseledets conformal \footnote{For $\mathbb R^k$-totally Anosov actions the condition for Oseledets splitting and metrics is needed for non-zero $\chi$, and for $\chi=0$ it holds trivially.}, %\ref{FP2}
then the action is finitely covered by a bi-homogeneous action.
%on some $K \backslash \HH/\Gamma$.

\end{theorem}
%\color{black}there exists a finite cover of $X$ %a double-homogeneous space $K \backslash \HH/\Gamma$ and such that 
%and a lift of the $\R^k$-action \color{black} 
%\todo{Comment (3) No change} 
% such that \color{black}the action\color{black} \, 
%which is $C^\infty$ conjugate to a bi-homogeneous action on $K \backslash \HH/\Gamma$.  
%\end{theorem}
%\color{black}SHALL WE SAY A FINITE COVER OF THE ACTION IS $C^\infty$ CONJUGATE TO A BI-HOMOGENEOUS ACTION?\color{black}

\color{black}
We remark that Theorem \ref{basic abelian} has also a low regularity version for {$C^{2}$}, see Theorem \ref{regularity of conjugacies}. 
%In fact the assumptions in the theorem imply certain structure along the leaves of coarse Lyapunov foliations which lead to existence of topological conjugacy even for only continuous actions with such leafwise structures, see Section \ref{sec:top-anosov}.

%Distributions $E_\chi$ are not necessarily integrable, but the sum of those corresponding to positively proportional Lyapunov functionals is integrable in the context of actions we consider here. We summarise these facts below and we refer the reader to Section \ref{abelian-prel} for precise statements and references. 

\subsubsection{Rigidity for \color{black}partially hyperbolic \color{black} $\mathbb R^k$-actions.}

Next, we state a theorem for \color{black} volume preserving totally partially hyperbolic \color{black} $\mathbb R^k$-actions. Here we also need some assumption that will make sure to eliminate non-rigid  examples (e.g. products). However, we will not assume \ref{FA1}, instead we assume certain {\it accessibility property}, namely that for every $\lambda\in \Delta$, $\ker\lambda$ contains an accessible partially hyperbolic diffeomorphism. Such accessibility holds for our main models,  left translation actions by  Cartan subgroups of semisimple higher-rank Lie groups, which motivates the definition of super accessibility. 

 %(\color{teal}The definition does not use $a$, could we just say the manifold or the action is $\Delta(\hat{\lambda})$-accessible? And then for generic singular $a$, $a$ is accessible in the usual sense.\color{black})

\begin{definition}
\label{def:strongly}
    If \color{black}$\R^k  \curvearrowright { X}$ is totally partially hyperbolic  \color{black} action, we say that it is {\it super accessible} if for every $\lambda \in \Delta$, there exists $a \in \ker \lambda$ such that $a$ is \color{black} accessible.\color{black}
\end{definition}

\begin{theorem}\label{abelian}
Let \color{black} $\R^k \curvearrowright { X}$ \color{black}  be a volume preserving, (measurably) Oseledets conformal, super accessible totally \color{black} partially hyperbolic \color{black} $C^\infty$ action.  
Then 
\color{black} then the action is finitely covered by a bi-homogeneous action.
%there exists a finite cover of $X$ %a double-homogeneous space $K \backslash \HH/\Gamma$ and such that 
%and a lift of the $\R^k$-action \color{black} 
%there exists a double-homogeneous space $K \backslash \HH/\Gamma$ such that the action  
%is $C^\infty$ conjugate to a bi-homogeneous action. %on some finite cover of $K \backslash \HH/\Gamma$.  
\end{theorem}

\color{black}
\begin{remark}\label{rema: FA1, 2 PH}
    In Section \ref{strong accessibility implies FA12} we will show that for partially hyperbolic actions as in Theorem \ref{abelian} the strong accessibility assumption in fact implies the key higher rank property \ref{FA1}. 
    So Theorem \ref {basic abelian} for super accessible Anosov actions is a consequence of Theorem \ref{abelian}. Otherwise, for general Anosov actions Theorem \ref {basic abelian} does not follow from Theorem \ref{abelian}. 
    %The proof of Theorem \ref{abelian}, just like the proof of Theorem \ref {basic abelian}, relies on the constructions in Section \ref{sec:top-anosov} which fundamentally relies on properties \ref{FP1} and \ref{FP2}.
\end{remark}

\color{black}
\begin{remark} The assumption that the actions are {\it totally} partially hyperbolic, is natural. In the context of bi-homogeneous actions, there is no difference between the notions of partially hyperbolic, and {\it totally} partially hyperbolic actions.
%The "totally partially hyperbolic" assumption is not overly restrictive in the context of homogeneous actions. 
%Indeed, for any bi-homogeneous action of a semisimple Lie group $G$ with all simple factors of real rank at least 2, the action is automatically totally partially hyperbolic. More generally, for homogeneous actions, the existence of a single partially hyperbolic element is sufficient to guarantee that the action is totally partially hyperbolic. This suggests that such dynamics are quite common in higher-rank settings.
In contrast, for non-homogeneous $\mathbb R^k$-actions, the situation can differ. For example, in the case of $\R^k$-actions, the third author constructs in \cite{VinhageExample} an Anosov action that fails to be totally Anosov. However, for actions of semisimple Lie groups 
$G$, it remains unclear whether a similar phenomenon occurs. We refer the reader to Section \ref{subsub:single-anosov} for further discussion.

%\color{black} REMOVE?? We refer the reader to Section \ref{subsub:single-anosov} for a discussion of ongoing work aimed at removing the "totally" assumption for $G$-Anosov actions.\color{black}
%Among homogeneous actions any partially hyperbolic action (positive entropy? or any homogeneous actions of higher rank groups is totally partially hyperbolic) is totally partially hyperbolic. In fact for homogeneous actions it suffices to assume existence of one partially hyperbolic element in order to have a totally partially hyperbolic action.

%In the non-homogeneous setting, an example of Anosov but not totally Anosov action is constructed in \cite{VinhageExample}.
\end{remark}

\begin{remark}
We note that, even though in the beginning of this section we assumed $X$ is connected, in fact the assumptions in all our results (both for $G$-actions and for $\mathbb R^k$-actions) imply directly  connectedness of the manifold $X$.
  
\end{remark}

\color{black}

\subsection{Questions and conjectures\color{black}}
%\label{conjectures}
\label{sec:questions}
\color{black}
There are several questions that are natural next steps after the results of this paper. 

\subsubsection{Rank one factors for $\mathbb{R}^k$-actions, Oseledets conformality, and a partially hyperbolic Katok-Spatzier conjecture}
\label{sec:PH-KS}

In this paper, topological transitivity of hyperplane actions is called the genuinely higher rank assumption \ref{FA1}. Indeed, in many settings, an action has a rank one factor if and only if there exists a hyperplane action which is not transitive. It was shown for a special class of totally Anosov actions, the so-called totally Cartan actions, in \cite[Theorem 2.1]{Spatzier-Vinhage}, and is easy to verify for homogeneous actions. We therefore formulate the following

\begin{conjecture}
\label{q:rank1-cond}
    Let $\R^k \curvearrowright X$ be an (essentially) accessible, partially hyperbolic action. If $H\subset \R^k$ is a hyperplane such that $H$ does not have a dense orbit, there exists a ($C^\infty$, $C^r$ or $C^0$) flow $\psi_t : Y \to Y$, a homomorphism $\sigma : \R^k \to \R$ such that $\ker \sigma = H$ and a submersion $\pi : X \to Y$ such that $\pi(a \cdot x) = \psi_{\sigma(a)}(x)$ for all $x \in X$.
\end{conjecture}

Partial results can be made in this direction by considering only Anosov actions, or actions with a dense set of regular elements. A positive answer to Conjecture \ref{q:rank1-cond} could lead to a version of Theorem \ref{basic abelian} for $\Z^k$-actions, replacing \ref{FA1} with a ``no rank one factor'' condition (note that there is no clean adaptation of \ref{FA1} to the setting of $\Z^k$-actions).  A related problem, which was shown for totally Anosov actions in \cite{KSp}, is

\begin{conjecture}
\label{conj:PH-cartan}
    Let $\R^k \curvearrowright X$ be an accessible, totally partially hyperbolic action such that every non-central coarse Lyapunov foliation is 1-dimensional. If the action satisfies \ref{FA1}, then the action is continuously Oseledets conformal.
\end{conjecture}

A reasonable first step towards proving Conjecture \ref{conj:PH-cartan} would be to strengthen the condition \ref{FA1} to super-accessibility (Definition \ref{def:strongly}). Note that the 1-dimensionality of the coarse Lyapunov foliations is required in Conjecture \ref{conj:PH-cartan}, since in this case the non-central leaves are modeled by $\R$, and every automorphism of $\R$ is semisimple. When the dimension is higher, Jordan blocks can appear, and a (bi-)homogeneous action may fail to be continuously Oseledets conformal. Furthermore, one may not relax accessibility to essential accessibility. If one takes $X = (SL(2,\R) \times SL(2,\R) \times SL(2,\R))/\Gamma$, with $\Gamma$ irreducible, and the action of the subgroup
\[B = \set{\begin{pmatrix}
    e^t & 0 \\ 0 & e^{-t}
\end{pmatrix} \times \begin{pmatrix}
    e^s & 0 \\ 0 & e^{-s}
\end{pmatrix} \times \begin{pmatrix}
    1 & t \\ 0 & 1
\end{pmatrix} : t,s \in \R},\]
then this action is essentially accessible and has 1-dimensional non-central coarse Lyapunov foliations, but is not continuously Oseledets conformal, since the center has a Jordan block.

Proving Conjectures \ref{q:rank1-cond} and \ref{conj:PH-cartan} would be a step toward proving the following

\begin{conjecture}[Partially hyperbolic Katok-Spatzier conjecture]
    Let $\mathbb{R}^k \curvearrowright X$ be an (essentially) accessible, totally partially hyperbolic action without rank one factors. Then the action is finitely covered by a bi-homogeneous action.
\end{conjecture}

We note that the adverb ``totally'' cannot be totally omitted due to the examples constructed by the third author in \cite{VinhageExample}.

\subsubsection{Actions with one regular element}
\label{subsub:single-anosov}

When $G$ is a simple Lie group of real rank at least 2, there are two key assumptions in the partially hyperbolic setting: a dense set of partially hyperbolic elements and accessibility. In the setting of abelian actions, the third author constructed actions with one regular element, but not a dense set of regular elements \cite{VinhageExample}.

\begin{conjecture}
    If $G \curvearrowright X$ is an action of $G$ and there exists $g \in G$ which is partially hyperbolic and accessible, then there exists an $\R$-split Cartan subgroup $A \subset G$ such that a dense set of $a \in A$ are partially hyperbolic and accessible with a common central distribution.
\end{conjecture}

The authors are pursuing this question for volume preserving Anosov $G$-actions.   The answer in the non-Anosov, partially hyperbolic case is more nuanced, as the center distribution may not have subexponential growth (one no longer has Lemma \ref{lem:sub-exp-center}). One may pose an easier version of this question by assuming the central distribution of $g$ to have subexponential growth.

\subsubsection{Essential accessibility}
\label{sec:ess-acc}
The other key assumption in our results is accessibility. This assumption seems natural, since if one does not assume accessibility, one may build new totally partially hyperbolic $G$-actions by taking product with trivial actions. While one may conjecture that  the trivial actions are also rigid, the methods used here are not immediately applicable. One may therefore ask about an intermediate case in which %every accessibility class is dense (or more strongly, that 
the partition into accessibility classes is trivial mod volume. This condition is called {\it essential accessibility}, see \cite{BW} for a precise definition, and fundamental ergodic consequences. For a purely topological, weaker condition, one may ask that the every accessibility class is dense.

Examples of essentially accessible $G$-actions which are not accessible are not obvious but not difficult to construct. Indeed, take a simple Lie group $G$, and consider an irreducible cocompact lattice $\Gamma \subset G^2=G\times G$. Then consider the action of $G$ which only acts non-trivially  on the first factor of $G^2$. Then every partially hyperbolic element has its accessibility classes equal to the $G$-orbits, which are dense and partition trivially mod volume since the lattice is irreducible. In fact, whenever one builds a bi-homogeneous action $G\curvearrowright K \backslash H / \Gamma$ which is accessible and for which $K$ is nontrivial and contains no normal subgroups of $H$, the translation action on $H /\Gamma$ is essentially accessible but not accessible (see Remark \ref{rem:essentially-acc}). We therefore ask:

\begin{conjecture}
    Let $G$ be a semisimple Lie group such that  all simple factors of $G$ have real rank at least 2.
    If $G\curvearrowright X$ is a (totally) partially hyperbolic, essentially accessible action, then the action finitely covered by a bi-homogeneous action.
\end{conjecture}

\color{black} Another question is whether one can prove this conjecture for actions which have say a dense set of accessible partially hyperbolic elements in a maximal split Cartan subgroup, but which are not totally accessible. \color{black}

\subsubsection{Systems with dominated splittings}

In its most radical form, the questions and conjectures about systems with dynamically-defined continuous splittings require no assumptions to eliminate counterexamples. In particular, we formulate the following

\begin{conjecture} Let $G$ be a semisimple Lie group such that  all simple factors of $G$ have real rank at least 2. 
    Suppose $G \curvearrowright X$ is a locally free action of $G$, and that there exists $g \in G$ and a continuous splitting $TX = E \oplus F$ such that for all unit vectors $v \in E$ and $w \in F$, $\norm{dg(v)} < \norm{dg(w)}$. Then the action is smoothly conjugated to a bi-homogeneous action.
\end{conjecture}

Note that we seem to have lost all structures assumed in the previous two conjectures, and there are several examples of actions satisfying these assumptions which are not topologically transitive (take, e.g., the product of a trivial action and an Anosov action). However, the property of preserving a continuous dominated splitting allows one to apply a continuous version of Zimmer's superrigidity theorem, which may have a noncompact ``noise'' group. In particular, one obtains many elements which preserve this splitting, and it will be difficult for extra hyperbolicity to appear.

\subsubsection{Products of groups of real rank 1}

Let $G$ be a semisimple Lie group such that the real rank of $G$ is at least 2, but $G$ has rank one factors (for instance, when $G = SL(2,\R) \times SL(2,\R)$. While several aspects of the proofs in this paper carry through, many use the assumption that every simple factor of $G$ has higher rank in a crucial way. Furthermore, one may no longer verify \ref{FA1} from the higher-rank assumption directly. It is therefore natural to ask:

\begin{conjecture}
    Assume $G$ is a semisimple Lie group and $G \curvearrowright X$ is a $C^\infty$, volume-preserving, totally partially hyperbolic action such that the restriction to an $\R$-split Cartan subgroup satisfies \ref{FA1}. Then the action of $G$ finitely covered by a bi-homogeneous action.
\end{conjecture}

\subsubsection{Invariant Volumes}

If $G$ is a higher-rank semisimple Lie group, it is not difficult to construct actions of $G$ which do not preserve a volume (for instance, the projectivized action of $SL(d,\mathbb{R})$ on $S^{d-1}$, or more generally $G \curvearrowright G / Q$ for some parabolic subgroup $Q \subset G$). However, all (bi-)homogeneous actions of $G$ which are Anosov always preserve a volume. We formulate the following

\begin{conjecture}
    If $G\curvearrowright X$ is a $C^\infty$ partially hyperbolic action, then $G$ preserves an invariant volume.
\end{conjecture}

This question can be made easier or more difficult by changing the assumptions to Anosov or essentially accessible partially hyperbolic, respectively.

\subsubsection{Actions of lattices}
\label{sec:lattice-q}

While we are able to obtain results for actions of semisimple Lie groups $G$, our assumptions make it difficult to obtain similar classifications for lattices. This is in contrast to the situation for abelian group actions, where classification of actions $\mathbb{R}^k \curvearrowright X$ usually leads to a classification for actions $\mathbb{Z}^k \curvearrowright X$ satisfying similar assumptions. The main difference is the following: when suspending an $\mathbb{R}^k$ action to a $\mathbb{Z}^k$ action, the suspension over a torus $\mathbb{T}^k$. If $a \in \mathbb{Z}^k$ is regular, and $v \in [0,1)^k$, then the monodromy of $v + ta$ differs from integer multiple of $a$ by at most a vector in $\mathbb{T}^k$ of $L^\infty$ norm 1. When suspending a $\Gamma$-action for some lattice $\Gamma \in G$, instead of covering a translation action on $\mathbb{T}^k$, the action covers the translation action on $G/\Gamma$. This action has hyperbolicity, and when moving along a one-parameter subgroup passing through a regular element $\gamma \in \Gamma$, the monodromy is no longer powers of $\gamma$.

One therefore obtains several new theorems by any answer to the following open ended-question:

\begin{question}
\label{q:lattices}
    Assume $\Gamma \subset G$ is a (uniform) lattice in a semisimple Lie group $G$, and that $\Gamma \curvearrowright$ is a $C^\infty$ action. Let $g_1 = \gamma \in \Gamma$ be an element of $\gamma$ which belongs to a 1-parameter subgroup $\set{g_t}$, and assume $\gamma : X \to X$ is partially hyperbolic. Under what conditions on the action $\Gamma \curvearrowright X$ is the action of the one-parameter subgroup $g_t$ acting on the suspension $G/\Gamma$ partially hyperbolic? Under what conditions is it accessible?
\end{question}
\color{black}
\subsubsection{Actions preserving affine connections}
We note that  bi-homogeneous actions   preserve natural affine connections themselves of algebraic nature.  

\begin{question}
Can one classify affine structures on compact manifolds whose automorphism group contains a higher-rank semisimple group or a higher-rank  lattice?
\end{question}

As one can take products of bi-homogeneous examples  with arbitrary affine manifolds, this will require additional conditions.  More generally one can consider actions with Gromov rigid structures \cite{gromov-rigid}. 

%\subsubsection{Local rigidity of bi-homogeneous actions}
%Fisher 

%Essentially accessible in place of accessible (example: product of Anosov and parabolic)

%One PH + center with 0-exponents??

\color{black}

\section{Outline of the arguments}\label{outline}
{%\color{cyan}{}

We will now describe the arguments in our work in more detail.  

In Section \ref{examples}, we describe  examples that exhibit  the various difficulties we encounter in our classification.  In particular, they explain the necessity to consider bi-homogeneous actions as models, both in the semisimple and also the higher rank abelian cases.  
%We also explain how Anosov and paritally hyperbolic automorphism actions by lattices give rise to Anosov and partially hyperbolic actions by semisimple groups, respectively.  

In Section \ref{abelian-prel} we collected some preparatory material. We introduce the needed  background  from smooth dynamics, especially on partially hyperbolic systems. We derive basic properties of totally partially hyperbolic abelian actions. Finally, a part of this section is dedicated to the   
%Critical for our work will be the notion of accessibility by   distributions $D_i$ of the tangent bundle.  Roughly this means that any two points can be reached from each other by broken paths with legs tangent to some $D_i$.  For a partially hyperbolic map the $D_i$ typically are  stable and unstable distributions of some partially hyperbolic diffeomorphism. 
powerful tool,  the \emph{invariance principle}, originally introduced by Ledrappier \cite{Ledrappier1986}. It was  further developed by Avila, Viana and Santamaria \cite{AV, ASV}.
%This was further developed by Avila, Viana and Santamaria \cite{ASV},
 Kalinin and Sadovskaya developed a version \cite{KalSad} which allows for an application to the partially hyperbolic setting with center-bunching and accessibility conditions. \color{black}Here we derive (with essentially the same proofs as in earlier works) the versions which are adapted to our purpose, i.e. to the setting of abelian group actions (where several foliations are involved), and to non-accessible situations we encounter such as principal fiber bundles over accessible systems and Anosov $\mathbb R^k$-actions.\color{black}  %More precisely, the principle shows H\"olderness of measurable conformal structures invariant under a partially hyperbolic system with suitable accessibility properties.  

 \color{black}
 %At the end of this section, we explain how to apply the Invariance Principle to $\R ^k$ partially hyperbolic actions.  For this, we introduce Lyapunov functionals and joint  Oseledets splitting of the $\R^k$ action as well as a common refinement,  the coarse Lyapunov spaces and  foliations they integrate to. These are simply sums of Oseledets space for positively proportional Lyapunov functionals. We later apply the Invariance  Principle to the action of an element of $\R^k$ which belongs to the kernel of a Lyapunov functional, a so called Lyapunov hyperplane. 
 
 The invariance principle has proved  very useful in the rigidity of group actions before, especially as used by   Damjanovic and Xu \cite{DX1}. They overcame one principal difficulty,  the regularity of Oseledets spaces within coarse Lyapunov foliations and related structures. The invariance principle could be circumvented in the works of Kalinin and Spatzier \cite{KSp}, and Spatzier and Vinhage \cite{Spatzier-Vinhage}, where the coarse Lyapunov foliations are one-dimensional and metric properties follow much easier. In this current work,   the invariance principle and accessibility feature prominently again due to the multidimensionality of the coarse Lyapunov foliations.  
 Naturally, we need to prove the needed accessibility properties using that we have an action of a semisimple Lie group. 
 
 Section \ref{sec:group-prelims} reviews background material from group theory.  Most important are the topological free product constructions for topological groups.  Crucial will be various   criteria  when a topological group is actually a Lie group,   most importantly for us one by Gleason and Palais \cite{gleason-palais}.  However, later on in the final proofs,  we will also employ the no small subgroups property of Montgomery and Zippin \cite{Montgomery-Zippin} and its application to inverse limits of Lie groups. {\color{black}At this point, we conclude Part \ref{part:background} of the paper, which focused on collecting previously established tools to be used later.
 
  Part \ref{part:reduction}  uses many of the tools from partial hyperbolicity and superrigidity first to prove results for $G$-actions from results for $\mathbb R^k$-actions (Section  \ref{sec:semisimple-proof}), and then to extract a common set of consequences of the assumptions made for $\mathbb R^k$-actions that will serve as a starting point for proving classification results for $\mathbb R^k$-actions (Section  \ref{strong accessibility implies FA12}).  
  
  Section \ref{sec:semisimple-proof} contains proofs of all of the results for $G$ actions in Section \ref{sec: G-action}, assuming Theorem \ref{abelian}. Here is a brief outline of the proof of Theorem \ref{G-action}.
  %To state our main result in minimal terms of the $G$-action, 
The basic step is to use a Howe-Moore type argument to get invariance of a volume form by $G$ from that of a suitable one-parameter subgroup of the split Cartan.
 %To state our main result in minimal terms of the $G$-action, we first a We remark that for an action of a semisimple Lie group $G$ of the noncompact type on a compact manifold $M$ and a volume form $\omega$ on $M$, $G$ invariance of $\omega$  follows from  invariance under a regular one-parameter subgroup of a split Cartan.  This allows us 
 Once the $G$-action preserves volume,  we use Zimmer's measurable cocycle rigidity theorem to get measurable  conformal structures along the coarse Lyapunov foliations, invariant under the Cartan subgroup $A$. This is the first place where we use higher rankness assumption on $G$. 
We also use higher rankness of $G$ in the next key step which shows that the accessibility assumption on the $G$-action implies super accessibility for the relevant Cartan subgroup $A$. To prove this, we employ the structure of the acting semisimple group, in particular  special facts about how to write elements of the Weyl group by products of unipotent elements.  Then the $A$ action satisfies assumptions of Theorem \ref{abelian}, which gives us the bi-homogeneous model for the $A$-action. By taking conjugates of $A$ within $G$, we can conclude that the restrictions of the $G$-action to Cartan subgroups are smoothly conjugate to bi-homogeneous actions on the same bi-homogeneous model space. Applying work by Zeghib on centralizers (see Appendix \ref{sec:centralizer-zhegib}) allows us to combine these conjugacies to get one for the whole $G$ action.

%The remaining arguments make sure that this actually models the whole $G$-action as well, the necessary tools for this are contained in Appendix \ref{sec:centralizer-zhegib}.
 %To prove that these conformal structures are \color{black} continuous we invoke the Invariance Principle.  For this, we need to prove that a suitable sub-collection of coarse Lyapunov foliations is still accessible (Theorem \ref{thm: sub acc}). %In our case this becomes the sum of strong stable and strong unstable distributions for an element in a Lyapunov hyperplane. 
 % To prove accessibility for it, 

% Now the split Cartan subgroup $A$  preserves a \color{black}partially \color{black} \holder conformal structure on each coarse Lyapunov foliation, has the right accessibility properties and has a dense set of \color{black}partially hyperbolic \color{black} elements by assumption.
% Thus we are ready to use the classification of such %\color{black}partially hyperbolic \color{black} actions, the second main result of this paper (Theorem \ref{abelian}). We can conclude that the restriction of the $G$-actions to Cartan subgroups are smoothly conjugate to bi-homogeneous actions. Applying work by Zeghib (see Appendix \ref{sec:centralizer-zhegib}) allows us to combine these conjugacies to get one for the whole $G$ action.
 
% \color{black}
In Section \ref{strong accessibility implies FA12}, we start working towards the two theorems on rigidity of abelian actions, Theorems \ref{abelian} and \ref{basic abelian}. From each collection of assumptions, we deduce two fundamental properties: genuinely higher-rank \ref{FP1} and H\"older Oseledets conformal \ref{FP2}. We arrive at these properties in different ways for each theorem, but in both settings the main tool used is the invariance principle (more specifically the results derived in Section \ref{abelian-prel}). This is a natural ``intermission'' of the paper, from that point forward we use these assumptions as a starting point.

{\color{black} After our first intermission, we turn to Part \ref{part:build-bundle}. The key subtlety to overcome before beginning topological group arguments is constructing nilpotent group actions parameterizing the coarse Lyapunov leaves. However, for many partially hyperbolic and Anosov bi-homogeneous actions, no such actions exist on the bi-homogeneous space. Instead, they only live on the homogeneous space which is a principal bundle with compact fibers. See Remark \ref{rem:no-param} for a precise example. Aim of Part \ref{part:build-bundle} is to construct such a bundle,  together with global nilpotent group actions that provide leafwise homogeneous structures, from purely dynamical assumptions.}
 
% In the remainder of the paper, Sections 9 though 12 we classify certain Anosov $\R^k \times K$ actions, and prove Theorems \ref{basic abelian} and \ref{abelian}.
 
 In Section \ref{extension}, we construct leafwise homogeneous structures which intertwine with the given $\R ^k$ action.  The idea is simple:  the Lyapunov hyperplanes act by isometries on the associated coarse Lyapunov foliations.  Each such hyperplane has a dense orbit by \ref{FP1}, so by taking limits they act transitively on coarse Lyapunov leaves to provide the homogeneous structures. 
 
 From this we would like to get a simply transitive action of a Lie group on the coarse Lyapunov leaves.
 There is a complication however.  When recurring to the initial point of a leaf, we may rotate by isometries.
 
 We resolve this problem in Section \ref{fibration}, where we construct a compact extension of the given $\R^k$-action to an action of $\R^k \times K$ for some suitable compact group $K$, essentially by passing to a suitable orthonormal frame bundle. For the lifted action we get group actions parameterizing the coarse Lyapunov foliations which interwine with the $\R^k$-action in Theorem \ref{thm:lifted-action}. In the subsequent section \ref{sec:top-anosov} we axiomatize such ``leafwise homogeneous'' actions  (see also Proposition \ref{prop:smooth-is-top}). Equipped with Theorem \ref{thm:lifted-action}, we take our second intermission.

 %This part  largely extends the work two of the authors \cite{Spatzier-Vinhage}. The classification of totally Anosov $\R ^k$  actions, for $k \geq 2$ under strong assumptions on the derivative cocycle along coarse Lyapunov foliations, e.g. measurable Oseledets  conformal actions.
 
 Part \ref{part:top-groups} begins with Section \ref{sec:top-anosov}, which introduces the notion of
  harnessed abstract partially hyperbolic actions (HAPHAs) and their classification in Theorem \ref{thm:technical}. The main goal is to once again recollect the structures studied before the intermission, and use only the abstract assumptions laid out, in a completely topological setting. Most of the axioms for smooth partially hyperbolic actions were established in the previous parts, and we delay checking them formally until Section \ref{sec:smooth-top}.
  
  HAPHAs are a vast generalization of  {\it topological Anosov actions} introduced in the work by Spatzier and Vinhage \cite[Definition 14.4]{Spatzier-Vinhage}.  The key idea for the proof of the central global classification result in Theorem \ref{thm:technical} is that we have a natural transitive action of an infinite dimensional topological group, a free product of the groups defining the homogeneous structures on the coarse Lyapunov leaves.    We show that this transitive action actually factors through a finite dimensional Lie group. 
  
  We think of elements fixing a given point $p$ as  cycles and need to show that they are independent of $p$.  When the cycles belong to coarse Lyapunov spaces coming from opposite Lyapunov functionals, this is done at the end of Section \ref{sec:top-anosov}.
  
  The other key case is handled in Section \ref{sec:fibers}, through a careful study of the so-called geometric commutators which correspond to taking Lie brackets in a Lie algebra.  One main point to remember here is that we do not have the necessary regularity to take brackets of vector fields as our objects are only topological. \color{black}
  
  %\color{black} partially \color{black} H\"older.
  
  One key lemma is Lemma \ref{eq:cocycle-like-bb} which shows that geometric commutators satisfy a cocycle like property with a polynomial correction term where the latter is independent of the base point.  This is crucial for proving that the cycles are constant in the base point. 
  
  In Section \ref{sec:pairwise-sufficient} we consider arbitrary paths,  show that they form  well-defined group relations  modulo the cycles, and thus   can be put in a canonical presentation. In particular, we may associate to an arbitrary path an equivalent one from a finite-dimensional family of presentations. The techniques and results of this section are similar to that of \cite{Spatzier-Vinhage}, with extra complications due to multidimensionality of the coarse Lyapunov foliations.
  
  Finally, in Section \ref{sec:smooth-top}, we verify the assumptions of a HAPHA for a smooth partially hyperbolic action. %The proof of regularity of the conjugacies (Theorem \ref{regularity of conjugacies}). %study arbitrary cycles and show that they can be put in canonical form. 

 %Next we want to use the action of a Cartan in G to get homogeneous structures.  However, as discussed by examples above, we also need to resolve the issue of bi-quotients which naturally occur in our setting.  That leads us to invent a natural compact extension for which an action of a compact by $\R^k$ group acts in an Anosov like way.  To get the extension, we think of orthonormal frame bundles in the metrics we constructed on the course Lyapunov spaces. 
 
 %As we show, everything miraculously  fits together beautifully: 
% we can prove the needed accessibility for an action by the Cartan subgroup using properties of actions of  semisimple group actions.   Zimmer cocycle superrigidity gives measurable solutions for the conformal structures.
 %By accessibility the invariance principle they 
}
\section{Examples}\label{examples}

Throughout this section, $G$ denotes a semisimple Lie group, and $\Gamma \subset G$ is a lattice.

\subsection{Suspension construction}

\begin{example}[Suspensions]
\label{ex:suspension}
Many $G$-actions come from a standard procedure called {\it suspension} or {\it induction}. Let $\Gamma \subset G$ be a (cocompact) lattice, and $\Gamma \curvearrowright X_0$ be a $C^\infty$ action of $\Gamma$ on $X_0$. The corresponding {\it suspension space} is the set 
%\todo{comment(4) Fixed typo}
$X = (G \times \color{black} X_0 \color{black}) / \sim$, where $\sim$ denotes the relation in which $(g_1,x_1) \sim (g_2,x_2)$ if and only if there exists $\gamma \in \Gamma$ such that $g_2 = g_1\gamma^{-1}$ and $x_2 = \gamma \cdot x_1$.

Notice that $G$ acts naturally on $X$ by $g \cdot (g',x) = (gg',x)$, which preserves equivalence classes. Furthermore,  $G / \Gamma$ is a factor of $X$ under the projection $\pi(g,x) = g\Gamma$, and the restriction of the $G$-action to $\Gamma$ preserves $\pi^{-1}(e)$. The action of $\Gamma$ on $\pi^{-1}(e)$ is clearly $C^\infty$ conjugated to the %\todo{comment(5)  "inverse" write a comment to the referee} 
\color{black} action of $\Gamma$ on $X_0$. {\color{black} When $\Gamma$ acts by automorphisms of a torus or nilmanifolds, and there exists an Anosov or partially hyperbolic $\gamma \in \Gamma$, then the suspended action is totally Anosov or totally partially hyperbolic, respectively (see Example \ref{ex:alg-suspension}).}
\end{example}

\begin{remark}
In general, it is not clear how to conclude that a suspension action constructed as in Example \ref{ex:suspension} is totally Anosov. Indeed, the difficulty lies in concluding that if the action of $\gamma$ on $X_0$ is Anosov, then the action of a 1-parameter subgroup passing through $\gamma$ on $X$ is normally hyperbolic with respect to its centralizer in $G$. Even for very regular 1-parameter subgroup, this relationship is complex and nontrivial.
\end{remark}

\subsection{Homogenenous examples}

We begin by describing an alternate construction to the suspension when the $\Gamma$-action is algebraic.

\begin{example}[Algebraic suspensions]
\label{ex:alg-suspension}    
When the $\Gamma$-action used to construct the suspended action in Example \ref{ex:suspension} is algebraic, another equivalent construction shows that the suspended action is totally Anosov. Indeed, suppose that $X_0$ is a nilmanifold $X_0 = 
N / \Lambda$, and that there is a representation $\rho : G \to \Aut(N)$ without zero weights such that $\rho(\Gamma)$ preserves $\Lambda$. In the case when $N = \R^d$ and $\Lambda = \Z^d$, this data corresponds to a homomorphism $\rho : G \to SL(d,\R)$ such that $\rho(\Gamma) \subset SL(d,\Z)$.

Let $H$ denote the semidirect product $H = G \ltimes_\rho N$, and $\hat{\Gamma}$ denote the group $\Gamma \ltimes_\rho \Lambda$. Then suspension space $X$ constructed in Example \ref{ex:suspension} is diffeomorphic to $H / \hat{\Gamma}$ and the $G$-action is conjugate the homogeneous (left-translation) action. The diffeomorphism can be constructed immediately by writing an element of $H$ as $(g,n)$ where $g \in G$ and $n \in N$. Furthermore, since the representation $\rho$ has no zero weights, the action of $G$ is totally Anosov.
\end{example}

\begin{remark}
The construction described in Example \ref{ex:alg-suspension} seems general, but is actually quite restrictive. Indeed, the difficulty lies in finding the representation $\rho$. Such representations seem to be guaranteed to exist from the Margulis superrigidity theorem. However, given a representation $\rho_0 : \Gamma \to \Aut(N)$ which preserves $\Lambda$, the extension of $\rho_0$ to $G$ is usually only guaranteed to exist {\it up to compact noise}. This can be resolved by considering the lattice $\Lambda$ not in $G$, but in the product of $G$ with a compact group $K$. The construction can proceed with these additional structures, but describes the suspension  of the action $\Gamma \curvearrowright X_0$ as a {\it bi-homogeneous} action, rather than a homogeneous action. See Example \ref{ex:SL3-SU3}.
\end{remark}

\begin{example}[Embedding $\R$-split orthogonal groups]
\label{ex:embedded-SO}
All actions described previously come from suspending a $\Gamma$-action. Here, we describe another class of actions which do not come from actions of lattices. Such actions were first described in \cite{Goetze-Spatzier}. Consider an embedding of the group $SO(n,n)$ into $SO(n,n+1)$. Notice that both groups have the same real rank, $n$, and hence that an $\R$-split Cartan of $SO(n,n)$ is automatically an $\R$-split Cartan subgroup of $SO(n,n+1)$. Furthermore, each group is $\R$-split, so the centralizer of an $\R$-split Cartan subgroup is discrete in both groups. Therefore, the translation action of $SO(n,n)$ on a compact quotient of $SO(n,n+1)$ will be a totally Cartan action. %or by considering an abelian action which does not extend to the action of a semismiple group (Example \ref{ex:SLnC}).
\end{example}

\begin{remark}
The special feature of the groups appearing in Example \ref{ex:embedded-SO} is that both groups are $\R$-split. This procedure can be adapted to produce bi-homogeneous actions when the smaller group is still $\R$-split, and the centralizer of the smaller group in the larger group is compact. This happens for the group $SO(n,n)$ sitting inside $SO(n,m)$, $m \ge n$. See Example \ref{ex:SO22-on-SO2n} for a precise description of this phenomena.
\end{remark}

\subsection{Bi-homogeneous examples}

{\color{black}
The first example of a bi-homogeneous action is an $\R^k$-action which does not extend to a semisimple group action.

\begin{example}[Weyl chamber flows on non-split groups]
\label{ex:SLnC}
Fix a cocompact lattice $\Gamma \subset SL(n,\C)$. Then let $X$ denote the double quotient space $\Diag_U \backslash SL(n,\C) / \Gamma$, where 
$$\Diag_U = \set{ \diag(e^{i\theta_1},\dots,e^{i\theta_n}) : \theta_1,\dots,\theta_n \in \R, \sum \theta_i = 0} \cong \mathbb{T}^{n-1}$$ is the group of unitary diagonal matrices. 

Then the left-translation action of $\Diag_\R = \set{ (e^{t_1},\dots,e^{t_n}) : t_1,\dots,t_n \in \R, \sum t_i = 0}$ is a totally Cartan $\R^{n-1}$ action.

\begin{remark}
\label{rem:no-param}
    Since $U$ is compact, the space $X$ has the structure of an orbifold, and if $\Diag_U \cap \Gamma = \set{e}$, then it is a manifold. Since $\Diag_\R$ commutes with $U$, the left-translation action of $\Diag_\R$ is well-defined. However, the groups $U_{ij}$ which consist of matrices with 1's on the diagonal, any complex number in the $(i,j)$th position, and 0's elsewhere, do not commute with $\Diag_U$. Hence while the coarse Lyapunov foliations have a canonical metric and each leaf has a fixed Euclidean structure, there does {\it not} exist a group action of $\C$ parameterizing the leaves. 
\end{remark}
\end{example}

\begin{example}[Suspensions of actions with ``compact noise'']
\label{ex:SL3-SU3}
Let $G = SL(3,\R) \times SU(3)$ and $\Gamma \subset G$ be the $\Z$-points of some $\Q$-rational embedding of $\rho : SL(3,\R) \times SU(3) \to SL(d,\R)$, $\Gamma = \rho^{-1}(SL(d,\Z))$. Then $\Gamma$ is a lattice in $G$ and $\rho$ is a representation of $G$. Define the semidirect product $H = G \ltimes_\rho \R^d$ in the usual way:

\[ (g_1,v_1) * (g_2,v_2) = (g_1g_2,\rho(g_2)^{-1}v_1+ v_2) \]

Let $\Lambda = \Gamma \ltimes \Z^d$ be the semidirect product of the corresponding lattices. We consider the space $X = SU(3) \backslash H / \Lambda$, and the action of $SL(3,\R)$ on $X$ by left translations. As in Example \ref{ex:SLnC}, the action is well-defined since $SU(3)$ commutes with $SL(3,\R)$. Furthermore, while $G$ splits as a direct product of $SL(3,\R)$ and $SU(3)$ and the representation $\rho$ may be restricted to $SL(3,\R)$,  this restriction does {\it not} give the corresponding action of $\Gamma$ on $\R^d$. Therefore, $X$ is {\it not} a $SL(3,\R) \ltimes \R^d$ homogeneous space.

If the representation $\rho$ does not have zero weights then this action of $SL(3,\R)$ is totally Anosov. Here, while the root spaces are parameterized by group actions, the weight spaces in $\R^d$ will {\it not} be parameterized by group actions of the corresponding weight spaces, exactly because the group $SU(3)$ rotates each weight space.

This example can be easily generalized to any $\R$-split Lie group $H$ defined over $\Q$, its compact real form $K$, and a $\Q$-representation $\rho : H \times K \to SL(d,\R)$.
\end{example}

\begin{example}[Embedding split groups in non-split groups]
\label{ex:SO22-on-SO2n}
Let $H = SO(2,n)^\circ$, $n \ge 3$. We write $\Lie(H)$ as the set of matrices 

\[
\left(
\begin{array}{cc|cc|cccc}
t_1 & u & 0 & a & r_1 & r_2 & \dots & r_{n - 2} \\
\hat{u} & t_2 & -a & 0 & s_1 & s_2 & \dots & s_{n-2} \\
\hline
0 & \hat{a} & -t_1 & -\hat{u} & \hat{r}_1 & \hat{r}_2 & \dots & \hat{r}_{n-2} \\
- \hat{a} & 0 & -u & -t_2 & \hat{s}_1 & \hat{s}_2 & \dots & \hat{s}_{n-2} \\
\hline
-\hat{r}_1 & -\hat{s}_1 & -r_1 & -s_1 & 0 & \theta_{12} & \dots  &\theta_{1(n-2)} \\
-\hat{r}_2 & -\hat{s}_2 & -r_2 & -s_2 & -\theta_{12} & 0 & \dots & \theta_{2(n-2)} \\
\vdots & \vdots & \vdots & \vdots & \vdots &  \vdots & \ddots & \vdots \\ 
-\hat{r}_{n-2} & -\hat{s}_{n-2} & -r_{n-2} & -s_{n-2} & -\theta_{1(n-2)} & -\theta_{2(n-2)} & \dots & 0
\end{array}
\right)
\]

Notice that $\mf{so}(2,2)$ sits inside $\Lie(H)$ with this presentation canonically as the upper left $4 \times 4$-block, and $\mf{so}(n-2)$ sits inside as the bottom right $(n-2) \times (n-2)$-block.  Furthermore, $\mf{so}(2,2)$ commutes with $\mf{so}(n-2)$. Therefore, if $G = \exp(\mf{so}(2,2)) \subset H$ and $K = \exp(\mf{so}(n-2)) \subset H$, we may construct an action of $G$ on $X = K \backslash H / \Lambda$, where $\Lambda$ is some fixed cocompact lattice in $H$. One may notice similarities with the previous examples: the action is Anosov and is well-defined because $G$ commutes with $K$. Furthermore, while the roots of $G$ act on $X$, the roots of $H$ do not act on the double quotient space. For instance, the subalgebra $\mf u$ spanned by the coordinates $r_1,\dots,r_{n-2}$ is a coarse Lyapunov foliation, normalized by $K$ and $K$ preserves an invariant metric on it. However, the action of $K$ is not trivial, so while $\exp(\mf u)$ acts on the homogeneous space $H / \Lambda$, it will not act on $K \backslash H / \Lambda$.

This is also an example of an abelian totally Anosov action by considering the action of the split Cartan subgroup. It can be further generalized to $SO(m,m)$ acting on $SO(m,n)$ quite easily or $SU(m,m)$ acting on $SU(m,n)$, respectively. However, the action of the Cartan subgroup of $SU(m,m)$ is not an abelian totally Anosov action: one must additionally quotient by $\Diag_U \subset SU(m,m)$ on the left.
\end{example}

\begin{example}[Combining phenomena]
We combine ideas in the last two examples to show one last feature. Let $H$ and $\Lambda$ be as in Example \ref{ex:SO22-on-SO2n}. $\Lambda$ is often obtained by taking a $\Q$-algebraic representation $\rho : H \times SO(n+2) \to SL(d,\R)$, and letting $\Gamma = \rho^{-1}(SL(d,\Z))$. This construction, called restriction of scalars, requires the group $SO(n+2)$ to be there. One may proceed as in Example \ref{ex:SL3-SU3} and construct and example of $H$ on $SO(n+2) \backslash (H \times SO(n+2))\ltimes \R^d / \Gamma \ltimes \Z^d$. The action of $H$ is Anosov in the sense of semisimple group actions, however, the restriction to the Cartan subalgebra of $H$ is {\it not} Anosov. Instead, the action of the {\it centralizer of the split Cartan}, $\R^2 \times SO(n-2)$ is Anosov. If one instead considers the quotient by $SO(n-2) \times SO(n+2)$ on the left, then one obtains a totally Anosov action of $\R^2$.
\end{example}

{\color{black}
\begin{proposition}
\label{prop:anosov-ex}
    Let $G$ be a semisimple Lie group without compact factors and $G \hookrightarrow H$ be an embedding into a Lie group $H$. Then there exists a bi-homogeneous space $K \backslash H / \Lambda$ such that the translation action $G \curvearrowright K \backslash H / \Gamma$ is Anosov if and only if $Z_H(G)$ is compact. In this case, $K^\circ = Z_H(G)^\circ$.
\end{proposition}

\begin{proof}
    Note that the derivative of the $G$-action on a bi-homogeneous space $K \backslash H / \Lambda$ is determined by the adjoint representation of $G$ on $\Lie(H) / \Lie(K)$. Since $G$ is semisimple, $\Lie(H)$ decomposes as a sum of irreducible representations of $G$. Note that the adjoint representation always appears as a subrepresentation since $\Lie(G)$ is a subalgebra. If the action is Anosov, no other representation of $G$ on $\Lie(H) / \Lie(K)$ can have zero weights. Thus, $\Lie(K)$, which must be invariant under $\Ad(G)$ for the action of $G$ to be well-defined, must contain all of the zero weights of $\Ad(G) \curvearrowright \Lie(H)$. Since the automorphisms of a compact group are compact, it follows that $\Ad(G)|_{\Lie(K)}$ is trivial. Thus, $K$ commutes with $G$, and since all other representations are nontrivial, $\Lie(K)$ is exactly the centralizer of $\Lie(G)$.
\end{proof}
}
%{  From this example, I think it makes sense to consider actions of abelian x compact groups. It makes the application for actions of non-split semisimple groups much easier. of course, one may always obtain a double-quotient space structure for the action of the split cartan first, then remember this afterwards to obtain the original semisimple group action was algebraic. not sure what the best course of action is....}

}

{\color{black}

\subsection{Partially hyperbolic examples\color{black}}

All examples discussed above are totally Anosov. In the abelian setting, one may consider restrictions of such actions to obtain partially hyperbolic examples. Here we describe a few more examples of partially hyperbolic, but not Anosov actions.

\begin{proposition}[Embedding groups]
\label{prop:PH-ex}
    Assume that $G$ and $L$ are simple Lie groups, $G \subset L$, $G$ has real rank at least two, and $\Lambda \subset L$ is a cocompact lattice in $L$. Then the action $G \curvearrowright L/\Lambda$ by left translations is $E^c$-totally partially hyperbolic for some homogeneous distribution $E^c$ and accessible.
\end{proposition}

\begin{remark}
    The difference between Proposition \ref{prop:anosov-ex} and \ref{prop:PH-ex} illustrates that the partially hyperbolic accessible actions are vastly more general in both number and variety. Indeed, it is not difficult to check that no  action on a bi-homogeneous space of an $\R$-split group $H$ (e.g., $SL(d,\R)$) will ever satisfy the assumptions of Proposition \ref{prop:anosov-ex}, since if $G \subset H$ is a proper semisimple subgroup, it has strictly smaller $\R$-rank, and there will be an $\R$-semisimple element of $H$ which commutes with $G$. 
    
    On the other hand, for $G$-actions, any accessible action is also super accessible, see Theorem \ref{thm: sub acc}. Note also that there is no assumption on the codimension of $G$ inside $L$, so potential examples include $SL(3,\R)$ embedding in $SL(8675309,\R)$ through any representation $\rho$.
\end{remark}

\begin{proof}
    Since $G \subset L$, the real rank of $L$ is at least the real rank of $G$. Furthermore, we may choose an $\R$-split Cartan subgroup $A_L \subset L$ such that $A_G := A_L \cap G$ is a Cartan subgroup of $G$. Let $\Delta_L$ denote the roots of $L$ with respect to $A_L$, and $D(\Delta_L)$ denote the subset of {\it detected roots}, ie, those roots $\beta \in \Delta_L$ such that $\beta|_{A_G} \not= 0$. Then the nonzero Lyapunov functionals of the $A_G$-action on $L / \Lambda$ are exactly the functionals $\beta|_{A_G}$, $\beta \in D(\Delta_L)$.

    Define $E^c = \Lie(A_L) \oplus \bigoplus_{\Delta_L \setminus D(\Delta_L)} E^\beta$ be the sum of the Lie algebra of $A_L$ and the undetected root spaces. Then if $a \in A_G$ satisfies $\beta(a) \not= 0$ for all detected roots $\beta$, the action of $a$ is $E^c$-partially hyperbolic, so the $A_G$-action is $E^c$-totally partially hyperbolic.

    We now show that the action is accessible. It suffices to show that the subalgebra $\mf h$ generated by the detected roots is all of $\Lie(L)$. To do this, we show that $\mf h$ is an ideal of $\Lie(L)$. Since the root space of $\Lie(L)$ generate $\Lie(L)$, we only need to show that if $\beta$ is an undetected root and $\gamma$ is a detected root, then $[E^\beta,E^\gamma]$ is contained in $\mf h$. Indeed, since $[E^\beta,E^\gamma] \subset E^{\beta + \gamma}$, and $(\beta+\gamma)|_{A_G} = \gamma|_{A_G}$ by assumption, $\beta+\gamma$ is detected. Therefore, the action is accessible.
\end{proof}

\begin{example}[Biquotients of semisimple Lie groups]
\label{ex:biquotient}
    Consider the example as above, but assume that some compact subgroup $K \subset L$ commutes with $G$ such that $\ell K\ell^{-1} \cap \Lambda = \set{e}$ for all $\ell \in L$. Then the action $G \curvearrowright K \backslash L / \Lambda$ is $E^c / \Lie(K)$-totally partially hyperbolic and accessible, where $E^c$ is as in the previous example.
\end{example}

\begin{example}[Partially hyperbolic algebraic suspensions]
\label{ex:general-headache}
One may consider semidirect product examples as in Examples \ref{ex:alg-suspension} and \ref{ex:SL3-SU3}, but remove the restriction that the representation $\rho$ has no zero weights. Instead, one insists that there are no trivial subrepresentations of $\rho$. This gives us totally partially hyperbolic accessible examples that are not Anosov. For each root $\beta$ such that $-\beta$ is a nontrivial weight of $\rho$, let $E^\beta$ be the $\beta$-root space and $V^{-\beta}$ be the $-\beta$-weight space of $\rho$. Then the sum of the distributions $[E^\beta,V^{-\beta}]$ form the 0-weight space of $\rho$, which makes the action accessible. 
    
    In summary, we describe (one of) the most general classes of examples we may address. Consider the following data as input:

    \begin{itemize}
        \item $G \subset L$ simple Lie groups such that $G$ has real rank at least 2.
        \item $M \subset L$ is a (possibly trivial) compact Lie subgroup commuting with $G$.
        \item $K$ is a compact Lie group.
        \item $\Lambda \subset K \times L$ is an irreducible cocompact lattice such that all conjugates of $\Lambda$ intersect $M \times K$ trivially.
        \item $\rho : K \times L \to SL(N,\R)$ is a representation with no trivial subrepresentations.
        \item $\rho(\Lambda) \subset SL(N,\Z)$.
    \end{itemize}

    Then the action $G \curvearrowright (K \times M) \backslash ((K \times L) \ltimes_\rho \R^N) / (\Lambda \ltimes_\rho \Z^N)$ defined by left translation is $E^c$-partially hyperbolic and accessible, where $E^c$ is the sum of the central distribution in Example \ref{ex:biquotient} and the zero weight spaces of $\rho$.
\end{example}

\subsection{Essentially accessible examples\color{black}}\label{rem:essentially-acc}

There are a handful of model actions which do not satisfy the accessibility assumptions of our theorem. In particular, they satisfy that they are partially hyperbolic and ergodic, but not accessible.

In Example \ref{ex:general-headache}, while one may take $L$ to be a trivial group, we may not omit the left quotient by the group $K$. Indeed, the accessibility classes will be the $L \ltimes_\rho \R^N$ orbits. By the irreducibility condition on $\Lambda$, these are dense in $((K \times L) \ltimes_\rho \R^N) / (\Lambda \ltimes_\rho \Z^N)$. Thus, while the action is {\it essentially} accessible (recall the notion of essential accessibility from  Section \ref{sec:ess-acc}), it is not accessible.

We conclude this examples section with another standard way of constructing an essentially accessible but not accessible actions.

\begin{example}
    Let $d \ge 3$, $n \ge 2$, and $G = SL(d,\R)^n$. Fix an irreducible lattice $\Gamma \subset G$, and some $3 \le k \le n$. Then $SL(k,\R)$ embeds into the first $SL(d,\R)$ factor of $G$ and acts on the homogeneous space $G / \Gamma$ by left translations. Then the action is partially hyperbolic, volume preserving, and ergodic, but not accessible. Instead, these examples are only {\it essentially} accessible. Indeed, regardless of $k$ and $d$ (among the restrictions given able), the accessibility classes of the $SL(k,\R)$-action are cosets of the first $SL(d,\R)$-factor. Unlike Example \ref{ex:SL3-SU3}, we cannot quotient by the factors other than the first factor since their foliations are minimal.

    This construction can be generalized in several ways. For example, one can form semidirect products as in Example \ref{ex:alg-suspension}, or use an $n$-fold product of semisimple Lie groups which include different real forms of the same group (by the lattice rigidity theorems of Margulis \cite[IX.4.5]{M91}, a product of groups has an irreducible lattice if and only if each group in the product is a real form of the same complex Lie algebra).
    
    Finally, not that this does not happen in the Anosov setting, since it comes from building a large central distribution from ``undetected'' simple factors in a semisimple Lie group.
\end{example}
}

\part{Background and preliminaries}
\label{part:background}

\section{Dynamical Preliminaries }\label{abelian-prel}

\subsection{Dominated splittings, partially hyperbolic diffeomorphisms and center bunching}
\label{sec:PH-prelims}
A $C^1$ diffeomorphism $f$ on a compact smooth Riemannian manifold $X$ is \emph{($E^c_f$-)partially hyperbolic} if there is a nontrivial $Df-$invariant splitting $E_f^s\oplus E_f^c\oplus E_f^u$ of the tangent bundle  $TX$ \color{black} and a Riemannian metric on { $X$} such that there exists continuous positive functions $\nu<1, \hat\nu<1, \gamma, \hat\gamma$ for which the following inequalities hold for any $x\in X$, $v^s\in E_f^s$, $v^u\in E_f^u$, $v^c\in E_f^c$: 
\begin{equation}\label{def: ph dif def}
\|Df(v^s)\|<\nu(x)<\gamma( x\color{black})<\|Df(v^c)\|<\hat\gamma(x)^{-1}<\hat\nu(x)^{-1}<\|Df(v^u)\|.
\end{equation}
%\begin{equation}\label{def: ph dif def}
%\|Df^k(v^s)\|<\nu(x)<\gamma( x\color{black})<\|Df^k(v^c)\|<\hat\gamma(x)^{-1}<\hat\nu(x)^{-1}<\|Df^k(v^u)\|.
%\end{equation}
%We always assume the bundles $E^{s}$ and $E^{u}$ are nontrivial. 
%If $E^c$ is trivial then $f$ is  \emph{Anosov}.
The subbundles $E_f^s, E^u_f$ and $E^c_f$ are continuous and they are called {\it stable, unstable and center}, respectively. 
The bundles $E_f^{s}$ and $E_f^{u}$ are integrable to (typically only) H\"older foliations $W^u_f$ and $W^s_f$. If $f$ is smooth than the leaves of  $W^u_f$ and $W^s_f$ are smooth. In general, the center bundle need not be integrable. 

\color{black} More generally, a {\it dominated splitting} for a diffeomorphism $f$ is a $Df$-invariant decomposition $TX=E_1\oplus \dots \oplus E_k$ such that $Df|_{E_i}$ dominates $Df|_{E_{i+1}}$ in the sense that for some $N\ge 1$ and all $x\in X$, $\|D_{x}f^N(u)\|\le \frac{1}{2} \|D_{x}f^N(v)\|$, for all unit vectors $u\in E_{i+1}$ and $v\in E_i$. (In particular,  $f$ with a dominated splitting $E_f^s\oplus E_f^c\oplus E_f^u $ is partially hyperbolic.)
\color{black}

Homogeneous actions examples in Section \ref{examples} contain many partially hyperbolic elements which all have a common center distribution which integrates to a foliation, and on which the action is isometric. In general we may not have such nice behavior in the center. For general partially hyperbolic diffeomorphisms we need to understand the extent to which non-conformality of $E^c$ is dominated by transversal contraction and expansion. Such domination is called \emph{center bunching}.  A diffeomorphism $f$ is {\it center bunched} if the functions $\nu, \hat\nu, \gamma, \hat\gamma$ can be chosen to satisfy 
$\nu<\gamma\hat\gamma \,\,\, \mbox{and}\,\,\, \hat\nu<\gamma\hat\gamma.$

%If $f$ is $C^{1+\beta}$ then we use a stronger version of the center bunching condition which can imply stronger results such as ergodicity and H\"older regularity of distributions $E_f^u\oplus E_f^c$,  $E_f^s\oplus E_f^c$ and $E_f^c$ (see \cite{KalSad} or \cite{BW} for more details).  A $C^{1+\beta}$ diffeomorphism $f$ is {\it strong center bunched} if there is $\theta\in (0, \delta)$ and functions $\kappa, \hat\kappa$ such that 
%$ \max\{\nu^\theta, \hat\nu^{\theta}\}<\gamma\hat\gamma$, 
%$\nu\gamma^{-1}<\kappa^{\theta}, \, \, \, \, \hat\nu\hat\gamma^{-1}<\hat\kappa^\theta,$ and for all $x\in M$, for all $v^s\in E^s_f(x)$ and  $v^u\in E^u_f(x)$:
%$$\kappa(x)<\|Df(v)\|, \, \, \, \,\,\, \|Df(v)\|<\hat\kappa(x)^{-1}.$$

%{\color{olive}
%\subsection{Partially hyperbolic abelian actions with common central distributions}

\subsection{Partially hyperbolic abelian actions with common center distribution}\label{PHactions}
\color{black}

%\begin{definition}
  %  Let $\alpha : \R^k \curvearrowright X$ be a smooth action, and $E^0 \subset TX$ be a continuous distribution on $X$. We say that $a \in \R^k$ is {\it $E^c$-partially hyperbolic} if there exists a splitting $TX = E^s_a \oplus E^c \oplus E^u_a$ which endows $a$ with the structure of a partially hyperbolic diffeomorphism. We say that the action $\alpha$ is {\it totally partially hyperbolic (with respect to a distribution $E^c$)} if for a dense set of $a \in \R^k$, $a$ is $E^c$-partially hyperbolic.
%\end{definition}

%We emphasize that this is a somewhat restrictive condition as not every $\R^k$-action with a dense set of individually partially hyperbolic elements is partially hyperbolic in this sense.  
\color{black}
\subsubsection{Oseledets decomposition} Let $\rho : \R^k \curvearrowright X$ be an action with an ergodic invariant measure $\mu$. (In this paper, $\mu$ is going to be the volume.)  The Oseledets  theorem for cocycles over abelian actions (\cite[Theorem 2.4]{BRHW}) applied to the derivative cocycle of $\rho$,  implies the existence of finitely many linear functionals $\chi\colon \R^k\to \R$ ({\it the Lyapunov functionals}), and a $\rho$-invariant measurable splitting $\oplus E^{ \chi}$  of { $TX$} ({\it the Oseledets decomposition}), on a full $\mu$-measure set,  such that for $a\in\R^k$
and $v\in E^{ \chi}(x)$: \[\lim_{a\to \infty}\frac{\log \|D_x\rho(a)v\|-\chi(a)}{\|a\|}=0.\]

The hyperplanes $\ker \chi\subset \R^k$ are 
 \emph{Weyl chamber walls}, and the connected components of $\R^k-\cup_\chi \ker \chi$ are the \emph{Weyl chambers} for the action (with respect to $\mu$). 
 
 Two nonzero Lyapunov functionals $\chi_i$ and $\chi_j$ are \emph{coarsely equivalent} if they are positively proportional: there exists $c>0$ such that $\chi_i=c\cdot \chi_j$. This is an equivalence relation on the set of Lyapunov functionals, and a \emph{coarse Lyapunov functional} is an equivalence class $[\chi]$ under this relation. We will often write just $\lambda=[\chi]$ for a coarse Lyapunov functional. 
\color{black}

\subsubsection{Coarse Lyapunov foliations}
\color{black} 
Let $\R^k \curvearrowright X$ be a partially hyperbolic action. Given an $E^c$-partially hyperbolic element $a \in \R^k$, the stable and unstable distributions integrate to foliations $W^s_a$ and $W^u_a$ which are in general only H\"older with smooth leaves. Given another $E^c$-partially hyperbolic element $b \in \R^k$, $Db|_{E^s_a}$ must admit a hyperbolic splitting, since  $a$ and $b$ share a common distribution. This allows us to subfoliate $W^s_a$ by considering points which are contracted under both $a$ {\it and} $b$. For a precise treatment of this construction see \cite[Corollary 4.6, Lemma 4.7]{Spatzier-Vinhage}, which is written with an Anosov assumption, but whose proofs work verbatim under a partial hyperbolicity assumption. Assume now that $\R^k \curvearrowright X$ is ($E^c$-) totally partially hyperbolic (Definition \ref{totally}). Proposition below summarizes the main features and structures which such actions have. 

\begin{proposition}\label{prop: fundm PH prop}
    Let $\R^k \curvearrowright X$ be a $C^r$, $E^c$-totally partially hyperbolic action, $r \ge 1$. Then 
    
    \begin{itemize}
        \item[(1)] there exists a set $\Delta \subset (\R^k)^*/ \sim$ (where $\chi \sim c\chi$, $c > 0$), of equivalence classes of linear functionals, 
        %linear functionals 
        %defined up to positive scalar multiple: 
        %$\Delta \subset (\R^k)^*/ \sim$ (where $\chi \sim c\chi$, $c > 0$),
        \item[(2)] for each $\lambda \in \Delta$ there exists a H\"older foliation $W^{\lambda}$ with $C^r$ leaves,
        \item[(3)] if $\lambda(a) > 0$, then $W^{\lambda}$ subfoliates $W^s_a$, 
        \item[(4)] $TX = E^c \oplus \bigoplus_{\lambda \in \Delta} TW^{\lambda}$, 
        \item[(5)] the set of $E^c$-\color{black}partially hyperbolic elements of $\R^k \curvearrowright X$ is exactly $\R^k \setminus \displaystyle\bigcup_{\lambda \in \Delta} \ker \lambda$. 
        \item[(6)] For each $a\in \R^k \setminus \displaystyle\bigcup_{\lambda \in \Delta} \ker \lambda$, $E^s_a=\bigoplus_{\lambda(a)<0} E_\lambda, \, E^u_a=\bigoplus_{\lambda(a)>0} E_\lambda$.
        %\color{black}(did we ever define $E^\beta$?)\color{black}
   \item[(7)] In addition, if $\R^k \curvearrowright X$ preserves an ergodic invariant measure, then $TW^{\lambda}=E_\lambda=\bigoplus_{c_i\chi \in \lambda} E^{c_i\chi}$, where $E^{c_i\chi}$ are Oseledets distributions.\color{black}
    \end{itemize}

    The set $\Delta$ is called the set of {\normalfont coarse Lyapunov functionals} of the action.
\end{proposition}

%{\color{olive}
We recall the following definitions from \cite[Section 5.7]{Spatzier-Vinhage}. We call elements of $\Delta$ {\it coarse weights} (or interchangeably {\it coarse Lyapunov exponents, coarse Lyapunov functionals}), and for $a \in \R^k$, let $\Phi\subset \{\gamma\in \Delta: \gamma (a)<0\}$ be a stable subset of coarse weights for $a$. 
%(ie, a set of weights such that $\gamma(a) < 0$ for all $\gamma \in \Delta$). % and distinct for all $\beta \in \Phi$. If $a$ is ultra-regular (Definition \ref{def:hyper-reg}), one may take $\Phi$ to be the set of {\it all} weights $\beta$ for which $\beta(a) < 0$. %Order the roots $\beta_n(a) < \beta_{n-1}(a) < \dots < \beta_1(a) < 0$ which have negative evaluation on $a$. Such a choice can be made since $\beta_1$ is the least negative root.
We introduce  a (not necessarily unique) order on the set $\Phi$. Choose $\R^2 \cong V \subset \R^k$ which contains $a$ and for which $\gamma_1|_V$ is proportional to $\gamma_2|_V$ if and only if $\gamma_1$ is proportional to $\gamma_2$ for all $\gamma_1,\gamma_2 \in \Phi$ (such choices of $V$ are open and dense). Fix some nonzero $\chi \in V^*$ such that $\chi(a) = 0$ ($\chi$ is not necessarily a coarse weight). Then $\beta|_V \in V^* \cong \R^2$ for every $\beta \in \Phi$ and $\Phi|_V = \set{\beta|_V : \beta \in \Phi}$ is contained completely on one side of the line spanned by $\chi$. We may introduce the angle $\angle(\gamma,\chi) := \arccos\dfrac{\left\langle\gamma,\chi\right\rangle}{\norm{\gamma}\norm{\chi}} \in [0,\pi)$, using the canonical inner product on $\R^2$.

\begin{definition}
\label{def:circular-ordering}
The ordering defined by: $\beta < \gamma$ if and only if %\todo{comment(28) indicate the range of angle values}
$\angle(\chi,\beta|_V) < \angle(\chi,\gamma|_V)$, is called the {\it circular ordering of $\Phi$ (induced by $\chi$ and $V$)} and is a total order on $\Phi$. Denote by $ |\alpha,\beta|_{\Phi}$ the set of coarse weights $\gamma \in \Phi$ such that $\alpha \le \gamma \le \beta$. If the set $\Phi$ is understood, we shorten the notation to $|\alpha,\beta|$.
%The order  on the weights in $\Phi$  the angle it makes with $\chi$, where $\chi$ itself having an angle of 0 will be the first element (while the angles may depend on the a choice of metric in $V^*$, the ordering does not). Such an ordering is called a {\it circular ordering (determined by $\chi$)}.
\end{definition}

 While each $\beta \in \Delta$ is only defined up to positive scalar multiple, this is still well-defined since the circular ordering on $\R^2$ is invariant under orientation-preserving linear maps.  
 
\begin{definition}
\label{def:canonical-order}
If $\alpha,\beta \in \Delta$, let $\Sigma(\alpha,\beta) \subset \Delta$ (called the {\em $\alpha,\beta$-cone}) be the set of $\gamma \in \Delta$ such that $\gamma = \sigma \alpha + \tau \beta$ for some $\sigma,\tau > 0$. We may identify $\Sigma(\alpha,\beta)$ as a subset of the first quadrant of $\R^2$ by using the coordinates $(\sigma,\tau)$. The {\em canonical circular ordering} on $\Sigma(\alpha,\beta) \cup \set{\alpha,\beta}$ is the counterclockwise order in the first quadrant. 
\end{definition}

Next proposition implies a local product structure of coarse Lyapunov foliations within a common stable foliation for several action elements.
\begin{proposition}
\label{prop:circ}
    Let $W^s_{a_1,\dots,a_\ell}$ be a common stable manifold for a collection of partially hyperbolic elements $a_i \in \R^k$. Let $\Phi = \set{\gamma \in \Delta : \gamma(a_i) < 0 \mbox{ for all }i=1,\dots,\ell}$, and $\Phi = \set{\gamma_1,\dots,\gamma_m}$ denote a circular ordering of $\Phi$. Let $W^{[1,j]}$ denote the foliation whose tangent distribution is $E_{\gamma_1} \oplus \dots \oplus E_{\gamma_j}$. Then for any $x \in X$, and collection of points $y_j \in W^{\gamma_j}(x)$, there is a unique sequence of points $(x_0,\dots,x_m)$ such that $x_0 = x$, $y_1 = x_1$, and $x_{j+1} \in W^{\gamma_{j+1}}(x_j) \cap W^{[1,j]}(y_j)$ when $j > 1$. Furthermore, the map $(y_1,\dots,y_m) \mapsto x_m$ is a homeomorphism between $\prod_{i=1}^m W^{\gamma_i}(x)$ and $W^s_{a_1,\dots,a_\ell}(x)$.
\end{proposition}

\begin{proof}
    We proceed by induction on $m$. When $m = 1$, the point $y_1 = x_1$ is uniquely defined, and certainly belongs to $W^{[1,1]}(x) = W^1(x)$.

    Suppose that we have the lemma for $m$, so that $(y_1,\dots,y_m) \mapsto x_m$ is a homeomorphism. Since the coarse weights $\gamma_i$ are listed in a circular ordering, we may find an element $a \in \R^k$ such that $-1 \ll \gamma_{m+1}(a) < 0$ and $\gamma_j(a) < -1$ for $j = 1,\dots,m$ (choose an element in $\ker \gamma_{m+1}$ and perturb). Then the splitting $TW^{[1,m+1]} = TW^{[1,m]} \oplus TW^{\gamma_{m+1}}$ is  dominated, and the leaf $W^{[1,m+1]}(x)$ is subfoliated by $W^{[1,m]}$-leaves which determine global holonomies $h_{y,z} : W^{\gamma_{m+1}}(y) \to W^{\gamma_{m+1}}(z)$ defined by $h_{y,z}(w) = W^{[1,m]}(w) \cap W^{\gamma_{m+1}}(z)$ between $W^{\gamma_{m+1}}$-leaves whenever $y \in W^{[1,m]}(z)$. See \cite[Section 4.3]{Spatzier-Vinhage} (which uses \cite{PSW04} to establish the regularity needed to leverage transversality of $W^{[1,m]}$ and $W^{\gamma_{m+1}}$).

    Therefore, we may define $x_{m+1} = h_{x,y_{m+1}}(x_m)$. This is well-defined and unique in a neighborhood, and using the intertwining property $a h_{y,z} = h_{ay,az} a$, we can extend the definition to the global leaf.    
\end{proof}
}

Notice that (5) of Proposition \ref{prop: fundm PH prop} shows that for an action $\R^k \curvearrowright X$ with a dense set of $E^c$-partially hyperbolic elements actually every element inside any Weyl chamber is $E^c$-partially hyperbolic. In the next section we show that the elements in the Weyl chamber walls $\ker \lambda$, $\lambda\in \Delta$, are also partially hyperbolic, moreover with zero Lyapunov exponent in the center direction.

\subsubsection{Vanishing of exponents on center distributions}

Statements of this section are a consequence of existence of a {\it dense set} of partially hyperbolic elements in an abelian action. They show {\it subexponential} growth along the center distribution for various action elements. This has been known and used for totally Anosov actions in various special contexts very early on, for example already in \cite{FKS}. From this we derive stronger properties for action elements, most importantly for our applications, we get partial hyperbolicity and center bunching for elements in the walls  $\ker \lambda$. 

\begin{definition}
    Let $\R^k \curvearrowright X$ be an action by diffeomorphisms of a Riemannian manifold $X$ and $E \subset TX$ be an action-invariant continuous distribution. We say that $E$ has {\it subexponential growth (for the action)} if for every $\ve > 0$, there exists some $C > 0$ such that for all $a \in \R^k$,

    \[ C^{-1}e^{-\ve\norm{a}} \le \norm{Da|_E} \le Ce^{\ve \norm{a}} .\]
\end{definition}

\bl \label{sub-exp growth}[Lemma 4.1 in \cite{DX1}] Let $\R^k \curvearrowright X$ be a totally partially hyperbolic action. For any $\lambda\in \Delta$,  if $a\in \ker \lambda$ then $a$ has subexponential growth along $W^\lambda$. %If $a\notin \ker \lambda$ then $a$ either contracts or expands $W^\lambda$ exponentially fast.
\el 

\begin{proof}
    \color{black}
    Fix a coarse exponent $\lambda=[\chi]\in \Delta$, and $\ve > 0$. Given $a \in \R^k$ and $R > 0$ consider the quantity

    \[ M_\chi(a,R) = \sup\set{\log \norm{D(ra)|_{E^\chi}} : 0 \le r \le R}. \]

    Note that $M_\chi(a,R)$ is continuous and subadditive in $a$ and $R$, so that
    \[M_\chi(a+b,R) \le M_\chi(a,R) + M_\chi(b,R) \qquad \mbox{and} \qquad M_\chi(a,R_1+R_2) \le M_\chi(a,R_1) + M_\chi(a,R_2).\] 
    %Furthermore, we know that if $a \in \ker \chi$, $M(a,R)$ is uniformly bounded above and below in $R$ by (FA-2).
    Define \[L_\chi(a) = \inf\set{ c : M_\chi(a,R) \le Rc \mbox{ for sufficiently large } R > 0} = \limsup_{R \to \infty} \frac{1}{R}M_\chi(a,R).\]
    Then if $\chi(a) > 0$, $L_\chi(a) >0$ and if $\chi(a) <0$, then $L_\chi(a) < 0$, since $E^\chi$ is part of the stable or unstable distribution of $a$, respectively. Hence, if we show that $L_\chi$ is continuous, we conclude that $L_\chi(a) = 0$ for all $a \in \ker \chi$.

    To show continuity of $L_\chi$, we show convexity. Notice that since $M_\chi$ is subadditive in $a$, $L_\chi$ is also subadditive: $L_\chi(a+b) \le L_\chi(a) + L_\chi(b)$. Furthermore, $L_\chi(ta) = tL_\chi(a)$ for all $t \in \R$. This implies that $L_\chi$ is convex and hence continuous in $a$.

    Now, for any $a \in \ker \chi$, we know that $L_\chi(a) = 0$. \color{black}\end{proof}
\color{black}
\begin{lemma}\label{Ec center bunching}
\label{lem:sub-exp-center}
    Let $ \R^k \curvearrowright X$ be a totally partially hyperbolic action with at least two independent coarse Lyapunov exponents with respect to some distribution $E^c$. Then $E^c$ has subexponential growth for the action. In particular, every partially hyperbolic element is center bunched.
\end{lemma}

\begin{proof}
    By Lemma \ref{sub-exp growth}, the action of $\ker \chi$ has uniformly 0 exponents with respect to $E^\chi$.
    Define the functions $M_0$ and $L_0$ as in the previous lemma, replacing the distribution $E^\chi$ by the distribution $E^c$. Then again, $L_0$ is a convex function, and for any partially hyperbolic element $a \in \R^k$, we know that $L_0(a) \le L_\chi(a)$. Then, since $\ker \chi$ is accumulated by partially hyperbolic elements, and both $L_0$ and $L_\chi$ are continuous, $L_0(a) \le 0$ for all $a \in \ker \chi$. Since $L_0(0) = 0$, by convexity we can conclude that $L_0(a) = L_0(-a) = 0$ for all $a \in \ker \chi$. This is exactly the conclusion that an element in $\ker \chi$ has uniformly 0 exponents on $E^c$. %Choose a sequence of elements $a_\ell$ such that $a_\ell$ is $E^0$-partially hyperbolic and expands $E^\chi$ but $a_\ell \to a_0 \in \ker \alpha \setminus\set{0}$. Then $\lambda(a_\ell)$ and converges to 0. Since $\norm{d\alpha(a_\ell)|_{E^\chi}}^{-1} \le \norm{d\alpha(a_\ell)|_{E^0}} \le \norm{d\alpha(a_\ell)|_{E^\chi}}$, it follows that $\norm{d\alpha(ra_\ell)|_{E^0}} \le Ce^{\lambda(a_\ell)r}$ for sufficiently large $r$. Taking limits as $\ell \to \infty$, it follows that $\norm{d\alpha(ra)|_{E^0}} \le Ce^{\ve r}$ for any $\ve > 0$

    Since we have assumed that there are at least two independent exponents, this implies the lemma. Indeed, if $\ker \chi_1 \not= \ker \chi_2$, then $\ker \chi_1 + \ker \chi_2 = \R^k$.

\end{proof}

\color{black}

\color{black}
Any $a$ which lies in $\ker \lambda$ for exactly one $\lambda\in \Delta$ is called \emph{generic singular}. In particular, in our setting for any $\lambda\in \Delta$ every generic singular element $a$ in $\ker \lambda$, $a$ can be viewed as a partially hyperbolic diffeomorphism acting on $X$, with center distribution \color{black} $E^c\bigoplus \oplus_{\lambda(a)=0} E_\lambda$
and from the above lemma we have
\color{black}

%\todo{Comment (6) about notation O for the orbit: we don't use this notation any more.}

\bl \label{lem: a bunch}  Let $\R^k \curvearrowright X$ be a totally partially hyperbolic action. For any $\lambda\in \Delta$, any  $a\in \ker \lambda$ is a center bunched partially hyperbolic diffeomorphism with respect to the center distribution  \color{black}$E^c\bigoplus \oplus_{\lambda(a)=0} E_{ \lambda}$\color{black}. In particular any generic singular element $a\in \ker \lambda$ is a center bunched partially hyperbolic diffeomorphism with respect to the center distribution \color{black} $E^c \oplus E_{ \lambda}$ if $-\lambda\notin \Delta$ or $E^c \oplus E_{ \lambda}\oplus E_{-\lambda}$ if $-\lambda\in \Delta$\color{black}. 
\el

%\subsection{ Normal forms for contracting foliations}

\subsection{Accessibility}\label{sec: acc} A finite collection of topological foliations $\mathcal F=\{\mathcal F_1$, \dots,  $\mathcal F_r\}$ %with $C^1$ leaves 
defines an accessibility relation on a topological manifold $X$. We say $x,y\in X$ are in the same {\em accessibility class} if they can be joined by a \color{black} continuous \color{black} path $\gamma : [0,n] \to X$ such that $\gamma|_{[i,i+1]}$ is contained in a single local leaf of one of the foliations $\mc F_{r_i}$. Each restriction $\gamma|_{[i,i+1]}$ is called a {\it leg} of the path $\gamma$. \color{black}In some cases, we are primarily interested the corner points $\gamma(i)$ of the path $\gamma$, which motivates the following definition.
\begin{definition}\label{def: F_i path}Given a finite collection of topological foliations $\mathcal F=\{\mathcal F_1$, \dots,  $\mathcal F_r\}$ of a topological manifold $X$, an $\{\mathcal F_1,\dots \mc F_r\}$-path connecting $x$ to $y$ is a sequence $[y_0=x, y_1, \dots, y_{N-1}, y_N=y]$ such that if for each $i$, $y_i, y_{i+1}$ lie in the same local $\mc F_{r_i}$-leaf for some $1\leq r_i\leq r$.
\end{definition}

\color{black}  

If there is a single accessibility class in $X$ we say that collection $\mathcal F$ is accessible. Given a partially hyperbolic diffeomorphism $f$ we say that $f$ is accessible if $\{W^s_f, W^u_f\}$ is accessible. A $\{W^s_f, W^u_f\}$-path $[y_0,\dots, y_N]$ will also be referred to as an $su$-path.
\color{black}
For a totally partially hyperbolic $\mathbb R^k$ action the coarse Lyapunov distributions $E_\la$, $\la\in \Delta$, are integrable to coarse Lyapunov foliations $W^\lambda$ (see Section \ref{PHactions}) and this action is accessible iff the collection of coarse Lyapunov foliations is accessible. Recall that in Definition \ref{def:strongly} we defined such an action to be super accessible if every $\ker\lambda$,  $\lambda \in \Delta$, contains an accessible partially hyperbolic element. Moreover we have:  
%by appealing to Lemma \ref{lem: a bunch} we have

%As before $\Delta$ denotes the set of coarse Lyapunov functionals.   Let  $\Delta(\hat \lambda)$ denote the set of coarse Lyapunov functionals \emph{not proportional} to $\lambda$.  A generic singular element $a \in \ker \lambda$ is \emph{$\Delta(\hat{\lambda})$-accessible} if any two points in $X$ can be connected by a broken path whose legs lie in leaves of foliations $W^\nu$ where $\nu\in \Delta(\hat \lambda)$.

%\begin{definition}
%\label{def:strongly}
 %   If \color{black}$\R^k  \curvearrowright { X}$ is a partially hyperbolic  \color{black} action, we say that it is {\it strongly accessible} if for every $\lambda \in \Delta$, there exists $a \in \ker \lambda$ such that $a$ is $\Delta(\hat \lambda)$-accessible.
%\end{definition}

\begin{lemma} \label{gen-sing-acc} If the totally partially hyperbolic $\R^k$-action is super accessible then for any $\lambda\in \Delta$ any generic singular element $a\in \ker \lambda$ is accessible (with respect to the center distribution  $E^c \oplus E_{ \lambda}$ if $-\lambda\notin \Delta$ or $E^c \oplus E_{ \lambda}\oplus E_{-\lambda}$ if $-\lambda\in \Delta$).
\end{lemma} 
\begin{proof}
By the definition of super accessibility, for any   $\lambda\in \Delta$ and for any generic singular element $a\in \ker \lambda$, there exists $a'\in \ker \lambda$ which is accessible with respect to $E^c_{a'}=E^c\bigoplus \oplus_{\lambda'(a')=0} E_{ \lambda'}$ (by Lemma \ref{lem: a bunch}). It is not hard to see that $E^c_{a'}\supset E^c_a$, $E^u_a\supset E^u_{a'}, E^s_a\supset E^s_{a'}$, so the accessibility of $a$ follows from the accessibility of $a'$. 
\end{proof}
\color{black}

\subsection{Some regularity theorems}

The following commonly used result will play an important role in proving the regularity of the conjugacies in Section \ref{sec:regularity}. We will rely on both the $C^\infty$ and $C^{1,\theta}$ versions in the corresponding settings.

\begin{theorem}[{Journ\'{e}, \cite{Journe}}]
\label{thm:journe}
Let $F : X \to Y$ be a continuous function between $C^\infty$ manifolds, $r \in \Z_+ \cup \set{\infty}$, $\theta > 0$, and assume $X$ has complementary transverse foliations $\mc F_1$ and $\mc F_2$ with uniformly $C^{r,\theta}$-leaves. Then if the restriction of $F$ to the leaves of the foliations $\mc F_i$ is uniformly $C^{r,\theta}$, then $F$ is $C^{r,\theta}$. 
\end{theorem}

A version of the following result was first proved by Calabi and Hartman. Concerns about its proof were resolved when a new proof was provided by M. Taylor, whose formulation we use.

\begin{theorem}[Taylor, \cite{Taylor} Theorem 2.1]
\label{thm:taylor}
Let $\mc O$, $\Omega$ be open subsets of $\R^n$ and carry metric tensors $g = (g_{jk})$ and $h = (h_{jk})$, respectively. Assume $r \in \Z_+ \cup \set{0}$, $\theta \in [0,1)$, $r+\theta > 0$, and $g_{jk} \in C^{r,\theta}(\mc O)$, $h_{jk} \in C^{r,\theta}(\Omega)$. Let $\varphi : \mc O \to \Omega$. The following are equivalent:

\begin{itemize}
\item[(1)] $\varphi$ is a distance-preserving homeomorphism,
\item[(2)] $\varphi$ is bi-Lipschitz and $\varphi^*h(x) = g(x)$, for almost every $x \in \mc O$,
\item[(3)] $\varphi$ is a $C^1$ diffeomorphism and $\varphi^*h = g$,
\item $\varphi$ is a diffeomorphism of class $C^{r+1,\theta}$ and $\varphi^*h = g$.
\end{itemize}
\end{theorem}

\begin{theorem}[Chernoff-Marsden, \cite{CM}]\label{MZ}
Let $r \in \Z$, $r \ge 1$, $X$ be a $C^r$ manifold, $G$ be a Lie group, and $\rho : G \curvearrowright X$  be a group action such that for every $g \in G$, $\rho(g)$ is a $C^r$ diffeomorphism. Then $\rho : G \times X \to X$ is a $C^r$ action.
\end{theorem}

\subsection{Invariance principle and deriving continuity properties \color{black}}\label{InvPrinI} 
In the course of the proof of our main results we will need to \color{black}
 derive continuity of certain objects from the leafwise continuity properties along different foliations.  
    It is an easy fact that a two variable function $\phi(x,y)$ which is continuous in $x$ and $y$ may not be a continuous function, e.g. a function $\phi(x,y)= \frac{xy}{x^2+y^2}$,  $(x,y)\neq (0,0)$, defined to be $0$ at $(0,0)$.
    
    %$\phi(x,y)=\begin{cases}\frac{xy}{x^2+y^2}, (x,y)\neq (0,0); \\ 0, (x,y)=(0,0).\end{cases}$
    In order to derive global continuity, we will need a stronger continuity property along a given foliation, as in \cite[Definition 2.13]{ASV}

 \begin{definition}\label{def: F-cont}
 Let $\mathcal F$ be a topological foliation of a topological manifold $X$ and let $\sigma$ be a %measurable 
 section %$\sigma: X \to \mc L(T\mc F)$ 
 of a continuous fiber bundle over $X$. Then $\sigma$ is called $\mathcal F$-{\it continuous} if the map $(x, y, \sigma(x))\mapsto \sigma(y)$ is continuous on the set of pairs of points $(x,y)$ such that $y\in \mathcal F_{loc}(x)$. More explicitly, $\sigma$ is $\mathcal F$-continuous if for every $\epsilon>0$ and
every $(x,y)$ with $y\in \mathcal F_{loc}(x)$ there exists $\delta>0$ such that $d(\sigma(y), \sigma(y')) < \epsilon$
for every $(x',y')$ with $$y'\in  \mathcal F_{loc}(x'),~~ d(x,x') < \delta,~~ d(y,y') < \delta,~~ d(\sigma(x), \sigma(x')) < \delta.$$ 

It is implicit in this formulation that the fiber bundle has been
trivialized in the neighborhoods of the fibers.
\end{definition}

The rest of this section is motivated by the following fundamental result: 

\begin{theorem}[Theorem E of \cite{ASV}]\label{ASV}
Let $f$ be a $C^1$ partially hyperbolic, accessible diffeomorphism with stable and unstable foliations $W^s$ and $W^u$, respectively.  Then every $W^s$- and $W^u$-continuous section of a continuous fiber bundle, is continuous. % is continuous in
\end{theorem}

The proof of the above Theorem in \cite{ASV} relies crucially on the following

\begin{proposition}[Proposition 7.2 of \cite{ASV}]\label{prop: 7.2 ASV}Suppose $f$ is an accessible partially hyperbolic diffeomorphism on a compact manifold $X$. Given $x_0\in X$, there is $w\in X$ and an $su$-path %with legs 
$[y_0(w)=x_0,\dots, y_N(w)=w]$ connecting $x_0$ to $w$ and satisfying the following property: for any
$\epsilon > 0$ there exist $\delta > 0$ and $L > 0$ such that for every $z \in B(w,\delta)$ there exists an $su$-path 
%with legs 
$[y_0(z),... , y_N(z)]$ connecting $x_0$ to $z$ and such that
$$d(y_j(z),y_j(w)) < \epsilon, \text{ and } d_{W^\ast}(y_{j-1}(z), y_j(z)) < L, j = 1,\dots N,$$
where $d_{W^{\ast}}$ denotes the distance along the stable or unstable leaf
common to the two points. 
\end{proposition}

In what follows we obtain different versions of Theorem \ref{ASV} suitable for application to our setting. Proposition \ref{prop: fancy inv prin} is for lifted dynamics on a principal bundle, Proposition \ref{prop: strong ASV thm E} is suitable for application in the context of group actions where there are several invariant foliations in place of stable and unstable ones, finally Corollaries \ref{coro: Rk Anosov cont ASV},  \ref{coro: Rk Anosov with bundle ASV} are suitable for application to Anosov actions. Proofs of  these results follow verbatim the proof of Theorem E in \cite{ASV} modulo Proposition \ref{prop: 7.2 ASV} of which we produce corresponding versions below. 

%\color{black} SHOULD we denote $M$ by $X$ instead? We used $M$ for the compact group..\color{black}
\begin{proposition}\label{prop: fancy inv prin} Let $\hat X$ be a principal bundle over a smooth manifold $X$ with a compact structure group $P$. 
Let $f$ be a partially hyperbolic accessible diffeomorphism on $X$ which lifts to a principal bundle morphism $\hat f$ on $\hat X$. Assume that $\hat f$ preserves continuous foliations $\tilde{W}^{u,s}$ which project to $W^{u,s}_f$ of $X$ respectively. Suppose $\sigma$ is a %measurable 
map $\hat X\to \hat X$ which is 
$\mathcal F$-continuous for each  $\mathcal F\in \{P, \tilde W^u, \tilde W^s\}$. Then $\sigma$ is continuous. \end{proposition}
\begin{remark} In our  application of this Proposition in Section \ref{BPsubbundle} the continuous foliations $\tilde W^{u,s}$ will come from the holonomies of automorphims of certain principal bundle extensions of partially hyperbolic systems,  continuity  of which is obtained in Appendix \ref{Appendix-holonomies}.
\end{remark}
\begin{proof}
Using Proposition \ref{prop: 7.2 ASV}, \cite{ASV} showed Theorem E, i.e. that any %measurable 
section of a continuous fiber bundle which is $W^s$-continuous and $W^u$-continuous is actually a continuous section. By exactly the same argument, to show our Proposition \ref{prop: fancy inv prin} we only need to show the following ``$\tilde s\tilde u P$'' version of Proposition \ref{prop: 7.2 ASV}. We define an $\tilde s\tilde uP$-path %with legs 
$[\hat y_0,\dots, \hat y_N]$ in $\hat {X}$ similar to that of an $su$-path in $X$, but we allow $\hat y_{j}, \hat y_{j-1}$ to be in the same local $\tilde W^s$, $\tilde W^u$ or $P$-leaves.
\begin{lemma}\label{lem: sup eff acc} Let $f,\hat f, X, \hat X, \tilde W^{u,s}, P$ as in Proposition \ref{prop: fancy inv prin}. Given $\hat x_0\in \hat X$, there is $\hat w\in \hat X$ and an $\tilde s\tilde uP$-path %with legs 
$[\hat y_0(\hat w)=\hat x_0,\dots, y_N(\hat w)=\hat w]$ connecting $\hat x_0$ to $\hat w$ and satisfying the following property: for any
$\epsilon > 0$ there exist $\delta > 0$ and $L > 0$ such that for every $\hat z \in B(\hat w,\delta)$ there exists an $\tilde s\tilde uP$-path
 %with legs 
$[\hat y_0(\hat z),... , y_N(\hat z)]$ connecting $\hat x_0$ to $\hat z$ and such that
$$d(\hat y_j(z),\hat y_j(w)) < \epsilon, \text{ and } d_{W^\ast}(\hat y_{j-1}(\hat z), \hat y_j(\hat z)) < L, j = 1,\dots N,$$
where $d_{W^{\ast}}$ denotes the distance along the $\tilde W^s, \tilde W^u, P$-leaf
common to the two points.
\end{lemma}
\begin{proof} We fix $\hat x_0\in \hat X$ and let $x_0:=\pi(\hat x_0)$. By Proposition \ref{prop: 7.2 ASV}, there is $w\in X$ and an $su$-path %with legs 
$[y_0(w)=x_0,\dots, y_N(w)=w]$ connecting $x_0$ to $w$ satisfying properties in Proposition \ref{prop: 7.2 ASV}. Take an arbitrary lift $\hat w\in \hat X$ of $w$, then by the lifting properties of foliations $\tilde W^{u,s}$ we get an $\tilde{s}\tilde{u}P$-path %with legs 
$$[\hat y_0(\hat w)=\hat x_0,\dots, \hat y_N(\hat w), \hat y_{N+1}(\hat w)=\hat w],$$such that $$\pi(\hat y_i(\hat w))=y_i(w), 0\leq i\leq N, \text{ and }\hat y_i, \hat y_{i+1} \text{ are in the same }\tilde W^s \text{ or }\tilde W^u \text{ leaf for }0\leq i\leq N-1.$$
(i.e. only the last segment in the path is a  $P$-path). We claim this $\hat s\hat u P$-path we picked satisfies Lemma \ref{lem: sup eff acc}. In fact for any $\hat z$ close to $\hat w$, let $\pi(\hat z)=z$, then we apply Proposition \ref{prop: 7.2 ASV} to get a $su$-path $[y_0(z)=x_0,\dots, y_N(z)=z]$ on $X$ connecting $x_0$ to $z$ such that $y_i(z)$ is close to $y_i(w)$, $0\leq i\leq N$. Again by the lifting property of $\tilde W^{u,s}$ there is %$[y_0(z),\dots, y_N(z)]$ to 
an $\hat s\hat u$-path $[\hat y_0(\hat z)=\hat x_0,\dots, \hat y_N(\hat z)]$ in $\hat X$ connecting $\hat x_0$ to some $\hat y_N(\hat z)$ in the same $P$-fiber as $\hat z$, and $\pi(\hat y_i(\hat z))=y_i(z)$. By continuity  of $\tilde W^{u,s}$ foliations and induction, each $\hat y_i(\hat z)$ is close to $\hat y_i(\hat w)$ for $0\leq i\leq N$, then letting $\hat y_{N+1}(\hat z)=\hat z$, we get the $\hat s\hat u P$-path we need, i.e. $$[\hat y_0(\hat z)=\hat x_0,\dots, \hat y_N(\hat z), \hat y_{N+1}(\hat z)=\hat z].$$
(if necessary we may let $L$ be much greater than the diameter of $P$.)
\end{proof}
Proposition \ref{prop: fancy inv prin} now follows directly from the part of the proof of Theorem E in  \cite{ASV} which  assumes Proposition \ref{prop: 7.2 ASV}, where instead of Proposition \ref{prop: 7.2 ASV} we use Lemma \ref{lem: sup eff acc}. 
\end{proof}

\begin{proposition}\label{prop: strong ASV thm E}Let $f$ be a $C^1$ partially hyperbolic, accessible diffeomorphism on a compact manifold $X$ with stable and unstable foliations $W^s$ and $W^u$, respectively. Assume that $E^s=\oplus_{i=1}^p E^s_i$, $E^u=\oplus_{j=1}^q E^u_j$ and each of $E^s_i$ and $E^u_j$ is tangent to a topological foliation $\mathcal F^s_i$ or $\mathcal F^u_j$ respectively. Suppose $\sigma$ is a %measurable
section of a continuous fiber bundle over $X$ which is $\mathcal F^s_i$- and $\mathcal F^u_j$-continuous. Then $\sigma$ is continuous.
\end{proposition}
\begin{proof} Again, we use  the part of the proof of Theorem E, \cite{ASV} which assumes Proposition \ref{prop: 7.2 ASV}, where Proposition \ref{prop: 7.2 ASV} is substituted by 
%The proof is very similar to that of Proposition  \ref{prop: fancy inv prin}, i.e. we only need to show the following variant of Proposition \ref{prop: 7.2 ASV}. The remaining proof of Proposition \ref{prop: strong ASV thm E} follows the part of the proof of Theorem E, \cite{ASV} when assuming Proposition \ref{prop: 7.2 ASV}.
\begin{lemma}\label{lem: multi fol ver 7.2 ASV}Let $f, X, \mathcal F^s_i, \mathcal F^u_j,  W^{u,s}$ as in Proposition \ref{prop: strong ASV thm E}. Given $x_0\in X$, there is $w\in  X$ and an $\{\mathcal F^s_i, \mathcal F^u_j\}$-path %with legs 
$[y_0(w)=x_0,\dots, y_N(w)=w]$ connecting $x_0$ to $w$ and satisfying the following property: for any
$\epsilon > 0$ there exist $\delta > 0$ and $L > 0$ such that for every $z \in B(w,\delta)$ there exists an $\{\mathcal F^s_i, \mathcal F^u_j\}$-path
%with legs 
$[y_0(z),... , y_N(z)]$ connecting $x_0$ to $z$ and such that
$$d(y_j(z),y_j(w)) < \epsilon, \text{ and } d_{W^\ast}(y_{j-1}(z), y_j( z)) < L, j = 1,\dots N,$$
where $d_{W^{\ast}}$ denotes the distance along the $\mathcal F^s_i, \mathcal F^u_j$-leaf common to the two points.
\end{lemma}
\begin{proof}Given $x_0\in X$, first apply Proposition \ref{prop: 7.2 ASV} we get $w\in X$ and an $su$-path %with legs 
$[y_0(w)=x_0,\dots, y_N(w)=w]$ connecting $x_0$ to $w$ and $\delta$, $L$ satisfying the property listed in Proposition \ref{prop: 7.2 ASV}. Without loss of generality we may assume $L$ is very small. By transversality of $E_i^s$ and $E_j^u$ we know locally $W^s$ is a product of $\mathcal F^s_i$ foliations, and $W^u$ is a product of $\mathcal F^u_j$ foliations. By the smallness of $L$ and the local product structure, we know each pair $y_{j-1}(w), y_j(w)$ can be connected by an $\{\mathcal F^s_i, \mathcal F^u_j\}$ path $[y'_{j-1,0}(w)=y_{j-1}(w), y'_{j-1, 1}(w),\dots, y'_{j-1, m_{j}-1}(w), y'_{j-1, m_{j}}(w)=y_{j,m_j}]$, with bounded length and bounded number of legs. Then the 
$\{\mathcal F^s_i, \mathcal F^u_j\}$-path $$[y'_{0,0}(w)=x_0,... , y'_{0,m_1}(w)=y_1(w), y'_{1,1}(w),\dots, y'_{1,m_2}(w)=y_2(w),\dots, y'_{N-1, m_N}(w)=y_N(w)=w]$$ satisfies Lemma \ref{lem: multi fol ver 7.2 ASV}, due to the local product structure and Proposition \ref{prop: 7.2 ASV}.
\end{proof}
\end{proof}
By using similar arguments, we get parallel results for $\R^k$-Anosov action, and the corresponding bundle version for a lift of an Anosov action to a principal fiber bundle. 
\begin{corollary}\label{coro: Rk Anosov cont ASV}Let $\R^k \curvearrowright X$ be an Anosov action on a compact manifold $X$ with a regular element $a\in \mathbb R^k$ whose stable and unstable foliations are $W^s$ and $W^u$, respectively. Assume that $E^s=\oplus_{i=1}^p E^s_i$, $E^u=\oplus_{j=1}^q E^u_j$ and each of $E^s_i$ and $E^u_j$ is tangent to a topological foliation $\mathcal F^s_i$ or $\mathcal F^u_j$ respectively. Suppose $\sigma$ is an  %measurable
$\R^k$-invariant section of a continuous fiber bundle over $X$ which is $\mathcal F^s_i$- and $\mathcal F^u_j$-continuous. Then $\sigma$ is continuous.
\end{corollary}
\begin{proof}$\R^k$-invariance of $\sigma$ implies that $\sigma$ is ($\R^k$-orbit foliation)-continuous in the sense of Definition \ref{def: F-cont}. Therefore the claim is an easy corollary of the proof of Proposition \ref{prop: strong ASV thm E} and the local product structure of $W^s, W^u$ and $\R^k$-orbit of the Anosov action.
\end{proof}

\begin{corollary}\label{coro: Rk Anosov with bundle ASV}Let $\R^k \curvearrowright X, a, X, W^s, W^u, \mathcal F_i^s, \mathcal F_j^u$ are defined as in Corollary \ref{coro: Rk Anosov cont ASV}. Assume that  $a$ lifts to a principal bundle morphism $\hat a$ on $\hat X$ over $X$ with fiber (group) $P$, such that $\hat a$ preserves continuous foliations $\hat {\mathcal F}_i^s, \hat {\mathcal F}_j^u$ which project to $\mathcal F_i^s, \mathcal F_j^u$ of $X$ respectively. Suppose $\sigma$ is a %measurable
map $\hat X\to \hat X$ which is $\mathcal F$-continuous for each $\mathcal F\in \{P,\hat{\mathcal F}_i^s, \hat {\mathcal F}_j^u, \mathbb R^k-\mathrm{orbit} \}$. Then $\sigma$ is continuous.
\end{corollary}
\begin{proof}
The proof is essentially the same as Corollary \ref{coro: Rk Anosov cont ASV}, using the local product structure of $P,\hat{\mathcal F}_i^s, \hat {\mathcal F}_j^u$, and the $\mathbb R^k-\mathrm{orbit}$.    
\end{proof}
%Then every $W^s$- and $W^u$-continuous section of a continuous fiber bundle, is continuous.
\color{black}
%UNIFY NOTATION FOR ACC PATH SINCE WE HAVE SU PATH SUP PATH F PATH BLAHBLAH, WE COULD DO IT AT SECTION 5.2.   

\subsection{Invariance principle and deriving additional regularity\color{black}}\label{InvPrinII}

\color{black}
We first define the additional regularity we will work with.
\begin{definition}\label{def: partial Hld}

\begin{itemize}
    \item[(1)] Let $f$ be a partially hyperbolic diffeomorphism on a compact manifold $X$. A function (or a map, a section, etc.) on $X$ is called {\it partially H\"older}  if it is continuous, and also uniformly H\"older continuous along $W^s_f$ and $W^u_f$.
    \item[(2)] Let $\R^k \curvearrowright X$ be a partially hyperbolic $\R^k$-action on a compact manifold $X$. A function (or a map, a section, etc.) on $X$ is called {\it H\"older 
    along 
    coarses} %foliations 
    if it is continuous, and also uniformly H\"older continuous along all coarse Lyapunov foliations $W^\lambda$. 
\end{itemize}
\end{definition}
\begin{remark}
By local product structure of coarse Lyapunov foliations (Proposition \ref{prop:circ}) within stable and unstable foliations of partially hyperbolic elements, we know that a map is H\"older along coarses if and only if it is \textit{partially H\"older} (in the sense of (1) of Definition \ref{def: partial Hld}). 
%for all partially hyperbolic elements $a\in \R^k$.
\end{remark}
%Define  (partially H\"older continuous).... 
\color{black}

The following is a variant  of  results based on the invariance principle as they appear in \color{black} \cite{ASV} and \color{black} \cite{KalSad}, but similar results can also be found in the initial works on the invariance principle \cite{AV}. \color{black}Recall that a $\beta$-H\"older continuous linear cocycle $F:E\to E$ of a $\beta$-H\"older vector bundle $E$ over a partially hyperbolic diffeomorphism $f:X\to X$ is called \emph{fiber bunched}, if 
$$\|F(x)\|\cdot \|F(x)\|^{-1}\cdot \nu(x)^\beta <1, ~  \|F(x)\|\cdot \|F(x)\|^{-1}\cdot \hat\nu(x)^\beta <1.$$
where $\nu, \hat \nu$ are defined in \eqref{def: ph dif def}.
\color{black}

\begin{theorem}\label{thm: acc imp hder} %\cite{KalSad}
Let $f$ be a \color{black} $C^{2}$ \color{black}  partially hyperbolic volume preserving,  %strongly 
center-bunched, accessible diffeomorphism of a closed manifold $X$. Let $\pi : E \to X$ be a H\"older vector bundle over $X$, and let $\phi$ be a fiber bunched H\"older linear cocycle over $f$, $\phi : E \to E$, with one Lyapunov exponent with respect to volume. Then the following hold:
\begin{itemize}
\item[(1)] Any almost-everywhere defined $\phi$-invariant measurable sub-bundle $V \subset E$ coincides almost everywhere with a \color{black} partially \color{black} H\"older one.
\item[(2)] Any almost-everywhere defined $\phi$-invariant conformal structure \color{black} of a measurable invariant sub-bundle $V\subset E$ \color{black} coincides almost everywhere with a \color{black} partially \color{black} H\"older one.
\end{itemize}
\end{theorem}
\color{black}
\begin{proof}For part (1), by Theorem 3.3 in \cite{KalSad}, any almost-everywhere defined $\phi$-invariant measurable sub-bundle $V$ coincides almost everywhere with a continuous bundle $V'$. By continuity $V'$ is $\phi$-invariant. But exactly the same proof as that in \cite{KalSad} (or using a version of Avila-Viana invariance principle, Theorem C. in \cite{ASV}), $V'$ is actually invariant under stable and unstable holonomy everywhere. %\footnote{DO WE NEED TO ADD THE DEFINITION OF FIBERED BUNCHED COCYCLE OVER PARTIALLY HYPERBOLIC SYSTEMS? or put all these fundamental materials in the appendix? or we directly put the whole proof of Theorem 8.2 to the appendix?}. 
By H\"older continuity of stable and unstable holonomies (see \cite{ASV} or \cite{KalSad}), $V'$ is partially H\"older. 

The proof of part (2) is not a direct consequence of Theorem 3.1 in \cite{KalSad}, due to lack of H\"older continuity of $V'$. By part (1) we know $V$ coincides with a continuous invariant and holonomy-invariant sub-bundle $V'$. 
Therefore we may consider a new cocycle  $\phi|_{V'}$ which is only partially H\"older. The point is that the stable and unstable Holonomies of $\phi|_{V'}$  are well-defined, given by the restriction of the canonical stable and unstable holonomies of $\phi$ to $V'$. Then by Proposition 4.4 of \cite{KalSad}, any almost-everywhere defined $\phi$-invariant conformal structure $\eta$ within $V$ is essentially invariant under stable and unstable holonomies of $\phi|_{V'}$. Notice that here the  H\"older continuity of $\phi|_{V'}$ is not needed (see Proposition 4.4 in \cite{KalSad}). Therefore by the same argument as the last paragraph of the proof of Proposition 4.4, that is by applying Theorem E of \cite{ASV}, we get that the conformal structure $\eta$ coincides with a continuous conformal structure $\eta'$ of $V'$ which is holonomy (that of $\phi|_{V'}$)-invariant. Since the stable and unstable holonomies of $\phi|_{V'}$ are partially H\"older, %(here we use part (1).), 
$\eta'$ is also partially H\"older.
\end{proof}

 \color{black}

\section{Algebraic preliminaries}
\label{sec:group-prelims}

%\subsection{Some Lie group theory facts}

%\begin{definition}  The Weyl group with respect to $A$ is defined by $N_G(A)/C_G(A)$, where The Weyl group with respect to $A$ is defined by $N_G(A)/C_G(A)$ where $N_G(A)$ is the normalizer of $A$. Denote by $w_\beta$ the reflection about the hyperplane perpendicular to the root $\beta$.
%\end{definition}

\color{black}

\subsection{Free products of topological groups}
\label{sec:free-prods}
Let $U_1,\dots,U_r$ be topological groups. The {\it topological free product} of the $U_i$, denoted $\mc P = U_1 * \dots * U_r$ is a topological group whose underlying group structure is exactly the usual free product of groups. That is, elements of $\mc P$ are given by

\[ u_1 * \dots * u_N \]

\noindent where each $u_k \in U_{i_k}$ for some associated sequence of symbols $i_k \in \set{1,\dots,r}$. We call the sequence $(i_1,\dots,i_k)$ the {\it combinatorial pattern} of the word. Each term $u_k$ is also called a {\it leg} and each word is also called a {\it path}. This is because in the case of a free product of connected Lie groups, the word can be represented by a path beginning at $e$, moving to $u_N$, then to $u_{N-1} * u_N$, and so on through the truncations of the word. The multiplication is given by concatenation of words, and the only group relations are given by

\begin{eqnarray}
\label{eq:freep-rel1} u * v & = & uv, \mbox{ if }u,v \mbox{ belong to the same }U_i \mbox{ and }\\
\label{eq:freep-rel2} e^{(i)} & = & e \in \mc P.
\end{eqnarray}

Notice that the relations \eqref{eq:freep-rel1} and \eqref{eq:freep-rel2} give rise to canonical embeddings of each $U_i$ into $\mc P$. We therefore identify each $U_i$ with its embedded copy in $\mc P$. The usual free product is characterized by a universal property: given a group $H$ and any collection of homomorphisms $\varphi_i : U_i \to H$, there exists a unique homomorphism $\Phi : \mc P \to H$ such that $\Phi|_{U_i} = \varphi_i$. The group topology on $\mc P$ may be similarly defined by a universal property, as first proved by Graev \cite{graev1950}:

\begin{proposition}\label{prop: gr}
There exists a unique topology $\tau$ on $\mc P$ (called the {\normalfont free product topology}) such that

\begin{itemize}
\item[(1)] each inclusion $U_i \hookrightarrow \mc P$ is a homeomorphism onto its image, and
\item[(2)] if $\varphi_i : U_i \to H$ are continuous group homomorphisms to a topological group $H$, then the unique extension $\Phi$ is continuous with respect to $\tau$.
\end{itemize}
\end{proposition}

In the case when each $U_i$ is a Lie group (or more generally, a CW-complex), Ordman found a more constructive description of the topology \cite{ordman1974}. Indeed, the free product of Lie groups is covered by a disjoint union of {\it combinatorial cells}. 

\begin{definition}
\label{def:comb-cells}
Let $\mc P$ be the free products of groups $U_{\beta}$, where the $\beta$ ranges over some indexing set $\Delta$. A {\normalfont combinatorial pattern} in $\Delta$ is a finite sequence $\bar{\beta} = (\beta_1,\dots,\beta_n)$ such that $\beta_i \in \Delta$ for $i=1,\dots,n$. For each combinatorial pattern $\bar{\beta}$, there is an associated {\normalfont combinatorial cell} $C_{\bar{\beta}} = U_{\beta_1} \times \dots \times U_{\beta_n}$. If each $U_\beta$ is a topological group, $C_{\bar{\beta}}$ carries the product topology from the topologies on $U_{\beta_i}$. Notice that each $C_{\bar{\beta}}$ has a map $\pi_{\bar{\beta}} : C_{\bar{\beta}} \to \mc P$ given by $(u_1,\dots,u_N) \mapsto u_1 * \dots * u_N$. Furthermore, if $C = \bigsqcup_{\bar{\beta}} C_{\bar{\beta}}$, $C$ carries a canonical topology in which each $C_{\bar{\beta}}$ is a connected component. Finally, $\pi : C \to \mc P$ is defined by setting $\pi(x) = \pi_{\bar{\beta}}(x)$ when $x \in C_{\bar{\beta}}$ (note that $\pi$ is onto).
%For each finite word $\bar{i} = (i_1,\dots,i_N)$, let $C_{\bar{i}} = U_{i_1} \times \dots \times U_{i_N}$ be the combinatorial cell for $\bar{i}$, with its usual product topology. Notice that each $C_{\bar{i}}$ has a map $\pi_{\bar{i}} : C_{\bar{i}} \to \mc P$ given by $(u_1,\dots,u_N) \mapsto u_1^{(i_1)} * \dots * u_N^{(i_N)}$. Furthermore, if $C = \bigsqcup_{\bar{i}} C_{\bar{i}}$ and $\pi : C \to \mc P$ is defined by setting $\pi(x) = \pi_{\bar{i}}(x)$ when $x \in C_{\bar{i}}$, then $\pi$ is onto.
\end{definition}

\begin{lemma}[\cite{vinhageJMD2015} Proposition 4.2]
\label{lem:continuity-criterion}
If each $U_i$ is a Lie group, $\tau$ is the quotient topology on $\mc P$ induced by $\pi$. In particular, $f : \mc P \to Z$ is a continuous function to a topological space $Z$ if and only if its pullback $f \of \pi_{\bar{\beta}}$ to $C_{\bar{\beta}}$ is continuous for every combinatorial pattern $\bar{\beta}$.
\end{lemma}

\begin{corollary}
\label{cor:connected}
If each $U_i$ is a connected Lie group, $\mc P$ is path-connected and locally path-connected.
\end{corollary}

Fix an indexing set $\Delta$ (which, in our applications, will be the set of coarse Lyapunov exponents), and for each $\alpha \in \Delta$, let $U_\alpha$ be an associated Lie group. Define $\mc P = \mc P_\Delta = U_{\alpha_1} * \dots * U_{\alpha_n}$ be the free product of the groups $U_\alpha$, $\alpha \in \Delta$. Given continuous actions $U_\alpha \curvearrowright X$, we may induce a continuous action of $\mc P$ on $X$ by setting:

\[ u_1 * u_2 * \dots * u_N \cdot x = u_1(u_2(\dots(u_N(x))\dots)) \]

This can be observed to be an action of $\mc P$ immediately, and continuity can be checked with either the universal property (considering each action $U_\alpha \curvearrowright X$ is a homomorphism from $U_\alpha$ to $\Homeo(X)$) or directly using the criterion of Lemma \ref{lem:continuity-criterion}. 

Suppose now that all groups $U_\alpha$ are nilpotent and simply connected. Given a word $u_1 * \dots * u_m$ (which we often call a path as discussed above), we may associate a path in $X$ defined by:

\[ \gamma\left(\frac{s+k-1}{m}\right) = u_{i_{m-k+1}}^t (x_{k-1}), \; s \in [0,1], \; k = 1,\dots,m, \]

\noindent where $u^t$ is the one parameter subgroup passing through $u$, $x_0$ is a base point and $x_k = u_{i_{m-k+1}}(x_{k-1})$. This gives more justification for calling each term $u_k$ a leg. The points $x_k$ are called the {\it break points} or {\it switches} of the path.

%\begin{remark}
%\label{rem:paths-no-action}
%Given a collection of oriented one-dimensional foliations $\mc W^{\beta_1},\dots,W^{\beta_N}$ and some Riemannian metric on $X$, one can define flows $\eta^i_t$ according to unit speed flow along $\mc W^{\beta_i}$. Therefore, even if there is no distinguished Riemannian metric to use, we will often use the terms paths, legs and break points in the foliations $\mc W^{\beta_i}$ (especially in Part I). The combinatorial patterns of such paths still make sense even without reference to a fixed parameterization.
%\end{remark}
%\begin{corollary}
%\label{cor:lie-from-const}
%If $\mc P$ is the free product of Lie groups, and $\eta : \mc P \curvearrowright X$ is an action of $\mc P$ on a manifold $X$ such that $\Stab_\eta(x)$ is independent of $x$, then $\eta$ factors through the action of a Lie group $H$.
%\end{corollary}

%\begin{proof}
%Let $\mc C = \Stab_\eta(x)$ be the common stabilizer of every point in $X$, and $H = \mc P / \mc C$.
%\end{proof}

Let $K$ be a compact group. We will work in the context of $\mathbb R^k\times K$ actions (see Section \ref{sec:top-anosov}). We assume that for each \color{black}$g \in \R^k\times K $\color{black}, there is an associated family of automorphisms $g_* : U_\alpha \to U_\alpha$ for every $\alpha \in \Delta$, and that the map $g \mapsto g_*$ is a homomorphism from \color{black} $\R^k\times K$ \color{black} to $\prod_{\alpha \in \Delta} \Aut(U_\alpha)$. Suppose that $u_1 * \dots * u_m \in \mc P$ is an element of combinatorial length $m$ with combinatorial pattern $(\alpha_1,\dots, \alpha_m)$. Then define $g_* : \mc P \to \mc P$ by:

\begin{equation}
\label{eq:induced-auto}
    u_1 * \dots * u_m \mapsto (g_*u_1) * \dots * (g_* u_m). 
\end{equation}

One may check that $g_*$ is a well-defined automorphism of $\mc P$ using relations \eqref{eq:freep-rel1} and  \eqref{eq:freep-rel2}, and noting that its inverse is $(g^{-1})_*$.

\begin{definition}
\label{def:P-groups}
Let \color{black} $\hat{\mc P} = (\R^k \times K )\ltimes \mc P$ \color{black}, with the semidirect product structure given by

\[ (g_1,\rho_1) \cdot (g_2,\rho_2) = (g_1g_2,({g_2^{-1}})_*(\rho_1)* \rho_2). \]

{\color{black}If $\Delta' \subset \Delta$ is a subset, let $\mc P_{\Delta'} \subset \mc P$ denote the subgroup of $\mc P$ generated by the groups $U_\alpha$, $\alpha \in \Delta$, and $\hat{\mc P}_{\Delta'}$ denote the semidirect product $(\R^k \times K) \ltimes \mc P_{\Delta'}$.}
%\noindent where the group operation of $\R^k$ is written multiplicatively.
\end{definition}

The following proposition follows almost immediately from the definitions. A proof of the case when each $U_\alpha \cong \R$ can be found in \cite[Proposition 4.7]{Spatzier-Vinhage}.

\begin{proposition}
\label{prop:integrality}
%Let $\psi_a$ denote the automorphism of $\mc P$ as defined in Definition \ref{def:P-groups}. Let $\mc C$ be any closed normal subgroup of $\mc P$, and
Let $\mc C$ be a closed, normal subgroup of $\mc P$, and $H = \mc P / \mc C$ be the corresponding topological group factor of $\mc P$. If $g_*\mc C = \mc C$ for all $g \in \color{black}\R^k\times K \color{black}$, then $g_*$ descends to a continuous homomorphism $g_*$ of $H$. Furthermore, if $H$ is a Lie group with Lie algebra $\mf h$, and $\pi_\alpha : U_\alpha \to H$ denotes the composition of inclusion into $\mc P$ and projection onto $H$, then

\begin{itemize}
\item[(1)] each generating group $U_\alpha \subset \mc P$, $d\pi_\alpha(\Lie(U_\alpha))$ is an invariant subspace of $dg_*$ for every $g \in \R^k\times K \color{black}$,
\item[(2)] if $X \in \Lie(U_\alpha)$ and $Y \in \Lie(U_\beta)$ are eigenvectors of $dg_*$ with eigenvalues of modulus $\lambda_1$ and $\lambda_2$, respectively, then $[d\pi_\alpha X,d\pi_\beta Y] \in \Lie(H)$ is an eigenspace of $dg_*$ with eigenvalue of modulus $\lambda_1\lambda_2$, and
\item[(3)] if $Z = [Z_1,[Z_2,\dots,[Z_N,Z_0]\dots]]$, with $Z_k = X$ or $Y$ for every $k$, then $Z$ is an eigenvector of $dg_*$ with eigenvalue of modulus $\lambda_1^u\lambda_2^v$ for some $u,v \in \Z_+$.
\end{itemize}
\end{proposition}

\subsection{Lie Criteria}

In this subsection, we recall two deep results for Lie criteria of topological groups. The first criterion was obtained by Gleason and Palais \cite[Corollary 7.4]{gleason-palais}:

\begin{theorem}[Gleason-Palais]
\label{lem:gleason-palais}
If $G$ is a locally path-connected topological group which admits an injective continuous map from a neighborhood of ${  e \in }\, G$ into a finite-dimensional topological space, then $G$ is a Lie group.
\end{theorem}

Theorem \ref{lem:gleason-palais} has an immediate corollary for actions of the path group $\mc P$ defined in Section \ref{sec:free-prods}:

\begin{corollary}
\label{cor:lie-from-const}
If $\eta : \mc P \curvearrowright X$ is a group action on a topological space $X$, and there exists $x_0 \in X$ such that $\mc C := \Stab_\eta(x_0) \subset \Stab_\eta(x)$ for every $x \in X$, then $\mc C$ is normal and the $\eta$ action descends to $\mc P / \mc C$.  If there is an injective continuous map from $\mc P /\mc C$ to a finite-dimensional space $Y$, then $\mc P / \mc C$ is a Lie group.
\end{corollary}

\begin{proof}
We first show that $\mc C$ is normal. Let $\sigma \in \mc C$ and $\rho \in \mc P$. Then $\sigma\cdot(\rho \cdot x_0) = \rho \cdot x_0$, since $\sigma$ stabilizes every point of $X$. Therefore, $\rho^{-1}\sigma\rho \cdot x_0 = x_0$, and $\rho^{-1}\sigma \rho \in \mc C$, and $\mc C$ is a closed normal subgroup. By Corollary \ref{cor:connected}, $\mc P$, and hence all of its factors, are locally path-connected. Therefore, by Theorem \ref{lem:gleason-palais}, $\mc P / \mc C$ is a Lie group if it admits an injective continuous map to a finite-dimensional space.
\end{proof}

The second, more well-known, criterion was obtained by Gleason and Yamabe \cite[Proposition 6.0.11]{hilbert5-2014}, and plays a crucial role in the general solution of Hilbert's fifth problem:

%\begin{theorem}[{\cite[Theorem 4.6]{MontZipp1955}}]
%\label{thm:mont-zipp}
%Let $G$ be a locally compact group such that $G / G^\circ$ is compact, where $G^\circ$ is the identity component of $G$. Let $U$ be an arbitrary neighborhood of $e$. Then there exists in $U$ a compact normal subgroup $H$ such that $G/H$ has no small subgroups (and is hence a Lie group).
%\end{theorem}

\bt[Gleason-Yamabe]  \label{thm:gleason-yamabe}
Let $G$ be a locally compact group. Then there exists an open subgroup $G' \subset G$ such that, for any open neighborhood $U$ of the identity in $G '$ there exists a compact normal subgroup $K \subset U \subset G'$  such that $G'/K$ is isomorphic to a Lie group. Furthermore, if $G$ is connected, $G'=G$.
\et

Recall that a locally compact group $G$ has the {\em no small subgroups property} if for $G'$ as in Theorem \ref{thm:gleason-yamabe}, there exists a neighborhood $U \subset G'$ such $U$ does not contain any compact normal subgroup besides $\{e\}$.  Such a group $G$ then is automatically a Lie group, by  Theorem \ref{thm:gleason-yamabe}.  
%Finally, we will require some  structure theory for path-connected, locally compact groups.

The following corollary is often used, but a citation is not available. We provide a proof for completeness:

\begin{corollary}
\label{cor:lc-lpc-structure}
If $G$ is a separable, Hausdorff, locally compact, locally path-connected group, then $G$ is an inverse limit of Lie groups with compact kernels. That is, we may construct the following commutative diagram describing $G$:

\begin{center}
% https://tikzcd.yichuanshen.de/#N4Igdg9gJgpgziAXAbVABwnAlgFyxMJZAJgBoAGAXVJADcBDAGwFcYkQAdDgC3p2ADiAXxBDS6TLnyEUAFlIBGanSat2AgPrlR4kBmx4CRAMyLlDFm0QhNCnRIPSiZJTQtrrm4vb2TDM5AUzN1UrGw1jH30pIxRyYJVLdi4AYygIHAQhZRgoAHN4IlAAMwAnCABbJABWGhwIJFkxEvKqxHkQeqRTRI89CJ8yyu66hsQyXrC0DW9mkCG2ia7EIMn2abs5haR4zrGOxnoAIxhGAAU-J2tSrDzuHBAQpOsARwGt1p3RkbXXmceQIcTudLrEQDc7g8PsNELtlhN3GE3nYaEDThdHGCIfdBp9Yd8VqjjujQTJATBig8nn03tpskIgA
\begin{tikzcd}
                 &                      & G \arrow[ld, "q_3"'] \arrow[d, "q_2"'] \arrow[rd, "q_1"'] \arrow[rrd, "q_0"] &                      &     \\
\cdots \arrow[r] & G_3 \arrow[r, "p_3"'] & G_2 \arrow[r, "p_2"']                                                               & G_1 \arrow[r, "p_1"'] & G_0
\end{tikzcd}
\end{center}
Here, each $G_i$ is a Lie group, $p_i : G_i \to G_{i-1}$ and $q_i : G \to G_i$ are surjective homomorphisms satisfying $q_n = p_{n+1} \of q_{n+1}$, $\ker p_n$ is compact, and $\bigcap_{n=0}^\infty \ker q_n = \set{e} \subset G$.
\end{corollary}

\begin{proof}
Notice that since $G$ is connected, $G^\circ = G$, and $G / G^\circ$ is compact. Choose a sequence $U_n$ of neighborhoods of $e$ such that $\bigcap_{n=1}^\infty U_n = \set{e}$ (the existence of the sequence follows from separability). By Theorem \ref{thm:gleason-yamabe}, there exists a compact normal subgroup $\tilde{K}_n \subset U_n$ such that $G_n := G / \tilde{K}_n$ is a Lie group.  Let $K_n = \bigcap_{i=1}^n \tilde{K}_i$.

We claim that $G_n' := G / K_n$ is a Lie group. Indeed, since $G_n$ is a Lie group, there is an open set $V_n \subset G_n$ such that the only subgroup of $G_n$ contained in $V_n$ is $\set{e}$. Consequently, the only subgroups contained in $V_n' \subset G$, the preimage of $V_n$ in $G_n$, must be contained in $\tilde{K_n}$. Let $W_n' = \bigcap_{i=1}^n V_i' \subset G$, and notice that since each $V_i'$ is saturated by $\tilde{K_i}$, $W_n'$ is saturated by $K_n = \bigcap_{i=1} \tilde{K}_i$. Then $W_n'$ is a $K_n$-saturated neighborhood of $K_n \subset G$. If $W_n$ denotes the image of $W_n'$ in $G_n'$, then $W_n$ is open in $G_n$. This implies that if $L$ is a subgroup of $G_n'$ contained in $W_n$, then its preimage in $G$ is a subgroup contained in $W_n'$. By construction, it must be contained in each $V_n'$ and therefore, must lie inside $\bigcap_{i=1}^n \tilde{K_n} = K_n$. Therefore, $L = \set{e}$ and $G_n'$ is a Lie group, since it is locally compact and has no small subgroups.

%Let $K_1 = \tilde{K_1}$ and $K_n =  \tilde{K_n} \cdot K_{n-1}$ for $n \ge 2$. Inductively, since $\tilde{K_n}$ and $K_{n-1}$ are compact normal subgroups, $K_n$ is again a compact normal subgroup contained in $U_{n-1}$. By construction, each $K_n$ is a compact normal subgroup, $K_{n+1} \subset K_n$ and $\bigcap_{n=1}^\infty K_n = \set{e}$. Since $K_n \supset \tilde{K}_n$, $G_n := G / K_n$ is a factor of $G_n'$ and $G_n$ is a Lie group as well.

 By construction, $G$ is exactly the projective limit of the groups $G_n'$. % (\color{black}For the arguments of this section, did we ever use the local product structure of stable, unstable and center?\color{black})
\end{proof}
%\begin{lemma}
%If $\R^k \curvearrowright X$ be a continuous action of $\R^k$ on a compact metric space $X$ with a dense $\R^k$-orbit, and $C \subset \R^k$ be any nontrivial open cone. Then the set of points $x$ such that $\overline{C \cdot x} = X$ is a dense $G_\delta$ subset of $X$.
%\end{lemma}

%\begin{proof}
%If $C$ contains a half-space, then the result follows from Lemma \ref{lem:dense-halfspace}. Otherwise, $\set{a \in \R^k : a + C \subset C}$ contains a half-space.
%\end{proof}

%\begin{corollary}
%If a totally Anosov action $\R^k$-action has a dense $\R^k$-orbit, then the set of periodic $\R^k$-orbits is dense.
%\end{corollary}
\part{Proofs for $G$-actions and deriving fundamental properties for $\mathbb R^k$-actions}
\label{part:reduction}

\section{Proof of the results for semisimple Lie group actions}
\label{sec:semisimple-proof}

\color{black}
From assumptions in Theorem \ref{G-action} there exists a totally \color{black}partially hyperbolic $A$-action of a split Cartan subgroup of $G$,  such that the $A$-action \color{black} contains an accessible element. We show below that the volume preservation assumption for a maximal split Cartan subgroup implies volume preservation for the whole $G$-action. In particular, $A$-action preserves volume. We show in this section that the $A$-action with these properties satisfies all the assumptions of Theorem \ref{abelian}. Applying Theorem \ref{abelian} gives us that the $A$ action is $C^\infty$ conjugate to a bi-homogeneous action (up to finite cover). Lastly we show that this conjugacy works for the whole $G$-action. In the rest of this section we derive all the other statements concerning global rigidity of $G$-actions. % on a finite cover of a double-homogeneous space $K \backslash \HH / \Gamma$. 
%By Lemma \ref{conjugacy} we get the proof of Theorem \ref{G-action}. 

%\begin{lemma} \label{conjugacy} Let $G$ and $A$ be as in the assumptions of Theorem \ref{G-action}. If $A$ is  conjugate to an algebraic action on $K \backslash H / \Gamma$ by a   $C^\infty$ diffeomorphism $\phi$, then the $G-$action is conjugate to an algebraic action  $K \backslash H / \Gamma$ by the same map $\phi$.
%\end{lemma}

\color{black}

 \subsection{The $\G$ action in Theorem \ref{G-action} preserves volume}\label{volume}

\begin{lemma}    \label{lem:ax+b}
Let  $X$ \color{black} be a manifold with a $C^1$-action of the  subgroup $ A \color{black} \ltimes \color{black} U$ on  $X$\color{black}, where $A$ is the maximal split Cartan subgroup in $\G$ and $U$ is a one parameter unipotent subgroup $U$ in $G$.  Suppose $A$ preserves a volume form $\omega$, and that $A$ has a dense set of recurrent orbits.  Then  $ A \color{black} \ltimes \color{black} U$ preserves $\omega$.
\end{lemma}

\begin{proof}
Let $u_s$ be the one-parameter group giving $U$, and $a_t$ the one generating $A$.  Then we have the commutation relation $  u_{s e^t} \circ a_t  =a_t \circ u_s$ - possibly after reparametrization of $a_t$. 
Let $g^s (x) = \log Jac \cdot u_s$.  Then $g:   X \color{black} \times \R\rightarrow \R$ is $C^1$, and satisfies the relation 
$$     g^{s \cdot e^t} \circ a_t =  g^s.$$
Now consider the derivative $h := \frac{\partial}{\partial s} g^s$.   Then
$$ e^t \cdot \left( h  \circ a_t \right)= h .$$
Then $h = 0$ along any recurrent orbit.  As the latter form a dense set in  $X$ \color{black}, $h \equiv 0$ on  $X$ \color{black}.  
Since $h = \frac{\partial}{\partial s} g^s$, the fundamental theorem of calculus implies that $g^s \equiv g^0 = 1$.  Hence the $u_s$ preserve $\omega$.
\end{proof}

\begin{proposition}   \label{prop:vol preserved}
Let $\G$ be a semisimple Lie group of the noncompact type, and $\rho: \G\to \Diff( X \color{black})$ an action on a manifold $X$.  Suppose that a regular one-parameter subgroup $A = \{a_t\}$ of the maximal split Cartan preserves a volume form $\omega$ on  $X$ \color{black} and has a dense set of recurrent orbits.  Then $\G$ preserves $\omega$.
\end{proposition}

\begin{proof}
It is well-known that  $\G$ as above is generated by finitely many unipotent root subgroups $U_i$ which form   skew product $A \color{black} \ltimes \color{black} U_i$.  Since $A$ is regular, the skew products are non-trivial.  Hence the  $A \color{black} \ltimes \color{black} U_i$ preserve $\omega$ by Lemma \ref{lem:ax+b}.
Therefore  $\G$ preserves $\omega$.
\end{proof}

%\todo{Comment(7) Fix semidirect products throughout}

%\noindent {\em Note:}  A similar argument should work if we know that some measure is quasi-invariant under the unipotent flow or flows. 

%We check that the abelian subaction of G satisfies conditions 1) and 2) of Theorem \ref{abelian}.

{\color{black}
\subsection{Measurable conformal structure}\label{cond2}

As in Theorem \ref{G-action} we let $\G$ be a real semisimple group  such that every simple factor of $\G$ has real rank at least 2. \color{black} In particular, $G$ has no compact factors.\color{black} 

The following is a foundational theorem in the study of higher rank Lie group actions, and actions of their lattices. 
%It will give that the action of a split Cartan subgroup is  Oseledets \color{black} conformal.  %(\color{teal}measurably conformal or  conformal? measurably conformal is really confusing, see also Lemma 7.4).\color{black} 

\color{black}
While this fundamental result is originally due to Zimmer, we state and use here the version which is due to Fisher and Hitchman %(\cite{FM}, Theorem 1.4).
(\cite{FH}, Theorem 1.3). Their result is for when the group $G$ is a semisimple Lie group with no compact factors and property (T) of Kazhdan. If $G$ is as in Theorem  \ref{G-action} then $G$ has Kazhdan's property (T) \cite[Section 1.6]{Bekka-Harpe}, and has no compact factors, hence it satisfies the assumptions of Theorem \ref{Zimmer}.

Let $\mu$ be a probability measure on the space $X$ and let $\rho$ be a $G$-action on $X$, preserving $\mu$. 
Following \cite{FH}, we say a cocycle $\beta: \G\times X\to GL(n, \mathbb R)$ over $\rho$ is $L^2$ if for any compact subset $K\subset G$, the function $\sup _{g\in K}\ln ^+\|\beta(g, x)\|$, $x\in X$, is in $L^2(X, \mu)$.\color{black}\,  Note that the derivative cocycle of a smooth $\mu$-preserving action $\rho$ on a compact smooth manifold $X$ is in $L^2(X, \mu)$.

\begin{theorem}[\cite{Z}, \cite{FM}, \cite{FH}]  \label{Zimmer}(Zimmer Cocycle Superrigidity Theorem)  Let $G$  be a real semisimple Lie group
with no compact factors and with property (T) of Kazhdan. Given an $L^2(X, \mu)$ cocycle $\beta: \G\times X\to GL(n, \mathbb R)$ over a measure preserving ergodic  $\G$-action $\rho$ on $X$, there exists:
\begin{itemize}
\item a measurable map $\psi: X\to GL(n, \mathbb R)$,
\item a continuous homomorphism $\pi: \G\to  GL(n, \mathbb R)$, and
\item a cocycle $c: G\times X\to SO(n)$ such that $c(G, \cdot)$ centralizes $\pi(G)$,
\end{itemize}
such that:
$$\beta(g, x)= \psi(\rho(g, x))\pi(g)c(g,x)\psi(x)^{-1}.$$
\end{theorem}

%\color{teal}
%Using of representation theory for real semisimple Lie group, another way to state Theorem \ref{Zimmer} is that . As a consequence we have, 
%\color{black}

\bl \label{lem: super}
Given the $\G$-action $\rho$ on $X$ as in Theorem \ref{G-action}, the $A$ action $\rho|_A$ is continuously Oseledets  conformal.
\el

%\todo{comment(8) added explanation for why the matrices are of the block-diagonal form instead of Jordan blocks}
\begin{proof}
Given the $\G$-action $\rho$ on $X$ as in Theorem \ref{G-action}, by applying Theorem \ref{Zimmer}, we obtain a measurable frame in which the derivative cocycle $D_x\rho(g)$ of the action $\rho$ is constant \color{black} up to a compact noise. The representation $\pi$ given by Theorem \ref{Zimmer} is a finite dimensional representation, it takes semisimple elements to semisimple elements. So we can decompose $\R^n$ as $\pi|_A$ (common) invariant eigenspaces, and this decomposition is preserved by the compact-group valued cocycle $c$ (since $c$ commutes with $\pi$). This gives a block-wise diagonal form for $\pi (g) c(g, \cdot)$ for all $g\in A$, which by application of $\psi$  gives that the derivative cocycle when restricted to $A$ has {\it measurably} 
%Then the Lyapunov spectrum  and Oseledets spaces for the totally \color{black} partially hyperbolic\color{black}\,  action $\rho|_A: A\times X\to X$ are (measurably) identified with those of \st{the representation $\pi|_A$. This implies that} the representation  
a block-wise diagonal form $\diag( e^{\chi_1}O(d_1), \dots, e^{\chi_r}O(d_r))$ where $O(d_i)$s are orthogonal groups  of dimension $d_i$. This gives the Lyapunov spectrum  and Oseledets spaces for the $A$-action.  The invariant subspaces $V_i$  carry a metric which  via $\psi$ defines a measurable metric $\|\cdot\|_i$ on each $E^{\chi_i}$ such that for $v\in E^{\chi_i}$ and $a\in A$ we have: $\|a_*v\|_i=e^{\chi_i(a)} \|v\|_i$. \color{black}
%\color{black} explain more \color{black}
\end{proof}
%\bl \label{lem: super}
%Superrigidity for cocycles, in particular we could identify weights with coarse Lypunov functionals. Moreover there is an $A$ action invariant measurable conformal structure on each Oseledec space.
%\el
\color{black}

\subsection{Super accessibility\color{black}}\label{cond1}
%(\color{teal}the third time we define accessibility in the paper...\color{black})
%Recall that by Definition \ref{def: totallyPH}

By assumptions of Theorem \ref{G-action} we have a maximal Cartan subgroup $A\subset G$ which acts as a \color{black} totally partially hyperbolic action on $X$ with respect to some $E^c$.  The structure for such actions is described in Section \ref{PHactions}. In particular, Lyapunov functionals in the $E^c$ direction are 0. As before   $\Delta$ denotes the set of non-zero coarse Lyapunov functionals.  Moreover, the $G$-action in Theorem  \ref{G-action} is accessible, which by Definition \ref{def: totallyPH} means that some $a\in A$ is an accessible partially hyperbolic map.   
%\begin{definition}
%Fix some $\Omega \subset \Delta$. If $x, y \in X$, we say $y$ is {\rm $\Omega$-accessible from $x$} if there exists $x = z_0, \dots, z_n = y$ such that for every $i = 0,\dots,n-1$ there exists $\lambda_i \in \Omega$ such that $z_{i+1} \in W^{\lambda_i}(z_i)$. The {\rm $\Omega$-accessibility class of $x$} is the set of all points $\Omega$-accessible from $x$, and is denoted $\Acc_{\Omega}(x)$. The action is called {\rm $\Omega$-accessible} if $\Acc_{\Omega}(x) = X$ for some (equivalently, for every) $x \in X$.
%\end{definition}
The goal of this section is to derive super accessibility  of the $A$-action, from the fact that it comes from the $G$-action. 

\begin{theorem}\label{thm: sub acc}
Let $G \curvearrowright X$ be an action satisfying the assumptions of Theorem \ref{G-action}. Then for any $\la\in \Delta$, there exists an accessible element in $\ker \la$. 
%$$\Delta(\hat{\lambda})$-accessible.
\end{theorem}

\begin{proof}
If the accessible element $a$ of the $G$-action is in $\ker \la$ for some $\la\in \Delta$, then for any $a'$ close to $a$, the stable (resp. unstable) distribution of $a$ is contained in stable (resp. unstable) distribution of $a'$, which implies accessibility of $a'$. So without loss of generality we may assume that $a$ is not in any $\ker \la$. This means that any two points on the manifold can be connected by a broken path whose legs are in leaves of foliations $W^\la$, $\la\in \Delta$.

Now we fix $\la\in \Delta$. Let  $\Delta(\hat \lambda)$ denote the set of coarse Lyapunov functionals \emph{not proportional} to $\lambda$. We want to show that any two points on a leaf of the foliation $W^\la$ can be connected by a broken path whose legs are contained in the leaves of $W^\beta$, where $\beta\in \Delta(\hat \lambda)$. If this holds, than any generic singular element in $\ker \la$ is necessarily accessible since its stable and unstable foliations are subfoliated by $W^\beta$, $\beta\in \Delta(\hat \lambda)$. To prove this we use the structure of $G$.

Recall that the Weyl group $N_{G}(A)/C_{G}(A)$  (where $N_{G}(A)$ is the normalizer of $A$) acts on the weights and roots. Denote by $w_\beta$ the reflection about the hyperplane perpendicular to the root $\beta$. It is well known that the Weyl group action on weights and roots is generated by $w_\beta$'s. Moreover, we have the following fact from the representation theory of semisimple Lie groups. 
\color{black}

\bl \label{lem0}Let $G$ be an $\R$-semisimple Lie group such that every simple factor has $\R$-rank at least 2, and $\rho : G \to E$ be a representation of $G$. Then for any nonzero weight $\lambda$ of $\rho$, there exists a root $\beta$, and weight $\gamma$ not proportional to $\lambda$ such that  $w_\beta(\gamma)=\lambda$.
\el 

\begin{proof}
    Since each simple factor has at least two, and $\ker \lambda$ is a hyperplane in $A$, there exists a simple factor $G_i$ such that if $\pi_i : G \to G_i$ is the factor map, $\pi_i(\ker\lambda)$ is codimension 1 in $A_i := \pi_i(A)$. By Proposition \ref{prop:weyl-standard}, there exists a root $\beta$ of $G_i$ such that $w_\beta(\ker\lambda|_{A_i})\not=  \ker \lambda|_{A_i}$. Since $\beta$ is also a root of $G$, and $w_\beta(\lambda)|_{A_i} = w_\beta(\lambda|_{A_i})$, it follows that $\gamma := w_\beta(\lambda)$ is not proportional to $\lambda$. Since $w_\beta$ is a reflection, $\lambda = w_\beta(\gamma)$
\end{proof}

%We return now to the proof of Theorem \ref{thm: sub acc}. 

Let $\lambda\in \Delta$ be fixed,  and consider $\beta, \gamma$ which satisfy the previous lemma. The following lemma can be found for example in \cite[Lemma 1.3]{Deodhar78}. 

\bl \label{l2} There is a representative of $w_\beta$ in $N_{\G}(A)$ can be written as $u_\beta\cdot v_{-\beta}\cdot u_\beta$, where $u_\beta$ and $v_{-\beta}$ are in the unipotent subgroups corresponding to the root $\beta$ and $-\beta$ respectively.
\el

By using the previous lemma we show the following: 

\bl \label{l1} For any representative $w$ of $w_\beta$, $w(W^\gamma)=W^\lambda$.
\el 
\begin{proof}Replacing $w$ by $w^{-1}$, it is enough to show $wW^\gamma\subset W^\lambda$. By definition of coarse Lyapunov foliation, $W^\gamma=\cap_{a: \gamma(a)<0}W^s_a, ~W^\lambda=\cap_{a: \lambda(a)<0}W^s_a$. Therefore we only need to prove the following properties:
\begin{itemize}
\item[(1)] $w(W^\gamma)$ is an $A-$invariant foliation.
\item[(2)] for any $a$ such that $\lambda(a)<0$, $a$ uniformly contracts $w(W^\gamma)$.
\end{itemize}

Property (1) follows from the fact  that for any $a\in A$, 
\begin{equation}aw(W^\gamma)=w(w^{-1}aw)(W^\gamma)=w(w_\beta(a))(W^\gamma)=w(W^\gamma).
\end{equation} 

To show  (2) observe that by the last equation we know that $a$ contracts $w(W^\gamma)$ if and only if $w_\beta(a)$ contracts $W^\gamma$, i.e. $\gamma(w_\beta(a))<0$, notice that $w_\beta(\lambda)=\gamma$ and $w_\beta$ is an isometry, so $$\gamma(w_\beta(a))=w_\beta(\lambda)(w_\beta(a))=\lambda(a)$$which implies (2).
\end{proof}
By Lemmas \ref{l2} and \ref{l1}, we know that any two points $x,y$ on the same $W^\lambda$ leaf can be linked by finitely many $W^\beta,W^{-\beta}, W^\gamma$ paths. Since $\pm\beta,\gamma\in \Delta(\hat{\lambda})$, we finish the proof.
\end{proof}

\color{black}
\subsection{Conclusion of the proof of Theorem \ref{G-action}}
\label{sec:main-Gproof}

Let  $\rho: G\to \Diff^\infty( X \color{black})$ be the $G$-action as in  Theorem \ref{G-action}. Then there is some split Cartan subgroup $A$ of $\G$ such that $\rho(A)$ is a totally \color{black} partially hyperbolic \color{black} volume preserving action. Due to Section \ref{volume} the $\G$ action $\rho$ preserves volume as well. Then Sections \ref{cond2} and \ref{cond1} imply that the action $\rho(A)$ satisfies the conditions of Theorem \ref{abelian}. Therefore there is a homogeneous space $ K \backslash H/\Gamma$ such that (a finite cover of) $\rho(A)$ is conjugate to a bi-homogeneous action on  $ K \backslash H/\Gamma$. Thus, it is clear that we may lift the $A$-action to a homogeneous action on $H/\Gamma$. By applying the lifting lemma (see Theorem \ref{lifting}), the whole $\rho(G)$ action lifts to a $G$-action on $H/\Gamma$ such that the maximal split Cartan subgroup $A$ of $G$ acts homogeneously, i.e. by left multiplication. 

We wish to show that the $G$ action is homogeneous, not just its restriction to $A$. For this, we use Theorem \ref{thm:affine-centralizer}, noting that since every simple factor of $G$ has rank at least two, for every root $\chi$ of $G$, the action of $\ker \chi \subset A$ is ergodic. Since $A$ is an $\R$-split Cartan subgroup, the action of $\ker \chi$ is $\R$-semisimple. Hence Theorem \ref{thm:affine-centralizer} applies, and since $U_\chi$, the root subgroup of $G$ corresponding to $\chi$, commutes with $\ker \chi$ and is $C^r$. So, by Theorem \ref{thm:affine-centralizer}, $U_\chi$ acts by affine maps.

Finally, observe that since the $U_\chi$ generate $G$ as $\chi$ varies over all roots of $G$, the action of $G$ is affine. Hence any conjugate of $A$ will also be homogeneous, since the conjugation of a homogeneous action by an affine transformation is also homogeneous. Since $G$ is semisimple, the conjugates of $A$ generate $G$ \color{black} and the action of $G$ on $H /\Gamma$ is by translations.

\subsection{Proof of Corollary \ref{G-Anosov}\color{black}} %SHOW ANOSOV G IS ACCESSIBLE, therefore satisfying 
\color{black}It suffices to verify that a totally Anosov $G$-action that satisfies 
assumptions of Corollary \ref{G-Anosov} is accessible as a totally partially hyperbolic $G$-action, then Corollary \ref{G-Anosov} is just a consequence of 
Theorem \ref{G-action}. 

Consider $A$ in Corollary \ref{G-Anosov} which is a split Cartan subgroup, the roots of the Lie algebra $\mathfrak{g}$ decompose into root spaces. The stable and unstable root spaces generate the entire Lie algebra \(\mathfrak{g}\). Consequently, the stable and unstable distributions \(E^s(a)\) and \(E^u(a)\) span the tangent space of the $G$-orbit at every point \(x \in X\), in particular, for any two points $x,y$ such that $y$ is in the $C_\G(A)$-orbit of $x$, $y$ can be connected through an $su$-path (even within $G$-orbit), which completes the proof.

\color{black}

{%\color{black}
\subsection{Proof of Corollary \ref{cor:semisimple}\color{black}}
\label{sec:cor-proof} %\color{black} For Anosov or PH? \color{black}
%\todo{comment(9) Proof needs a fix}
We must show that the assumptions of Corollary \ref{G-Anosov} \color{black} are satisfied when either of the assumptions of Corollary \ref{cor:semisimple} are satisfied. Assume that $\mc J(\mc H)$ has dense image \color{black} in an arbitrary given $\R$-split Cartan subgroup $A$ of $G$. Recall that an element $g\in G$ is called a hyperbolic element for $G\curvearrowright X$ if it satisfies Definition \ref{def: hyp elem}. In particular for $g$ and an arbirary conjugacy $hgh^{-1}$ of $g$, $g$ is a hyperbolic element if and only if of $hgh^{-1}$ is a hyperbolic element. Moreover we have the following useful lemma.
\begin{lemma}\label{lem: dense hyp imp tot Ans}
Let $A\subset G$ be an arbitrary $\R$-split Cartan subgroup, and let $G\curvearrowright X$ be a $C^\infty$ action. If there is a dense subset $J\subset A$ formed by hyperbolic elements, then $G\curvearrowright X$ is a totally Anosov action.
\end{lemma}
\begin{proof}%since for a fixed Cartan subgroup of $G$, there is a dense subset which is conjugate to a hyperbolic element. 
Since an open dense subset $\mc R$ of regular elements $a$ in $A$ satisfies that $\Lie(C_G(A))=Z_\mathfrak{g}(a) = \mf g^0_a$, it follows that any element $a$ in the dense intersection set $\mc R\cap \mc J$ in $A$ is partially hyperbolic respect to the distribution tangent to the $C_G(A)-$orbit. Then Lemma  \ref{lem: dense hyp imp tot Ans} follows Definition \ref{def: tot Ans}.
\end{proof}\color{black}
We claim that if the Jordan\color{black}-Chevalley \color{black} decomposition of $g\in G$ is $g = kan$, and $g$ is a hyperbolic element for $G \curvearrowright  X \color{black}$, then $a$ is a hyperbolic element for $G\curvearrowright  X$ as well. This implies the result, \color{black}since for the Jordan-Chevalley decomposition $g=kan$ of a hyperbolic element $g$, $a$ is also a hyperbolic element. Then  $\mc J(g)$ is a hyperbolic element for $G\curvearrowright X$ as well since it is conjugate to $a$. Then the image of $\mc J(\mc H)$ is a dense subset of a Weyl chamber of $A$ formed by  hyperbolic elements. Modulo the action of the Weyl group, we get a dense subset of $A$ formed by hyperbolic elements of $G\curvearrowright X$. Then  by Lemma \ref{lem: dense hyp imp tot Ans}, $G\curvearrowright X$ is a totally Anosov $C^\infty$ action. Then Corollary \ref{cor:semisimple} follows Corollary \ref{G-Anosov}.
\color{black}
\begin{proof}[Proof of the claim]
Indeed, we may choose a Riemannian metric on  $X$ \color{black} which is invariant under $k$, since $k$ belongs to a compact subgroup. Furthermore, the condition that $g$ is hyperbolic for $G \curvearrowright  X \color{black}$ \color{black} implies $Dg$ preserves a dominated splitting $E^s_g\oplus E^0_g\oplus E^u_g$. Since commuting diffeomorphisms share invariant dominated splitting, \color{black}see e.g. Lemma 13 of \cite{DWX}, %is exactly that it is normally hyperbolic with respect to the orbit foliation of $\exp(E^0_g)$. In particular, the bundles $E^s_g$ and $E^u_g$ are unique
\color{black} and since $k,a$ and $n$ commute with $g$, each of them must preserve the subbundles as well.  Hence, $\norm{Dg|_{E^s_g}(x)} = \norm{D(an)|_{E^s_g}(x)}$. Since $E^s_g$ is uniformly contracting under $g$, there exists $C> 0$ and $0 < \lambda < 1$ such that $\norm{Dg^k|_{E^s_g}(x)} \le C\lambda^k$. We now appeal to the following

\begin{lemma}
\label{lem:nilpotent-exponents}
Let $n \in G$ be an $\ad$-unipotent element. Then for every $\ve >0$, there exists some $C' > 0$ such that $C'e^{-k\ve} \le \norm{Dn^k(x)} \le C'e^{k\ve}$.
\end{lemma}

Let us apply the lemma before proceeding with the proof. Notice that

%\todo{comment(10) provide more explanation for the first equality in equation below}

\[ \norm{Da^k|_{E^s_g}(x)} = \norm{Dn^{-k} \of Dg^k|_{E^s_g}(x)} \le C'e^{k\ve} \cdot Ce^{k\lambda} = CC' e^{(\lambda + \ve)k}. \]

Since $\ve$ is arbitrary, we may choose it to be $-\lambda/2$, so that $\lambda + \ve$ is still negative. Therefore, $a$ contracts  the bundle $E^s_g$ exponentially. A symmetric argument works for $E^u_g$ \color{black}, therefore $E^{s/u}_g=E^{s/u}_a$. Since $\mf g_g^0=\mf g_a^0$ (by Lemma \ref{lem:nilpotent-exponents} $g=kan$ and $Dn$ has subexponential growth along $\mf g_g^0$, so $\mf g_g^0=\mf g_a^0$), we have $E^g_0=E^a_0$, which completes the proof of Corollary \ref{cor:semisimple} under the first assumption.\color{black}
\end{proof}
\begin{proof}[Proof of Lemma \ref{lem:nilpotent-exponents}]
Fix $v \in T X \color{black}$, and let

\[ \chi(v) = \limsup_{k\to \infty} \frac{1}{k} \log \norm{Dn^k(v)} \]

denote the (upper) Lyapunov exponent of $n$ on the vector $v$. Observe that \[\chi(v) \le \sup_{v \in T X \color{black}} \norm{Dn(v)}/\norm{v} < \infty,\] so $\chi$ is a bounded function on $T X \color{black}$. Furthermore, since $n$ is $\ad$-unipotent and belongs to a semisimple Lie group, if it is nontrivial, there exists a renormalizing element $b \in G$ such that $b^{-1}nb = n^2$ (this follows from the Jacobson-Morozov theorem). Then direct computation shows that

\begin{multline*}
    \chi(Db(v)) = \limsup_{k \to \infty} \frac{1}{k} \log \norm{Dn^k(Db(v))} = \limsup_{k \to\infty} \frac{1}{k} \log \norm{D(b^{-1}n^kb)(v)} \\ = 2\cdot \limsup_{k\to\infty} \frac{1}{2k} \log\norm{D(n^{2k})(v)} = 2\chi(v).
\end{multline*} 

Even though this is taken along the subsequence of even iterates, notice that $D(n^{2k+1}) = Dn^{2k} \of Dn = Dn^{2(k+1)} \of Dn^{-1}$. Since $Dn$ is uniformly bounded, the $\limsup$ along even terms and odd terms coincides. Since $\chi$ must remain bounded, we conclude that $\chi(v) = 0$ for all $v \in T X-\{0\} \color{black}$. 
\color{black}
%Thus we have that  for the cocycle $Dn$ over the dynamics $n$ on $X$,  $\chi(v)=0$ for all $v\in TX-\{0\}$, 
which  implies that $Dn$ has $0$ Lyapunov exponents for all invariant measures. 
By applying exactly the same arguments as in \cite{Kalinin}\, (which uses   subadditive sequences  \cite{Schr}), if a linear cocycle over a dynamical system  has $0$ Lyapunov exponent for all the invariant measures, then it has sub exponential growth. Hence $\|Dn^k\|$ has sub-exponential growth\color{black}. The other inequality follows from an identical analysis using $\liminf$ and the lower exponent.
\end{proof}

\color{black}
Now we turn to the second assumption, that $\mc H$ intersects the set of $\R$-semisimple elements in a dense set \color{black} $\mc S$. If $\mc J(\mc H)$ has dense image then we get the proof by the first assumption. Otherwise, assume that $\mc J(\mc H)$ does not intersect an open set $U\subset A$. By shrinking $U$ if necessary we may assume $U$ is formed by regular elements. Then by the definition of Jordan-Chevalley decomposition and the assumption that $\mc J(\mc H)\cap U$ is empty,  we know that any semisimple element $g\in G$ such that $\mc J(g)=a$ for some $a\in U$ cannot be in $\mc S$. 
%which shares the same spectrum (for $Ad_G$) as an element $a\in U$, cannot be in $\mc S$, since $\mc J(g)=a$. 
The set of such $g$ forms an open set in the $\R$-semisimple elements (since any element $g$ that shares the same spectrum, for $Ad_G$, as an element $a\in U$, satisfies $\mc J(g)=a$), which contradicts the second assumption of the Corollary. \color{black}

%We would like to prove it using 

%Let $\mc S$ denote the set of all such elements. Fix a Cartan subgroup $A \subset G$. Any $\R$-semisimple element of $G$ is conjugate to an element of $A$, and the map $\mc J|_{\mc S}$ is onto. In particular,  $\mc J(\mc H)$ must be dense if $\mc H \cap \mc S$ is dense in $\mc S$.
%(\color{black}Revise later use eigenvalues characterise semisimple elements.\color{black})
}
%\color{green} Now we use the result of Zeghib \cite{...} (see  Section...in (Kurt-Ralf paper) \cite{} )  to conclude that  
 %$\rho(\G)\subset Aff(\HH/\Gamma)$. It is an easy observation that a conjugate of a homogeneous action by an affine map is also a homogeneous action. By conjugating the maximal split Cartan subgroup $A$ in $\G$ by an automorphism of $\G$ we can get that any other Cartan subgroup of $\G$ also acts homogeneously i.e. by multiplication. Since Cartan subgroups generate $\G$, it follows that the lift of $\rho(\G)$ to $\HH/\Gamma$ is a homogeneous action.  
 
% \color{green} in Appendix?
 
% Step 1: Use the Zeghib result to conclude that an element of a coarse Lyapunov has a homogeneous graph, since it commutes with its corresponding kernel

%Step 2: Translate the homogeneous graph in the product of H/Lambda with
%itself so that it intersects the identity coset. Use the tangent space
%of the graph to define an automorphism of the Lie algebra of H (and
%hence of H). Conclude that the action is affine in H/Lambda

\color{black}

\subsection{Proof of Theorem \ref{thm:genuinely-HR}\color{black}}

Note that the assumptions of Theorem \ref{thm:genuinely-HR} are almost the same as those of Theorem \ref{G-action}. The key difference is that we assume that only some subgroup of an $\R$-split Cartan subgroup $A$ acts totally partially hyperbolically. We therefore need to show that under the genuinely higher-rank assumption, we still have the assumptions of Theorem \ref{abelian}.

We can still apply Zimmer's cocycle superrigidity theorem to obtain that the action of $B$ is measurably Oseledets conformal. To see that the action is super-accessible, we follow the same strategy as before. We first show that for every weight $\lambda$, there exists $a \in \ker \lambda|_B$ such that the accessibility classes of $a$ are saturated by $G$-orbits. Note that since $B$ is genuinely higher rank, for every weight $\lambda \in \Delta$, $\ker \lambda$ projects non-trivially onto each simple factor (since by the genuinely higher-rank condition, $\pi_i(B)$ has dimension at least 2, so $\pi_i(\ker \lambda|_B)$ has dimension at least 1). This means that if $U_a^+,U_a^- \subset G$ denote the expanding and contracting subgroups of some $a \in \ker \lambda|_B$, $U_a^{\pm} \cap G_i$ is a nontrivial subgroup. Since $G_i$ is simple, and $\langle U_a^+,U_a^-\rangle$, the group generated by $U_a^\pm$ must be all of $G$. It follows that $a$-accessibility classes are saturated by $G$-orbits.

Once $a$-accessibility classes are saturated by $G$-orbits, one may apply Lemmas \ref{lem0}-\ref{l1} to obtain that $G$-actions satisfying the assumptions of Theorem \ref{thm:genuinely-HR} are super accessible, replacing Lemma \ref{lem0} with Corollary \ref{cor:detected-weyl}.

\color{black}

\section{Consequences of assumptions in Theorem \ref{basic abelian} and Theorem \ref{abelian}%Strong accessibility implies \ref{FP1} and \ref{FP2}
}
\label{strong accessibility implies FA12}
%We show that strong accessibility assumption in Theorem \ref{abelian}, implies  the fundamental assumptions \ref{FP1} and \ref{fundamental2}}.

%The remaining sections are then dedicated to the description and study of leafwise homogeneous Anosov actions.

 \subsection{Assumptions of Theorem \ref{abelian} imply property \ref{FA1}}
 %ergodicity and topological transitivity of $\ker \lambda$ actions
 
%\todo{comment(11) Fixed the inconsistent use of Oseledets vs Oseledec}
% Now we show that the accessibility assumption in Theorem \ref{abelian}
%implies condition \ref{FP1}. %\st{in Theorem \ref{basic abelian}}. \color{black}
\color{black} We use here the result by Burns-Wilkinson \cite{BW} that a partially hyperbolic, volume preserving center bunched diffeomorphism is ergodic. Namely, by the super accessibility assumption, for any $\lambda\in \Delta$ there exists $a\in \ker\lambda$ such that $a$ acts on $X$ as an accessible partially hyperbolic diffeomorphism.  By Lemma \ref{Ec center bunching}, it is also center bunched. Thus by \cite{BW} we get ergodicity of this element with respect to invariant volume,  and therefore  topological transitivity of $\ker \lambda$ actions i.e. the property \ref{FA1} holds.

\color{black}

%\color{black} (UNIFY NOTATIONS $ker \lambda$ and $Ker \lambda$. DONE, Disheng) \color{black}

\subsection{Assumptions of Theorem \ref{abelian} imply additional regularity of Oseledets objects\color{black}}\label{regularity-metrics}

\color{black}

In this section we show that the regularity of the  measurable invariant objects (distributions  and metrics) in the assumptions of Theorem \ref{abelian}, can be improved to continuous, and even H\"older along invariant foliations. 
%given the assumptions in Theorem \ref{abelian}. 

We first do this for the Oseledets spaces corresponding to non-zero Lyapunov functionals. The final statement is Proposition \ref{Hoelder metric}.

%The following theorem is a version of results obtained in many papers by several authors, including ....  Our proof is functionally a collection of references to proofs of each aspect as stated.

Let $\mathbb R^k\curvearrowright X$ be a $C^\infty$  action 
 satisfying assumptions of Theorem \ref{abelian}. Recall that in Sections \ref{PHactions} and \ref{sec: acc} we established that totally partially hyperbolic actions with super accessibility property actually have {\it all} elements outside the walls $\ker \lambda$, $\lambda \in \Delta$, as well as all generic singular elements in the walls,  partially hyperbolic and accessible. Since the action preserves volume, all these elements are necessarily ergodic with respect to volume, and topologically transitive.  
 
The other  assumption is that the action is measurably Oseledets conformal, i.e. Oseledets spaces carry measurable conformal structure and measurable metric invariant under the action. Recall that for any structure over the action,  partial H\"olderness means that the structure is continuous on $X$ and in addition H\"older along the coarse Lyapunov foliations. Recall also that every coarse Lyapunov distribution $E_\la$ is a direct sum $\bigoplus_i E^{c_i\la}$ of Oseledets distributions, where $c_i$ are positive constants. 
%Then by applying Theorem \ref{thm: acc imp hder} to generic singular elements in $\rho$ we obtain the following corollary:
\color{black}
\bl \label{lem: hder conf str} \color{black} Every Oseledets distribution $E^{c_i\lambda}$  coincides with a partially H\"older one almost everywhere. \color{black} Within $E^{c_i\lambda}$ there is an $\mathbb R^k-$invariant \color{black} partially \color{black} H\"older 
invariant conformal structure \color{black}almost everywhere\color{black}.
\el
\begin{proof}
 In our setting for any coarse Lyapunov foliation $W^\lambda$, a generic singular element $a$ in $\ker \lambda$ is by assumption a partially hyperbolic diffeomorphism with center distribution \color{black} $E^c\oplus E_\lambda \oplus E_{-\lambda}$\color{black}, or \color{black}  $E^c\oplus E_\lambda$\color{black}. In the rest of the argument we assume the former, since it is the more general case. Moreover, the generic singular $a$ is accessible (by Lemma \ref{gen-sing-acc}, under the assumptions of Theorem \ref{abelian}) and by Lemma \ref{lem: a bunch}
 %\marginpar{or refer to Lemma 5.4 instead} 
 it is center bunched partially hyperbolic. Since it is a generic singular element, it has zero Lyapunov exponents in the direction of the center distribution \color{black} $E^c\oplus E_\lambda \oplus E_{-\lambda}$\color{black}  \, 
and non-zero exponents in all other directions.
%\bl \label{lem: a bunch}For any coarse Lyapunov foliation $W^\lambda$, any generic singular $a\in \ker \lambda$, $a$ is center bunched partially hyperbolic diffeomorphism.
%\el 

  Consider the H\"older continuous linear cocycle $Da|_{E_\lambda\oplus E_{-\lambda}}$ over $a$. Since $a\in \ker \lambda$, $Da|_{E_\lambda\oplus E_{-\lambda}}$ has zero Lyapunov exponents. By Oseledets theorem and by assumption on Oseledets conformality in Theorem \ref{abelian}, the Oseledets subspaces within $E_{\lambda}$ and $E_{-\lambda}$ are measurable $\mathbb R^k-$invariant subbundles, and within each of them there is an $\mathbb R^k-$invariant measurable conformal structure. Therefore by Theorem \ref{thm: acc imp hder} all the measurable invariant objects (Oseledets spaces and conformal structures on them) for the cocycle $Da|_{E_\lambda\oplus E_{-\lambda}}$ coincide with \color{black} partially \color{black} H\"older 
  invariant ones almost everywhere.
\end{proof}

\color{black}

We will use the following result from \cite{W} to obtain \color{black} partial H\"older regularity \color{black} of the measurable metric in the assumption of the main Theorem \ref{abelian}. 
\begin{theorem}\label{thm: livsic ph}\cite{W} 
Let $f:X\to X$ be a \color{black} $C^{2\color{black}}$ %strongly 
center bunched conservative partially hyperbolic diffeomorphism.  Assume that $f$ is accessible. Let $\phi : X \to \mathbb{R}$ be a H\"older continuous \color{black} (resp. partially H\"older) \color{black} function, then any measurable solution of the cohomological equation $$\phi=\Phi\circ f-\Phi+c$$ coincides with a H\"older continuous \color{black} (resp. partially H\"older) \color{black} solution almost everywhere. \end{theorem}
\color{black}
\begin{remark}
The version of Theorem \ref{thm: livsic ph} with partially H\"older regularity is not stated in \cite{W}, but it is a direct consequence of the construction of solution to coboundary equation using periodic cycle functionals as in \cite{W}. The partial H\"olderness of solutions has also been noticed for accessible systems in \cite{KatKon} where the periodic cycle functionals were used for the first time. \color{black}
\end{remark}

%\color{black}
%The following proposition will confirm that \ref{FP2} holds for actions in Theorem \ref{abelian}. %\footnote{Definition of H\"older continuous along all coarse foliations, we need uniformity assumption.}
%\color(black}

\color{black}

\begin{proposition} 
\label{Hoelder metric}\color{black} For the action as in Theorem  \ref{abelian} each Oseledets subspace is H\"older continuous and   \color{black} the metric \color{black}within each Oseledets subspace \color{black} is partially H\"older. %foliations. %\color{black}(within each Oseledets subspace)EXPLAIN\color{black}.
\end{proposition} 
\begin{proof}
\color{black}

 The strategy of the proof is that first we will use the ``partially H\"older'' version of Theorem \ref{thm: livsic ph} to show that the measurable metric is $C^0$. The main step for this is Lemma \ref{coboundary lemma}.  Then we will use this together with the higher rank assumption to improve the regularity of Oseledets splitting from continuous to H\"older. After that we use the H\"older version of Theorem \ref{thm: livsic ph} and Lemma \ref{coboundary lemma} 
 to improve the regularity of the metric from continuous to partially H\"older.

%\begin{lemma}\label{metric is C0}
%The measurable metric on an Oseledets subspace for actions as in Theorem  \ref{abelian} is continuous. 
%\end{lemma}
%\begin{proof}

\color{black}
By Lemma  \ref{lem: hder conf str} the Oseledets subspaces %and measurable conformal structures 
in the assumption of Theorem  \ref{abelian} \color{black}  actually coincide with partially H\"older ones (which we still denote by $E^\lambda$) almost everywhere\color{black}. %As a consequence, to show our proposition we have to solve some cocycle equations. 
We fix a smooth background metric on $X$. It induces a \color{black}partially \color{black} H\"older metric $g_\lambda$ within each Oseledets subspace $E^\lambda$. Then $g_\lambda$ induces a \color{black}partially \color{black}H\"older volume form $\nu_{\lambda}$ within $E^\lambda$. For any \color{black}$a\in \mathbb R^k$\color{black}, if we let %\todo{comment(12) fixed} $q(a,x):=\rm{Jac}_{\nu_\lambda}(Da|_{E^\lambda})$, then $q$ is \color{black}partially \color{black} H\"older cocycle over the $\mathbb R^k$-action. Moreover we have

\begin{lemma}\label{coboundary lemma}
There exists a \color{black}partially \color{black} H\"older function $\phi: X\to \mathbb{R}$ such that for any \color{black}$a\in \mathbb R^k$ \color{black}\color{black} \begin{equation} \label{eqn: def phi solv cobry}
\phi(x)\cdot \phi(ax)^{-1}=e^{-\dim(E^{\lambda})\lambda(a)}\cdot q(a,x)
\end{equation}
\end{lemma}\color{black}
\begin{proof}First let $a$ be a generic singular element in $\ker \lambda$. Assumption of Theorem  \ref{abelian} that the action is measurably Oseledets conformal implies the existence of a measurable invariant metric, which then induces a measurable invariant volume form. By defining $\phi$ is the Radon-Nikodym derivative of the invariant measurable volume form with respect to the volume form $\nu_\lambda$ we obtain the  existence of a measurable  function satisfying our lemma for $a$. (Notice that $\lambda (a)=0$ in this case.) Then by Theorem \ref{thm: livsic ph}, we can upgrade regularity of  $\phi$ to partially H\"older function, since $a$ is (at least) $C^2$, volume preserving, accessible and center-bunched (from Lemma \ref{lem: a bunch}), and $q$ is partially H\"older.

\color{black}  Now let $b$ be a general element in $\mathbb R^k$. The fact that  \eqref{eqn: def phi solv cobry} holds for $b$ as well is a direct consequence of transitivity of elements in $\ker \la$ and commutativity of the action. This is a general fact: for any ($\R$-valued multiplicative) cocycle $\beta$ over an $\mathbb R$-action, if $\beta(a,x)$ is cohomologous to 1 (i.e. if $\beta(a,x)=\psi(x)\psi(ax)^{-1}$), and $a$ is transitive,  then $\beta(b,x)$ is cohomologous to a constant for any other $b\in \mathbb R$ via the same transfer map $\psi$. The reason is that commutativity of $a$ and $b$ and the cocycle property for $\beta$ imply $\beta(a, bx)\beta(b,x)=\beta(b, ax)\beta(a,x)$ hence $\psi(bx)\psi(abx)^{-1}\beta (b,x)=\beta(b,ax)\psi(x)\psi(ax)^{-1}$
which implies 
$\psi(bx)\psi(x)^{-1}\beta (b,x)=\psi(bax)\psi(ax)^{-1}\beta(b,ax)$ which means that $\psi(bx)\psi(x)^{-1}\beta (b,x)$ is $a$-invariant and therefore (since $a$ is transitive), it is constant. 

By applying this general reasoning to the cocycle $q$ and the fact that  that $b$ acts on $E^\lambda$ with Lyapunov exponents equal to $\lambda(b)$, we get \eqref{eqn: def phi solv cobry} for any $b$ as well. \color{black}
%Define a \color{black}partially \color{black} H\"older volume form on $E^\lambda$ by  $\nu_\lambda':=\phi\cdot \nu_\lambda$. By our construction of $\phi$, $\nu_\lambda'$ is invariant under the action of $b$ for $b\in \ker\lambda$. By commutativity of $a$ and $b$, \,  $a_\ast \nu_\lambda'$ is also invariant under $b$ and there exists a continuous function $\psi$ such that $a_\ast\nu_\lambda'=\psi\cdot  \nu_\lambda'$. 
%is also a $b-$invariant volume form on $E^\lambda$. 
%Since  $b$ acts transitively on $X$,  $\psi$ is a constant function. Notice that $a$ acts on $E^\lambda$ with Lyapunov exponents equal to $\lambda(a)$. Therefore because $\psi$ is constant, we complete the proof of the lemma.
\end{proof}

\color{black}
As a consequence, the measurable volume form induced by measurable metric in the assumption of Theorem  \ref{abelian} within $E^\lambda$ coincides with a continuous volume form almost everywhere. %within the continuous bundle $E^\lambda$ such that it behaves well under the $\R^k$ action. 

By Lemma  \ref{lem: hder conf str} again, the measurable conformal structure induced by measurable metric in the assumption of Theorem  \ref{abelian} also coincides with a partially H\"older one almost everywhere within $E^\lambda$. Therefore by using the continuity of the volume form, the measurable metric on $E^\lambda$ subspace for actions as in Theorem  \ref{abelian} coincides with a continuous metric everywhere. 
%By continuity this metric still satisfy the equation in \ref{FP2}. 

The key fact which follows from the continuity of the metric and the higher-rankness of the action is that if we take a generic element $b\in \R^k$ such that $\chi(b)\neq \chi'(b)$ for any two different Lyapunov functionals $\chi, \chi'$, then the Oseledets splitting of $\R^k$ 
%by  \ref{FP2} 
is actually a dominated splitting of the cocycle $Db$ over $b$. Then it is a classical fact that the splitting is H\"older continuous (see for example Theorem 4.11 of \cite{CP}). So the Oseledets splitting for the $\R^k-$action is H\"older continuous. 

%\begin{lemma}\label{Holder Oseledets}Last step: The Oseledets splitting for actions as in Theorem  \ref{abelian} is H\"older continuous. 
%\end{lemma}

To complete the proof of Proposition \ref{Hoelder metric}, it suffices to verify that the measurable  metric is  H\"older along coarses. Therefore we need to show that the volume form and the conformal structure induced by the metric are partially H\"older.

For the volume form, since the cocycle $q$ now is actually H\"older continuous due to the H\"older continuity of $E^\lambda$, by the H\"older part of Theorem \ref{thm: livsic ph} we get the H\"older continuous analogue of Lemma \ref{coboundary lemma}, i.e. any continuous solution $\phi$ of \eqref{eqn: def phi solv cobry} is actually H\"older continuous. Hence the volume form induced by the metric is also H\"older continuous (since it corresponds to the solution $\phi$ of \eqref{eqn: def phi solv cobry}).

For the conformal structure, we take a generic element $b$ that does not belong to any $\ker \chi$. Then $Db|_{E^\lambda}$ is a fibered bunched H\"older linear cocycle with coinciding Lyapunov exponents over a $C^2$ partially hyperbolic volume preserving, center-bunched accessible diffeomorphism $b$, therefore by Theorem \ref{thm: acc imp hder} we know the conformal structure induced by the metric on $E^\lambda$ is H\"older along $W^s_b$ and $W^u_b$, hence by genericity of $b$ it is partially H\"older.

%Notice that in Lemma \ref{coboundary lemma} when $E^\lambda$ is H\"older all the arguments go through verbatim with "H\"older" in place "partially H\"older". 

%Last Last step: Oseledets spaces Holder---> Conformal structure is H\"older (by considering new cocycle: projectification of $Db$ for some $b$ not in the kernel acting on projectification of a H\"older continuous bundle $E^\lambda$, then by ASV we know it is holonomy-invariant hence Holder ). \color{black}

%Combining the H\"older continuous volume form $\nu_\lambda'=\phi\cdot \nu_\lambda$ and the H\"older continuous conformal structure on $E^\lambda$ we easily get the H\"older continuous metric as claimed in the \color{black} Proposition \color{black} \ref{Hoelder metric}.

Now we consider the $\R^k$-invariant measurable metric on $E^c$, see Remark \ref{Rem: 0 exp expl}. First $E^c$ is a H\"older continuous bundle since $E^c$ is a sub-bundle of a dominated splitting. Then by the same proof as Lemma \ref{lem: hder conf str}, the conformal structure induced by the measurable metric within $E^c$ is partially  H\"older. By the same proof as Lemma \ref{coboundary lemma}, the volume form induced by the metric within $E^c$ is actually H\"older. So the measurable metric within $E^c$ is partially H\"older.
\end{proof}
\subsubsection{Consequences on dynamical coherence} \label{DynCoh}

\color{black}
As a corollary of the discussions above, %of Section \ref{regularity-metrics}, 
for any generic singular element $a\in \ker \lambda$ for some Lyapunov functional $\lambda$, $Da$ preserves a continuous metric in $E^c_a=E^c\oplus E_\lambda$ if $-\lambda\notin \Delta$ or $E^c \oplus E_\lambda\oplus E_{-\lambda}$ if $-\lambda\in \Delta$, hence by Proposition 6 of \cite{B03}, $E^c_a$ and $$ E^{cs}_a:=E^c\oplus E_\lambda\oplus E_{-\lambda}\oplus \bigoplus_{\lambda'(a)<0, \lambda'\ne \lambda} E_{\lambda'}$$ are uniquely integrable. Similarly by taking $a$ not in $\ker \lambda$,  but close to $\ker \lambda$, we know $E^c$ is also uniquely integrable.

\subsection{Assumptions of Theorem \ref{basic abelian} imply additional regularity of Oseledets objects\color{black}} \label{Oseledets regularity Anosov}
\begin{proposition}\label{Hoelder Anosov}
 For the action as in Theorem  \ref{basic abelian} each Oseledets subspace and  the metric within each Oseledets subspace  are H\"older continuous.
 \begin{proof}The proof of Proposition \ref{Hoelder Anosov} is essentially the same as the proof of Proposition \ref{Hoelder metric}, for completeness we sketch the proof here. First we consider a generic element $b\in \R^k$ not in any $\ker\lambda$. Using the $C^0$ metric, the Oseledets splitting is a dominated splitting of $Db$, hence it is H\"older continuous. 
 
 Second we show the metric is partially H\"older. For the volume form, it is a corollary of Livsic theorem for Anosov $\R^k$-actions (by this we mean that every measurable solution of a coboundary equation over the action is H\"older.) It is  contained in Theorem 2.1 in \cite{Goetze-Spatzier-Duke}, or can be argued as in \cite{W} (or as in  the last paragraph of this section.) For the conformal structure, by Avila-Viana invariance principle \cite{AV}, it is essentially holonomy-invariant, but since holonomies and the conformal structures are uniformly continuous, the conformal structure is holonomy-invariant everywhere. Hence by H\"older continuity of holonomies along each leaf of each coarse foliation, we get the conformal structure is partially H\"older.

% Then by the same arguement as that in Proposition \ref{Hoelder metric} we know 
 
 By the scaling equation in Oseledets conformal assumption,
 %\ref{FP2},
 the metric is also H\"older continuous along the $\R^k$-orbit direction. By transversality of the sum of all coarse Lyapunov distributions and the orbit direction, any two nearby points $x,y$ can be connected by a finite (with uniformly bounded number of legs) local broken path such that each step is either lying in a local leaf of a coarse Lyapunov foliation or $\R^k$-orbit, and the length of each step is bounded (up to a global constant) by some uniform power of $d(x,y)$, which implies that the metric is globally H\"older continuous. \end{proof}
%along %all 
%coarse. 

%For $\R^k\curvearrowright X$ action that satisfies the assumptions in Theorem \ref{basic abelian}
    
\end{proposition}

\color{black}
\subsection{Summary and preview of Part 3.\color{black} }
%Sections \ref{extension} and \ref{fibration}} %\color{black} MAKE CHANGES IN THIS SECTION\color{black}
\label{sec:intermission1}

\color{black}
We summarize now the conclusions of this section.  Let $\rho$ be an action satisfying assumptions of Theorem \ref{abelian} or those of Theorem \ref{basic abelian}. What we proved in this section is that $\rho$ has the following {\it fundamental properties}:

\begin{itemize}%[label=\textnormal{(\arabic*)}]
\item [\mylabel{FP1}{\rm (GHR)}]   For every $\lambda \in \Delta$, $\ker \lambda$ has a dense orbit.
\item [\mylabel{FP2}{\rm (HOC)}] 
Each nonzero coarse Lyapunov distribution $E_\lambda$ of $\rho$ decomposes into Oseledets spaces with exponents positively proportional to $\lambda$: $E_{\lambda}=\oplus _{i=1} ^{n_{\lambda}} E^{c_i \lambda} $ and the following hold: 
\begin{itemize}
\item[(a)]  There exists a H\"older continuous Oseledets decomposition. 
$$ TX = E^c \oplus\bigoplus _{\lambda \in \Delta} (\oplus_{I=1} ^{n_{\lambda}} E^{c_i \lambda} ). $$
\item[(b)] There exist  %H\"older 
continuous Riemannian metrics (inner products) $\langle \:, \: \rangle _{c_i \lambda}$ on $E^{c_i \lambda} $ \color{black}which are partially H\"older and satisfy that for every $v \in E^{c_i\lambda}$ and $a \in \R^k$: %H\"older continuous along coarse \color{black} such that for all $ v \in E^{c_i \lambda}$ 
$$ \| a_* v\| = e^{c_i \lambda (a)} \|v\| .$$
\item[(c)] There exists a continuous, partially H\"older Riemannian metric $\langle \;,\;\rangle_0$ on $E^c$ invariant under the $\R^k$-action.
\item[(d)] Let $a \in \ker \lambda$, and $y \in W^s _a (x)$.  Let $H^{s,a} _{x,y}: {E} _{\lambda}(x)  \to {E} _{\lambda} (y)$  be the stable holonomy map for $a$.  Then $H^{s,a} _{x,y} (E^{c_i \lambda} _x) = E^{c_i \lambda} _y$ and 
$$ H^{s,a} _{x,y} :  E^{c_i \lambda} _x \to E^{c_i \lambda} _y$$
is an isometry with respect to the inner products above constructed from \\ $a_y ^{-n} I_{a^n x, a^n y} a_x ^n | _{{E} _{\lambda} (x) } \rightarrow H^{s, a} _{x,y} $  where for $x, y$ two nearby points we let $I_{xy} : E_\lambda(x) \to  E_\lambda(y)$ be a linear identification which is H\"older close to the identity.\color{black}
\end{itemize}
\end{itemize}
\color{black}

\color{black} Claims (a), (b) and (c) are proved in Sections \ref{regularity-metrics} and \ref{Oseledets regularity Anosov}. \color{black} We note that the last claim  is an application of the invariance principle \cite{AV} or Theorem \ref{thm: acc imp hder} above. Since $a$ has $0$ Lyapunov exponents on $E_{\lambda}$, by Theorem \ref{thm: acc imp hder}, the stable holonomy $H^{s,a}$ preserves each invariant sub-bundle almost everywhere. And by H\"older continuity of $H^{s,a}$ and $E^{c_i\lambda}$, we know the stable holonomy preserves each $E^{c_i\lambda}$ everywhere. By \cite{KalSad}, the stable holonomy preserves the conformal structure within each $E^{c_i\lambda}$ everywhere. To show (d) we only need to prove $H^{s,a}$ preserves the volume form induced by the %H\"older continuous 
metric within each $ E^{c_i\lambda}$. But it is not hard to see that the Jacobian of $H^{s,a}_{x,y}|_{E^{c_i\lambda}}$ is exactly the holonomy of the one-dimensional cocycle $\mathrm{Jac}(Da|_{E^{c_i\lambda}})$, the stable holonomy preserves the invariant volume form within each $E^{c_i\lambda}$, thus the stable holonomy is an isometry.

\color{black}

%\st{Moreover,  since the $M$-action in the assumptions of Theorem \ref{abelian} commutes with the $\R^k$ action, we may assume without loss of generality that the $\norm{\cdot}_{\alpha,i}$ are invariant under the $M$-action by averaging its pushforwards under $M$ with respect to the Haar measure on $M$. } 
%All this implies that   \ref{FP2}  holds for actions satisfying assumptions in Theorem \ref{abelian}. %\footnote{It seems now we could make our FA-2 weaker, make it only $C^0$, which actually implies H\"older continuous along coarse, be our FA-2' that will heavily use.}. 

%\begin{enumerate}[label={\rm (FA-\arabic*)}]
%\item \label{fundamental1} For every $\alpha \in \Delta$, $\ker \alpha {\color{black} \times M?}$ has a dense orbit.
%\item \label{fundamental2} The Oseledets splitting is H\"older continuous, and for each Oseledets space $E^{c_i^\alpha\alpha}$, there exists a H\"older norm $\norm{\cdot}_{\alpha,i}$ on the bundle $E^{c_i^\alpha\alpha}$ such that for every $a \in \R^k \times M$ and $v \in E^{c_i^\alpha\alpha}$,

%\[ \norm{a_*v}_{\alpha,i} = e^{c_i^\alpha\alpha(a)}\norm{v}_{\alpha,i}. \]
%\end{enumerate}

%\begin{remark}
%Since the $M$-action commutes with the $\R^k$ action, we may assume without loss of generality that the $\norm{\cdot}_{\alpha,i}$ are invariant under the $M$-action by averaging its pushforwards under $M$ with respect to the Haar measure on $M$.
%\end{remark}
\color{black}
In the next section we perform the initial step in the proof of \color{black} Theorems \ref{basic abelian} and \ref{abelian}. From the assumptions in these theorems, in the previous section we derived the fundamental properties \ref{FP1} and \ref{FP2}.
In Section \ref{extension} we use these properties to construct a simply transitive subgroup of isometries of a leaf $W^\lambda(x)$ of coarse Lyapunov foliation, for each $\lambda$. For this purpose we assume that each Oseledets subbundle is %the leaves of the foliations are 
orientable, otherwise we just may lift the action to a suitable finite cover.  
\color{black}

After that, in Section \ref{fibration}, we construct  a principal bundle extension (with a compact structure group) over $X$ which, as it turns out, satisfies all the condition of a genuinely higher-rank \color{black}\emph{harnessed abstract partially hyperbolic action}, (defined in Section \ref{sec:top-anosov})\color{black}. These are the topological actions which we prove are modeled by the homogeneous ones in Part 4. of the paper. 
%Roughly speaking the main feature of these actions is that they come together with groups of isometries acting transitively on the leaves of coarse Lyapunov foliations. 

A key point which motivates the need for the construction in Section \ref{fibration}, is the following: for each $x \in X$, we will be able to construct a simply transitive subgroup of isometries of $W^\lambda(x)$. All such groups will be isomorphic to one another, denoted by $\leftN$. However, this does not immediately give an action of a group on the entire manifold. Indeed, in Example  \ref{ex:SLnC}, each coarse Lyapunov leaf is a copy of $\C$, but there is no global action of $\C$ on $X$ which parameterizes the leaves on $X$. This can be resolved by passing to some compact fiber bundle over X (which is hinted at in Example \ref{ex:SLnC}).
%\todo{comment(13) completed the sentence} 
In case when  $W^\lambda$ are 1-dimensional, as in \cite{Spatzier-Vinhage}, \color{black} this problem (failing to construct a globally-defined group action) \color{black} does not occur. In Section \ref{fibration}, we will carefully choose a compact fiber bundle extension $\hat X$ over $X$ so that the group $N^\lambda$ acts on it in a canonical way. 
Similar problem occurs for higher-rank semisimple Lie group actions such as Example \ref{ex:SL3-SU3}.

\part{Compact group extensions and trivialization of Lyapunov frames}

\label{part:build-bundle}

\section{Construction of a large group of isometries %and construction of the fibration
}\label{extension}

%From this we want to construct an action of one group $N^\lambda$ acting transitively on leaves of $W^\lambda$ by isometries.  But a key point here is that there is no canonical way to model isometries of $W^\lambda$ leaves  $W^\lambda(x)$ for various  $x$ by $N^\lambda$  (see example \ref{ex:SLnC}). In case when  $W^\lambda$ are 1-dimensional, as in \cite{Mega paper} does not occur. To overcome this problem, in the next section, we will carefully choose a compact fiber bundle extension $\hat X$ over $X$ such that $N^\lambda$ acts on it in a canonical way.

% In this section we fix a coarse Lyapunov foliation $W^\lambda$. 

%\begin{definition} \label{metric-W-lambda}
%\subsection{Constructing a large group of isometries} 
%\color{black}%(GO THROUGH THIS SECTION MORE SERIOUSLY)
%\label{sec:isometries}
Throughout this section we assume that \color{black} $\R^k \curvearrowright X$ is a $C^\infty$ totally partially hyperbolic \color{black} action satisfying \ref{FP1} and \ref{FP2}. In particular, for $v,w \in T_x W^{\lambda} = \oplus _{I=1} ^{n_{\lambda}} E^{c_i \lambda} (x)$ we have the metric 

$$\langle v,w \rangle = \sum  _{I=1} ^{n_{\lambda}} \langle v_i,w_i \rangle  _{c_i \lambda},$$
where $v_i, w_i$ are the $E^{c_i \lambda}$ components of $v$ and $w$ respectively. Denote by  $\Isom( W^{\lambda} (x))$  the group of isometries with respect to this metric.

\color{black} From this point on we assume that each Oseledets subbundle is
%the leaves of the coarse Lyapunov foliations are 
orientable, otherwise we lift the action to a finite cover. \color{black}

%\end{definition}

\begin{definition}

Call an isometry $\phi :  W^{\lambda} (x) \to   W^{\lambda} (x)$ {\em harnessed} if $\phi _* $ preserves the Oseledets subbundles and their orientations.  Let $\Isom^{\lambda}_H(x):=\Isom_H (W^{\lambda}(x))$ be the group of harnessed isometries.

\end{definition}

Note that the group of harnessed isometries $\Isom^{\lambda}_H(x)$ is a closed subgroup of  $\Isom( W^{\lambda} (x))$ because the limits in $\Isom^{\lambda}(x)$ preserve Oseledets frame.

The main outcome of this section is the following proposition:

\begin{proposition}
    \label{prop:isometry-construction} Let \color{black}  $\R^k \curvearrowright X$ be a $C^{2}$ (or $C^\infty$ resp.) totally partially hyperbolic \color{black} action satisfying \ref{FP1} and \ref{FP2}.

Fix a coarse Lyapunov foliation $W^{\lambda}$.  Then there exists a nilpotent Lie group $\leftN$  such that for all $x \in X$ there exists a subgroup $\leftN_x \subset \Isom_H( W^{\lambda} (x))$ with the following properties: 

\begin{itemize}

\item[(1)] $\leftN_x$ isomorphic to $\leftN$,

\item[(2)]  for all $h \in \leftN_x$, $h_*$ preserves the Oseledets splitting $T W^{\lambda} = \oplus _{I=1} ^{n_{\lambda}} E^{c_i \lambda}$,

\item[(3)] $\leftN_x$ acts transitively on $W^{\lambda} (x)$,

\item[(4)] $a \, {\leftN_x} \,a^{-1} = \leftN_{ax}$ for all $x \in X$ and for all $a \in \color{black}\R ^k\color{black}$.  

\item[(5)]  $\leftN_x$ acts by  $C^{(1,\theta)}$ (or $C^\infty$ resp.) 
 diffeomorphisms {\color{black}of $W^\lambda(x)$}. 

\end{itemize}

%-such that for all $h \in H_x$, $h_*$ preserves the Oseledets splitting $T W^{\lambda} = \oplus _{I=1} ^{n_{\lambda}} E^{c_i \lambda}$ and $H_x$ acts transitively on $W^{\lambda} (x)$.  

\end{proposition}

%\begin{proof} 

We first easily show the following  property of harnessed isometries:

\bl \label{aHa-inv}
For every $a \in \R^k$,
$\Isom^{\lambda} _H (ax) = a \circ \Isom^{\lambda} _H (x) \circ a^{-1}$.

\el

\bp

Choose any harnessed isometry $\phi\in \Isom^{\lambda} _H (x)$ at $x$. $\phi$ is a harnessed isometry if and only if  $ \phi $ at every point of $W^{\lambda } (x)$ is an orthogonal matrix that preserves the Oseledets subspaces.  Then note that $a \circ \phi \circ a^{-1}$ satisfies the same properties. It clearly preserves the Oseledets splitting, and $a$ undoes what $a^{-1}$ does on each Oseledets space. 

\ep

We will use the following Lemma, the proof of which is essentially the same  as the proof of Theorem 2.8 in \cite{Goetze-Spatzier}, for completeness we sketch the proof here.

%The sketch of the proof of the Lemma below can be found in the Appendix.

\bl \label{GS} Let $x,y\in X$. If there exists a sequence $a_k\in \ker \lambda$ such that $\lim_{k\to \infty}a_k(x)\to y$. Then there exists a subsequence $k_j$ and an isometry map $a_0:W^\lambda(x)\to W^{\lambda}(y)$ such that $a_0=\lim_{j\to\infty} a_{k_j}|_{W^\lambda(x)}, Da_0=\lim_{j\to\infty} Da_{k_j}|_{W^\lambda(x)}$, where $D$ denotes the derivative.

\el

\bp For the first claim, the proof is essentially the same as that of Proposition 2.9 in \cite{Goetze-Spatzier}. For completeness we sketch the proof here, the idea is to use $C^\infty$ metric to approximate original H\"older continuous metric to show that if $a_{k_n}(x_i)\to y_i, i=1,2$ then $d_{W^\lambda}(x_1, x_2)\geq (1-\epsilon) d_{W^\lambda}(y_1,y_2)$ for any small $\epsilon>0$. Notice that by classical diagonal argument, we can find a subsequence $k_j$ such that $a_{k_j}$ converges on a dense set $\{x_i\}\subset W^\lambda(x)$. Then combining with the last inequality we can extend $\lim a_{k_j}$ to a Lipchitz continuous map $a_0: W^\lambda(x)\to W^\lambda(y)$, with Lipchitz constant bounded by $1$.

By similar approximation argument, we show that $a_0$ is invertible and $a_0^{-1}$ is also Lipchitz continuous with Lipchitz constant bounded by $1$, therefore $a_0$ is an isometry on $W^\lambda(x)$.

For the second claim, by the first claim and \ref{FP1}, as the proof of Theorem 2.8 of \cite{Goetze-Spatzier} we may easily build a  homogeneous space structure on every $W^\lambda$ leaf such that for any point $x$ the group $\Isom(W^\lambda(x))$ acts transitively on $W^\lambda(x)$. 
 %\color{black} Since the metric is H\"older continuous, by Theorem \ref{thm:taylor} 
 \color{black}
 By Lemma \ref{lem:coarse-lyap-reglow} we know the action of $\Isom(W^\lambda(x))$ is $C^{1+}$. \color{black} Therefore the $C^{\infty}$ homogeneous space structures of $W^\lambda$ leaves are $C^{1+}$ equivalent to the original differentiable structure. Therefore by the corresponding proof in \cite{Goetze-Spatzier}, we know that although the metric along $W^\lambda$ is only \color{black} H\"older continuous, but locally the uniqueness of length minimizing geodesics holds. As a consequence, if we take the subsequence $a_{k_j}\to a_0$ in the first claim, since the derivative of an isometry is determined by how geodesics are mapped, we get that $Da_{k_j}$ also converges to $Da_0$. For more details see \cite{Goetze-Spatzier}.\ep

Similarly we can show the transitivity of the $\Isom^{\lambda}_H(x)$ action for points with dense $\ker \lambda$ orbit.

\bl \label{transitive-x} If $x$ has a dense $\ker \lambda$ orbit, then $\Isom^{\lambda}_H(x)$ acts transitively on $W^\lambda(x)$.

\el

\bp  If $x$ has a dense $\ker \lambda $ orbit, then for any $x'\in W^\lambda(x)$, we can find a sequence $a_k\in \ker \lambda$ such that $\lim_{k\to \infty}a_kx=x'$. As a consequence of lemma \ref{GS} we can pick a subsequence of $a_k$ such that this subsequence uniformly converges to a harnessed isometry $a_0$ on  $W^\lambda(x)$ such that $a_0(x)=x'$. Since $x'$ was arbitrary, this implies that  $\Isom^{\lambda}_H(x)$ acts transitively on $W^\lambda(x)$.

\ep

The next lemma combines the previous two lemmas.

%\todo{comment(14) Clarify whether the isomorphism is (necessarily) canonical or not}

\bl \label{iso-x-iso-y}
Suppose $x\in X$ has a dense $\ker \lambda-$orbit, then for any $y\in X$
the group  $\mathrm{Isom}^{\lambda}_H(y)$ is isomorphic to $\mathrm{Isom}^{\lambda}_H(x)$, and $\Isom^{\lambda}_H(y)$ acts transitively on $W^\lambda(y)$.
\el

\bp  Let $y\in X$. If $x\in X$ has a dense $\ker \lambda$-orbit, there exists $a_k\in \ker \lambda$ such that $\lim_{k\to \infty} a_kx=y$. By Lemma \ref{GS} there exists a subsequence $k_j$ and an isometry map $a_0:W^\lambda(x)\to W^{\lambda}(y)$ such that $a_0=\lim_{j\to\infty} a_{k_j}|_{W^\lambda(x)}, Da_0=\lim_{j\to\infty} Da_{k_j}|_{W^\lambda(x)}$. Since each $a_{k_j}$ is a harnessed isometry from $W^\lambda(x)$ to $W^\lambda(a_{k_j}x)$, $a_0$ is a harnessed isometry from $W^\lambda(x)$ to $W^{\lambda}(y)$ as well. And $a_0$ induces an isomorphism between $\mathrm{Isom}^{\lambda}_H(x)$ and $\mathrm{Isom}^{\lambda}_H(y)$ by Lemma \ref{aHa-inv}. This proves the first claim in the Lemma.

 For the second claim in the lemma, pick any $y'\in W^\lambda(y)$; as in the previous part of the proof we can construct a harnessed isometry $a_0$ from $W^\lambda(x)$ to $W^\lambda(y)$ and $a_0(x)=y$. Let $a_0^{-1}\cdot y':=x'$. By Lemma  \ref{transitive-x} , for any $x'\in W^\lambda(x)$, there exists an $a'\in \Isom^{\lambda}_H(x)$ such that $a'x=x'$, then $a_0\circ a'\circ a_0^{-1}$ is a harnessed isometry in $\Isom^{\lambda}_H(y)$ which maps $y$ to $y'$.

\ep

The remaining part of the argument is to show that the group from the previous lemma has a large nilpotent subgroup $\leftN$ which is the group we need in Proposition \ref{prop:isometry-construction}.

{\color{black}
\bl \label{nilpot.}
If $p$ is a point with a dense $\ker \lambda$-orbit, then $\Isom^{\lambda} _H (p)=K_p^{\lambda} \ltimes \leftN_p$ where $K_p^{\lambda}$ is a compact Lie group and $\leftN_p$ is a normal simply connected nilpotent Lie subgroup, and $\leftN_p$ acts simply transitively on $W^\lambda(p)$. 
\el
%\bl \label{nilpot.}
%At a periodic point $p$, $\Isom^{\lambda} _H (p)=K_p^{\lambda} \ltimes N_p^{\lambda}$ where $K_p^{\lambda}$ is a compact Lie group and $N_p^{\lambda}$ is simply connected nilpotent Lie group, and $\leftN_p$ acts simply transitively on $W^\lambda(p)$. 
%\el
}

 Lemma \ref{nilpot.} is a direct corollary of the following more general proposition.

\bl \label{Independent prop}

Let $({ X},d)$ be a connected complete Riemannian manifold, $G$ be a locally compact topological group acting isometrically and transitively on ${ X}$. If there exists a strictly contracting $a\in \Diff({ X})$ such that $a$ normalizes   the $G$-action\color{black}, then

\begin{itemize}

\item[(1)] there exists a unique nilpotent subgroup $N$ normal in $\Isom ({ X})$ and acting simply and transitively on ${ X}$,

\item[(2)] one can isometrically identify $X$ with $N$, the metric $d$ on ${ X}$ is identified with the left-invariant metric on $N$,

\item[(3)] for any $p_0\in { X}$, $\Isom({ X})$ is the semidirect product of $N$ with the isotropic group $K_{p_0}:=\{g\in \Isom(X), kp_0=p_0\}$,

\item[(4)] for any $p_0\in { X}$, $G$ is the semidirect product of $N$ with the isotropic group $K'_{p_0}:=\{g\in G, kp_0=p_0\}$.
%\color{teal}$K$ or $K'$?\color{black}
\end{itemize}
\el
\bp 
By Theorem \ref{thm:taylor} we know all elements of $G$ are $C^1$, hence by \cite{CM}, $G$ is a locally compact Lie group and the action by $G$ is a $C^1$ Lie group action. By completeness of $({ X},d)$, there exists a unique $a$-fixed point $p\in { X}$. Since $aGa^{-1}=G$, the conjugacy by $a$ induces a Lie group automorphism of $G$, hence a Lie algebra automorphism $\Phi_a$ of $\mathfrak g:=\mathrm{Lie}
(G)$.

%By contracting property of $a$, there is no   non-trivial \color{black} $g\in G$ such that $a^{-n}g a^{n}, n\to \infty$ converges to identity on a compact neighborhood of $p$. Hence the automorphism $\Phi_a:\mathfrak g \to \mathfrak g$ is non-expanding, i.e. no eigenvalue has modulus greater than $1$.

%{\color{teal}I forgot the reason why it is non-expanding. The proof written here at least not true if $g=id$.}

Denote by $V_\mu$ the generalized eigenspace of $\Phi_a$ for eigenvalue $\mu \in \mathbb C$. Using the fact that %\todo{comment(15) May require more justification-see comment} 
$[V_{\mu_1}, V_{\mu_2}]\subset V_{\mu_1\cdot \mu_2}$ we get that $\mathfrak n:=\oplus_{|\mu|<1} V_\mu$ forms a nilpotent Lie subalgebra of $\mathfrak g$. We denote by $N$ the connected nilpotent Lie group $\exp(\mathfrak n)$, then $aNa^{-1}=N$.

\color{black} We show first that $ \oplus_{|\mu|>1} V_\mu \subset \Lie(K_p)$, where $K_p$ is the stablizer of $p$ in $G$. Assume, for a contradiction, that $Z \in \oplus_{|\mu|>1} V_\mu$ has $\exp(tZ)\cdot p \not= p$ for all sufficiently small $t > 0$. Let $\delta_t = d(\exp(tZ)\cdot p,p)$, and note that $\delta_t > 0$ for all $t>0$, and $\delta_t \to 0$ as $t \to 0$. Since $a$ is contracting, for sufficiently small $\ve >0$ and some $0 < \lambda < 1$, 

\begin{equation}
    \label{eq:annuli-move}
    a^{-1}\cdot (B(p,\ve)\setminus B(p,\lambda\ve)) \cap B(p,\lambda\ve) =\emptyset.
\end{equation}
By the intermediate value theorem, we may find, for every $\ve >0$, some $t$ such that $\delta_t \in (\lambda\ve,\ve)$. Finally, without loss of generality, we may assume that $\Ad(a^{-1})$ preserves $B(0,\delta) \subset \Lie(G)$ for every $\delta > 0$ by replacing $a$ with a sufficiently large power. 

Now, observe that $a^{-1} \exp(tZ) \cdot p = \exp(t\Ad(a^{-1}Z))\cdot p$. Since $Z$ is in the sum of positive generalized eigenspaces, 
%\todo{comment(16) clarify wether $a^{-1}$ should be replaced by a larger power}
$\exp(t\Ad(a^{-1})Z)\cdot p \in B(p,\lambda \ve)$. This is a contradiction to \eqref{eq:annuli-move}, so $\oplus_{|\mu|>1} V_\mu \subset \Lie(K_p)$.
\color{black}

%  Let $\ve > 0$ be any sufficiently small positive number, and $O_\ve$ be the annulus $B(p,2\ve) \setminus B(p,\ve)$. Since $a$ is a contraction, it follows that $a\cdot B(p,\ve) \subset B(p,\ve)$.% Fix a small neighborhood $O_p$ of $p$ such that $a(\overline{O_p})\subset \mathrm{Int}(O_p)$.Then by definition By the contracting property of $a$, for any $g=\exp(Z)$ such that $Z$ small enough and $Z\in \oplus_{|\mu|>1} V_\mu-\{0\}$, then $a^{-n}g a^{n}$ converges to the identity on a compact neighborhood of $p$ as $n\to \infty$. On the other hand if $g$ moves $p$, then  $a^{-n}g a^{n}(p)$ can not converge to the identity on a compact neighborhood of $p$ as $n\to \infty$. Hence we show that $ \oplus_{|\mu|>1} V_\mu \subset K_p$. }

Now let $\mathfrak k:=\oplus_{|\mu|=1} V_\mu$.   \color{black} Note that $\mf k$ is a subalgebra since $\Ad(a)$ is an automorphism of the Lie algebra, so whenever $\norm{\Ad(a)^kv}$ and $\norm{\Ad(a)^kw}$ are bounded above and below by polynomials for positive and negative values of $k$, so is $\norm{\Ad(a)^k[v,w]} = \norm{[\Ad(a)^k v,\Ad(a)^kw]}$. This property characterizes the generalized eigenspaces of modulus 1.\color{black}

We claim that $\exp(tZ)\cdot p=p$ for all ${ Z}\in \mathfrak k$. Let $\mathfrak k_0$ be the subalgebra in $\mathfrak k$ such that for all ${ Z}\in \mathfrak k_0,t\in \mathbb R$, $\exp(t{ Z})\cdot p=p$. Assume now  $\mathfrak k_0\ne  \mathfrak k$. Choose $\epsilon$ small enough such that $\exp$ is injective and close to identity at an $\epsilon$ ball around $0\in \mathfrak g$. A useful fact is that for fixed $\epsilon'\ll\epsilon$ small enough, for any $Y$ in $\mathfrak k\setminus \mathfrak k_0$ such that $\|Y\|<\epsilon'$, $d(\exp (Y)\cdot p, p)$ positive and has order $O(\|Y\|\cdot |\angle(Y, \mathfrak k_0)|)$ (if $\mathfrak k_0=\{0\}$, then without loss of generality we could assume the angle to be constant $1$).

We pick now an arbitrary ${ Z}\in \mathfrak k \setminus \mathfrak k_0$ and denote by $Y_n:= \Phi_a^{n}({ Z})$, then $\|Y_n\|$ has order at least $\|{ Z}\|$, \color{black} up to a polynomial factor of $n$\color{black}. Take $t_n=\frac{\epsilon'^2}{\|Y_n\|}$. Then notice that both $t_n$ and the angle $\angle(Y_n, \mathfrak k_0)$ are bounded and either do not decay as $n\to\infty$ or if they decay they do so at most polynomially fast. So we have $d(\exp(t_nY_n))\cdot p, p)$ has the order $O(\epsilon'^2\cdot |\angle(Y_n, \mathfrak k_0)|)$, which cannot decay exponentially fast.

On the other hand:
%\todo{comment(17) fixed}
\begin{equation}
d(\exp(t_nY_n))\cdot p, p)=d(a^{n}\exp(t_n{ Z})a^{-n}\cdot p,p)=d(a^{n} \exp(t_n{ Z})\cdot p,a^{n}\cdot p)\leq O (\color{black} \|Da\|\color{black}^n d(\exp(t_n{ Z})\cdot p, p))
\end{equation}
decays exponentially fast due to our choice of $t_n$. Then we get a contradiction. In summary, $\mathfrak k_0=\mathfrak k$.

Recall that $G$ acting on ${ X}$ transitively, therefore for arbitrary small open neighborhood $B$ of identity in $G$, $B\cdot p$ contains $p$ as an interior point. Since $\exp(\mathfrak k)$ fixes $p$, we get $(B\cap N)\cdot p$ contains an  open neighborhood of $p$. Notice that $N=a^{-n}\circ N\circ a^{n}$, therefore $N\cdot p\supset a^{-n}\cdot (B\cap N) \cdot a^n(p), n\to \infty$ contains exponentially large neighborhood of $p$, hence $N\cdot p={ X}$, i.e. $N$ acting transitively on ${ X}$.

In summary, we showed that $G$ hence $\Isom({ X})$ contains a nilpotent Lie subgroup acting transitively on ${ X}$. {  Thus $X$ is a homogeneous nilmanifold as defined in  \cite{Wilson}.  Then  (1)-(3) of Lemma  \ref{Independent prop} follow from { Theorem 2} in \cite{Wilson}, and (4) is an easy corollary of (3).}
\ep

\bp[\it Proof of Lemma \ref{nilpot.}]

 Let $\Isom^\lambda_H(p)$ and  $W^\lambda(p)$ of Lemma \ref{nilpot.} be $G$ and $X$\color{black} \, respectively, in the Lemma \ref{Independent prop}. \color{black} By Lemma \ref{Independent prop}, to\color{black} \, complete the proof of Lemma \ref{nilpot.} we only need to show the existence of a \color{black} function which contracts $W^\lambda(p)$ and normalizes the $G$-action. Choose some $a_0 \in \R^k$ such that $\lambda(a_0) < 0$. Since $p$ has a dense $\ker \lambda$-orbit, so does $a_0 \cdot p$. Hence there exists a sequence $b_n \in \ker \lambda$ such that $(a_0 + b_n) \cdot p = b_n \cdot (a_0 \cdot p)$ converges to $p$. Then as in Lemma \ref{GS}, we may pass to a convergent subsequence to get a contraction of $W^\lambda(p)$ which fixes $p$ and normalizes the $G$-action, as desired. 
 %\color{black} an element of the $\mathbb R^k$-action which contracts $W^\lambda(p)$ and normalizes the $G$-action. This is guaranteed by the existence of $a\in \mathbb R^k$ such that $\lambda (a) <0$ and by Lemma \ref{aHa-inv}.
 %This is possible since the stabilizer of $p$ contains $\Z^k$  as a cocompact subgroup. Hence the stablizer of $p$ in $C=\mathbb R^k\times M$ is not contained in the kernel of any Lyapunov functional. 
 \ep

Lemma \ref{nilpot.} together with Lemma \ref{iso-x-iso-y} implies that the group $(\Isom ^{\lambda}_H (x))$ for any $x$ is isomorphic to $K^{\lambda} \ltimes \leftN$. We show that the splitting is canonical in the sense that $K^{\lambda}$ and $\leftN$ are isomorphic to $K^{\lambda}_x$ and $\leftN_x$, respectively.  As in Lemma \ref{nilpot.} $K^{\lambda}$ is compact and  $\leftN$ is simply connected nilpotent. Since  $ \Isom ^{\lambda}_H (x)$ acts transitively, so does the subgroup  $\leftN_x$. This completes the proof of Proposition \ref{prop:isometry-construction}.

Now we connect the Lie algebra structure of $\leftN$ to the grading given by the Oseledets splitting. 
\bl
\label{lem:rich-automorphism}
For any $x$, $\Lie(\leftN)$ is canonically isomorphic to $\oplus E^{c_i \lambda} (x)$ as a vector space. For every $r \in \R$, the map scaling each $E^{c_i \lambda} $ by $e^{c_i r}$ is a {\color{black} harnessed} automorphism of $\Lie(\leftN)$.
\el

\bp

As the proof of Lemma \ref{nilpot.} we know that we can identify $\leftN$ harnessed isometrically with any $W^\lambda(x)$. Therefore we can canonically identify (harnessed isometrically) $\Lie(\leftN)$ with the tangent space $\oplus E^{c_i \lambda} (x)$. And it induces a splitting of $\Lie(\leftN)$. As the proof of Lemma \ref{iso-x-iso-y}, this splitting is actually independent of the choice of $x$. As the proof of Lemma \ref{nilpot.} we could take a $p$ with a compact $C$-orbit and $ap=p$ such that $\lambda= \lambda(a)<0$, then without loss of generality we can assume that $\Lie(\leftN)$ is %\todo{comment(18)} 
isomorphic  to $\Lie(\leftN_p)=\oplus E^{c_i \lambda}_p$.

Picking suitable basis elements for the $E^{c_i \lambda}$ to shows the claim we only need to show that $[E^{c_i \lambda}, E^{c_j \lambda}] \subset  E^{(c_i+c_j) \lambda}$. %Since 

Let $X,Y$ are invariant vector fields tangent to $E^{c_i\lambda}, E^{c_j\lambda}$ respectively, with the non-vanished $[X, Y]$. Then we can find a subsequence ${n_j}\to \infty$ such that

$$ \frac {1}{n_j} \log \| a_* ^{n_j} [X , Y ]\| = \frac {1}{n_j} \log \| [a_* ^{n_j}  X  ,a_* ^{n_j}  Y ]\|  =
\frac {1}{n_j} \log e^{(c_i + c_j) \lambda n_j}\| [(k_1  ^{n_j})_*  X  ,(k_2 ^{n_j})_*  Y ]\| \rightarrow  {(c_i + c_j) \lambda}\| [  X  ,  Y ]\|
$$
for suitable $k_i \in SO(E^{c_i \lambda})$ for which $k_i ^{n_j} \rightarrow id$, $i=1,2$. { By \ref{FP2}}, the limit of $\frac {1}{n} \log \| a_* ^{n} [X , Y ]\| $ can be decided by its behavior along a subsequence $n_j$. This implies that $[E^{c_i \lambda} ,  E^{c_j \lambda} ] \subset E^{(c_i + c_j) \lambda}$.

\ep

\begin{lemma}

\label{lem:coarse-lyap-reglow}

If \, \color{black}$\R^k \curvearrowright X$ is a $C^{2}$ totally partially hyperbolic \color{black} action satisfying \ref{FP1} and \ref{FP2}, then for every $\lambda$, each  $g \in \Isom(W^{\lambda}(x))$ 
%\st{(which acts transitively by Lemma} \ref{prop:isometry-construction}) 
is a $C^{1,\beta}$ diffeomorphism. \color{black}

\end{lemma}

\begin{proof} From properties  \ref{FP1} and \ref{FP2} we have that the regularity of the metric in the Oseledets spaces is \color{black} uniformly H\"older along $W^\lambda$ \color{black} with some H\"older exponent $\beta$.  Theorem \ref{thm:taylor} implies that the isometries are $C^{1, \beta}$-transformations on each leaf. 

\color{black}

\end{proof}

If in addition if the given action is $C^\infty$, we obtain two additional results:

\begin{lemma}

If the action is $C^\infty$, both the subspaces $E^{c_i\lambda}$ and the %H\"older 
metric $\langle,\rangle$ are $C^\infty$ along the coarse Lyapunov leaves of $W^\lambda$.

\end{lemma}

\begin{proof}

Recall that the action of the harnessed isometry group $\Isom_H(W^\lambda(x))$ is obtained from taking limits of elements $a \in \ker \lambda$. In normal form coordinates (cf. Appendix \ref{app: nrml form}), these maps are given by sub resonance polynomials. Therefore, the limits of such maps are given by polynomials, and are hence $C^\infty$. Therefore, $\Isom_H(W^\lambda(x))$ has a subgroup of $C^\infty$ isometries acting on $W^\lambda(x)$. Since $E^{c_i\lambda}_x = h_*(E^{c_i\lambda}_y)$ for any harnessed isometry $h$ such that $h(x) = y$. Therefore each $E^{c_i\lambda}$ can be viewed a $C^\infty$ homogeneous graph of a continuous mapping from $W^\lambda(x)$ to corresponding Grassmannian space over $W^\lambda(x)$, hence this graph is locally compact, therefore by the result in \cite{RSS} we conclude that the splitting is $C^\infty$ along $W^\lambda$. Similarly, the a priori %H\"older 
continuous metric $<,>$ can be viewed as a $C^\infty$ homogeneous graph of a continuous mapping from $W^\lambda(x)$ to corresponding modular space of quadratic forms over $W^\lambda(x)$, then by the same proof we know it is actually $C^\infty$ along $W^\lambda(x)$.
\end{proof}

As a direct corollary of last lemma and Theorem \ref{thm:taylor}, we have

\begin{lemma}
\label{lem:coarse-lyap-smooth}
If \, \color{black} $\R^k \curvearrowright X$ is a $C^\infty$ totally partially hyperbolic \color{black} action satisfying \ref{FP1} and \ref{FP2}, then for every $\lambda$, each $g \in\Isom(W^\lambda(x))$  is a $C^\infty$ diffeomorphism.
\end{lemma}

\section{Construction of the fibration}\label{fibration}
%{  need to look over and be careful; make sure v.f.'s integrate}

Throughout this section  we are assuming that we have an action $\mathbb R^k\curvearrowright X$ satisfying assumptions of Theorem \ref{abelian} or Theorem \ref{basic abelian}, which therefore has properties \ref{FP1} and \ref{FP2}. Recall that we assumed in Section \ref{extension} that the coarse Lyapunov foliations of the action are oriented.  
%(see Section \ref{sec:intermission1}).

{\color{black}
    In this section, we combine the groups $\leftN_x$ of leafwise isometries constructed in Section \ref{extension} to form an action of a nilpotent group $N^\lambda$ on some compact fibration of $X$ which does not depend on the basepoint. The construction here is rather subtle, due to two key complexities:
    
    \begin{itemize}
        \item The isometries built in the previous section are obtained as limits of the kernel action. In the setting of homogeneous spaces, the coarse Lyapunov foliations are orbits of a nilpotent group $N^\lambda$. Given a point $g\Gamma \in X$, a sequence $a_k \in \ker \lambda$ and an element $u \in N^\lambda$, we obtain an isometry of $N^\lambda g\Gamma = W^\lambda(g\Gamma) = \leftN_{g\Gamma}\cdot g\Gamma$ as a limit of $a_k$. That is, if $a_k g\Gamma \to u g\Gamma$, and $v \in N^\lambda$, then $a_kvg\Gamma = va_kg\Gamma \to vug\Gamma$. Importantly, note that the isometry action determined by the sequence $a_k$ is a {\it right} translation action on $N^\lambda g\Gamma$, and in fact must be if it is an isometry. These actions can never be extended to a global action on $X$, even in this homogeneous setting.

        Therefore, the strategy is as follows: we convert the right actions on each individual leaf $W^\lambda(x)$ (which we think of as $N^\lambda g\Gamma$) into left actions by considering the action generated by the vector fields invariant under $\leftN_x$. This will indeed yield the desired action. See Lemma \ref{lem:dual-brackets} and Corollary \ref{Chi}.

        \item The other complexity arises from the fact that not every model is a homogeneous space, but only a bi-homogeneous space. In particular, while to define an $\R^k$ action on $K \backslash H / \Lambda$, $\R^k$ must commute with $K$, it is not true that the nilpotent subgroups of $H$ normalized by $\R^k$ must commute with $K$. In particular, no left action can exist on the manifold. Indeed, one must ``undo'' the quotient by $K$. This process amounts to building a principal $K$-bundle over the bi-quotient. To do so, we build a bundle which frames the nilpotent groups correctly, see Lemma \ref{lem:big-bundle}.
    \end{itemize}

    Resolving these two complexities simultaneously is the main achievement of this section. The main difficulty is to make sure that there is a nilpotent Lie algebra $\mf n^\lambda$ such that given a ``good'' framing of the $T_xW^\lambda(x)$, the corresponding invariant vector fields correspond to the Lie algebra $\mf n^\lambda$, see Corollary \ref{cor:lie-const}. These culminate in defining the lifted $\R^k$-action on the principal bundle (Equation \eqref{eq:lift-def}) and the corresponding actions of $N^\lambda$ in Theorem \ref{thm:lifted-action}.
}

\subsection{Lie algebras and harnessed isometries on vector bundles\color{black}}
\label{sec:lie-bundles}

    For a given coarse Lyapunov exponent $\lambda$, let $\leftN$ denote the group constructed in Proposition \ref{prop:isometry-construction}, and $\leftN_x$ denote the simply transitive subgroup of $\Isom_H(W^\lambda(x))$. Then let $\mathfrak{X}_\lambda(x)$ denote the set of all vector fields on $W^\lambda(x)$ invariant under the action of $\leftN_x$. Recall that $W^\lambda$ have {\color{black}$C^r$ leaves, $r = (1,\theta)$ or $r = \infty$. 

    We establish some general structure theory regarding Lie structures on bundles.
    Let $\mc V$ be a continuous vector bundle over a smooth manifold $X$, and $V_x \subset \mc V$ denote the fiber over $x$. A Lie algebra structure on $V_x$ is a bilinear antisymmetric functional $[\cdot,\cdot]_x : V_x \times V_x \to V_x$ satisfying the Jacobi identity. It is hence an element of $L_x(V) = \Lambda^2(V_x^*) \otimes V_x$. Hence the vector bundle $\mc L(\mc V)$ with fibers $L_x(V)$ contains Lie algebra structures on the fibers of $\mc V$ (the Jacobi identity imposes an additional linear relation).

    Now, assume further that $\mc V$ decomposes as a sum of continuous subbundles $V_x = E^1_x \oplus \dots \oplus E^m_x$, and each $E^i_x$ has a Riemannian metric and orientation. We say that a framing $\varphi : \R^{\ell_1} \times \dots \times \R^{\ell_m} \to V_x$ is {\it harnessed} (by the decomposition) if $\varphi(\R^{\ell_i}) = E^i_x$ and $(\varphi(e_1),\dots,\varphi(e_{\ell_1}))$ is positively oriented. $\varphi$ is an {\it orthonormal harnessed} framing if $\varphi$ is harnessed and preserves the metric. A {\it harnessed isometry} of $V$ is a linear isomorphism which is an orientation-preserving isometry of each subspace $E_i$, and we let $SO_H(V_x) := SO(E^1_x) \times \dots \times SO(E^m_x)$ denote the set of harnessed orthogonal transformations. When $V = \R^n$ and $E^i = \set{0} \times \R^{\ell_i} \times \set{0}$, we denote the group by $SO_H(n)$.

%     %Notice that if $\mf f = (v_1,\dots,v_n)$ is an orthonormal harnessed frame, then there is a unique linear map defined by $e_i \mapsto v_i$ which is an isometry between $\R^{\ell_i}$ and $E_i$ for every $i$. 
     Given a Lie algebra structure $[\cdot,\cdot]_x$ on $V_x$ and a harnessed orthonormal framing $\varphi$ on $\R^{\dim(V_x)}$, it pulls back to a unique Lie algebra structure on $\R^n$ which we denote by $[\cdot,\cdot]_\varphi$.

    \begin{lemma}
    \label{lem:dual-brackets}
        Let $H$ be a Lie group with a simply transitive action on a $C^1$ manifold $X$ by $C^1$ diffeomorphisms. Then the set of $H$-invariant vector fields are $C^0$ and uniquely integrable, and generate an action of $H^{\operatorname{op}}$, where $H^{\operatorname{op}}$ is $H$ with the multiplication $*$ defined by

        \[ h_1 * h_2 = h_2h_1.\]
        Furthermore, the $H$-invariant vector fields are only $C^0$, but have a well-defined Lie bracket.
    \end{lemma}

\begin{proof}
    Fix $x_0 \in X$, and let $\phi : H \to X$ be defined by $\phi(h) = h\cdot x_0$. Notice that $\phi \of L_g(h) = gh \cdot x_0 = g \of \phi(h)$. By taking inverses, it follows that $L_g \of \phi^{-1} = \phi^{-1} \of g$.
        
    The map $\phi$ is a $C^1$ diffeomorphism by Theorem \ref{MZ}, since each $g$ is a $C^{1}$ diffeomorphism.  Therefore, if $v$ is an $H$-invariant vector field on $X$, it is $C^0$. Furthermore, $\hat{v} := \phi^{-1}_*v$, the pushforward of $v$ under $\phi^{-1}$, is a left-invariant vector field on $H$, since

        \[ \hat{v}(gh) = d\phi^{-1}(v(gh \cdot x_0)) = d\phi^{-1}dgdh(v(x_0)), \mbox{ and}\]

        \[ (L_g)_*\hat{v}(h) = (L_g)_*d\phi^{-1}(v(h\cdot x_0)) = d\phi^{-1} dgdh(v\cdot x_0).\]

    Since on $H$, the left-invariant vector fields are $C^\infty$, each $\hat{v}$ is uniquely integrable and $[\hat{v},\hat{w}]$ always exists. Unique integrability is preserved under $C^1$-conjugacy, so we conclude that $v$ is uniquely integrable. Furthermore, we can define a Lie bracket on the set of $H$-invariant by

    \[ [v,w] := \phi_*[\hat{v},\hat{w}].\]

    Finally, observe that on $H$, the left-invariant vector fields generate right translations on $H$. Therefore, an $H$-invariant vector field $v$ on $X$ integrates to the flow:

    \[ \psi_t^v(gx_0) = \phi(\psi_t^{\hat{v}}(g)) = \phi(gh_t) = gh_tx_0 \]
    where $h_t$ is the one-parameter subgroup of $H$ generated by $v$. Since the action is by right translations in the group structure, it is an action of the opposite group, as claimed.
\end{proof}

Note that for any group, $H \cong H^{\operatorname{op}}$ via the isomorphism $h \mapsto h^{-1}$.}

\begin{corollary}\label{Chi}
$\mf X_\lambda(x) \cong \Lie(\leftN^{\operatorname{op}})$ for every $x \in X$, and $\mf X_\lambda(x) = \mf X_\lambda(y)$ if $y \in W^\lambda(x)$.
\end{corollary}

\begin{proof}
%This follows from a general fact: if a Lie group $H$ acts simply transitively on a space $Y$ by diffeomorphisms, the space of invariant vector fields under the action can be pulled back to the space of either right- or left-invariant vector fields on $H$, which is by definition $\Lie(H)$. 
{\color{black}Recall that the $\mf X_\lambda(x)$ is the set of $\leftN_x$-invariant vector fields on $W^\lambda(x)$. The first claim follows from the Lemma \ref{lem:dual-brackets} and Proposition \ref{prop:isometry-construction}.} The last claim follows from the fact that $\mf X_\lambda(x)$ depends only on the manifold $W^\lambda(x)$ and does not depend on $x$ itself (recall that the normal subgroup of harnessed isometries acting simply transitively on a leaf is uniquely determined by Lemma \ref{Independent prop}(1)).
\end{proof}

\subsection{Continuity of Lie structures\color{black}}
\color{black}
In this section we assume thet we have an action as in Theorem \ref{basic abelian} or Theorem \ref{abelian} and we fix a coarse foliation $W^\lambda$ of the action. 
Let $\sigma : X \to \mc L(T W^\lambda)$ denote the section defined by $x \mapsto \mf X_\lambda(x)$, noting that each vector $v \in T_xW^\lambda(x)$ induces a vector field on $W^\lambda$ via the action of $\leftN_x$ and vice-versa.

  %  \begin{theorem} \label{thm:holder-lie} $\sigma$ is $W^\beta$-continuous %along every coarse Lyapunov leaf $W^\beta$ for $\beta \not= -\lambda$ and $\R^k$-continuous. If the action is Anosov, it is $W^{-\lambda}$-continuous   as well. (Add more details?)\end{theorem}

    %\todo{UNIFY curly Ws and non curly}

    The main result of this section is

    \begin{theorem}
    \label{cor:cont-lie}
        $\sigma$ is continuous.
    \end{theorem}

    %which we show in the two subsequent lemmas. The first one, Lemma \ref{lem:holder-lie1}, is going to be used both for the partially hyperbolic accessible case and the Anosov case and it shows stronger leafwise continuity as in Definition \ref{def: F-cont} along the foliations $W^\beta$, $\beta \not= \pm \lambda$. The reason this stronger form of continuity is needed is so that we can apply the invariance principle (Theorem E of \cite{ASV}) in the partially hyperbolic accessible case. The second one, Lemma \ref{lem:holder-lie2},  is showing a less strong continuity property and is restricted to (and will be used for) the Anosov case.  

    %We break up the proof of Theorem \ref{thm:holder-lie} into pieces, namely Lemmas \ref{lem:holder-lie1} and \ref{lem:holder-lie2}. 

We will prove continuity of $\sigma$ by using heavily the versions of invariance priniciple from Section \ref{InvPrinI}. The main point is that continuity is obtained as a consequence of strong continuity (see Definition \ref{def: F-cont}) along several foliations.  The starting point is the following lemma which establishes strong continuity along all the coarse Lyapunov foliations, except for $W^\lambda$ and $W^{-\lambda}$.

\color{black}
    \begin{lemma}
    \label{lem:holder-lie1}
        $\sigma$ is $W^\beta$-continuous, $\beta \not= \pm \lambda$.
    \end{lemma}

    \begin{proof}
        Choose an element $a \in \R^k$ very close to $\ker \lambda$, so that $c_i\beta(a) \le -1$ for all $i$ and $-1 \ll c_i\lambda(a) < 0$ for all $i$. With $a$ sufficiently close to $\ker \lambda$, we may guarantee that $E^s_a = E_\lambda \oplus E^{ss}_a$ is a dominated splitting, where $E^{ss}$ is the sum of the stable coarse Lyapunov distributions which are not $E_\lambda$.

        Since $E^s_a = E_\lambda \oplus E^{ss}_a$, there is a well-defined ``fast stable'' foliation $W^{ss}$, and given $y \in W^{ss}(x)$ a holonomy $h_{x,y} : W^\lambda(x) \to  W^\lambda(y)$ defined by $h_{x,y}(x') = W^{ss}(x') \cap W^\lambda(y)$.

        This holonomy map $h_{x,y}$ has very good properties, notably it is a $C^{1}$-diffeomorphism and in fact by \cite[Theorem 1.3]{Sag} $Dh_{x,y}$ depends continuously on pair $(x,y)$ as long as $y\in W^{ss}(x)$. We justify this more precisely. The result \cite[Theorem 1.3]{Sag} is about stable holonomy between center leaves. In our case here the role of the stable leaves is played by $W^{ss}$ and the role of the center leaves is played by $W^c_a$ which is the foliation tangent to $E^c_a=E_\lambda\oplus E^c$ if $-\lambda\notin \Delta$ or $E_\lambda\oplus E^c\oplus E_{-\lambda}$ if $-\lambda\in \Delta$. \cite[Theorem 1.3]{Sag} requires two things: dynamical coherence and enough (stable) bunching. Dynamical coherence in our situation we have from Section \ref{DynCoh}. By our choice of $a$ sufficiently close to $\ker\lambda$, we can make sure to have as much bunching as needed for the application of \cite[Theorem 1.3]{Sag}), which then gives result for the holonomy between $W^c_a$ leaves. Then by just restricting to the holonomy between $W^\lambda$ leaves, we obtain the needed result for $h_{x,y}$.%Now we consider $Dh_{x,y}$, the key is to identify it with the holonomy of certain continuous cocycle, more precisely the holonomy of $Da|_{E^\lambda}$, since by our choice of $a$, as the holonomy of the derivative cocycle}, 
    %{\color{black} USE OBATA or SAGHIN PAPER for $C^1$/$C^0$ version (rather than $C^{1,\theta}$/$C^\theta$)}, as long as $a$ is chosen sufficiently close to $\ker \lambda$ (see \cite{}). 
     
     We claim that the holonomy is a harnessed isometry. Indeed, note that if we choose $b \in \ker \lambda$ such that $E^s_b = E^{ss}_a$, then $b$ is an isometry between $W^\lambda(x)$ and $W^\lambda(bx)$. Then

        \[ h_{x,y}= (tb)^{-1} \of h_{(tb)x,(tb)y} \of (tb)\]
        and the right hand side converges to a harnessed isometry since $b$ is a harnessed isometry, and $h_{(tb)x,(tb)y}$ is  $C^1$  close to $\id$ when $t \to \infty$. Since the left hand side does not change, it must be a harnessed isometry itself. 

        Now, since $h_{x,y}$ is a harnessed isometry, it follows that $h_{x,y} \of \leftN_x \of h_{x,y}^{-1} = \leftN_y$. It follows that the conjugation map $c_{x,y}$ induced by $h_{x,y}$   is an isomorphism between $\leftN_x$ and $\leftN_y$ and is hence $C^\infty$ in the smooth structures of Lie groups. In particular, it induces an isomorphism of the left-invariant vector fields and the corresponding Lie algebra structures. We conclude that the Lie bracket at $x$ is the pullback of the Lie bracket at $y$ via $Dh_{x,y}(x)$, which varies continuously in $(x,y)$ as long as $y$ is in local $W^{ss}_a$ leaf. It follows that the Lie algebra structures are $W^{ss}_a$-continuous, and hence $W^\beta$-continuous, since $W^\beta$ subfoliates $W^{ss}_a$.
    \end{proof}

\color{black}
   \begin{proof}[ Proof of Theorem \ref{cor:cont-lie} for partially hyperbolic super accessible action as in Theorem \ref{abelian}] Let $a$ be a singular element in $\ker \lambda$. Then $a$ is partially hyperbolic and accessible with stable and unstable foliations decomposing into $W^\beta$s where $\beta\ne \pm \lambda$. Then the fact that $\sigma$ is $W^\beta$-continuous from previous lemma, together with  Proposition \ref{prop: strong ASV thm E} directly implies continuity of $\sigma$.
   \end{proof}

   Now we assume that we have an Anosov action as in Theorem \ref{basic abelian}. In this case we don't necessarily have super accessibility, so to prove Theorem \ref{cor:cont-lie} in the Anosov case we need 
two more lemmas. Fix a coarse Lyapunov foliation $W^\lambda$, and let $E^{c,\lambda}$ denote the distribution $T\R^k \oplus E_\lambda$ if $-\lambda\notin\Delta$ or  $T\R^k \oplus E_\lambda \oplus E_{-\lambda}$ if $-\lambda\in \Delta$.

\begin{lemma}
\label{lem:isom-section}
    If the $\R^k$-action is Anosov, the distribution $E^{c,\lambda}$ uniquely integrates to a foliation $W^{c,\lambda}$, and $\Isom_H(W^{c,\lambda}(x))$ is a Lie group acting transitively on each leaf $W^{c,\lambda}(x)$, where

    \[ \Isom_H(W^{c,\lambda}(x)) = \set{ f \in \Isom(W^{c,\lambda}(x)) : Df \mbox{ preserves Oseledets subbundles and their orientations}}.\]
    Furthermore, for every $x \in X$, there exists a neighborhood $x \in U \subset W^{c,\lambda}(x)$ and a continuous local section $\tau : U \to \Isom_H(W^{c,\lambda}(x))$ such that $\tau(y)x = y$ and $\tau(x) = \id$.
\end{lemma}

\begin{proof}
By the discussions in Section \ref{DynCoh}, for any generic singular element $a\in \ker \lambda$, $a$ is partially hyperbolic with isometric center distribution $E^c_a=E^{c,\lambda}$ defined above. And $E^c_a$ is
%            Choose a singular regular element $a \in \ker \lambda$, so $\beta(a) \not= 0$ for all $\beta\not= \pm \lambda$. Then $a$ is partially hyperbolic with isometric center, so by \cite[Proposition 6]{B03}, the distribution $E^c = T\R^k \oplus E_\lambda \oplus E_{-\lambda}$ is 
uniquely integrable to a central foliation $W^c_a$ for $a$.

        Put a metric on each $ W^c_a(x)$ by declaring the decomposition $T\R^k \oplus E_\lambda \oplus E_{-\lambda}$ or $T\R^k \oplus E_\lambda$ to be into orthogonal subspaces, and putting the intertwined metrics of \ref{FP2} on $E_\lambda$ and $E_{-\lambda}$, and the standard inner product on $\R^k$. We claim that the isometry group of $W^c_a(x)$ acts transitively on it. Indeed, the proof goes as in Lemmas \ref{transitive-x} and \ref{iso-x-iso-y}, by noting that the $\ker \lambda$-orbit is dense, and if $b \in \ker \lambda$, then $b : W^c_a(x) \to W^c_a(bx)$ is an isometry. Then if $y \in  W^c_a(x)$, we may choose $b_k \in \ker \lambda$ such that $b_k x \to y$, and in the limit obtain an isometry $f : W^c_a(x) \to W^c_a(x)$ such that $f(x) = y$.

        Notice also that the isometries obtained this way are harnessed (ie, they preserve the splitting into Oseledets sub-bundles and their orientations). Hence, the harnessed isometry $\Isom_H(W^c_a(x))$ group also acts transitively. By construction, the harnessed isometry group acts by {\color{black}$C^1$} diffeomorphisms by Theorem \ref{thm:taylor}, so it is a Lie group by \cite[Theorem 13]{MontZip1}. %We claim that it is a Lie group. Indeed, consider the orthonormal frame bundle $\mc F$ over $\mc W^c(x)$, and note that it is a topological manifold. $\Isom_H(\mc W^c(x))$ also acts on $\mc F$, and the action is free.
        
        %The stabilizer of $x$ is a closed subgroup of $SO(T\mc W^c(x))$. Therefore, $\Isom_H(\mc W^c(x))$ is a Lie group.

        Consider the evaluation map $\Isom_H(W^c_a(x)) \to W^c_a(x)$ defined by $E : f \mapsto f(x)$. This map has constant rank, since it is an evaluation map ($dE(f) = df \of dE(e) \of dR_{f^{-1}}$). Onto maps of constant rank, even in low regularity, are submersions, see \cite{BTV} (note that we cannot directly apply the usual version of Sard's theorem since it requires maps to be $C^r$, where $r$ is a function of the dimension). Since the evaluation is a submersion, there is a local $C^1$ section $\tau$ which associates to each $y \in W^c_a(x)$ sufficiently close to $x$ an element $\tau(y) \in \Isom_H(W^c_a(x))$ such that $\tau(y) \cdot x = y$. 
\end{proof}

    \begin{lemma}
    \label{lem:holder-lie2}
        If the $\R^k$ action is Anosov, $\sigma$ is $W^\lambda$-continuous and $W^{-\lambda}$-continuous.
    \end{lemma}

    \begin{proof}
    Fix $x \in X$ and $y \in W^{-\lambda}(x)$, and consider the section $\tau$ of Lemma \ref{lem:isom-section}. Notice that since each such $\tau(y)$ is harnessed, it is an isometry between $W^\lambda(x)$ and $W^\lambda(y)$. Since  $\tau(y)$ varies continuously in the $C^{1}$ topology, it conjugates the group $\leftN_x$ to the group $\leftN_y$. Again, since the map $y \to \tau(y)_*$ is continuous, it follows that the corresponding Lie algebra structures are related by $\tau(y)_*$ and vary continuously.
    \end{proof}

    \begin{proof}[Proof of Theorem \ref{cor:cont-lie} for Anosov actions as in Theorem \ref{basic abelian}]
    
    When the action is Anosov, by Lemmas \ref{lem:isom-section} and  \ref{lem:holder-lie2} we have strong leafwise continuity of $\sigma$ along all the coarse foliations of the action. Then we may use  Proposition \ref{prop: strong ASV thm E} again Corollary \ref{coro: Rk Anosov cont ASV} to conclude  $\sigma$ is continuous.
\end{proof}

 {\color{black} 
 \subsection{Harnessed orthonormal frames, compact extension $\tilde X$ and lifting of foliations\color{black}}\label{sec: Harn ort frm}We define a fiber bundle $\tilde{X}^\lambda$ over $X$, which will be a principal $K$-bundle, where $K = SO_H(\dim(W^\lambda(x)))$. This bundle will not have a canonical H\"older structure, but will have a H\"older structure along the leaves of the coarse Lyapunov foliations. Such structures are discussed in detail in Appendix \ref{app:brin-pesin}. Even though the material from Appendix \ref{app:brin-pesin} is fairly standard in partially hyperbolic dynamics, we could not find reference for continuous bundles. This is why we provide detailed arguments in Appendix \ref{app:brin-pesin} and advise the reader to consult it as we will refer to it in this section. 
 
 Recall that $\Lie(\leftN)$ has a grading, and let $\R^{\dim(\leftN)}$ have the canonical grading with matching dimensions $\R^{\dim(\leftN)} = \R^{\ell_1} \oplus \dots \oplus \R^{\ell_{n_\lambda}}$ (so $\dim(E^{c_i\lambda}) = \ell_i$). Let $SO_H(n)$ denote the group of harnessed isometries of $\R^{\dim(\leftN)}$ with such a grading, as discussed in Section \ref{sec:lie-bundles}.

\begin{definition}\label{def: X til lambda X til}
    Let $\tilde{X}^\lambda$ denote the bundle over $X$ 

    \[ \tilde{X}^\lambda = \set{ (x,\varphi_\lambda) : x \in X, \,\varphi_\lambda : \R^{\dim(\leftN)} \to T_xW^\lambda \mbox{ is a harnessed isometry}}.\]

Let $\tilde{X}$ denote the product bundle as $\lambda \in \Delta$ varies over all possibilities, so the fiber over $x$ is the tuples of frames $(\mf \varphi_\lambda)_{\lambda \in \Delta}$, and $\tilde{K} = \prod_{\lambda \in \Delta} \tilde{K}^\lambda$ denote the structure group. 
\end{definition}
\color{black}In the following, we first show that $\tilde X$ fits Definition \ref{def:reg-along-fol}. This allows us to use Proposition \ref{prop:holder-holonomies}, and obtain that the action and the invariant foliations lift from $X$ to $\tilde{X}$ in the canonical way.%, and since it is a compact extension, the foliations do as well. %(the proof is identical).
   % Then $\tilde{X}^\lambda$ is a continuous principal $\tilde{K}^\lambda$-bundle which is H\"older along coarses, where $\tilde{K}^\lambda = SO_H(\dim(\leftN))$.
   
A principal bundle over $X$ with a compact structure group is called \textit{partially H\"older} if it is $W^s_a$-H\"older continuous in the sense of Definition \ref{def:reg-along-fol} for every partially hyperbolic $a\in \R^k$.

\begin{proposition}\label{prop: partial Hol prin bd}
 The bundles $\tilde X^\lambda$ 
and $\tilde X$ over $X$ are partially H\"older. %Then by Proposition \ref{prop:holder-holonomies}
\end{proposition}
\begin{proof} This amounts to verifying the two conditions of Definition \ref{def:reg-along-fol} in the two subsequent lemmas.
\begin{lemma}\label{lem: prtl chrt}
The bundles $\tilde X^\lambda$ %in Lemma \ref{lem:big-bundle}, as well as the product bundle
and $\tilde X$ %in Section \ref{sec: Harn ort frm} 
satisfies (1). of 
%are H\"older along every coarse foliation in the sense of 
 Definition \ref{def:reg-along-fol}.
\end{lemma}
\begin{proof}To verify $\tilde X^\lambda$ is partially H\"older,  %H\"older along coarse in the sense of Definition \ref{def:reg-along-fol}, 
it is equivalent to show that for every $p\in X$ there is a continuous local section of $\tilde X^\lambda$ near $p$ which is H\"older along some $W^s_a$.

First by H\"older continuity of Oseledets splittings we know that the \textit{harnessed frame bundle} (not necessary orthonormal) 
\begin{multline}\label{eqn: def hns gen frm bdle}
    {\overline X^\lambda}:=\{(x, \phi_\lambda):x\in X,  \phi_\lambda:\R^{\ell_1}\oplus\cdots \oplus \R^{\ell_n}\to E^{c
_1\lambda}_x\oplus \cdots \oplus E^{c
_n\lambda}_x \\ \text{ preserves the splitting and their orientations}\}.
\end{multline} 
Clearly $\overline X^\lambda$ can be viewed as the product (fix $\lambda$ and let $i$ run over $1$ to $n$) of the frame bundles $GL^+(E^{c_i\lambda})$ of each Oseledets subspace $E^{c_i\lambda}$, therefore it is a H\"older continuous bundle, and contains $\tilde X^\lambda$ as a subbundle.

So for any $p\in X$, there is a local H\"older continuous section $\overline s_p$ of the bundle $\overline X^\lambda$. A priori, $\overline s_p$ may not intersects $\tilde X^\lambda$, i.e. those frames are represented by $\overline s_p$ may not be orthonormal frames, we would like to apply standard Gram-Schmidt process to get orthonormal frames from $\overline s_p$.
By taking a harnessed Gram-Schmidt process (i.e. taking Gram-Schmidt process within the subspaces of the splitting), with respect to the metrics on $E^\lambda$, we can  revise $\overline s_p$ to be a local section $\tilde s_p$ of $\tilde X^\lambda$. Since the metric on $E^\lambda$ is continuous and partially H\"older (Proposition \ref{Hoelder metric}) %along coarse, and by the local product structure of $\mc W^s_a$
the section $\tilde s_p$ is also partially H\"older. %along coarse. 
The proof of $\tilde X$ follows.
\end{proof}

\begin{lemma}\label{lem: loc idfy}
The bundles $\tilde X^\lambda$ and $\tilde X$ satisfies (2). of Definition \ref{def:reg-along-fol},% there exists a continuous family of identifications $\psi_{x,y}: F_x\to F_y$ defined for any $(x,y)$ in a neighborhood of the diagonal of $\tilde X^\lambda\times \tilde X^\lambda$ such that %$x,y \in X$ are sufficiently close and $y \in \mc W(x)$, there exists a map $\psi_{x,y} : F_x \to F_y$ such that for all $k \in K^\lambda$, $k\psi_{x,y}(p) = \psi_{x,y}(kp)$. And $\psi_{x,y}$ varies H\"older continuously as $y \in \mc W(x)$ varies along the local leaf  of coarse foliations passing through $x$, and $\psi_{x,x} = \id$ for all $x \in X$. The same conclusion also holds for $\tilde X$.
\end{lemma}
\begin{proof}The proof of $\tilde X$ directly follows the result of $\tilde X^\lambda$, so we only show Lemma \ref{lem: loc idfy} for $\tilde X^\lambda$. As the proof of Lemma \ref{lem: prtl chrt},  $\tilde X^\lambda$ could be viewed as a continuous subbundle of $\overline X^\lambda$ defined in \eqref{eqn: def hns gen frm bdle}. 

In Section 2.2 of \cite{KalSad}, the authors constructed a family of linear  identifications $I_{x,y}:F_x\to F_y$ which are H\"older continuous in $x,y$ and $I_{x,x}=\id$, for any H\"older continuous vector bundle $F$ over a smooth manifold. We apply this result to the bundle $E^{c_i\lambda}$ we get a H\"older continuous family of linear identifications $I^\lambda_{x,y}:E^{c_i\lambda}_x\to E^{c_i\lambda}_y$. The family $I^\lambda_{x,y}$ naturally induces a H\"older continuous family of  identifications between frame bundles $$\overline{I^{c_i\lambda}_{x,y}}: GL^+(E^{c_i\lambda})_x\to GL^+(E^{c_i\lambda})_y.$$
In particular we could restrict $\overline{I^{c_i\lambda}_{x,y}}$ to $SO(E^{c_i\lambda})_x$. A priori, the image of $\overline{I^\lambda_{x,y}}$ restrict to $SO(E^{c_i\lambda})_x$ may not intersect $SO(E^{c_i\lambda})_y$. Composing with a further Gram-Schmidt ``operator'' (with respect to the metric in $E^{c_i\lambda}$) if necessary, we get a family of identifications 
$$\tilde{I}^{c_i\lambda}_{x,y}: SO(E^{c_i\lambda})_x \to SO(E^{c_i\lambda})_y.$$
Since the metric in $E^{c_i\lambda}$ is continuous and H\"older continuous along coarse foliations (Proposition \ref{Hoelder metric}), $\tilde{I}^{c_i\lambda}_{x,y}$ is continuous and partially H\"older. %continuous along coarses. 
Moreover since Gram-Schmidt operator is just the identity if it acts on an orthonormal frame, we have $\tilde{I}^{c_i\lambda}_{x,x}=\id$. Glue all $\tilde{I}^{c_i\lambda}_{x,y}$ in an obvious way we get the family of identifications in Definition \ref{def:reg-along-fol} (2)  we want.
\end{proof}
\end{proof}
\color{black}

\begin{lemma}
\label{lem:big-bundle}
     The $\R^k$-action on $X$ lifts to an $\R^k$-action on $\tilde{X}$ by bundle automorphisms, and each foliation $W^\beta$ on $X$ lifts to a continuous foliation $\tilde{W}^\beta$ on $\tilde{X}$. 
\end{lemma}

For each $a \in \R^k$, let $\psi_{a,\lambda} : \R^{\dim(\leftN)} \to \R^{\dim(\leftN)}$ denote the linear automorphism uniquely determined by $\psi_{a,\lambda}|_{\R^{\ell_i}} = e^{c_i\lambda(a)} \id_{\R^{\ell_i}}$.

\begin{proof}
    %That $\tilde{X}^\lambda$ is a continuous principal bundle which is H\"older along coarses is clear from the fact that the metrics on $E_\lambda$ are continuous and H\"older along coarses. We just need to verify that the $\R^k$-action lifts to an action by bundle automorphisms.

    Given $(x,\varphi_\lambda) \in \tilde{X}^\lambda$, define:

    \begin{equation}\label{eq:lift-def} a\cdot (x,\varphi_\lambda) = (ax,a_*\of \varphi_\lambda \of \psi_{-a,\lambda}),\end{equation}
 where $a_* : T_xW^\lambda \to T_{ax}W^\lambda$ is the pushforward. Note that since $a_*$ and $\psi_{-a,\lambda}$ both preserve the Oseledets decomposition and are conformal restricted to each Oseledets space with inverse scaling constants, the map $a_* \of \varphi_\lambda \of \psi_{-a,\lambda}$ is a harnessed isometry between $\R^{\dim(\leftN)}$ and $T_{ax}W^\lambda$. The family of maps $a_* \of \varphi_\lambda \of \psi_{-a,\lambda}$ are connected to $\id$, so they must be orientation-preserving in each Oseledets space. Thus, each map is well-defined. 
 
 Since the transformations $\psi_{a,\lambda}$ all commute, it follows that it is an action of $\R^k$. Finally, it is an action by bundle automorphisms, since any harnessed isometry of $\R^{\dim(\leftN)}$ will commute with the maps $\psi_{-a,\lambda}$, and the structure group acts by precomposition on the right.
 
\color{black}
 Finally, observe that by Proposition \ref{prop: partial Hol prin bd} and Proposition \ref{prop:holder-holonomies}, all the (un)stable foliations of $\R^k$-action lift to foliations of $\tilde{X}$, hence all 
 the coarse Lyapunov foliations of the $\R^k$-action lift as well,  which we will denote by $\tilde W^\lambda$.
\end{proof}

\color{black}

\color{black}
\subsection{The Brin-Pesin subbundle\color{black}}\label{BPsubbundle}

Let $\hat{X}$ denote a Brin-Pesin subbundle for the lifted $\R^k$-action on $\tilde{X}$. This construction is outlined in Proposition \ref{prop:brin-pesin-construct}, and has the properties that

\begin{itemize}
    \item $\hat{X}$ contains all $\tilde{W}^\beta$-leaves for all $\beta \in \Delta$.
    \item $\hat{X}$ is $\R^k$-invariant and has a dense $\R^k$-orbit.
    \item $\hat{X}$ is unique up to translation by an element of $\tilde{K}$.
    \item The structure group $\hat{K}$ is unique up to conjugacy by the same element of $\tilde{K}$.
\end{itemize}

Let $I_{x,\lambda} : T_xW^\lambda \to \mf X_\lambda(x)$ denote the map which associates to a vector $v$ the unique $\leftN_x$-invariant vector field on $W^\lambda(x)$. Recall the block homotheties $\psi_{a,\lambda}$ introduced before Lemma \ref{lem:big-bundle}.

\begin{lemma}
\label{lem:psi-auto}
    $\psi_{a,\lambda}$ is an automorphism of $\R^{\dim(\leftN)}$ with respect to the Lie algebra structure \[[\cdot,\cdot]_{\hat{\varphi}_\lambda} := [\cdot,\cdot]_{I_{x,\lambda} \of \varphi_\lambda}.\]
\end{lemma}

\begin{proof}
    Since commutators are bilinear and $\psi_{a,\lambda}$ is a homothety in each $\R^{\ell_i}$ subspace, it suffices to show that for every $v \in \R^{\ell_i}$ and $w \in \R^{\ell_j}$, $[v,w]$ belongs to $\R^{\ell_r}$, where $c_i + c_j = c_r$. This follows as in Lemma \ref{lem:rich-automorphism}, since $\norm{a_*v} = e^{c_i\lambda(a)}\norm{v}$ and $\norm{a_*w} = e^{c_j\lambda(a)}\norm{w}$, by bilinearity $\norm{a_*[v,w]} = e^{(c_i+c_j)\lambda(a)}\norm{[v,w]}$.
\end{proof}

\begin{lemma}
\label{lem:lie-const}
    The map $(\varphi_\lambda)_{\lambda \in \Delta} \mapsto [\cdot,\cdot]_{\hat{\varphi}_\lambda}$ from $\tilde{X}$ to $\mc L(\R^{\dim(\leftN)})$ is continuous, and constant on $\R^k$-orbits.
\end{lemma}

\begin{proof}
    To see the continuity, we obtain the map as a composition of continuous functions. First, note that the map $\sigma$ which assigns the Lie algebra structure to $T_xW^\lambda$ is continuous in $x$  by Corollary \ref{cor:cont-lie}. Furthermore, let $\mc P$ denote the bundles whose fiber at $x$ is $\mc L(T_xW^\lambda) \times \tilde{X}_x$ (ie the product bundle of $\mc L(TW^\lambda)$ and $\tilde{X}$. we may define a map $\psi : \mc P \to \mc L(\R^{\dim(W^\lambda)})$ via

    \[ \psi(\omega,(\varphi_\lambda)_{\lambda \in \Delta}) = \varphi_\lambda^*\omega.\]
Then $\psi$ is clearly continuous, and we obtain the function in the statement as

    \[ ((\varphi_{\lambda})_{\lambda \in \Delta}) \mapsto \psi(\sigma(x),(\varphi_{\lambda})_{\lambda \in \Delta})\]
where $x$ is the basepoint of $(\varphi_\lambda)_{\lambda \in \Delta}$. Hence it is continuous.

    Finally, observe that $a_*$ takes $\leftN_x$-invariant vector fields to $\leftN_{ax}$-invariant vector fields, and since it is a diffeomorphism between the leaves, is an automorphism of the algebra of left-invariant vector fields. Since by Lemma \ref{lem:psi-auto}, $\psi_{a,\lambda}$ is an automorphism as well, it follows from the definition of the lifted action (Equation \eqref{eq:lift-def}) that the Lie algebra sturcture is constant on $\R^k$-orbits.
\end{proof}

\begin{corollary}
\label{cor:lie-const}
    $[\cdot,\cdot]_{\hat{\varphi}_\lambda}$ is constant on $\hat{X}$.%For every $(\varphi_\lambda)_{\lambda \in \Delta}$, $(\varphi_\lambda')_{\lambda \in \Delta} \in \hat{X}$ based at $x,x' \in X$, respectively, 
    
    %\[ I_{x',\lambda} \of \varphi_\lambda' \of {\varphi_\lambda}^{-1} \of I_{x,\lambda}^{-1} : \mf X_\lambda(x) \to \mf X_\lambda(x')\]

    %is a Lie algebra isomorphism for every $\lambda \in \Delta$.
\end{corollary}

\begin{proof}
%    First note that the Lie algebra structure $[\cdot,\cdot]_{\hat{\varphi}_\lambda}$ which is the pullback of the structure on $\mf X_\lambda(x)$ by $\varphi_\lambda \of I_{x,\lambda}$ varies continuously on $\tilde{X}$ by Corollary \ref{cor:cont-lie}. Furthermore, it is constant on $\R^k$-orbits, since...
    Since the $\R^k$-action has a dense orbit on $\hat{X}$, this follows immediately from Lemma \ref{lem:lie-const}.
\end{proof}

\begin{corollary}
\label{cor:compact-auto}
    $\hat{K} \subset SO_H(n)$ acts by automorphisms of the common Lie structure $[\cdot,\cdot]_{\hat{\varphi}_\lambda}$.
\end{corollary}

\begin{proof}
    Let $k \in \hat{K}$. By definition, $[\cdot,\cdot]_{\widehat{k\varphi}_\lambda} = k_*[\cdot,\cdot]_{\hat{\varphi}_\lambda}$. Since $k\varphi_{\lambda}$ must lie in $\hat{X}$ and the Lie algebra structures are the same by Corollary \ref{cor:lie-const}, it follows that $k_*$ preserves the Lie algebra structure.
\end{proof}

\subsection{The lifted actions\color{black}}

 Let $N^\lambda$ denote the simply connected Lie group whose Lie algebra is $\R^{\dim(\leftN)}$ with the bracket $[\cdot,\cdot]_{\hat{\varphi}_\lambda}$. Note that since $\psi_{a,\lambda}$ and $k_*$ are both automorphisms of the Lie algebra, they lift to unique automorphisms of $N^\lambda$, which by abuse of notation we denote by the same symbols.

 The following result is the crucial outcome of this part of the paper. It provides homogeneous structures along leaves of lifted foliations on the Brin-Pesin subbundle constructed in the last section. 
\begin{theorem}
\label{thm:lifted-action}
    For each $\beta \in \Delta$, there exists a continuous action of $N^\beta$ on $\hat{X}$, such that

    \begin{itemize}
        \item[(1)] $N^\beta \cdot (\varphi_\lambda)_{\lambda \in \Delta} = \tilde{W}^\beta((\varphi_\lambda)_{\lambda \in \Delta})$.
        \item[(2)] For every $u \in N^\beta$ and $a \in \R^k$, 

        \[ au = \psi_{a,\beta}(u)a.\]
        \item[(3)] For every $u \in N^\beta$ and $k \in \hat{K}$,

        \[ ku = (k_*u)k.\]
    \end{itemize}
\end{theorem}

\begin{proof}
    We define the actions as follows: given $v \in \R^{\dim(N^\beta)}$, and $(\varphi_\lambda)_{\lambda \in \Delta} \in \hat{X}$ based at $x \in X$, let $\bar{v}$ denote the $\leftN_x$-invariant vector field on $W^\beta(x)$ such that $\bar{v}(x) = \varphi_\beta(v)$. Let $y$ denote the image of $x$ under the time one map of the flow generated by $\bar{v}$, and $(\varphi_\lambda')_{\lambda \in \Delta}$ be the point of the lifted manifold $W^\beta((\varphi_\lambda)_{\lambda \in \Delta})$ which covers $y$. Then define $\exp(v)\cdot (\varphi_\lambda)_{\lambda \in \Delta} := (\varphi_\lambda')_{\lambda \in \Delta}$.

    To see that this is an action, since $\tilde{W}^\beta(x)$ covers $W^\beta(x)$ homeomorphically, and by Corollary \ref{Chi}, the invariant vector fields of the $W^\beta$-leaf do not change. Furthermore, by Corollary \ref{cor:lie-const}, it follows that $\varphi_\lambda : \R^{\dim(N^\beta)} \to \mf X_\beta(x)$ is a Lie algebra isomorphism. The action is the definition of the lift of a Lie algebra homomorphism to the corresponding Lie group (note that this is sufficient since all nilpotent groups are exponential). Properties (1)-(3) follow directly from the construction. 

{\color{black}
    The proof of continuity of $N^\beta$ action requires more care and we break it into several steps as in the proof of Theorem \ref{cor:cont-lie}, except that now we will use the bundle versions of the invariance principle Proposition \ref{prop: fancy inv prin} and Corollary \ref{coro: Rk Anosov with bundle ASV} instead. 

    In what follows, we shorten the notation for tuples $(\varphi_\lambda)_{\lambda \in \Delta}$ to $(\varphi_\lambda)$. }

    %Lemma \ref{lem:action-cont1}, Lemma \ref{lem:action-cont2} and Theorem \ref{thm:action-cont}.}%Finally, we show continuity. We will show each each $\exp(v)$ is a continuous map of $\hat{X}$, utilizing the invariance  principle again. That is, we will show that $\exp(v)$ is continuous along every coarse Lyapunov foliation... Does this work??? We don't have accessibility on the fiber, maybe we need to use the intertwining property with the $\hat{K}$-action? Also worried about $-\beta$, maybe need to split into P.H. and Anosov again.

%Definition of bi-continuous. In \cite{ASV}, the definition of $s$ and $u$ continuity is slightly stronger than continuous along stable and unstable leaf. See Definition 1.3 of \cite{ASV}. For our purpose, each time we apply this version of invariance principle, we basically need to verify the section is subordinate to holonomies. and need continuity of holonomies.

\color{black}
\begin{lemma}
\label{lem:action-cont1}
    The action of $N^\beta$ on $\hat X$ is $\tilde{W}^\mu$-continuous, $\mu \not= - \beta$.
\end{lemma}
\begin{proof}
    Fix $u \in N^\beta$. We will show the following: if $(\varphi_\lambda)_n$ and $(\varphi_\lambda')_n$ are sequences of tuples in $\hat{X}$ such that $(\varphi_\lambda')_n \in \tilde{W}^\mu((\varphi_\lambda)_n)$, which converge to $(\varphi_\lambda)_0$ and $(\varphi_\lambda')_0$, respectively, and $u\cdot (\varphi_\lambda)_n \to u \cdot (\varphi_\lambda)_0$, then $u \cdot (\varphi_\lambda')_n \to u \cdot (\varphi_\lambda')_0$.

    It suffices to show that the convergence happens on $X$. That is, if $\pi : \hat{X} \to X$ is the projection map, then $\pi(u \cdot (\varphi_\lambda')_n) \to \pi(u \cdot (\varphi_\lambda')_0)$. Indeed, assume we have this convergence. We also know that $(\varphi_\lambda')_n \to (\varphi_\lambda')_0$, and the action of $u$ is defined by lifting the $W^\beta$-leaves to $\tilde{W}^\beta$-leaves. Since the $\tilde{W}^\beta$-foliation is continuous by Lemma \ref{lem:big-bundle}, we conclude that the action is continuous on $\hat{X}$ from the convergence criterion on $X$. 

    Let us recall how the action is defined. Assume that $(\varphi_\lambda)_n'$ covers a sequence of points $y_n$, and that $u = \exp(v)$ for some $v \in \Lie(N^\beta)$. The projection of $u\cdot (\varphi_\lambda')_n$ is image of the time-1 map of the flow whose generating vector field is the ${}^{\beta}\hspace{-1.5pt}N_{y_n}$-invariant vector field on $W^\beta(y_n)$, whose evaluation at $y_n$ is $\varphi_{\beta,n}'(v)$. Call this vector field $V_n'$. Similarly, there exists a corresponding sequence of vector fields $V_n$ on $W^\beta(x_n)$, where $x_n$ is the projection of $(\varphi_\lambda)_n$.
    
    We will show that there exist $C^1$-diffeomorphisms $h_n : W^\beta(x_n) \to W^\beta(y_n)$ converging to a diffeomorphism $h_0 : W^\beta(x_0) \to W^\beta(y_0)$ uniformly on compact sets such that $(h_n)_*V_n= V_n'$. It follows that if the time-1 map of $V_n$ converges to a diffeomorphism of $W^\beta(x_0)$, the time-one maps of $V_n'$ converge to the time one map of $(h_0)_*V_0$. With the existence of such diffeomorphism, we conclude the lemma. 

    Of course, the diffeomorphisms $h_n$ are given by stable holonomy. Choose an element $a_0 \in \ker \beta$ such that $\mu(a_0) < 0$ and $\lambda(a_0) \not= 0$ for all $\lambda \not= \pm \beta$. Then choose a nearby partially hyperbolic element $a$ such that $-1 \ll \beta(a) < 0$, so that $E_\beta$ is the slow foliation in $E^s_a$, and the complementary fast distribution in $E^s_a$, $E^{ss}_a$, contains $E_\mu$. %(\color{black}what is $E_\mu$?\color{black}). 
    Then the fast distribution $E^{ss}_a$ integrates to a foliation $W^{ss}_a$ which is smooth in $W^s_a$, and whenever $y \in W^{ss}_a(x)$ there exists a smooth holonomy map $h_{x,y} : W^\beta(x) \to W^\beta(y)$ which varies continuously in $x$ and $y$ in the compact-open topology, and satisfies that

    \begin{equation}\label{eq:holonomy-intertwine} a\of h_{x,y} = h_{ax,ay} \of a.\end{equation}

    So we just need to show that $h_{x,y}$ intertwines the action of ${}^{\beta}\hspace{-1.5pt}N_x$ and ${}^{\beta}\hspace{-1.5pt}N_y$. Indeed, notice that $a_0$ is a harnessed isometry between $W^\beta(x)$ and $W^\beta(a_0x)$, and hence intertwines the actions of ${}^{\beta}\hspace{-1.5pt}N_x$ and ${}^{\beta}\hspace{-1.5pt}N_{a_0x}$ by the uniqueness property in Lemma \ref{Independent prop}. Then by \eqref{eq:holonomy-intertwine},

    \[ h_{x,y} = \lim_{n\to \infty} (-na_0) \of h_{na_0x,na_0y} \of (na_0).\]

    Since $d(na_0x,na_0y) \to 0$, $h_{na_0x,na_0y} \to \id$  as $C^1$-diffeomorphisms by \cite[Theorem 1.3]{Sag}. It follows that $h_{x,y}$ is a harnessed isometry, and has the desired intertwining properties.

    We define $h_n = h_{x_n,y_n}$. To complete the proof, we must show that $(h_n)_*V_n = V_n'$. Since $h_n$ is a harnessed isometry, we know that $(h_n)_*V_n$ is a ${}^{\beta}\hspace{-1.5pt}N_{x_n'}$-invariant vector field, so it suffices to show that $dh_n(\varphi_{\beta,n}(v)) = \varphi_{\beta,n}'(v)$. We will in fact show more, that for any $(\xi_\lambda) \in \hat{X}$ and $(\xi_\lambda)' \in \tilde{W}^\beta((\xi_\lambda))$ covering points $y \in W^\beta(x)$, 
    
    \begin{equation}
    \label{eq:ugh}dh_{x,y} \of \xi_\beta = \xi_\beta'.
    \end{equation}

%    By definition of the lifted action \eqref{eq:lift-def}, $a_0 \cdot \xi_\beta = da_0\of \xi_\beta \of \psi_{-a_0,\beta}$.
    To show \eqref{eq:ugh}, we use the lifted action as defined in \eqref{eq:lift-def}, which we denote by $a \cdot \xi_\beta$. It suffices to show that $d((na_0)\cdot (dh_{x,y} \of \xi_\beta),(na_0)\cdot \xi_\beta) \to 0$, since there is a unique point covering $y$ in $\tilde{W}^\beta(x)$. But by \eqref{eq:holonomy-intertwine},

    \[(na_0)\cdot (dh_{x,y} \of \xi_\beta)= dh_{na_0x,na_0y} \of (na_0)\cdot \xi_\beta\]
    and since $h_{na_0x,na_0y} \to \id$ as $n \to \infty$ (again, as a sequence of $C^1$-diffeomorphisms), we conclude \eqref{eq:ugh}. This concludes the lemma.
\end{proof}

The above lemma will suffice for showing continuity of $N^\beta$- action on the space $\hat X$ when the base action on $X$ is super accessible totally partially hyperbolic as in Theorem \ref{abelian}. When the base action is Anosov as in Theorem \ref{basic abelian}, it need not be accessible, so we need a bit more, namely for Anosov base actions we have the following: 

\begin{lemma}
\label{lem:action-cont2}
    If the base $\R^k$-action is Anosov, then the $N^\beta$-action on $\hat X$ is $\tilde{W}^{-\beta}$-continuous.
\end{lemma}

\begin{proof}
    Consider two nearby tuples $(\varphi_\lambda)$ and $(\varphi_\lambda')$ in the same $\tilde{W}^{-\beta}$-leaf, covering points $x,y$ in the same $W^{-\beta}$-leaf on $X.$ That is, if $\pi : \hat{X} \to X$ is the projection map, then $\pi((\varphi_\lambda)) = x$ and $\pi((\varphi_\lambda'))= y$. It suffices to show that for $u \in N^\beta$, $\pi(u\cdot (\varphi_\lambda))$ and $ \pi(u \cdot (\varphi_\lambda'))$ are close in $X$, since the action is defined by lifting coarse Lyapunov leaves, and the foliation $\tilde{W}^\beta$ is continuous by Lemma \ref{lem:big-bundle}.

    By definition of the action, if $u = \exp(v) \in N^\beta$, then $u\cdot (\varphi_\lambda)$ covers the point $x'$, which is the image of $x$ under the time-one map of the flow generated by vector field $\bar{v}_x$ such that $\bar{v}_x(x) = \varphi_\beta(v)$. 
    
    Recall the construction of a local $C^1$-section $\tau$ in Lemma \ref{lem:isom-section}, $\tau : W^{c,\beta}_{\loc}(x) \to \Isom_H(W^{c,\beta}(x))$ such that $\tau(z)x = z$ for all $z \in W^{c,\beta}_{\loc}(x)$. The section $\tau$ is continuous and each $\tau(y)$ is harnessed isometry, $\tau(y)$ is $C^1$-close to $\id$, takes $W^\beta(x)$ to $W^\beta(y)$ and intertwines the ${}^{\beta}\hspace{-1.5pt}N_x$ - and ${}^{\beta}\hspace{-1.5pt}N_y$ -actions. Therefore, $\tau(y)_*\bar{v}_x$ must be $C^0$-close to $\bar{v}_y$. It follows that since $\sigma(y) \to \id$ as $y \to x$, the action of $N^\beta$ is continuous.
\end{proof}

%\begin{theorem}
%\label{thm:action-cont}
 Now we show that   the action of $N^\beta$ is continuous.
%\end{theorem}
The proof goes along the same lines as the proof of Theorem \ref{cor:cont-lie}, upgraded to use {\color{black} Proposition \ref{prop: fancy inv prin}} for principal bundle extensions, instead of the usual invariance principle. 
    The $N^\beta$ action is automatically $\hat K$-continuous by construction (item (3) of Theorem \ref{thm:lifted-action}). In addition by Lemma \ref{lem:action-cont1}, we have $W^\lambda$-continuity for all $\lambda \in \Delta$, $\lambda \not= \pm \beta$.
    In the case when the base action is partially hyperbolic as in Theorem \ref{abelian}, because of the super accessibility by choosing an element in $\ker \beta$ and applying to it  Proposition \ref{prop: fancy inv prin} we get continuity of the $N^\beta$ action in this case.  In the case the base action is Anosov as in Theorem \ref{basic abelian}, we use in addition  Lemma \ref{lem:action-cont2}, so the result follows from local product structure and Corollary \ref{coro: Rk Anosov with bundle ASV}
\end{proof}

\color{black}
\begin{remark}
    The actions by $N^\beta$ are {\it not} necessarily by isometries of $W^\beta(x)$ and hence do not coincide with the previously constructed isometry group actions. In fact, it may be that no such global actions exist, as this would correspond to a right action on a homogeneous space $H/ \Lambda$.
\end{remark}

\subsection{Preview of Part 4.\color{black}}
In conclusion of the preparatory steps in the proof of Theorems \ref{basic abelian} and \ref{abelian}, given an action $\rho$ as in Theorem \ref{basic abelian} or \ref{abelian} we have constructed an extension of $\rho$ as given by \eqref{eq:lift-def} and corresponding normalized nilpotent group actions in Theorem \ref{thm:lifted-action}. This is the starting point for the rest of the proof of Theorems \ref{abelian} and \ref{basic abelian}. In the next section we describe the class of actions to which the extension constructed in Theorem \ref{thm:lifted-action} belongs. We call them \emph{harnessed abstract partially hyperbolic actions} (see Definition \ref{def:top-anosov}) and we show that, if genuinely higher rank, such actions are essentially homogeneous. This is our main technical result: Theorem \ref{thm:technical} in the subsequent section. \color{black} Sections \ref{sec:top-anosov}, \ref{sec:fibers} and \ref{sec:pairwise-sufficient} are dedicated to its proof. 
\color{black}

\part{Path group, cycle structure and classification of $\mathbb R^k$-actions}
\label{part:top-groups}

\section{Topological partially hyperbolic actions}
\label{sec:top-anosov}

%Here is the technical theorem we should try to prove

%\todo{comment(27) define action and space} 

%\todo{comment(30) see general comment on section 10}

%Let $M$ be a connected compact group. 
In this section, we define and develop properties of certain $\R^k$ %\times M$ 
actions on topological spaces which include the important features of smooth partially hyperbolic actions as axioms (see Proposition \ref{prop:smooth-is-top}). We begin by defining such actions, after developing some definitions and notations:

\subsection{Definitions and notations}%First define a topological variant of abelian Anosov actions 

%Suppose $X$ is a finite dimensional compact metric space.
%(+ maybe other structures, properties?).
%Let $M$  be a compact connected Lie group, and
Let $\Delta$ be a set of nonvanishing real linear functionals on $\R^k$ up to positive scalar multiple (i.e. elements of $(\R^k)^*/ \sim$, where $f \sim g$ if and only if $f = \lambda g$ for some $\lambda > 0$). For each $\alpha \in \Delta$, we fix the subset $[\alpha]$ of linear functionals positively proportional to  $\alpha$. We abusively let $\alpha$ denote the ``smallest'' element of $[\alpha]$. That is, we write:

\[ [\alpha] = \set{ \alpha = c_1 \alpha, c_2\alpha, c_3\alpha, \dots, c_{\ell} \alpha}, \mbox{ where } 1 = c_1 < c_2 < \dots < c_{\ell} \]

\noindent for some $\ell = \ell(a)$. Importantly, each $\alpha$ should not be thought of as a linear functional, but an equivalence class, and we add coefficients $c_i$ when picking specific functionals in the class. {\color{black}Each $c_i \alpha$ is called a {\it weight}, or interchangeably {\it Lyapunov exponent} or {\it  Lyapunov functional}, and each such $\alpha$ is called a {\it coarse weight},  {\it coarse Lyapunov exponent}.} Given such an $[\alpha]$, an element $a \in \R^k$, and a graded vector space $V = \bigoplus_{i=1}^\ell E^i$, there exists a uniquely defined isomorphism $a_* : V \to V$, by letting $a_*|_{E^i}$ be defined by scalar multiplication by $e^{c_i\alpha(a)}$. We call each $E^i$ an {\it Oseledets space} and $a_*$ the {\it graded homothety induced by $a$}.

{\color{black}
\begin{definition}
\label{def:del-harnessed}
    A {\it $\Delta$-harnessed nilpotent Lie group} is a pair $N^\alpha = (N,\alpha)$ such that $N$ is a simply connected nilpotent Lie group, $[\alpha] = \set{c_1\alpha,\dots,c_\ell \alpha}$ is a coarse Lyapunov exponent of some finite subset $\Delta \subset (\R^k)^*/ \sim$, %(\color{black}Since it is up to positively scalar, so we may need to modulo equivalent classes, so maybe here is a good place to define $\Omega$ to be the set of all Lyapunov exponents, note her e we never defined what Lyapunov exponent (functional) is in the whole HAPHA section. Then if we do it, the proofs at Page 65 makes sense. And just let $\Delta$ defined as above. On the other hand we use the word ``weight" really in two ways, sometimes coarse sometimes non-coarse.\color{black}),
    $\Lie(N)$ has a vector space decomposition $\bigoplus_{i=1}^\ell E^i$, and for every $a \in \R^k$, the graded homothety $a_*$ induced by $a$ is an automorphism of $\Lie(N)$.
\end{definition}

\begin{remark}
    We will often denote the group $N$ in Definition \ref{def:del-harnessed} by $N^\alpha$. We use this notation to indicate that we have a nilpotent group paired with a family of automorphisms indexed by $\R^k$ whose eigenvalues are determined by $[\alpha]$. Furthermore, since the nilpotent group is simply connected, the automorphisms $a_*$ will always lift, and we will let $a_*$ denote both the automorphism of the group or algebra, as determined by the context.
\end{remark}

We collect several useful definitions:

\begin{definition}
Let $\Delta$ be set of equivalence classes of functionals as described above.
\begin{itemize}
    \item Given a collection $a_1,\dots,a_n \in \R^k$, let $\Delta^-(\set{a_i}) = \set{\chi \in \Delta :  \chi(a_i) < 0 \mbox{ for every }i=1,\dots,n}$ (we similarly define $\Delta^+$ as the set of coarse weights with positive evaluations on every $\chi \in \Delta$).
    \item A subset $\Phi \subset \Delta$ is called {\it stable} if there exists $a \in \R^k$ such that $\Phi \subset \Delta^-(a)$.
    \item The {\it Weyl chambers} of $\Delta$ are the connected components of $\R^k \setminus \bigcup_{\alpha \in \Delta} \ker \alpha$.
    \item An element $a \in \R^k$ is called {\it regular} if $a$ belongs to a Weyl chamber.
\end{itemize}
\end{definition}} 

 The following lemma first appeared in \cite[Lemma 5.32]{Spatzier-Vinhage}, and the proof is identical (it is purely linear algebra). Recall the definitions and notations for circular ordering (Definition \ref{def:circular-ordering}), $\Sigma(\alpha,\beta)$, and canonical circular ordering (Definition \ref{def:canonical-order}).

\begin{lemma}
\label{lem:root-comb}
Let $\alpha,\beta \in \Delta$ be linearly independent and $\mathfrak C_1,\dots,\mathfrak C_m$ be the Weyl chambers such that $\alpha$ and $\beta$ are both negative on every 
%\todo{comment(29) define $\mathcal W_i$} 
$\mathfrak C_i$. For each such chamber, choose an arbitrary $a_j \in \mathfrak C_j$. Then $\Sigma(\alpha,\beta) \cup \set{\alpha,\beta} = \Delta^-(\set{a_i})$.
\end{lemma}

{\color{black}
We will use one more piece of useful terminology:

\begin{definition}
    Let $\R^k \curvearrowright X$ be a continuous group action on a metric space. We say that the action is {\it totally recurrent} if for every $a \in \R^k$, the set of $a$-recurrent points is dense.
\end{definition}
  Note that by Poincar\'{e} recurrence, actions preserving a measure of full support are totally recurrent.
  
%\begin{definition}
%    Let $X$ be a compact metric space with Lie group actions $H_1,H_2 \curvearrowright X$ with right-invariant metrics. We say that a group action $H_1 \curvearrowright X$ is {\it H\"older along $H_2$} if for every $h \in H$, there exists some $C(h),\theta > 0$ such that if $d(hx,hgx) < C(h)d_{H_2}(e,g)^\theta$ for all sufficiently small $g \in H_2.$
%\end{definition}  
  
  }

  \subsection{Harnessed abstract partially hyperbolic actions (HAPHAs)\color{black}}

  As noted at the start of this section, we axiomatize the key structures obtained from the dynamics of an accessible partially hyperbolic $\R^k \times K$-action. Most of the following properties were deduced in the previous section for smooth partially hyperbolic actions. The reason for formulating the definition in the topological setting is twofold: first, it allows us to identify which properties of smooth systems are important and reference those properties quickly. Second, it highlights  the breadth of the geometric approach we employ. In particular, we will be able to obtain a purely topological rigidity result without the use of derivatives from these conditions.

\begin{definition}
\label{def:top-anosov}
{\color{black}Let $K$ be a compact connected Lie group and $X$ be a compact connected metric space {\color{black} of finite topological dimension}. Consider a continuous, locally free action of $\R ^k \times K$ on $X$, equipped with a nonempty finite subset $\Delta \subset (\R^k)^* / \sim$, which is decomposed into coarse Lyapunov exponents $\Delta = \alpha_1 \cup \dots \cup \alpha_\ell$ with distinguished finite subsets $[\alpha_i] \subset (\R^k)^*/ \sim$. If an action $\R^k\times K \curvearrowright X$ has properties \ref{ta1}-\ref{ta9} as described below, we call it a {\it harnessed abstract partially hyperbolic action (HAPHA)} of $\R^k \times K$.}
\begin{itemize}%[label=(HA-\arabic*)]

%\todo{com(31)}
%\todo{com(32)}
\item [\mylabel{ta1}{(HA-1)}] The $\R^k$ action has a dense orbit.
%\item \label{ta3} For all $\alpha \neq \beta \in \Delta$, $\alpha$ and $\beta$ are not positively proportional.
 \item [\mylabel{ta4}{(HA-2)}] For all $\alpha \in \Delta$, there is a {\color{black}$\Delta$-harnessed} nilpotent Lie group $N^{\alpha}$ with {\color{black}locally free continuous} actions $N^{\alpha}  \curvearrowright X$, with corresponding action of the free product $\mc P$ of the groups $N^\chi$, $\chi \in \Delta$. %{\color{olive} which are H\"older along the $N_\beta$-action for every $\beta \in \Delta$}.%  These actions { have locally bi-Lipschitz evaluation maps} onto their images and are globally H\"older. \color{black} We don't really need the bi-Lipshits contidion? Just need bi-Hölder?  \color{black}
\item [\mylabel{ta5}{(HA-3)}]For all $\alpha \in \Delta$, {\color{black}there exists an automorphism action $\R^k \times K \to \Aut(N^\alpha)$ denoted by $g \mapsto g_*$} such that {for all $u\in N^\alpha$
\[ g ug^{-1} \cdot x = (g_*u)\cdot x.\]

Furthermore for all $a \in \R^k$ and $u \in N^\alpha$, the map $a_*$ is the automorphism coming from the harnessed assumption (see Definition \ref{def:del-harnessed}).% also satisfies

%\[ aua^{-1} \cdot x = (a_*u)\cdot x.\]
%Assume also that the $ \Lie(N_\alpha)$ automorphism induced by $g_*$ is a graded homothety for $g \in \R^k$. 
%In addition, if $g \in \R^k$, we assume that induced automorphism $g_*:\Lie(N_\alpha)\to \Lie(N_\alpha)$ is a graded homothety.
%(Note: we will use the same notation $g_*$ for both the automorphism of the Lie group and Lie algebra. Since these operate on different spaces, the context will make clear which is intended.) }%\color{teal} did we explain semi-simple automorphism? We actually need for any $m\in M$, $Dm$ preserves the Oseldets decomposition.}
%induced by { $g$}. %and for all $g \in \R^k \times M$ { , }:

%{\color{teal}$u\in N_\alpha?$.} %I am a little bit confused here, the action of any element in $\R^k$ is homothety, is the action of an element in $M$ also a homothety? }
%\item \label{ta6}The action of $M$  normalizes all $N_{\alpha}$ actions.
\item [\mylabel{ta8}{(HA-4)}] {\color{black} The $\R^k$-action is totally recurrent.}%For every $a \in \R^k$, the set of $a$-recurrent points is dense.}%The set of $\R^k \times M$-closed orbits is dense. 
%\item the action is accessible (ie, the free product of the groups $N_\beta$ is transitive).
 \item [\mylabel{ta10}{(HA-5)}] If $a_1,\dots,a_m \in \R^k$ is a list of regular elements, and %, $\Phi \subset \Delta$ is the subset of weights such that $\chi(a_i) < 0$ for all $i = 1,\dots,m$, and 
% \todo{comment(33)}
$\set{\chi_1,\dots,\chi_r}$ is a circular ordering of $\Delta^-(\set{a_i})$, the restriction of the evaluation maps $\mc P \to X$ defined by $\rho \mapsto \rho x$ to $C_{(\chi_1,\dots,\chi_r)} = N^{\chi_1} \times \dots \times N^{\chi_r}$ are injective (recall %the definition of $C_{(\chi_1,\dots,\chi_r)}$ from
Definition \ref{def:comb-cells}). Their images are denoted by $W^s_{(a_1,\dots,a_m)}(x)$.%{  can we do it with just $m = 1$? unclear how to make induction work}
 \item [\mylabel{ta3}{(HA-6)}] If $a_1,\dots, a_m \in \R^k$ are regular elements, then $W^s_{(a_1,\dots,a_m)}(x)$ as defined in \ref{ta10} is exactly the set of points $y \in X$ such that $d({a^t_i}x,{a^t_i}y) \to 0$  for every $i = 1,\dots, m$. %If $a_1,\dots,a_m \in \R^k$ are as in \ref{ta10},
 %then for any combinatorial pattern $\bar{\beta}$ whose letters are all from $\Delta^-(\set{a_i})$, $C_{\bar{\beta}} x \subset W^s_{(a_1,\dots,a_m)}(x)$ for every $x \in X$.
\item [\mylabel{ta9}{(HA-7)}] The action satisfies at least one of the following properties:

\begin{itemize}%[label=(\alph*)]
    \item [\mylabel{(HA-7a)}{ (a)}]{\color{black} For every coarse Lyapunov exponent $\alpha \in \Delta$ and $x,y \in X$, there exists an element $\rho \in \mc P_{\hat{\alpha}}$, the free product of the groups $N^\beta$, $\beta \not= \pm \alpha$, such that $\rho \cdot x \in  K \cdot y$}.
    \item [\mylabel{(HA-7b)}{ (b)}] {\color{black}If $w = (\alpha_1,\dots,\alpha_m)$ is a listing of the coarse Lyapunov exponents in which every exponent appears exactly once, then the map 

\[ h_{w,x} : C_w \times (\R^k \times K) \to X\]

defined by $h_{w,x}(\rho,g) = \rho \cdot g \cdot x$ takes arbitrarily small neighborhoods of $((e,e,e,\dots,e),e)$ to open neighborhoods of $x$.}
\end{itemize}%There exists a tuple of weights $\bar{\beta} = (\beta_1,\dots,\beta_n)$ (where we allow each weight to be listed more than once) such that the map from $p_x : C_{\bar{\beta}} \times (\R^k \times M) \to X$ defined by
%\[p_x((u_1,\dots , u_n),g) = u_1\cdot \dots \cdot u_n \cdot g \cdot x\]
%takes arbitrarily small neighborhoods of the trivial path to open neighborhoods of $x$.} 
Actions satisfying \ref{ta9}(a) will be called {\it super accessible} (cf. Definition \ref{def:strongly}). Actions satisfying \ref{ta9}(b) will be called {\it abstract Anosov}. These two cases are reflections of the differing assumptions in Theorems \ref{abelian} and \ref{basic abelian}.

\hspace{1cm}

%For any ordering $\bar{\beta} = (\beta_1,...,\beta_n)$ of $\Delta$ which lists every weight exactly once, there exist arbitrarily small open sets $U\subset C_{\bar{\beta}} \times (\R^k \times M)$ containing the trivial path such that for every $x \in X$, the restriction of the evaluation map at $x$ from $U \to X$ is onto a neighborhood of $x$.\\

\setcounter{enumi}{7}

\noindent We say that the action is {\it genuinely higher rank} if it also satisfies:\\

\item [\mylabel{ta2}{(HA-8)}] %for all hyperplanes $H \subset \R ^k $,
For every $\alpha \in \Delta$,  $(\ker \alpha \times K) \cdot  x$ is dense for some $x \in X$. \\

 \noindent We say that the action {\it has SRB measures} if it also satisfies:\\

\item [\mylabel{ta-srb}{(HA-9)}] for every pair $\alpha,\beta \in \Delta$, there exists a fully supported (not necessarily ergodic) measure $\mu$ which is invariant under $\ker \beta \times K$ and has \color{black}absolutely continuous disintegrations along the $N^\alpha$-orbit foliation. More precisely, this means that for every foliation box $B$ of the $N^\alpha$-orbit foliation, the pull back of the conditional measure of $\mu$ restricted to $B$ is absolutely continuous with respect to the Haar measure on $N^\alpha$, for $\mu$-almost every leaf. }

\color{black}
 %if we parametrize each $N_\alpha$ orbit leaf by the $N_\alpha$ action, then the conditional measure of $\mu$ restricted on $B$ on the intersection of almost every $N_\alpha$-orbit leaf with the foliation box is absolutely continuous with respect to the Haar measure on $N_\alpha$. }

%\item \label{ta7} For every $\alpha$, if $\mc P_{\hat{\alpha}}$ is the group freely generated by $N_\beta$, $\beta \not= \pm \alpha$, then there exists a neighborhood $U \subset (\R^k \times M) \ltimes \mc P_{\hat{\alpha}}$ such that $U \cdot x$ is a neighborhood of $x$ for every $x \in X$.
\end{itemize}

\end{definition}

%\begin{remark}
% Several properties listed here may follow from others. In particular, we believe that properties \ref{ta8}, \ref{ta3} {  and \ref{ta-srb}} can be deduced from the other conditions of the definition under transitivity assumptions. Since the proofs are straightforward when working with smooth systems, we add them as additional conditions here. {  In Section \ref{sec:smooth-top} we deduce that the extension constructed in Section \ref{extension} of a {\it smooth} partially hyperbolic action satisfying conditions \ref{FP1} and \ref{FP2}   is a HAPHA with SRB measures. } See also Section 4.6 and Remark 14.6 of \cite{Spatzier-Vinhage}.
%\end{remark}

\begin{remark}
While many examples of smooth partially hyperbolic $\R^k \times K$ satisfy assumptions \ref{ta1}-\ref{ta-srb}, some do not. In particular, actions with nontrivial Jordan blocks will fail to satisfy \ref{ta5} and actions with rank one factors will fail to satisfy \ref{ta2}. Among homogeneous actions, these are the actions which  fail to satisfy the conditions and we believe them to be the only ways smooth partially hyperbolic $\R^k \times K$-actions fail to satisfy the conditions (although conjecturally they are all still homogeneous). In particular, $\R^k \times K$ actions which are restrictions of actions of semisimple groups are all HAPHAs (see Proposition \ref{prop:smooth-is-top}).
\end{remark}

%\begin{remark}
%\color{black} ....add a remark about (HA-7)a and (HA-7)b , can be substitutet by effective accessibility...\color{black}
%\end{remark}

%\begin{remark}
%\color{olive} The accessibility condition \ref{ta9}(a) can be relaxed to accessibility using all coarse Lyapunov foliations not proportional to some fixed $\alpha$ {\it and} legs along the orbit direction $\R^k$
%\color{black} ....add a remark about "transversally accessible"...\color{black}
%\end{remark}

Note that {\color{black}both the genuinely higher rank and super accessibility assumptions} force that $k \geq 2$ unless we are discussing trivial actions (when $\Delta = \emptyset$ and the $\R^k \times K$ action is transitive). %{ We will see that these generalize the notion of smooth actions in Section \ref{sec:smooth-top}.}% The following shows that our new notion is a direct generalization of smooth actions.
Our main (and most general) technical result of the next few sections follows. {  We require an additional assumption of {\it integral Lyapunov coefficients, defined after Definition \ref{def:rhohat}}. This will always hold for $C^2$-actions by Lemma \ref{lem:brown}.}

\begin{theorem}
\label{thm:technical}
Let $\alpha$ be a  genuinely higher-rank HAPHA {  with integral Lyapunov coefficients and SRB measures}.  Then there exists a Lie group $H$, an embedding of $\R ^k \times K$ into $H$, a cocompact lattice $\Lambda \subset H$, and a homeomorphism $\phi : X \to H/\Lambda$ which conjugates the $\R ^k \times K$-action to the natural actions by left translation on $H/\Lambda$ by $\R ^k \times K$.  
\end{theorem}

Throughout the remainder of Section \ref{sec:top-anosov}, we assume that $\R^k \times K \curvearrowright X$ is a genuinely higher-rank HAPHA with SRB measures and integral Lyapunov coefficients unless otherwise stated.

\subsection{Basic dynamical properties}% of leafwise homogeneous topological Anosov actions}

\begin{lemma}
\label{lem:coarse-Lyapunov-dynamical}
 {\color{black}For each $\alpha \in \Delta$, the action of $N^\alpha$ is free} and for every $x \in X$, $W^\alpha(x) := N^\alpha x$ consists of the set of points $y \in X$ such that $d(a^tx,a^ty) \to 0$  as $t \to \infty$ for all $a$ such that $\alpha(a) < 0$. Furthermore, $W^\alpha(x) = \bigcap_{\alpha(a) < 0} W^s_a(x)$, and $W^s_{(a_1,\dots,a_m)}(x) = \bigcap W^s_{a_i}(x)$.
\end{lemma}
%https://www.overleaf.com/project/62c5c5a2ce05c20f6867eb98
%{  could maybe prove this without TA-6. Do we need the full version?}

\begin{proof}
The proof is immediate from \ref{ta10} and \ref{ta3}. Indeed, notice that given any coarse Lyapunov exponent $\alpha$ we may choose a collection $\set{a_1,\dots,a_m}$ such that $\Delta^-(\set{a_1,\dots,a_m}) = \alpha$. {\color{black} Since the evaluations are injective by \ref{ta10}, the action of $N^\alpha$ is free.} Indeed, given any collection which already contains $\alpha$, if $\chi \not\in \Delta^-(\set{a_1,a_2,\dots,a_m})$ for some $\chi \in \Delta\setminus \set{\alpha}$, we may remove a coarse exponent $\chi$ by adding $a$ such that $\alpha(a) < 0$ and $\chi(a) > 0$. Again applying \ref{ta3} to each $W^s_{a_i}(x)$ independently,  it is clear that their intersection must be $W^\alpha(x)$.
\end{proof}

\color{black}
We now describe a basic operation that is critical to our analysis of the way the group actions of $N^\alpha$ interact with one another: the geometric commutator. This is in contrast to the infinitesimal commutator, which only exist when the space a has a smooth structure and the actions of $N^\alpha$ are known to be at least $C^1$. Even in the smooth setting, our foliations have smooth leaves but are only H\"older transversally, so we instead use a coarser version of the commutator which works even for HAPHAs.

Recall that if $\alpha$ and $\beta$ are coarse Lyapunov exponents, $\Sigma(\alpha,\beta)$ is the set of \color{black} coarse \color{black} exponents which can be written as $\sigma \alpha + \tau \beta$, where $\sigma,\tau > 0$ (Definition \ref{def:canonical-order}). Notice that while $\alpha$ and $\beta$ are only coarse exponents, the set $\Sigma$ is still well-defined since we consider all positive linear combinations. Given a group $\mc G$ and elements $g,h \in \mc G$, we use the following convention for group commutators:

\[ [u,v] = v^{-1}u^{-1}vu. \]

\begin{lemma}
\label{lem:geo-comm}
Fix a HAPHA. Let $\alpha$ and $\beta$ be non-proportional coarse Lyapunov exponents, $u \in N^\alpha$ and $v \in N^\beta$. Then for every $x$, there exists a unique collection of elements $w_i := \rho^{\alpha\beta}_{\gamma_i}(u,v,x) \in N^{\gamma_i}$, where $\gamma_i$ ranges over all coarse Lyapunov exponents in $\Sigma(\alpha,\beta)$ listed so that $(\alpha,\gamma_1,\gamma_2,\dots,\gamma_n,\beta)$ is the canonical circular ordering, satisfying:

\[ w_n * \dots * w_2 * w_1 * [u,v] \cdot x = x. \]

Furthermore, the functions $\rho^{\alpha\beta}_{\gamma_i}$ are continuous in all three variables, and satisfy the following equivariance property for any $g \in \R^k \times K$:

\begin{equation}
\label{eq:rho-equivariance}
 g_*\rho^{\alpha\beta}_{\gamma_i}(u,v,x) = \rho^{\alpha\beta}_{\gamma_i}(g_*u,g_*v,gx).
 \end{equation}
\end{lemma}

\begin{proof}
Consider the points $y = [u,v] \cdot x$, and notice that if $a \in \R^k$ satisfies $\alpha(a),\beta(a) < 0$, then $d(a^n\cdot x,a^n \cdot y) \to 0$. In particular, if $\set{a_1,\dots,a_m}$ are elements such that $\Delta^-(\set{a_1,\dots,a_m}) = \set{\alpha,\gamma_1,\dots,\gamma_n,\beta}$ (such a choice of $\set{a_i}$ exists by Lemma \ref{lem:root-comb}), then $y \in W^s_{(a_1,\dots,a_m)}(x)$. Hence, by \ref{ta10} there exist unique elements $u' \in N^\alpha$, $v' \in N^\beta$, $w_i \in N^{\gamma_i}$ such that 

\begin{equation}
\label{eq:comm-eq1}
y =  u' * w_1 * \dots * w_n * v'  {\,  {\cdot} \,} x. 
\end{equation}

 Choose $a \in \ker \alpha$ such that $\beta(a) < 0$. Then since $\gamma_i \in \Sigma(\alpha,\beta)$ for every $i$, $\gamma_i(a) < 0$. {\color{black}In particular, using the automorphisms induced by \ref{ta5} and \eqref{eq:induced-auto}, 
 
 \begin{eqnarray*}
     \lim_{n \to \infty} (a^n)_* (u' * w_1 * \dots * w_n * v') & = & \lim_{n \to \infty} (a^n)_*u' * (a^n)_*w_1 \dots * (a^n)_*v' \\
     & = & u' * e * \dots * e * e \\
     &= & u',
 \end{eqnarray*}
 
 since for $v'$ and $w_i'$, the automorphism $a_*$ has eigenvalues determined by $e^{\beta(a)},e^{\gamma_i(a)} < 1$, and the automorphism $a_*$ on $N^\alpha$ is the identity, since it is a graded homothety, and all eigenvalues are $e^{\alpha(a)} =1$.} Therefore, $\lim_{n \to \infty} d(a^nx,a^ny) > 0$ unless $u' = e$. However, since $(a^n)_*[u,v] = [u,(a^n)_*v] \to [u,e] = e$, $d(a^nx,a^ny) \to 0$. Therefore, $u' = e$. Similarly, $v' = e$ by choosing $a \in \ker \beta$ such that $\alpha(a) < 0$. Since $u',v' =e$, \eqref{eq:comm-eq1} is the desired expression {\color{black}after multiplying by inverses}.
\end{proof}

{\color{black}
We have the following useful special case:
\begin{corollary}
    If $\alpha$ and $\beta$ are non-proportional coarse Lyapunov exponents of a HAPHA and $\Sigma(\alpha,\beta) = \emptyset$, then $N^\alpha$ and $N^\beta$ commute.
\end{corollary}}

\begin{definition}
\label{def:geo-comm}
The functions $\rho^{\alpha\beta} : N^\alpha \times N^\beta \to\displaystyle \prod_{i=1}^{\#\Sigma(\alpha,\beta)} N^{\gamma_i}$ and $\rho^{\alpha,\beta}_{\gamma_i} : N^\alpha \times N^\beta \to N^{\gamma_i}$ are called {\it geometric commutators} (we do not endow the target space with an algebraic structure, it is only a topological space). The element:

\begin{equation}
\label{eq:comm-relation}
 \rho^{\alpha,\beta}(u,v,x) * [u,v]
 \end{equation}
 
\noindent of $\mc P$ is called a {\it commutator relation} at $x$. It is an element of the cycle subgroup at $x$.
\end{definition}

\begin{lemma}
\label{lem:presentation-uniqueness}
Let $\set{\alpha_1,\dots,\alpha_n}$ be a {\color{black} stable} collection of coarse Lyapunov exponents listed in a circular ordering and $u_i,v_i \in N^{\gamma_i}$. If $u_1 * \dots * u_n \cdot x = v_1 * \dots * v_n \cdot x$ for every $x \in X$, then $v_i = u_i$ for $i=1,\dots,n$.
\end{lemma}

\begin{proof}
 Suppose that $u_1 * \dots * u_r\cdot x = v_1 * \dots * v_r \cdot x$ for every $x \in X$, where $u_i,v_i \in N^{\alpha_i}$. Then

\[ u_1 * \dots * u_r * {v_r}^{-1} * \dots * {v_1}^{-1} \]

\noindent stabilizes every point of $X$. Picking some $a \in \ker \alpha_r$ such that $\alpha_i(a) < 0$ for all $i = 1,\dots,r-1$ implies that

\[ (a_*u_1) * \dots * (a_*u_{r-1}) * u_r * {v_r}^{-1} * (a_*v_{r-1})^{-1} \dots * (a_*v_1)^{-1} \]

\noindent also stabilizes every point of $X$. Letting $\alpha_i(a) \to \infty$, $i < r$ implies that $u_r{v_r}^{-1}$ stabilizes every point of $X$. Since the action of $N^{\alpha_1}$ is faithful, $u_r= v_r$. Iterating this procedure by choosing $a_i \in \ker \beta_i$ such that $\alpha_j(a_i) < 0$ for $j = 1,\dots, i-1$ inductively shows $u_i = v_i$, $i=1,\dots,r-1$.
\end{proof}

\subsection{Groups generated by opposite coarse weights} In this section, we study the interaction of the groups $N^\alpha$ and $N^{-\alpha}$, where $\alpha \in \Delta$ is a coarse weight such that $-\alpha \in \Delta$. We will show that they fit into a Lie group action, and establish certain structural features.

\begin{lemma}
\label{lem:closure-containment}
If $y \in \overline{\ker \alpha \cdot x}$, then $\overline{\ker \alpha \cdot y} \subset \overline{\ker \alpha \cdot x}$.
\end{lemma}

\begin{proof}
Suppose that $z \in \overline{\ker \alpha \cdot y}$ and $a \in \ker \alpha$ be such that $d(ay,z) < \ve$. Since $a : X \to X$ is continuous, we may choose $\delta > 0$ such that if $d(y,y') < \delta$, then $d(ay',z) < \ve$. Then choose $b \in \ker \alpha$ such that $d(bx,y) < \delta$. Then by construction, $d(abx,z) < \ve$ and $z \in \overline{\ker \alpha \cdot x}$.
\end{proof}

{  The following lemma gives a topological analogue of an abstract ergodic decomposition. We will apply it in different settings for the action of $\ker \alpha$.

\begin{lemma}
\label{lem:keralpha-foliation}
Fix $\alpha \in \Delta$. There is an {$\R^k$-invariant} residual set of points $x_0 \in X$ such that $(K \times \ker \alpha)\cdot x_0$ is dense and either 

\begin{itemize}
\item[(1)] $\ker \alpha \cdot x_0$ is dense, or 
\item[(2)] $\mc F_\alpha(m) := m \cdot \overline{\ker \alpha \cdot x_0}$, $m \in K$ are a family of closed sets that partition $X$ and each atom is saturated by $W^\beta$-leaves for every $\beta \in \Delta$ (including $\beta = \alpha$). {  The indexing of the partition $\set{\mc F_\alpha(m) : m \in K}$ by $m$ depends on $x_0$, but the partition itself is independent of the choice of $x_0$ from the residual set.}
\end{itemize}

 { In case (2), for every $a \in \R^k$ and $m \in K$, there exists $m' \in K$ such that $a \mc F_\alpha(m) = \mc F_\alpha(m')$}.
\end{lemma}

{\color{black}
\begin{proof}
    We wish to apply Lemma \ref{lem:top-decomp}. We begin by considering the partition $\mc W$ induced by the equivalence relation that $x \sim y$ if and only if there exists $a \in \ker \alpha$ such that $ax$ and $y$ are connected by a path in the foliations $W^\beta$, $\beta\not= \alpha$ with finitely many legs. Our application of Lemma \ref{lem:top-decomp} will rely on whether we are assuming \ref{ta9}(a) or \ref{ta9}(b).

    Under assumption \ref{ta9}(a), $K\mc W(x)= X$ for all $x \in X$. We may then apply Lemma \ref{lem:top-decomp} with $\mc G = \set{e}$ and get that a partition consisting of sets of the form $\overline{\mc W(mx_0)}$. We may without loss of generality assume that $x_0$ is a point which satisfies that $\mc W(x_0) = \overline{(\ker \alpha) \cdot x_0}$, since this condition is generic when the action is totally recurrent (see \ref{ta8}, \cite[Lemma 10.2]{Spatzier-Vinhage} and Lemma \ref{lem:residual-recurrence}). The results follows.

    Under assumption \ref{ta9}(b), we use the same equivalence relation to define the partition $\mc W$, but introduce the group $\mc G$. Let $L \subset \R^k$ be any line transverse to $\ker \alpha$, and let $\mc G = L\ltimes (N^\alpha * N^{-\alpha})$ be the semidirect product of $L$ and the group freely generated by $N^\alpha$ and $N^{-\alpha}$. Note that since geometric commutators of $\alpha$-legs with $\beta$-legs, $\beta\not= \alpha$ only produce new legs $\gamma$ with $\gamma \not= \alpha$, it follows that for any finite path $\rho$  in the foliations $W^\beta$, $\beta \not=\alpha$ which begins at a point $x$ and ends at a point $y$, $hy$ is the endpoint of some path $\rho'$ which begins at $hx$ for all $h \in \mc G$. That is, $h\mc W(x) = \mc W(hx)$ for all $h \in K \ltimes \mc G$. Finally, by assumption $K\mc W(x_0) \supset (K \times \ker \alpha)\cdot x_0$ is dense for some $x_0$. To apply Lemma \ref{lem:top-decomp} it remains to show property (*). This follows by taking the compact subsets of $\mc G$ which consist of a short $L$-leg, a single short $\alpha$-leg and a single short $-\alpha$-leg, listed in that order. Then \ref{ta9}(b), the local product structure assumption, yields (*). The remainder of the proof is as in the previous case.
 \end{proof}
}

Fix a coarse weight $\alpha \in \Delta$ such that $-\alpha \in \Delta$ as well. {\color{black}Our next immediate goal is to fit the actions of $N^\alpha$ and $N^{-\alpha}$ into the action of a single Lie group $G_\alpha$.

\begin{lemma}
\label{lem:group-fibration}
    Let $\set{\mc F_\alpha(m) : m \in K}$ denote the partition of $X$ into $\ker\alpha$-orbit closures from Lemma \ref{lem:keralpha-foliation}. There exists a continuous fiber bundle $\mc G_\alpha$ over $K$ such that

    \begin{itemize}
        \item[(1)] The fiber above $m \in K$ is a simply connected Lie group $G_\alpha(m)$.
        \item[(2)] For every $m \in K$, $N^\alpha$ and $N^{-\alpha}$ embed into $G_\alpha(m)$, and generate $G_\alpha(m)$.
        \item[(3)] For each $m \in K$, $G_\alpha(m)$ acts on $\mc F_\alpha(m) \subset X$.
        \item[(4)] The restrictions of the $G_\alpha(m)$-action to the embeddings of $N^\alpha$ and $N^{-\alpha}$ coincide with the actions from \ref{ta4}.
        \item[(5)] The vector bundle with fibers $\Lie(G_\alpha(m))$ is a continuous vector bundle over $K$, with $\Lie(N^\alpha)$ and $\Lie(N^{-\alpha})$ as continuous subbundles.
    \end{itemize}
\end{lemma}

\begin{proof}
Fix $x \in X$, and consider the stabilizer of a point, $\mc C_\alpha(x)$, under the action of the free product $N^\alpha * N^{-\alpha}$ on $X$. Then $\ker \alpha$ takes any cycle with legs only from $N^\alpha$ and $N^{-\alpha}$ to the same cycle at a new basepoint (since by \ref{ta5}, $\ker \alpha$ acts trivially on every leg), i.e. $\mc C_\alpha(x) = a_*\mc C_\alpha(x) = \mc C_\alpha(a \cdot x)$. Choose $x_0$ as in Lemma \ref{lem:keralpha-foliation} such that $x_0$  has a dense $\R^k$-orbit.
%(a residual property), and let $\mc F_\alpha$ be the corresponding partition. 
Since the atoms of $\mc F_\alpha$ are saturated by coarse Lyapunov leaves, $N^\alpha * N^{-\alpha}$-orbits are contained in a single $\ker \alpha$-orbit closure. {\color{black} Hence, the cycles at any point $x \in \mc F_\alpha(m)$ contain those of $m \cdot x_0$.} Therefore, by  Corollary \ref{cor:lie-from-const} at each $m \in K$ the group $\bar{G}_\alpha(m) = (N^\alpha * N^{-\alpha}) / \mc C_\alpha(m\cdot x_0)$ is Lie, and has a canonical continuous action $\bar{G}_\alpha(m) \curvearrowright \mc F_\alpha(m)$ into which the group actions of $N^\alpha$ and $N^{-\alpha}$ canonically embed. Let $G_\alpha(m)$ denote the universal cover of $\bar{G}_\alpha(m)$, so that $G_\alpha(m)$ also acts on $\mc F_\alpha(m)$, and is determined by its Lie algebra. {\color{black} In particular, we get two key features: that $\Lie(N^\alpha)$ and $\Lie(N^{-\alpha})$ can both be considered subalgebras of $G_\alpha(m)$, and $G_\alpha(m)$ acts locally freely (by definition, the evaluation map $g \mapsto g mx_0$ is a continuous bijection for $\bar{G}_\alpha(m)$).}

Finally, we wish to show that $\mc G_\alpha$ has the structure of a continuous fiber bundle. Indeed, it is a trivial fiber bundle. If $\rho \in N^\alpha * N^{-\alpha}$ and $m \in K$, and $m_* : N^\alpha * N^{-\alpha} \to N^\alpha * N^{-\alpha}$ is the automorphism from \ref{ta5}, the trivialization given by

\[ \phi : G_\alpha(e) \times K \to \mc G_\alpha \qquad \phi(\rho \, \mc C_\alpha(e),m) = (m_*\rho) \mc C_\alpha(m)\]

Since $\mc G_\alpha$ is a continuous fiber bundle, it follows immediately that the Lie algebras of the fibers also vary continuously.
\end{proof}
}

\begin{lemma}
\label{lem:Galpha}
%There exists a continuous factor $\bar{\pi} : \hat{X} \to \mathbb{T}^s$ for some factor $s$ such that

%\begin{enumerate}
%\item the $\R^k$-action descends to a linear action on $\mathbb{T}^s$ with a dense orbit,
%\item for every $z \in \mathbb{T}^s$, the actions of $N_\alpha$ and $N_{-\alpha}$ embed into an action of a Lie group $G_\alpha(z)$ on $\bar{\pi}^{-1}(z)$, 
%\item if $g \in \R^k \times M$, then there is an isomorphsim $g_* : G_\alpha(z) \to G_\alpha(g z)$ such that if $u \in N_{\pm \alpha}$, then $g\cdot (g_*u \cdot x) = u \cdot gx$, and
%\item if the action of $\ker \alpha$ has a dense orbit on $\mathbb{T}^s$, then $G_\alpha(z)$ does not depend on $z$, and there is a global action on $\hat{X}$.

%\end{enumerate}
There exists a {\color{black} simply} connected Lie group $G_\alpha$ and a continuous group action $G_\alpha \curvearrowright X$ such that

\begin{itemize}
\item[(1)] $\Lie(G_\alpha) \cong \Lie(N^\alpha) \oplus \mf g_0 \oplus \Lie(N^{-\alpha})$ {\color{black}for some subalgebra $\mf g_0$},% where $\mf g_0 \subset \R^k \oplus \Lie(M)$,
\item[(2)] the inclusion of $\Lie(N^{\pm \alpha})$ in $\Lie(G_\alpha)$ induces a local isomorphism from $N^{\pm \alpha}$ onto its image,
\item[(3)] the action of $\exp_{G_\alpha}(\Lie(N^{\pm \alpha}))$ coincides with the existing action of $N^{\pm \alpha}$, 
\item[(4)] the subgroups $\exp_{G_\alpha}(\Lie(N^{\alpha}))$ and $\exp_{G_\alpha}(\Lie(N^{-\alpha}))$ generate $G_\alpha$, and
\item[(5)] the action of {\color{black}$G_{0,\alpha} := \exp_{G_\alpha}(\mf g_0)$ commutes with the $\R^k$-action}.%coincides with the existing action as a subgroup of  $\R^k \times M$.
\end{itemize}
\end{lemma}

\begin{proof}

Consider the groups $G_\alpha(m) = (N^\alpha * N^{-\alpha}) / \mc C_\alpha(m)$ from Lemma \ref{lem:group-fibration}. We will show the existence of a normal subgroup $\mc C_0 \subset N^\alpha * N^{-\alpha}$ such that $(N^\alpha * N^{-\alpha})/\mc C_0$ is a Lie group and $\mc C_0 \subset \mc C_\alpha(m)$ for every $m \in K$. This will imply that there is a global group action $G_\alpha \curvearrowright X$ (since in this case, $G_\alpha(m)$ factors through a common quotient of the free product $N^\alpha * N^{-\alpha}$). Recall that if $g \in \R^k \times K$, $g$ normalizes the $N^{\pm \alpha}$-actions. Let $g_*$ denote the induced automorphism of $N^\alpha * N^{-\alpha}$. Any globally-defined map preserving the coarse Lyapunov foliations will take cycles to cycles, so $\mc C_\alpha(g \cdot n) = g_*\mc C_\alpha(n)$. Since $g_*$ is an automorphism of $N^\alpha * N^{-\alpha}$ taking $\mc C_{\alpha}(m)$ to $\mc C_{\alpha}(g \cdot m)$, $g_*$ induces an isomorphism $g_* : G_\alpha(m) \to G_{\alpha}(g \cdot m)$. 

As a remark, note that this determines the isomorphism class of $G_\alpha(m)$, but that this is insufficient for our purposes. To obtain a group action on $X$, we need to know the cycles $\mc C_\alpha(m)$ are constant, which is to say that the generating relations for the groups are the same. %This is equivalent to showing that the maps $\beta_m : \Lie(N_\alpha) \otimes \Lie(N_{-\alpha}) \to \Lie(G_\alpha)$ defined by $\beta_m(Y,Z) = [Y,Z]_{\Lie(G_\alpha(m))} \in \Lie(G_\alpha(m))$ are independent of $m$ in some sense {\color{olive} which we will make precise shortly}. It is possible that the groups $G_\alpha(m)$ have a harnessed family of isomorphisms, even though they do not induce the same actions. In fact, we will show that each $\Lie(G_\alpha(m))$ sits inside a global vector space independent of $m$ and that the maps $\beta_m$ are constant there, proving constancy of relations and hence of $\mc C_\alpha(m)$.  This is made more precise below.
%{\color{teal}so here we use $X$ also as a vector, I guess within the proof of this lemma we can freely replace all these $X$ by $Z$.}

We now consider the Lie algebra structure of $G_\alpha(m)$. Notice that $N^{\pm \alpha}$ has a grading that corresponds to the Oseledets splitting, $\Lie(N^\alpha) = E^{\alpha} \oplus E^{c_2\alpha} \oplus \dots \oplus E^{c_{\ell_1} \alpha}$ and $\Lie(N^{-\alpha}) = E^{-d_1\alpha} \oplus E^{-d_2\alpha} \oplus { \cdots \oplus}E^{-d_{\ell_2}\alpha}$. 
Suppose that ${ Z} \in E^{c_i\alpha}$ and $Y \in E^{-d_j\alpha}$. {\color{black}Then for any $x \in \mc F_\alpha(m)$, $\exp_{G_\alpha(m)}(t[{ Z},Y]_{G_\alpha(m)})x$} is a {\color{black}continuous curve in $X$, but a priori only a continuous curve.} %(recall that by \ref{ta4}, the evaluation maps of the $N_{\pm \alpha}$ are assumed to be locally bi-Lipschitz, but the global actions of $N_{\pm \alpha}$, and hence $G_\alpha(m)$, are only continuous).}
For each $a \in \R^k$, $a_*$ is an isomorphism between $G_\alpha(m)$ and $G_\alpha(m')$, where $m'$ is such that  $a\mc F_\alpha(m) = \mc F_\alpha(m')$. Let us consider the asymptotic behavior \color{black}$\exp_{G_\alpha(m)}(t [{  Z},Y]_{G_\alpha(m)})x$. \color{black} If $d_j > c_i$, then as $\alpha(a) \to \infty$:

\begin{eqnarray*}
d(ax,a\exp_{G_\alpha(m)}(t[{ Z},Y]_{G_\alpha(m)})x) & = & d(\exp_{G_\alpha(m')}(ta_*[{ Z},Y]_{G_\alpha(m)})ax,ax) \\
& = &  d(\exp_{G_\alpha(m')}(te^{(c_i-d_j)\alpha(a)}[{ Z},Y]_{G_\alpha(m')})ax,ax) \\
 & \to &  0.
\end{eqnarray*}

Importantly, note that we do not yet know the precise rate of convergence using the distance on $X$, since the group action itself is only continuous. However, if $d_j > c_i$, then the curve $\exp_{G_\alpha(m)}(t[{ Z},Y]_{G_\alpha(m)})$ is contained in the $N^{-\alpha}$-orbit by Lemma \ref{lem:coarse-Lyapunov-dynamical} (and similarly for the $N^{\alpha}$ -orbit for $c_i > d_j$). 

{\color{black} Since the $G_\alpha(m)$-action has a locally free orbit, we conclude that for all $t \in \R$, $\exp(t[Z,Y]_{G_\alpha(m)}) \in N^{\pm \alpha}$ whenever  $c_i \not= d_j$. That is, $[{ Z},Y]_{G_\alpha(m)} \in \Lie(N^{\pm \alpha})$.} 

Choose any continuous family of norms on $\Lie(G_\alpha(m))$ such that for $g \in K$, the maps $g_* : \Lie(G_\alpha(m)) \to \Lie(G_\alpha(gm))$ induced by \ref{ta5} are isometric (one may average over the closed subgroup of $K$ which fixes $\mc F_\alpha(m)$).

Assume that $c_i > d_j$, so $[{ Z},Y]_{G_\alpha(m)} \in \Lie(N^\alpha)$ (the opposite case is identical). Choose any $b \in \R^k$ such that $\alpha(b) > 0$, and let $m \in K$. Then $b\mc F_\alpha(m) = \mc F_\alpha(gm)$ for some $g \in K$, and $bg^{-1}$ fixes $\mc F_\alpha(m)$. Then $(bg^{-1})_*$ is a harnessed automorphism of $G_\alpha(m)$. Let ${ Z}' := g_*{ Z}$ and $Y' := g_*Y$. Then,   using that the $K$-action preserves the norms: % is isometric:
\color{black}

\begin{eqnarray*}
(bg^{-1})_*[{ Z},Y]_{G_\alpha(m)} & = & [(bg^{-1})_*{ Z},(bg^{-1})_*Y]_{G_\alpha(m)}, \mbox{ so} \\
\norm{(bg^{-1})_*[{ Z},Y]_{G_\alpha(m)}} & = & \norm{[e^{c_i\alpha(b)}{ Z}',e^{-d_j\alpha(b)}Y']_{G_\alpha(m)}} \\
 & = & e^{(c_i-d_j)\alpha(b)}\norm{[{ Z}',Y']_{G_\alpha({ gm})}}\\
 & = &    e^{(c_i-d_j)\alpha(b)}\norm{[{ Z},Y]_{G_\alpha(m)}}. %\text{(using $M$-action is isometric).}
\end{eqnarray*}
%{\color{teal}(Apparently here we mixed $d(bg^{-1}_*)$ with $bg^{-1}_*$.)
{ Moreover since $(bg^{-1})_\ast$ preserves the Oseledets splitting in $\Lie(N^{\pm \alpha})$, by a similar proof we have 
\begin{eqnarray*}
  \norm{(bg^{-1})^n_*[{ Z},Y]_{G_\alpha(m)}} =    e^{n(c_i-d_j)\alpha(b)}\norm{[{ Z},Y]_{G_\alpha(m)}}.  
\end{eqnarray*}}

Since $[{ Z},Y]_{G_\alpha(m)} \in \Lie(N^\alpha)$, this implies that $[{ Z},Y]_{G_\alpha(m)}$ is in the {  direct sum of generalized} eigenspaces of $(bg^{-1})_*$ { with  eigenvalues }whose moduli are $e^{(c_i-d_j)\alpha(b)}$ %{(\color{teal} here it is not eigenspace but more like just Oseledec space or subspace of singular value decomposition, and mapping with correct ratio of the norm does not mean it is in Oseledec space, unless we allow iteration. But the conclusion here is correct since here the equation also holds for $(bg^{-1})^n$.)}
and hence that $[{ Z},Y]_{G_\alpha(m)} \in E^{(c_i-d_j)\alpha}$.

 In the case when $c_i = d_j$,  the flow generated by $[{ Z},Y]_{G_\alpha(m)}$ commutes with the $\R^k$ action, {\color{black}as claimed.} %, and therefore its orbits are contained in the $\R^k \times M$-orbits by Lemma \ref{lem:dynamical-characterizations}. Define $\varphi(x)$ to be the element of $A \cdot M$ such that $\exp([{ Z},Y]) \cdot x = \varphi(x) \cdot x$. Then $\varphi$ is constant on $\ker \alpha$ orbits, and hence $\varphi(x)$ is constant on $\mc F_\alpha(m)$. Therefore, the flow generated by $[{ Z},Y]$ is a one-parameter subgroup of $\R^k { \times}M$ on $\mc F_\alpha(m)$. %Notice that $N_{\pm \alpha}$ each embed as Lie subgroups of $G_\alpha(x)$, and generate the group. Furthermore, $\Lie(N_{\pm \alpha})$ comes equipped with a grading $\Lie(N_\alpha) = \mf n_{c_1\alpha} \oplus \mf n_{c_2\alpha} \oplus \dots \oplus \mf n_{c_p\alpha}$ and $\Lie(N_{-\alpha}) = \mf n_{-d_1\alpha} \oplus \mf n_{-d_2\alpha} \oplus \dots \oplus \mf n_{-d_q\alpha}$, where $0 < c_1 < c_2 < \dots < c_p$ and $0 < d_1 < d_2 < \dots d_q$. Furthermore, since the dynamics of each $a \in \R^k$ acts by automorphism, we know that if $c_i \not= d_j$, then $[\mf n_{c_i\alpha},\mf n_{-d_j\alpha}] \subset \mf n_{(c_i-d_j)\alpha}$. If $c_i = d_j$, then the corresponding one-parameter subgroup generated by the commutator must commute with the $\R^k$-action, and therefore be an element of $\R^k \times M$. 
{\color{black}Let $\mf g_0(m) \subset \Lie(G_\alpha(m))$ denote the vector subspace of the Lie algebra spanned by elements of the form $[Z,Y]_{G_\alpha(m)}$ for some  $Z \in E^{c_i\alpha}$ and $Y \in E^{-c_i\alpha}$. Since $G_\alpha(m)$ is generated by $N^\alpha$ and $N^{-\alpha}$, we wish to establish the vector space decomposition

\begin{equation}
\label{eq:s-u-c decomp}
    \Lie(G_\alpha(m)) = \Lie(N^\alpha) \oplus \mf g_0(m) \oplus \Lie(N^{-\alpha}).
\end{equation} It suffices to check that 
%\color{black}Label bug here\color{black}
\begin{itemize}
%[label=(\roman*)]
    \item [\mylabel{(i)}{(i)}] $[\mf g_0(m),E^{c_i\alpha}] \subset E^{c_i\alpha}$ (similarly for $E^{-d_j\alpha}$), and
    \item [\mylabel{(ii)}{(ii)}]$[\mf g_0(m),\mf g_0(m)] \subset \mf g_0(m)$ (ie, $\mf g_0(m)$ is a subalgebra). 
\end{itemize}

In fact, we show a stronger version of (i), that if $W = [Z,Y] \in \mf g_0(m)$, and $Z' \in E^{c_j\alpha}$, then $[W,Z'] \in E^{c_j\alpha}$. This follows since $a_*[W,Z'] = [a_*W,a_*Z'] = e^{c_j\alpha(a)}[W,Z']$. Hence $[W,Z']$ is in the $E^{c_j\alpha}$-eigenspace of $a_*$, and hence part of $\Lie(N^\alpha)$ by Lemma \ref{lem:coarse-Lyapunov-dynamical} (since dynamically it contracts for any $a$ such that $\alpha(a) < 0$). It follows that it must be in the corresponding eigenspace.

To see (ii), note that if $Z_1 \in E^{c_i\alpha}$, $Y_1 \in E^{-c_i\alpha}$, $Z_2 \in E^{-c_j\alpha}$ and $Y_2 \in E^{-c_j\alpha}$, then by the Jacobi identity

\begin{equation}\label{eq:quad-bracket} [[Z_1,Y_1],[Z_2,Y_2]] = -[[[Z_2,Y_2],Z_1],Y_1] - [[Y_1,[Z_2,Y_2]],Z_1].
\end{equation}

By (i), this is a sum of elements of $\mf g_0(m)$, so $\mf g_0(m)$ is a subalgebra.}
Therefore, $\Lie(G_\alpha(m))$ has a canonical splitting $\Lie(N^\alpha) \oplus {\color{black}\mf g_0(m)}\oplus \Lie(N^{-\alpha})$.

{\color{black}Finally, we wish to show the independence of $\mf g_0(m)$ on $m$. %We will first show that for any $a \in \R^k$, if $ax_0 \in \mc F(m)$, then $G_\alpha(m) \cong G_\alpha(e)$ via an isomorphism restricting to the identity of $N_{\pm \alpha}$. In particular, this implies that $\mc C^c_\alpha(ax_0) = \mc C^c_\alpha(x_0)$. %following to conclude independence of $m$: For any $m,m' \in M$, there exist canonical isomorphisms $\psi_{m,m'} : G_\alpha(m) \to \Lie(G_\alpha(m'))$ such that $\psi_{m,m} = \id$, $\psi_{m,m'}|_{\Lie(N_{\pm \alpha})} = \id$ and $\psi_{m',m''} \of \psi_{m,m'} = \psi_{m,m''}$.
Given $m \in K$, let \[\omega_m : \bigoplus_\alpha \bigoplus_i E^{c_i\alpha} \wedge E^{-c_i\alpha} \to \mf g_0(m) \qquad \mbox{be defined by } \qquad \omega_{m}(Y\wedge Z) = [Y,Z]_{G_\alpha(m)}.\] Notice that $\omega_m$ is onto $\mf g_0(m)$ for every $m$. We claim that $\ker \omega_{m}$ is a constant subspace of $\omega_{m}$. Indeed, let $a \in \R^k$, $\sum Y_\ell \wedge Z_\ell \in \ker \omega_{m}$, $m \in K$ and  $m'$ be such that $a\mc F_\alpha(m) = \mc F_\alpha(m')$. Since $a_* : \mf g_0(m) \to \mf g_0(m')$ is a Lie algebra isomorphism, we get

\begin{multline*} 0 = a_*\omega_{m}\left(\sum Y_\ell\wedge Z_\ell\right) = \sum a_*[Y_\ell,Z_\ell]_{G_\alpha(m)}   = \sum [a_*Y_\ell,a_*Z_\ell]_{G_\alpha(m')} \\ = [e^{c_{i_\ell}\alpha(a)}Y_\ell,e^{-c_{i_\ell}\alpha(a)}Z]_{G_\alpha(m')} = \sum[Y_\ell,Z_\ell]_{G_\alpha(m')} = \omega_{m'}\left(\sum Y_\ell \wedge Z_\ell\right).
\end{multline*}

Thus, $m \mapsto \ker \omega_{m}$ is $\R^k$-invariant. Notice that the dimension is constant since the groups $G_\alpha(m)$ are all isomorphic (and hence $\mf g_0(m)$ must have constant dimension). Since the Lie algebra structures vary continuously, it follows that $\omega_m$ is a continuous map from the bundle with fibers $\Lie(G_\alpha(m))$ over $K$. In particular, we get that $\ker \omega_m$ is independent of $m$.

Define $\psi_{m,m'} : \mf g_0(m) \to \mf g_0(m')$ by 
\[ \psi_{m,m'}(\omega_m(Y \wedge Z)) := \omega_{m'}(Y\wedge Z),\]
whenever $Z \in E^{c_i\alpha}$ and $Y \in E^{-c_i\alpha}$ whenever $E^{c_i\alpha}$ and $E^{-c_i\alpha}$ are both Oseledets spaces. Since the $\omega_m$ are always onto and have a common kernel, this map is well-defined. Note that since the groups $G_\alpha(m)$ are all isomorphic, the dimension of each $\mf g_0(m)$ is always the same, and hence $\psi_{m,m'}$ is always an isomorphism of vector spaces. Hence, the maps $\psi_{m,m'}$ give an identification of the Lie algebras $\mf g_0(m)$ for any $m \in K$.}

{\color{black}Extend $\psi_{m,m'}$ to $\Psi_{m,m'} : \Lie(G_\alpha(m)) \to \Lie(G_\alpha(m'))$ using the decomposition \eqref{eq:s-u-c decomp}, setting $\Psi_{m,m'}|_{N^{\pm \alpha}} = \id$. Then $\Psi_{m,m'}$ is a vector space isomorphism for any $m,m'$, and we wish to show that it is a Lie algebra isomorphism. Notice that it suffices to show that 
\[ [Y,Z]_{G_\alpha(m)} = [Y,Z]_{G_\alpha(m')},  \; [W,Y]_{G_\alpha(m)} = [\psi_{m,m'}W,Y]_{G_\alpha(m')}, \mbox{ and } [W,Z]_{G_\alpha(m)} = [\psi_{m,m'}W,Z]_{G_\alpha(m')}\] for any $m,m' \in K$, ${ Z} \in E^{c_i\alpha}$, $Y \in E^{-d_j\alpha}$ and $W \in \mf g_0(m)$. Indeed, since $\Lie(N^{\pm \alpha})$ are already subalgebras, the only missing brackets are of the form $\psi_{m,m'}[W,W']_{G_\alpha(m)} = [\psi_{m,m'}W,\psi_{m,m'}W']_{G_\alpha(m')}$. But these can be computed from \eqref{eq:quad-bracket}, so it suffices to check the relations above.%, since if $W = [Y',Z'] \in \mf g_0(m)$,

%\[ [W,Y] = [[Y',Z'],Y] = -[[Z',Y],Y']-[[Y,Y'],Z]. \]
} We first %assume that
{prove the claim when} $\mc F_\alpha(m') = a\mc F_\alpha(m)$. We will leverage the fact that $[{ Z},Y]_{G_\alpha(m)} \in E^{(c_i-d_j)\alpha(a)}$ if $c_i \not= d_j$. 
Given $Z \in E^{c_i\alpha}$ and $Y \in E^{-d_j\alpha}$, with $c_i \not= d_j$, we see that since $a_*$ is a Lie algebra isomorphism,

\begin{eqnarray*}
[{ Z},Y]_{G_\alpha(m)} & = & (a^{-1})_*[a_*{ Z},a_*Y]_{G_{\alpha}(m')} \\ 
 & = & e^{(d_j-c_i)\alpha(a)}[e^{c_i\alpha(a)}{ Z},e^{-d_j\alpha(a)}Y]_{G_\alpha(m')} \\
 & = & [{ Z},Y]_{G_\alpha(m')}.
\end{eqnarray*}

Now, if $c_i = d_j$, the $[Z,Y]_{G_\alpha(m)} \in \mf g_0(m)$ and commutes with the $\R^k$-action. Hence if $W =[Y',Z']_{G_\alpha(m)} \in \mf g_0(m)$ is the image of a primitive element of $E^{c_i\alpha} \wedge E^{-c_i\alpha}$ under $\omega_m$, we can write

\begin{eqnarray*}
    [W,Y]_{G_\alpha(m)} & = & [[Y', Z']_{G_\alpha(m)}, Y]_{G_\alpha(m)}\\
    & =& (a^{-1})_*[[a_*Y',a_*Z']_{G_\alpha(m')}, a_*Y]_{G_\alpha(m')} \\
    & = & e^{-c_i\alpha(a)}[\psi_{m,m'}W,e^{c_i\alpha(a)}Y]_{G_\alpha(m')} \\
    &= & [\psi_{m,m'}W,Y]_{G_\alpha(m')}.
\end{eqnarray*}

The last case is identical. Therefore, the map $\Psi_{m,m'}$ is an isomorphism whenever $\mc F_\alpha(m)$ and $\mc F_\alpha(m')$ belong to the same $\R^k$-orbit. Since $G_\alpha(m)$ is simply connected, these isomorphisms lift to isomorphisms from $G_\alpha(m) \to G_\alpha(m')$. %the group $G_\alpha(m)$ is constant on $\R^k$-orbits,
Since they are the identity on $N^{\pm \alpha}$, $\mc C^c_\alpha(x_0)$, the space of (contractible) cycles in $N^\alpha$ and $N^{-\alpha}$ is constant on a dense set by \ref{ta1}. The sets $\mc C_\alpha(x)$ are semi-continuously varying in the following sense: if $\lim_{n \to \infty} x_n = x$, and $\sigma_n \in \mc C_\alpha(x_n)$ is a sequence of cycles converging to a cycle $\sigma$, then $\sigma \in \mc C(x)$. By density of the $\R^k$ orbit and this semicontinuity, $\mc C_\alpha(x_0)$ is contained in $\mc C_\alpha(x)$ for every $x \in X$. Therefore, there is a Lie group $G_\alpha = (N^\alpha * N^{-\alpha}) / \mc C_\alpha(x_0)$ which acts on the total space $X$, as described.
\end{proof}

\section{Polynomial forms of geometric commutators}
\label{sec:fibers}

Recall the geometric commutator  functions $\rho^{\alpha,\beta}: N^\alpha \times N^\beta \times X \to \prod_{i=\#\Sigma(\alpha,\beta)}^1 N^{\gamma_i}$ from Lemma \ref{lem:geo-comm} and Definition \ref{def:geo-comm}.

\begin{theorem} \label{thm:constant pairwise cycle structure}
If $\R^k \times K \curvearrowright X$ is a HAPHA satisfying the assumptions of Theorem \ref{thm:technical}, then the functions $\rho^{\alpha,\beta}(u,v,x)$ are independent of $x$.
\end{theorem}

We will in fact show in Corollary \ref{cor:phi-polynomial} that the functions $\rho^{\alpha,\beta}(u,v,x)$ are polynomials independent of $x$ using natural coordinates on $N^\alpha$ and $N^\beta$.

%\begin{remark}
%In this section, we require that the action is $C^{1,\theta}$ rather than a topological Cartan action. The only place in which this is used is in an application of Lemma \ref{lem:at-least-one}. Under the assumption that $\chi = u\alpha + v\beta \in [\alpha,\beta]$ implies $u \ge 1$ or $v \ge 1$, one can prove Theorem \ref{thm:constant pairwise cycle structure} for topological Cartan actions, as well.
%\end{remark}

\subsection{Reduction to Oseledets subspaces}

%The trichotomy obtained in Lemma \ref{lem:fiber-types} will play an important role in proving \hyperlink{base-fiber}{(BF)} and \hyperlink{base-base}{(BB)}. While the reductive case will not require establishing coordinates on the coarse Lyapunov subgroups of the fiber, we will need them in the abelian times compact and nilpotent times compact cases. Fortunately, the structure of the groups $V_\alpha^\Omega$ are very simple in both cases: $V_\alpha^\Omega$ is abelian. This is obvious in the abelian case, and in the nilpotent case, we notice the only nontrivial commutators come from negatively proportional weights (since the commutator must generate an element of $\R^k \times M$).

%Now if $V_\alpha^\Omega$ is abelian, we may identify it with its Lie algebra. In particular, by the definition of a topological Anosov action and \ref{ta5}, $V_\alpha^\Omega$ has a grading $V_\alpha^\Omega = E_{c_1\alpha}^\Omega \oplus E_{c_2\alpha}^\Omega \oplus \dots E_{c_{\ell(\alpha)}\alpha}^\Omega$ such that if $v \in E_{c_i\alpha}^\Omega$, then $a_*v = e^{c_i\alpha(a)}v$.

If $\alpha$ is a coarse Lyapunov exponent, recall from the definition of a HAPHA that $\Lie(N^\alpha)$ splits as a direct sum of Oseledets subspaces. That is, $\Lie(N^\alpha) = E^{c_1^\alpha \alpha} \oplus \dots \oplus E^{c_{\ell(\alpha)}^\alpha \alpha}$. Let $e(u) = \exp(u)$ for ease of notation, and $\log$ denote its inverse (which exists because $N^\alpha$ is nilpotent). Then by \ref{ta5}, if $u \in E^{c_i^\alpha\alpha}$, {  $a\in \R^k$}, 

\begin{equation}
\label{eq:intertwining}
a_*e(u) = e(e^{c_i^\alpha \alpha(a)}u).
\end{equation}

\begin{definition}
\label{def:rhohat}
Given Oseledets spaces $E^{c_i^\alpha \alpha}$ and $E^{c_j^\beta \beta}$, let $\hat{\rho}^{c_i^\alpha \alpha,c_j^\beta \beta}_{{ c_m^\gamma}\gamma} : E^{c_i^\alpha \alpha} \times E^{c_j^\beta \beta} \times X \to \Lie(N^\gamma)$ be defined by \[\hat{\rho}^{c_i^\alpha \alpha,c_j^\beta \beta}_{c_m^\gamma \gamma}(u,v,x) = \pi_m^\gamma(\log \rho^{\alpha,\beta}_{\gamma}(e(u),e(v),x)),\]

\noindent where $\pi_m^\gamma : \Lie(N^\gamma) \to E^{c_m^\gamma \gamma}$ is the projection onto the corresponding Oseledets space induced by the Oseledets splitting.

  We now turn to the last technical assumption which appears in Theorem \ref{thm:technical}. If $c_m^\gamma\gamma = \sigma \, c_i^\alpha \alpha + \tau \, c_j^\beta \beta$ for some $\sigma,\tau > 0$, we call $\sigma$ and $\tau$ the {\it Lyapunov coefficients} of $c_m^\gamma \gamma$ with respect to $c_i^\alpha \alpha$ and $c_j^\beta \beta$. We say that an action has {\it integral Lyapunov coefficients} if \color{black} for any non proportional $\alpha,\beta\in \Delta$, any $c_i^\alpha\alpha\in [\alpha], c_j^\beta\beta\in [\beta]$, %(here maybe we use the notation $\Omega$ and $\Omega(\alpha,\beta)$) 
  \color{black} $\hat{\rho}^{c_i^\alpha\alpha,c_j^\beta\beta}_{c_m^\gamma\gamma} \equiv 0$ whenever both Lyapunov coefficients \color{black}(with respect to $c_i^\alpha\alpha$ and  $c_j^\beta\beta$) \color{black} are less than 1.
\end{definition}
\color{black}
\begin{definition}\label{def:  Omega, Omega_al be} We denote by
$$\Omega:=\{c^\alpha_i\alpha, \alpha\in \Delta\}$$the set of all Lyapunov functionals. For $\alpha,\beta \in \Delta$, similar to $\Sigma(\alpha, \beta)$, let $\Omega(\alpha,\beta) \subset \Omega$ be the set of $c^\gamma_k\gamma \in \Omega$ such that $c^\gamma_k\gamma = \sigma c^\alpha_i\alpha + \tau c^\beta_j \beta$ for some $\sigma,\tau > 0$. 
\end{definition}
\color{black}

\begin{lemma}
\label{lem:oseledets-sufficient}
If for every pair of Oseledets subspaces $E^{c_i^\alpha \alpha}, E^{c_j^\beta \beta}$ with $\alpha$ and $\beta$ non-proportional, the functions $\hat{\rho}^{c_i^\alpha \alpha,c_j^\beta \beta}_{ c_m^\gamma \gamma}$% $(u,v) \mapsto \rho^{\alpha,\beta}_\Omega(e(u),e(v),x)$
are independent of $x$, then the functions $\rho^{\alpha,\beta}$ are independent of $x$ for every $\alpha,\beta \in \Delta$.
\end{lemma}

\begin{proof}
Let $u \in N^\alpha$ and $v \in N^\beta$. Then we may write \[ u = e(u_1)e(u_2)\cdots e(u_{\ell(\alpha)}) \mbox{ and }v = e(v_1)e(v_2) \cdots e(v_{\ell(\beta)}),\] where $u_i \in E^{c_i^\alpha \alpha}$ and $v_j \in E^{c_j^\beta \beta}$. We wish to compute $[u,v]$ using commutator relations coming from commutators from the exponentials of Oseledets subpaces. Using their expression in terms of Oseledets subspaces, we may write

\[ [u,v] = e(u_1)\cdots e(u_{\ell(\alpha)})\cdot e( v_1)\cdots e(v_{\ell(\beta)}) \cdot e(-u_{\ell(\alpha)})\cdots e(-u_1) \cdot e(-v_{\ell(\beta)}) \cdots e(-v_1). \]

 We wish to push the $e(-u_{\ell(\alpha)})$ term past the $e(v_i)$-terms to cancel with $e(u_{\ell(\alpha)})$. Since we have constant commutator relations among the exponentials of Oseledets spaces, we may do so, but accumulate their commutators along the way. We may choose to put them on the left or right whenever we commute, we choose to put them on the left. That is, we write $e(v_i)e(-u_{\ell(\alpha)}) = g e(-u_{\ell(\alpha)}) e(v_i)$, where $g$ is a geometric commutator of $e(-v_i)$ and $e(u_{\ell(\alpha)})$ determined by the $\rho$-function, which is constant. Recall that $g$ is a product of elements from the coarse Lyapunov groups strictly between $\alpha$ and $\beta$, which are constant by the assumption of the lemma.

The element $g$ itself may be decomposed as a product of Oseledets subspaces, which we may push past $e(v_1)\cdots e(v_{i-1})$ by the same method, accumulating new terms on the left in the process between the coarse weights of $g$ and $\beta$. These terms are independent of $x$ since we only commute terms coming from exponentials of Oseledets spaces. Since $g$ will always take values in coarse weights between $\alpha$ and $\beta$, there is a clear induction on $\#\Sigma(\alpha,\beta)$, which terminates since there are only finitely many such \color{black} coarse \color{black} weights. We may express $[u,v]$ as

\[ [u,v] = e(u_1)\cdots e(u_{\ell(\alpha)}) g e(-u_{\ell(\alpha)})  e( v_1)\cdots e(v_{\ell(\beta)}) \cdot e(-u_{\ell(\alpha)-1})\cdots e(-u_1) \cdot e(-v_{\ell(\beta)}) \cdots e(-v_1), \]

\noindent where $g$ is a product of exponentials of Oseledets space from \color{black} $\Omega(\alpha,\beta)$\color{black}, %from coarse Lyapunov exponents that lie strictly between $\alpha$ and $\beta$
which is still independent of $x$. Using the same procedure as above, we may push all $g$-terms to the far left, and cancel the $e(u_{\ell(\alpha)})$-term with its inverse to obtain the following expression, with $g'$ independent of $x$:

\[ [u,v] = g' e(u_1)\cdots e(u_{\ell(\alpha)-1})\cdot  e( v_1)\cdots e(v_{\ell(\beta)}) \cdot e(-u_{\ell(\alpha)-1})\cdots e(-u_1) \cdot e(-v_{\ell(\beta)}) \cdots e(-v_1). \]

We now repeat this process until we have canceled each $e(u_i)$-term which then further allows for the cancelling of all $e(v_j)$-terms, leaving only a product exponentials of Oseledets spaces from \color{black} $\Omega
(\alpha,\beta)$\color{black}, which is independent of $x$.

We have now reduced $[u,v]$ to a product of exponentials of Oseledets spaces of $\gamma$, where $\gamma \in \Sigma(\alpha,\beta)$. To deduce that $\rho^{\alpha,\beta}(u,v,x)$ is independent of $x$, we write \color{black}$\Sigma(\alpha,\beta)$ \color{black}  in a circular ordering. Then push all of the exponential terms from each the Oseledets subspace of the first coarse exponent to the left. Since {\color{black} the commutators $\rho^{\alpha,\beta}(e(u_i),e(v_j))$ are determined by the functions $\hat{\rho}$, which are assumed to be constant}, %we have constant relations among the exponentials of the Oseledets spaces,
this is possible, accumulating their commutators on the right. The corresponding reduction yields an element, written in circular ordering, which is independent of $x$. Uniqueness follows from Lemma \ref{lem:presentation-uniqueness}.
\end{proof}

\subsection{Setting up the inductions}

We prove Theorem \ref{thm:constant pairwise cycle structure} using Lemma \ref{lem:oseledets-sufficient} by showing each $\hat{\rho}^{c_i^\alpha\alpha,c_j^\beta\beta}_{c_m^\gamma\gamma}$ is constant. We use use three inductions. The outermost induction is on $\#\Sigma(\alpha,\beta)$, we call this Induction I. 
 In each step of the induction, we will show for Lyapunov exponents $c_i^\alpha\alpha$ and $c_j^\beta\beta$,

\begin{flalign}
 &  \label{74out-induction1}\mbox{if $\hat{\rho}^{c_i^\alpha\alpha,c_j^\beta\beta}_{c_m^\gamma\gamma} \not\equiv 0$, then $c_m^\gamma\gamma = \sigma \, c_i^\alpha\alpha + \tau \, c_j^\beta\beta$ for some  $\sigma,\tau \in \Z_+$, and} \\
 &  \label{74out-induction2} \mbox{if $c_m^\gamma\gamma = \sigma \, c_i^\alpha\alpha + \tau \, c_j^\beta\beta$, then $\hat{\rho}^{c_i^\alpha\alpha,c_j^\beta\beta}_{ c_m^\gamma\gamma}(u,v,x)$ is a polynomial which is $\sigma$-homogeneous in $u$,} \\ 
 &  \nonumber \hspace{2cm} \mbox{$\tau$-homogeneous in $v$ and independent of $x$.}
\end{flalign}

We first state a key consequence of the induction. Let $\mc P_{|\alpha,\beta|}$ denote the group freely generated by the groups $N^\gamma$, $\gamma  \in \Sigma(\alpha,\beta) \cup\set{\alpha,\beta}$.

%{\color{olive} Move up or state a slightly more general version, with this as a Corollary? We need similar, but slightly broader statements later, with a nearly indentical proof,,,}
\begin{lemma}
\label{lem:polynomial-to-nil}
If \eqref{74out-induction1} and \eqref{74out-induction2} hold for all linearly independent $\alpha,\beta$ such that $\#\Sigma(\alpha,\beta) \le n$, then for any such $\alpha,\beta$,

\begin{itemize}
\item  the action of $\mc P_{|\alpha,\beta|}$ factors through the action of a nilpotent Lie group $N^{|\alpha,\beta|}$,
\item  $\Lie(N^{|\alpha,\beta|}) = \bigoplus_{\gamma \in \Sigma(\alpha,\beta)\cup \set{\alpha,\beta}} \Lie(N^\gamma)$, and
\item the family of automorphisms $a_*$ defined in \ref{ta5} descend to an automorphism of  $N^{|\alpha,\beta|}$.
\end{itemize}
\end{lemma}

\begin{proof}

Write $\Sigma(\alpha,\beta) \cup \set{\alpha,\beta} = \set{\alpha = \gamma_1,\gamma_2,\dots,\gamma_r = \beta}$ in the induced circular ordering. Let $\mc G$ denote the factor of the group \color{black}$\mc P_{\abs{\alpha,\beta}}$ \color{black} modulo the commutator relations \eqref{eq:comm-relation}. By \eqref{74out-induction2} and Lemma \ref{lem:oseledets-sufficient}, they are independent of $x$. We first claim that every $\rho \in \mc G$ can be written as

\begin{equation}
\label{eq:nil-presentation}u_1 * \dots * u_r,
\end{equation}

\noindent where each $u_i \in N^{\gamma_i}$ is unique. Indeed, any $\rho \in \mc G$ can be written as $\rho = v_1* \dots * v_k$ (where each $v_i \in N^{\beta_{n_i}}$). %By Lemma \ref{lem:oseledets-sufficient}, since we have \eqref{74out-induction1}, we know that the function $\rho^{\alpha,\beta}_\gamma$ are independent of $x$.

 We may begin by pushing all of the terms from the $\beta_1$ component to the left. We do this by looking at the first term to appear with $\beta_1$. Each time we pass it through, we may accumulate some $[u^{\beta_1},v^{\beta_j}]$ which may be rewritten as $\rho(u^{\beta_1},v^{\beta_j})$, having no $\beta_1$ terms, since we have quotiented by the commutator relations \eqref{eq:comm-relation}. So we have shown that in $\mc G$, $\rho$ is equal to $u_1 * \rho'$, where $\rho'$ consists only of terms without $\beta_1$, and $u_1 \in N^{\beta_1}$.

We now proceed inductively. We may in the same way push all $\beta_2$ terms to the left. Notice now that each time we pass through, the ``commutator'' $\rho(u^{\beta_2},v^{\beta_j})$, $j \ge 3$ has no $\beta_1$ or $\beta_2$ terms.  Iterating this process yields the desired presentation of $\rho$.
%Uniqueness will follow from an argument similar to Lemma \ref{lem:extending-charts}. %{\color{olive} and Remark \ref{rem:rho-ambiguity}}.

Thus, every element of $\mc G$ has a unique presentation of the form \eqref{eq:nil-presentation}, where the uniqueness follows from Lemma \ref{lem:presentation-uniqueness}. The map which assigns an element $\rho$ to such a presentation gives a an injective map from $\mc G$ to $\prod N^{\beta_i}$ (but the map may not be a homomorphism of groups). By Lemma \ref{lem:continuity-criterion}, it will be continuous once its lift to $\mc P_{\abs{\alpha,\beta}}$ is continuous. In each combinatorial cell $C_{\overline{\beta}}$, the map is given by composition of the group multiplications in each $N^{\beta_i}$ and the functions $\rho^{\alpha,\beta}(\cdot,\cdot)$ evaluated on cell coordinates, which are continuous. Therefore, the lift is continuous, so the map from $\mc G$ is continuous. 

Therefore, there is an injective continuous map from $\mc G$ to a finite-dimensional space, and $\mc G$ is a Lie group by \color{black} Corollary \color{black} \ref{cor:lie-from-const}. Fix $a$ which contracts every $\beta_i$. The fact that $\mc G$ is nilpotent follows from the fact that it has a contracting automorphism.
\end{proof}

{\color{olive}
%Every weight $\chi \in [\alpha,\beta] \cap \Delta_f$ is of the form $u \alpha + v \beta$, so we may introduce a lexicographical ordering on $[\alpha,\beta] \cap \Delta_f$ by saying $u \alpha + v \beta \prec u' \alpha + v'\beta$ if $u < u'$ or $u = u'$ and $v < v'$. Without loss of generality assume $\Omega = \set{\chi_0,\dots,\chi_m}$ is ordered so that $\chi_0 \prec \dots \prec \chi_m$.

%\begin{corollary}[Base case]
%If $[\alpha,\beta] \cap \Delta_b = \emptyset$, for each $\chi_i \in \Omega$, $\rho^{\alpha,\beta}_{\chi_i}(s,t,x)$ is a cocycle over $\eta^\alpha$ in $s$ when fixing $t$, and is a cocycle over $\eta^\beta$ in $t$ when fixing $s$.
%\end{corollary}

%\begin{proof}
%This is essentially the same as the proof of Lemma \ref{lem:cocycle-property}. Indeed, the $\chi \in \Omega$ cannot be obtained as commutators with either $\alpha$ or $\beta$ and hence  ``slide'' past $\alpha$ or $\beta$ legs unchanged possibly introducing higher level legs however.
%\end{proof}

}

Each step of the outer induction on $\#\Sigma(\alpha,\beta)$ will be proved using two further inductions. We introduce a partial order on $\set{(i,j) : 1 \le i \le \ell(\alpha), 1 \le j \le \ell(\beta)}$ by saying that $(i_1,j_1) \preceq (i_2,j_2)$ if and only if $i_1 \le i_2$ and $j_1 \le j_2$. The second induction will utilize this partial order: we will show \eqref{74out-induction1} and \eqref{74out-induction2} for a pair $c_i^\alpha \alpha$ and $c_j^\beta\beta$ assuming that we have concluded it for all choices of $c_m^\gamma\gamma$ and all $c_{i'}^\alpha\alpha$, $c_{j'}^\beta\beta$ such that $(i,j) \preceq (i',j')$ and $(i,j) \not= (i',j')$. The base of this induction will then be $(\ell(\alpha),\ell(\beta))$, the unique maximal element. It is clear from the structure of the partial order that such an induction will exhaust all choices of $(i,j)$. We call this induction Induction II.

 Given $c_i^\alpha\alpha$ and $c_j^\beta \beta$, let 
\begin{equation}\label{def: [] for expon and Omega l}
 [c_i^\alpha \alpha,c_j^\alpha \beta] = \set{c_m^\gamma \gamma : \hat{\rho}^{c_i^\alpha \alpha,c_j^\beta \beta}_{c_m^\gamma \gamma} \not\equiv 0} \mbox{ and }\Omega_l = \set{  \sigma \, c_i^\alpha \alpha + \tau \, c_j^\beta \beta : \sigma + \tau = l, \; \sigma,\tau > 0} \cap \color{black}\Omega(\alpha,\beta)\color{black}. %\set{c_m^\gamma\gamma : \gamma \in \Delta}.
\end{equation}

 Then there are finitely many values $l_0 < l_1 < \dots < l_m$ such that \[\Omega(\alpha,\beta) = \bigcup_{p=0}^m \Omega_{l_p}.\]

\noindent Given a subset $S$ of Lyapunov exponents, and a weight $c_i^\alpha\alpha \in \color{black}\Omega\color{black}$, let 
\[ [c_i^\alpha\alpha,S] = \set{c_m^\gamma\gamma : \rho^{c_i^\alpha\alpha,c_j^\beta\beta}_{c_m^\gamma\gamma}\not\equiv 0 \mbox{ for some }c_j^\beta\beta \in S}.\]

%\color{black} Should we have a summary of notations in section 11 perhaps? or 12???\color{black}
%$\Delta_0 = \Omega$ and $\Delta_{i+1} = \big((\Delta_i + \alpha) \cup (\Delta_i + \beta)\big) \cap ([\alpha,\Delta_i] \cup [\beta,\Delta_i]) \cap [\alpha,\beta]$.

%\begin{corollary}
%\label{cor:delt-increasing}
%$\bigcup \Delta_i = [\alpha,\beta]$ %, % $\Delta_i \cap \Delta_j = \emptyset$ if $i \not= j$, 
%and if $i \le j$, $[\alpha,\Delta_j] \cap \Delta_i = [\beta,\Delta_j] \cap \Delta_i = \emptyset$. 
%\end{corollary}
%\begin{proof}
%The first claim follow from Lemma \ref{lem:uppertriangular}. {  Second claim?????}
%\end{proof}

Let $\mc P_{(\alpha,\beta)}$, $\mc P_{|\alpha,\beta)}$ and $\mc P_{(\alpha,\beta|}$ denote the groups freely generated by the coarse weights of $\Sigma(\alpha,\beta)$, $\Sigma(\alpha,\beta) \cup \set{\alpha}$ and $\Sigma(\alpha,\beta) \cup \set{\beta}$, respectively.  By the induction hypothesis, the action of each group factors through Lie groups which we denote by $N^{(\alpha,\beta)}$, $N^{|\alpha,\beta)}$ and $N^{(\alpha,\beta|}$, respectively. Let $F_l = \bigoplus_{c_m^\gamma\gamma \in \Omega_l} E^{c_m^\gamma\gamma}$, so that $\Lie(N^{(\alpha,\beta)}) = \bigoplus_{a=0}^m F_{l_a}$.

It now
 suffices to show Claims \eqref{74out-induction1} and \eqref{74out-induction2} for each $\chi \in \Omega_{l_p}$. We will do this using a final induction on $p$, starting from $p=0$. We call this induction Induction III. We summarize each induction below, noting that the proof of \eqref{74out-induction1} and \eqref{74out-induction2} runs in a lexicographical ordering: for each step of Induction I, we do every step of Induction II, and for each of Induction II, we do every step of Induction III:
 
 \begin{itemize}
 \item Induction I: $\#\Sigma(\alpha,\beta)$, base case $\#\Sigma(\alpha,\beta) = 0$.
 \item Induction II: Partial order on $(c_i^\alpha\alpha,c_j^\beta\beta)$, base case maximal element, induction moves downward.
 \item Induction III: $c_m^\gamma\gamma \in \Omega_l$, induction on $l$, base case $l = l_0$ smallest coefficients.
 \end{itemize}
 
 % The base of the induction will be $i = -1$, with $\Omega_{l_{-1}} = \emptyset$. 

\subsection{Proving the inductive steps}
Fix $l$, and given $g \in N^{(\alpha,\beta)}$, write $\log g =\check{g} + \bar{g} + \hat{g}$, where $\check{g}$ is the component of $\log g$ from the weight spaces of $\Omega_{l'}$, $l' < l$, $\bar{g}$ is the component from the weight spaces of $\Omega_l$ and $\hat{g}$ is the component from the weight spaces of $\Omega_{l'}$, $l' > l$.

By Induction I and Lemma \ref{lem:polynomial-to-nil}, there are groups $N^{(\alpha,\beta)}$, $N^{|\alpha,\beta)}$ and $N^{(\alpha,\beta|}$ generated by the coarse weights $\Sigma(\alpha,\beta)$, $\set{\alpha}\cup \Sigma(\alpha,\beta)$ and $\set{\beta} \cup \Sigma(\alpha,\beta)$. The following important lemma uses these group structures to describe the action of ${ N^\alpha}$ on $N^{(\alpha,\beta)}$.

\begin{lemma}
\label{lem:jordan-block-flag}
If $u \in E^{c_i^\alpha\alpha}$, then the conjugation action of $\exp(u)$ on  $N^{|\alpha,\beta)}$ preserves $N^{(\alpha,\beta)}$. Furthermore, $\ad(u)$  a nilpotent automorphism such that $\ad(u)(F_l) \subset F_{l+1}$.
\end{lemma}

\begin{proof}
That $\ad(u)$ is unipotent follows from the fact that $\ad(u)(F_l) \subset F_{l+1}$, which we now show. Assume $c_m^\gamma\gamma \in \Omega_l$. The  Lie group $N^{|\alpha,\beta)}$ also carries an automorphism $a_*$ which expands $E^{c_i^\alpha\alpha}$ by $e^{c_i^\alpha\alpha(a)}$ and $E^{c_m^\gamma\gamma}$ by $e^{c_m^\gamma\gamma(a)}$. Therefore, $[E^{c_m^\gamma\gamma},E^{c_i^\alpha\alpha}]$ consists of vectors which are expanded by { $e^{c_m^\gamma\gamma(a) + c_i^\alpha\alpha(a)}$}. Since $a$ is arbitrary we conclude that \color{black}$[E^{c_m^\gamma\gamma},E^{c_i^\alpha\alpha}] \subset E^{c_m^\gamma\gamma + c_i^\alpha\alpha}$\color{black}. Since $c_m^\gamma\gamma \in \Omega_l$, the result follows.
\end{proof}

A completely symmetric version holds for $u \in E^{c_j^\beta\beta}$. The following are immediate consequences of Lemma \ref{lem:jordan-block-flag}, and the Baker-Campbell-Hausdorff formula. 

%\color{green} and Lemma \ref{lem:polynomial-restrictions}\color{black}.

\begin{corollary}
\label{cor:omega-shear}
If $u \in E^{c_i^\alpha\alpha} \cup E^{c_j^\beta\beta}$, and $g = \exp(\check{g} + \bar{g} + \hat{g}) \in N^{(\alpha,\beta)}$, let $v = \exp(u) * g * \exp(-u)$. Then:

\begin{itemize}
\item[(1)] $v \in N^{(\alpha,\beta)}$,
\item[(2)] $\check{v}$ is a polynomial of $u$ and $\check{g}$,
\item[(3)] $\bar{v}$ takes the following form:
%is the sum of a polynomial in $u$ and $\check{g}$ and a linear function in $\bar{g}$ whose coefficients are polynomials in $u$,
\[ \bar{v} = p(u,\check{g}) + \bar{g} \]

for some polynomial $p $ such that $p(0,\cdot) = p(\cdot, 0) = 0$, and
\item[(4)] $\hat{v}$ is a polynomial in $u$, {\color{black}$\check{g}$, $\bar{g},$ and $\hat{g}$}.
\end{itemize}
\end{corollary}

\begin{corollary}
\label{cor:multiplication-polynomial}
If $g_1,g_2 \in N^{(\alpha,\beta)}$ and we write $g_3 = g_1g_2$, then

\[ \check{g}_3 = p_1(\check{g_1},\check{g_2}) \qquad \bar{g}_3 = \bar{g}_1 + \bar{g}_2 + p_2(\check{g}_1,\check{g}_2) \qquad \color{black}\hat{g}_3 =  p_3(\check{g}_1,\bar{g}_1,\hat{g}_1,\check{g}_2,\bar{g}_2,\hat{g}_2),\color{black} \]

\noindent for some polynomials $p_1$, $p_2$ and $p_3$.
\end{corollary}

Notice that $\rho^{\alpha,\beta}(u,v,x)$ is a formal product of elements from the groups $N^\gamma$, $\gamma \in \Sigma(\alpha,\beta)$, written in a circular ordering. Therefore, it represents a unique element of $N^{(\alpha,\beta)}$, which we abusively denote with the same notation. Fix $c_m^\gamma\gamma \in \Omega_l$, and define a function 
 $r^{c_i^\alpha\alpha,c_j^\beta\beta}_{c_m^\gamma\gamma}(u,v,x)$ to be the $E^{c_m^\gamma\gamma}$-component of $\log \rho^{\alpha,\beta}(e(u),e(v),x)$. \color{black}
\begin{corollary}
$r^{c_i^\alpha\alpha,c_j^\beta\beta}_{c_m^\gamma\gamma}(u,v,x) = \hat{\rho}^{c_i^\alpha\alpha,c_j^\beta\beta}_{c_m^\gamma\gamma}(u,v,x) + p(u,v)$ for some polynomial $p : E^{c_i^\alpha\alpha} \oplus E^{c_j^\beta\beta} \to E^{c_m^\gamma\gamma}$ independent of $x$.
\end{corollary}

\begin{proof}
The definitions of $r^{c_i^\alpha\alpha,c_j^\beta\beta}_{c_m^\gamma\gamma}$ and $\hat{\rho}^{c_i^\alpha\alpha,c_j^\beta\beta}_{c_m^\gamma\gamma}$ are quite similar, the only difference being that $\hat{\rho}^{c_i^\alpha\alpha,c_j^\beta\beta}_{c_m^\gamma\gamma}$ uses the $c_m^\gamma\gamma$-component of  $\log \rho^{\alpha,\beta}_\gamma$, while $r^{c_i^\alpha\alpha,c_j^\beta\beta}_{c_m^\gamma\gamma}$ regards $\rho^{\alpha,\beta}$ as an element of $N_{(\alpha,\beta)}$, then takes the $\log$ and the $c_m^\gamma\gamma$-component. Therefore, we wish to compare the standard exponential coordinate system on $N_{(\alpha,\beta)}$ and the coordinate system given by $(v_1,\dots,v_n) \mapsto e(v_1)\dots e(v_n)$, where $v_i \in N^{\gamma_i}$ and the $\gamma_i$ are listed in a circular ordering. By Corollary \ref{cor:multiplication-polynomial}, the $c_m^\gamma\gamma$-components will differ only by polynomials that depend on the terms of the commutator coming from $\Omega_{l'}$, $l' < l$ (ie, the $\check{\cdot}$\,-terms). Since by Induction III, such terms are polynomials, $\hat{\rho}^{c_i^\alpha\alpha,c_j^\beta\beta}_{c_m^\gamma\gamma}$ and $r^{c_i^\alpha\alpha,c_j^\beta\beta}_{c_m^\gamma\gamma}$ differ by a polynomial in $u$ and $v$ which is independent of $x$.
\end{proof}

Figure \ref{fig:inner-induction} gives an example of the structures above. We give a description of the features available for one step of the induction for this particular example. We assume that we are at the stage of the induction to analyze $\hat{\rho}^{c_2^\alpha\alpha,c_2^\beta\beta}_{c_j^{\gamma_m}\gamma_m}$. Then $\Omega_{l_0} = \set{c_1^{\gamma_1}\gamma_1}$, $\Omega_{l_1} = \set{c_1^{\gamma_2}\gamma_2}$ and $\Omega_{l_2} = \set{c_2^{\gamma_1}\gamma_1,c_2^{\gamma_2}\gamma_2}$, since they are the intersections of lines parallel to the one passing through $c_2^\beta\beta$ and $c_2^\alpha\alpha$. At the first stage of the innermost induction, we would analyze only the function $\hat{\rho}^{c_2^\alpha\alpha,c_2^\beta\beta}_{c_1^{\gamma_1}\gamma_1}$. The crucial feature for the base step is that,  by Lemma \ref{lem:jordan-block-flag}, no terms in the Oseledets space  $c_1^{\gamma_1}\gamma_1$ can appear by commuting the $c_i^{\gamma_j}\gamma_j$ with another weight in the figure. The second induction is necessary, due to the fact that $E^{c_2^\alpha\alpha}$ may only be a {\it vector} subspace of $\Lie(N^\alpha)$ and not be a subalgebra, and the algebraic properties of this subspace will be crucial in understanding the dependence of $\hat{\rho}^{c_2^\alpha\alpha,c_2^\beta\beta}_{c_1^{\gamma_1}\gamma_1}(u,v,x)$ on $u$. Luckily, some algebraic features remain. If $u_1,u_2 \in E^{c_2^\alpha\alpha}$, we may write $e(u_1+u_2) = e(u_1)e(u_2)\cdot g$ for some $g \in N^\alpha$. In fact, such a $g$ must lie in $\exp(E^{c_3^\alpha\alpha} \oplus E^{c_4^\alpha\alpha})$. By the second induction on the pairs $c_i^\alpha\alpha$ and $c_j^\beta\beta$, we know that this additional term $g$ will have polynomial relations with $c_2^\beta\beta$ and has polynomial relations with each $c_j^{\gamma_m}\gamma_m$ by the first induction on $\#\Sigma(\alpha,\beta)$. This allows the analysis to go through.

In the next step of the induction on the $\Omega_{l_i}$, $c_1^{\gamma_2}\gamma_2$ terms {\it may} appear when commuting the $c_1^{\gamma_1}\gamma_1$ terms with a multiple of $\alpha$ or $\beta$ (which will be needed in Lemma \ref{eq:cocycle-like-bb}), but for this example, this is the only Lyapunov exponent {\it strictly} between $\alpha$ and $\beta$ with this property (by considering Figure \ref{fig:inner-induction} and Lemma \ref{lem:jordan-block-flag}). This can and will appear when analyzing how the function $\hat{\rho}^{c_2^\alpha\alpha,c_2^\beta,\beta}_{c_1^{\gamma_2}\gamma_2}(u,v,x)$ depends on $u$. Such terms will contribute polynomials by the induction hypothesis, leading to the final polynomial form.

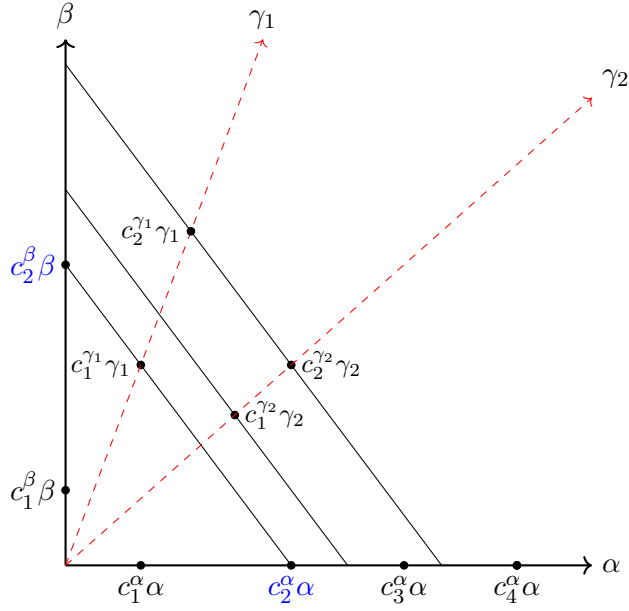
\begin{figure}[!ht]
\begin{tikzpicture}
\draw[thick,->] (0,0) -- (7,0);
\draw[thick,->] (0,0) -- (0,7);
\draw[fill] (1,0) circle [radius=0.05];
\node [below] at (1,0) {$c_1^\alpha\alpha$};
\draw[fill] (3,0) circle [radius=0.05];
\node [below] at (3,0) {\color{black}$c_2^\alpha\alpha$};
\draw[fill] (6,0) circle [radius=0.05];
\draw[fill] (6,0) circle [radius=0.05];
\node [below] at (6,0) {$c_4^\alpha\alpha$};
\draw[fill] (4.5,0) circle [radius=0.05];
\node [below] at (4.5,0) {$c_3^\alpha\alpha$};
\draw[fill] (0,1) circle [radius=0.05];
\node [left] at (0,1) {$c_1^\beta\beta$};
\draw[fill] (0,4) circle [radius=0.05];
\node [left] at (0,4) {\color{blue}$c_2^\beta\beta$};
\draw (0,4) -- (3,0);
\draw (0,5) -- (3.75,0);
\draw (0,20/3) -- (5,0);
\node [left] at (1,8/3) {\small $c_1^{\gamma_1}\gamma_1$};
\draw[fill] (1,8/3) circle [radius=0.05];
\draw [dashed,red,->] (0,0) -- (21/8,7);
\node [left] at (5/3,40/9) {\small $c_2^{\gamma_1}\gamma_1$};
\draw[fill] (5/3,40/9) circle [radius=0.05];
\node [right] at (9/4,2) {\small $c_1^{\gamma_2}\gamma_2$};
\draw[fill] (9/4,2) circle [radius=0.05];
\node [right] at (3,8/3) {\small $c_2^{\gamma_2}\gamma_2$};
\draw[fill] (3,8/3) circle [radius=0.05];
\draw [dashed,red,->] (0,0) -- (7,56/9);
\node [above] at (0,7) {$\beta$};
\node [right] at (7,0) {$\alpha$};
\node [above] at (21/8,7) {$\gamma_1$};
\node [above right] at (7,56/9) {$\gamma_2$};
\end{tikzpicture}
\caption{Lyapunov exponents in $\Sigma(\alpha,\beta)$}
\label{fig:inner-induction}
\end{figure}

We now return to the formal proof, assuming the induction hypotheses. Assume Claims \eqref{74out-induction1} and \eqref{74out-induction2} hold for $c_m^\gamma\gamma \in \Omega_{l_q}$, $q < p$.

Fix $v \in E^{c_j^\beta\beta}$ and let $\varphi(u,x) = \hat{\rho}^{c_i^\alpha\alpha,c_j^\beta\beta}_{c_m^\gamma\gamma}(u,v,x)$. Notice that  when $a \in \ker \beta$ and $c_m^\gamma\gamma = \sigma \, c_i^\alpha\alpha + \tau\, c_j^\beta\beta$, \eqref{eq:rho-equivariance}, \ref{ta5} and the definition of $\hat{\rho}^{c_i^\alpha\alpha,c_j^\beta\beta}_{c_m^\gamma\gamma}$ implies that
\begin{equation}
\label{eq:phi-pullout}
\varphi(u,x) = e^{-\sigma c_i^\alpha 
\alpha(a)}\varphi(e^{c_i^\alpha\alpha(a)}u,a\cdot x).
\end{equation}

\color{black}

We are now ready to establish the key lemma which gives a cocycle-like property to the function $\varphi$. While the proof requires checking some complicated details, the following lemma follows from two simple ideas: splitting a commutator into a sum of two commutators requires a conjugation and reordering, and with careful bookkeeping, the reordering and conjugation can be shown to contribute polynomial terms only. By the Baker-Campbell-Hausdorff formula, for homogeneous systems where the functions $\varphi$ are compositions of multiplication in a nilpotent Lie group, such polynomials will be nonvanishing unless $\sigma=\tau=1$, and the cocycle equation without them will not hold. We assume that the induction hypotheses hold.

\begin{lemma}
\label{eq:cocycle-like-bb}
$\varphi(u_1+u_2,x) = \varphi(u_1, x) + \varphi(u_2,e(u_1) x) + p(u_1,u_2)$  for some polynomial $p : E^{c_i^\alpha\alpha} \times E^{c_i^\alpha\alpha} \to E^{c_m^\gamma\gamma}$ such that $p(0,\cdot) \equiv 0$ and $p(\cdot ,0) \equiv 0$. % Furthermore, $\varphi_{\chi}(s,x) = cs^k$ for some $c \in \R$.
\end{lemma}

\begin{proof}

We assume that $c_m^\gamma \gamma \in \Omega_{l_p}$. Recall that $\varphi(u,x) = { \hat{\rho}^{c_i^\alpha\alpha,c_j^\beta\beta}_{c_m^\gamma\gamma}}(u,v,x)$ is the $c_m^{ \gamma}\gamma$-component of the unique path $\rho^{\alpha,\beta}_\gamma(u,v,x)$, written in circular ordering of the coarse weights in $\Sigma(\alpha,\beta)$, which connects $[e(u),e(v)] \cdot x$ and $x$. Given $u_1,u_2 \in E^{c_i\alpha}$, there exists $q(u_1,u_2) \in \bigoplus_{i' > i} E^{c_{i'}^\alpha\alpha}$ such that $e(u_1 + u_2) = e(q(u_1,u_2))e(u_2)e(u_1)$. Since $N^\alpha$ is nilpotent, $q$ is a polynomial in $u_1$ and $u_2$. Notice that using only the free product relations, we get that:

\begin{eqnarray*}
[e(u_1+u_2),e(v)] & = & e(-v) * e(-u_1-u_2) * e(v) * e(u_1+u_2) \\
 & = & e(-v) * e(-u_1-u_2) * e(v) *  e(q(u_1,u_2))  * e(u_2) * e(u_1)  \\
 & = & e(-v) * e(-u_1-u_2) * e(v) *  e(q(u_1,u_2))  * e(u_2) * \big( e(-v) * e(u_1) * e(v)\big) \\
  & & \qquad * \; \big(e(-v)  *  e(-u_1) * e(v)) * e(u_1)    \\
 & = &  e(-v)  * e(-u_1) * e(-u_2) *  e(-q(u_1,u_2)) * e(v) *   e(q(u_1,u_2))  * e(u_2)  \\
 & & \qquad  * \; \big(e(-v) * e(u_1) * e(v)\big)*  [e(u_1),e(v)]  \\
 & = & e(-v)  * e(-u_1) * e(-u_2) *  e(-q(u_1,u_2)) * e(v) *   e(q(u_1,u_2)) * e(-v) * e(u_2) * e(v)  \\
  & & \qquad *\;  [e(u_2),e(v)]  * \big(e(-v) * e(u_1) * e(v)\big)*  [e(u_1),e(v)]  \\
  & = & \big(e(-v)  * e(-u_1) * e(-u_2)\big) *[e(-v),e(q(u_1,u_2))] * \big( e(u_2) * e(u_1) * e(v)\big)  \\
  & & \qquad *\; \big(e(-v) * e(-u_1) * e(v)\big) * [e(u_2),e(v)]  * \big(e(-v) * e(u_1) * e(v)\big)*  [e(u_1),e(v)]  \\
    & = & { \color{red} \big(e(-v)  * e(-u_1) * e(-u_2)\big) *[e(-v),e(q(u_1,u_2))] * \big( e(u_2) * e(u_1) * e(v)\big) } \\
  & & \qquad *\; {\color{blue} \big(e(-v) * e(-u_1) * e(v)\big) * [e(u_2),e(v)]  * \big(e(-v) * e(u_1) * e(v)\big)} * {\color{olive} [e(u_1),e(v)] } \\
\end{eqnarray*}

The last equality is simply the second-to-last expression rewritten with color-coding. First, consider the red term. Since $q$ takes values in $\bigoplus_{i' > i} E^{c_{i'}^\alpha\alpha}$, we know the commutators of $q$ with $e(-v)$ are polynomial and independent of their basepoint by Induction II. Therefore we may rewrite the red term as $ \big(e(-v)  * e(-u_1) * e(-u_2)\big) * e(\tau_0(u_1,u_2,v)) *  \big( e(u_2) * e(u_1) * e(v)\big) $, where $\tau_0 \in \Lie(N^{(\alpha,\beta)})$ is independent of $x$, depending polynomially on $u_1$, $u_2$ and $v$. Then by Induction I, we know how each of the conjugating terms act on the term $\tau_0$, which must be polynomially. Therefore, the entire first red term is independent of $x$ and can be replaced by some $e(\tau(u_1,u_2,v))$ for some polynomial $\tau$ taking values in $\Lie(N^{(\alpha,\beta)})$.

We now turn to the blue terms. %Given a coarse Lyapunov exponent $\gamma \in \Sigma(\alpha,\beta)$, there exists a unique index $r(\gamma)$ such that $c_n^\gamma\gamma \in \Omega l$ with $l < l_p$ if and only if $n \le r(\gamma)$. By definition of $\hat{\rho}^{c_i^\alpha\alpha,c_j^\beta\beta}_{c_m^\gamma\gamma}$ and the induction hypotheses, we may rewrite the action of $[e(u_i),e(v)]$ as
%\begin{equation}
%\label{eeq:controlled-uncontrolled}
% [e(u_i),e(v)] \cdot x = \prod_{n = \#\Sigma(\alpha,\beta)}^1 e\left( \sum_{d=1}^{r(\gamma_n)} p_{n,d}(u_i,v) + \sum_{d=r(\gamma_n)+1}^{\ell(\gamma_n)} F_{n,d}(u_i,v,x) \right) \cdot x
% \end{equation}
%\noindent for $i=1,2$, where $p_{n,d}^\gamma$ is a polynomial independent of $x$ taking values in $E^{c_d^{\gamma_n}\gamma_n}$, and $F_{n,d} = \hat{\rho}^{c_i^\alpha\alpha,c_j^\beta\beta}_{c_d^{\gamma_n}\gamma_n}$.
Let $y = e(u_1) \cdot x$, so that  $h_2(y) := \rho^{\alpha,\beta}(e(u_2),e(v),y) \in N^{(\alpha,\beta)}$ acts on $y$ in the same way as $[e(u_2),e(v)]$. By induction, we may apply the conjugation of $h_2(y)$ by $e(-v) * e(u_1) * e(v)$ as 3 independent ones, which have well-understood forms by Corollary \ref{cor:omega-shear}. Indeed, applying parts (2) and (3) of the corollary three times shows that the blue terms act on $y$ in the same way that $g(y) = \exp(\check{g} + \bar{g}(y) + \hat{g}(y)) \in N^{(\alpha,\beta)}$ does, where $\check{g}$ is polynomial in $u_2$ and $v$ (by applying part (2) and Induction I), $\bar{g}(y) = \bar{h}_2(y) + p(u_2,v)$ for some polynomial $p$ (by applying part (3)) and $\hat{g}(y) \in \bigoplus_{l' > l_p} \bigoplus_{c_m^\gamma\gamma \in \Omega_{l}} E^{c_m^\gamma\gamma}$ is, for now, uncontrolled.

The final green term is straightforward, letting $h_1(x) = \rho^{\alpha,\beta}(e(u_1),e(v),x)$ be the element of $N^{(\alpha,\beta)}$ acting on $x$ in the same way as $[e(u_1),e(v)]$. Then $\check{h}_1$ is a polynomial in $u_1$ and $v$ independent of $x$ by Induction III. Putting the conclusions together yields that $[e(u_1+u_2),e(v)]$ acts on $x$ in the same way as:

\[ h(x) = e(\tau(u_1,u_2,v)) * g(y) * h_1(x),\]
so by Corollary \ref{cor:multiplication-polynomial}, $\check{h}$ is a polynomial in $u_1$, $u_2$ and $v$ which is independent of $x$, and $\bar{h}(x) = \bar{h}_1(e(u_1) \cdot x) + \bar{h}_2(x) + p(u_1,u_2,v)$ for some polynomial $p$ in $u_1$, $u_2$ and $v$ (depending only on the group structure of $N^{\alpha,\beta}$ and is hence independent of $x$). Therefore, the $\Omega_{l_p}$ terms have exactly the prescribed form.
\end{proof}

 {
\begin{lemma}
\label{lem:stationary-cocycle-like}
If $V$ and $W$ are vector spaces and $f : V \to W$ is a continuous function such that $f(0) = 0$ and 

\begin{equation}
\label{eq:stationary-cocycle-like}
f(v_1+v_2) = f(v_1) + f(v_2) + p(v_1,v_2)
\end{equation}

\noindent for some polynomial $p$, then $f$ is a polynomial.
\end{lemma}

\begin{proof}
First, observe that since $f(0+v) = f(0) + f(v) + p(0,v)$, it follows that $p(0,v) = 0$. Therefore, $p(v,w)$ has no constant terms and $p$ is symmetric by \eqref{eq:stationary-cocycle-like}. Furthermore, every term of $p$ must have degree at least one in {\it both} $v$ and $w$ since it is symmetric. If $q(v) = p(v,v)$, then every term of $q$ has degree at least two. In particular, since each nonzero term of $q$ is multiplied by at least $2^n \cdot 2^{-2n}$, $\sum 2^{n} q(2^{-n} v) $ is summable and the sum converges to a polynomial of the same degree.

From \eqref{eq:stationary-cocycle-like}, it follows that $f(v) = 2f\left(\frac{1}{2}v\right) + q\left(\frac{1}{2}v\right)$. Inductively, it follows that $f(v) = 2^nf\left(2^{-n}v\right) + \sum_{i=1}^n 2^{i-1}q\left(2^{-i}v\right)$. Since the sum on the right hand side converges to a polynomial, it follows that $\lim_{n\to\infty} 2^n f\left(2^{-n}v\right)$ converges to a vector uniformly bounded as $v$ varies in a compact set. Therefore, the map $D: v \mapsto \lim_{n\to\infty} 2^n f\left(2^{-n}v\right)$ is well-defined and satisfies

\begin{multline*} D(v+w) = \lim_{n\to\infty} 2^n f\left(2^{-n}(v+w)\right) = \lim_{n\to\infty} 2^n\left(  f\left(2^{-n}v\right) + f\left(2^{-n}w\right) + p(2^{-n}v,2^{-n}w)\right) \\ = D(v) + D(w) .
\end{multline*}

Furthermore, since $D(v) = f(v) - \sum_{i=1}^\infty 2^{i-1}q\left(2^{-i}v\right)$, it follows that $D$ is continuous and hence linear. Therefore, $f(v) = D(v) +  \sum_{i=1}^\infty 2^{i-1}q\left(2^{-i}v\right)$, and $f$ is a polynomial.
%First, observe that $f(2v) = 2f(v) + p(v,v)$, and if $s \not= 1$, then $f(2v) = 2^sf(v)$, so $f(v) = \frac{1}{2^s-2}p(v,v)$. Therefore, $f$ is a polynomial and $s \in \N$. If $s = 1$, first notice that $p(v,w) = f(v+w) - f(v) - f(w)$ is a symmetric polynomial, and that for every $\lambda \in \R$, $v,w \in V$, we have that $p(\lambda v,\lambda w) = f(\lambda v + \lambda w) - f(\lambda v) - f(\lambda w) = \lambda p(v,w)$. Therefore, the only nonzero terms of the polynomial must be linear functionals in $v$ and $w$, with constant coefficients. But since $p(0,w) = p(v,0) = 0$ for all $v,w \in V$, it follows that $p \equiv 0$. That is, $f$ is linear.
\end{proof}}

% Recall that we are addressing a weight $c_m^\gamma \gamma = sc_i^\alpha\alpha  + tc_j^\beta\beta$ for some $s,t \in \R_+$.

% \begin{lemma}
% \label{lem:s-integrality}
% The coefficient $s$ of $c_i^\alpha\alpha$ satisfies $s \in \Z_+$.
% \end{lemma}

% \begin{proof}
% Consider the group $\R^k \ltimes N_\alpha$, whose semidirect product structure is determined by \ref{ta5}.
% \color{black}Then $\R^k \ltimes N_\alpha$  has a canonical action on $X$. Since the group is solvable, there exists a Borel probability measure $\mu$ on $X$ which is invariant under $\R^k \ltimes N_\alpha$. Define

% \begin{equation}
% \label{eq:integrated-varphi}
% \Phi(u) = \int_X \varphi(u,x) \, d\mu(x).
% \end{equation}

% Then by Lemma \ref{eq:cocycle-like-bb} and invariance of $\mu$ under $N_\alpha$, $\Phi$ satisfes

% \[ \Phi(u+v) = \int_X \varphi(u+v,x) \, d \mu(x) = \int_X \varphi(u,x) + \varphi(v,e(u)\cdot x) + p(u,v) \, d\mu(x) = \Phi(u) + \Phi(v) + p(u,v). \]

% By Lemma \ref{lem:stationary-cocycle-like}, $\Phi$ is a polynomial. Fix $\lambda > 0$ and choose $a \in \ker \beta$ with $\lambda = e^{c_i^\alpha\alpha(a)}$. Then by \eqref{eq:phi-pullout} and invariance of $\mu$ under $\R^k$,

% \[ \Phi(\lambda u) = \int_X \varphi(e^{c_i^\alpha\alpha(a)}u,x) \, d\mu(x) = \int_X \lambda^s \varphi(u,a\cdot x)\, d\mu(x) = \lambda^s\Phi(u).\]

% Since $\Phi$ is a polynomial, it follows that $s \in \Z_+$.
% \end{proof}
% }

{  Recall that if we write $c_m^\gamma \gamma = \sigma \, c_i^\alpha \alpha + \tau \, c_j^\beta \beta$, we say that $\sigma$ and $\tau$ are the Lyapunov coefficients. Since we have assumed integral Lyapunov coefficients (Definition \ref{def:rhohat}), we may assume that either $\sigma \ge 1$ or $\tau \ge 1$. We without loss of generality assume that $\sigma \ge 1$.}

\begin{corollary}
\label{cor:phi-polynomial}
The function $\varphi(u,x)$ is a polynomial in $u$, whose coefficients are functions of $x$ which are constant along each the sets $\mc F_\beta(m)$ defined in Lemma \ref{lem:keralpha-foliation}.
\end{corollary}

\begin{proof}
By Lemma \ref{lem:polynomial-restrictions}, it suffices to show the following claim:
\vspace{.1cm}

{\noindent \bf Claim 12.12.1.}  For every $u,v \in E^{c_i^\alpha\alpha}$ such that $\norm{u} = 1$, $\varphi(tu+v,x)$ is a polynomial in $t$ whose coefficients are functions of $x$ which are constant along the sets $\mc F_\beta(m)$. 

\vspace{.1cm}

Claim 12.12.1 will follow from the following weaker claim:

\vspace{.1cm}

{\noindent \bf Claim 12.12.2. }
For every $u \in E^{c_i^\alpha\alpha}$ such that $\norm{u} = 1$, $\varphi(tu,x)$ is a polynomial in $t$ whose coefficients are functions of $x$ which are constant along the sets $\mc F_\beta(m)$. 

\vspace{.1cm}

Let us deduce Claim 12.12.1 from Claim 12.12.2. By Lemma \ref{eq:cocycle-like-bb},

\[ \varphi(tu+v,x) = \varphi(tu,x) + \varphi(v,e(tu)\cdot x) + p(tu,v). \]

If we can show Claim 12.12.2, then it follows that $\varphi(tu,x)$ is a polynomial in $t$ whose coefficients are functions of $x$ which are constant along the sets $\mc F_\beta(m)$. By Lemma \ref{lem:keralpha-foliation} the atoms of $\mc F_\beta$ are saturated by the leaves of $W^\alpha$. Since $\varphi(sv,y)$ is a polynomial in $s$ whose coefficients are functions of $y$ which are constant along atoms of $\mc F_\beta$, it follows that $\varphi(v,e(tu)\cdot x)$ is independent of $t$ and is constant as $x$ moves within the sets $\mc F_\beta(m)$. We have shown that $\varphi(tu+v,x)$ is the sum of a polynomial in $t$ whose  coefficients are functions of $x$ which are constant along the sets $\mc F_\beta(m)$, a  function of $x$ which is constant along the sets $\mc F_\beta(m)$ and a polynomial in $t$. Claim 12.12.1 follows.

So we aim to prove Claim 12.12.2. Fix $u \in E^{c_i^\alpha\alpha}$ with $\norm{u} = 1$, and for notational convenience, let $g(t,x) = \varphi(tu,x)$. We claim that for every $x \in X$ and Lebesgue almost every $t_1 \in \R$, $\displaystyle\left.\frac{d}{dt}\right|_{t=t_1} g(t,x)$ exists. To prove the claim, it suffices to show that $g$ is locally Lipschitz. By Lemma \ref{eq:cocycle-like-bb} and \eqref{eq:phi-pullout},

\begin{eqnarray*}
\norm{g(t,x) - g(t_1,x)} & = & \color{black}\norm{g(t-t_1,e(t_1u)\cdot x) + p(t_1u,(t-t_1)u)} \\
 & \le & \norm{t-t_1}^\sigma\norm{g\left(1,a \cdot e(t_1u) \cdot x\right)} + { \norm{p(t_1u,(t-t_1)u)}} \color{black}\\
\end{eqnarray*}

\noindent for a suitable choice of $a \in \ker\beta$. { Recall that $\sigma \ge 1$ by the  integrality assumption and the choice made directly before the statement of Corollary \ref{cor:phi-polynomial}. Hence, since $g(1,\cdot)$ is bounded, $p$ is a polynomial and $\sigma \ge 1$, %Lemma \ref{lem:s-integrality},
$g(t,x)$ is Lipschitz in any compact neighborhood of $t_1$.} Therefore, for almost every $t_1 \in E^{c_i^\alpha\alpha}$, $g$ is differentiable in $t$ at $t_1$. Therefore, since $g$ is differentiable in $t$ at $t = 0$ on a dense set of each orbit of the one-parameter subgroup generated by $u$, $g$ is differentiable in $t$ at $t = 0$ on a dense subset of $X$.

By \eqref{eq:phi-pullout}, the set of points for which $g$ is differentiable in $t$ is invariant under $\R^k$. Therefore, if $f(x)$ denotes the derivative of $g$ in $t$ at $t = 0$, $f$ exists on a dense subset of $X$. We claim that $\norm{f(x)} \le B$ for some $B \in \R$ whenever it exists. Indeed,

\[ \abs{f(x)} = \lim_{\ve \to 0} \frac{1}{\ve}\abs{g(\ve ,x)} = \lim_{\ve \to 0} \frac{1}{\ve} \ve^\sigma \abs{g(1,a_\ve \cdot x)} \le \sup_{y \in X} \abs{g(1,y)} \]

\noindent where $a_\ve \in \ker\beta$ is chosen appropriately (using 
\eqref{eq:phi-pullout}), since $\sigma \ge 1$.
Notice that if $a \in \ker \beta$, then again by \eqref{eq:phi-pullout} and the chain rule,

\[ f(a \cdot x) = e^{(\sigma-1)c_i^\alpha\alpha(a)}f(x). \]

Therefore, either $\sigma = 1$ or $f \equiv 0$, since otherwise one may apply an element $a \in \ker \beta$ with $\alpha(a)$ arbitrarily large to contradict the boundedness of $f$. Since $\sigma = 1$ or $f \equiv 0$, $f$ is constant along $\ker \beta$ orbits. 
%We claim that $f$ is also constant along $W^\alpha$ leaves whenever it exists. Indeed, if $x_2 \in W^\alpha(x_1)$ and $f$ exists at both $x_1$ and $x_2$, { choose $a_k \in \ker \beta$ such that $\alpha(a_k) \to \infty$ and $a_k \cdot x_1 \to z$ for some $z \in X$ (first choose any sequence such that $\alpha(a_k) \to \infty$, then choose a convergent subsequence). %Since $\alpha$ is contracted by $a_k$, $a_k \cdot x_2 \to z$ as well. 
%Then for any $u \in E^{c_i^\alpha\alpha}$:

%\[ f(x_i) = \lim_{\ve \to 0} \frac{1}{\ve} g(\ve ,x_i) = \lim_{k\to \infty} e^{c_i^\alpha\alpha(a_k)} g(e^{-c_i^\alpha\alpha(a_k)},x_i) = \lim_{k\to\infty} g(1,a_k \cdot x_i) = g(1,z). \]

{ 
We claim that $f$ is also constant along $W^\alpha$ leaves whenever it exists. { Assume it is not identically 0, otherwise the claim follows immediately. In this case $\sigma = 1$, and $g$ must be a cocycle over the $e(tu)$-action by comparing Lemma \ref{eq:cocycle-like-bb} and \eqref{eq:phi-pullout}. %Then pick a Weyl chamber $\mc W$ for which $\ker \beta$ bounds $\mc W$ and $\alpha(\mc W) > 0$. 
By \ref{ta-srb}, \color{black} there exists an SRB-like measure $\mu$ which is invariant under the $\ker\beta$-action and has absolutely continuous disintegrations along $N^\alpha$-leaves\color{black}. By standard Hopf argument, it follows that each ergodic component for the ergodic decomposition of $\mu$ with respect to the $\ker \beta$-action is also absolutely continuous along $N^\alpha$-leaves.

%We may then build the SRB measure $\mu_{\mc W}$ as in Section \ref{sec:SRB}, and disintegrate it into ergodic components for the $\ker \beta$ action. By Proposition \ref{prop:SRB}, it follows that 
Hence, for $\mu$-almost every point $x$, $f$ is constant at Lebesgue-almost every point of $W^\alpha$ (since it is a $\ker \beta$-invariant function, and defined at Lebesgue almost every point of {\it every} $W^\alpha$ leaf). Since the $e(tu)$-orbit foliation is a smooth subfoliation of $N^\alpha$, it follows that {\color{black}the disintegration of $\mu$} is absolutely continuous on almost every such orbit. {\color{black}Hence $f$ is constant on almost every $e(tu)$-orbit.} Since $f$ is the derivative of $g$, and $g$ is a cocycle over the $e(tu)$-flow, it follows that

\[ g(t,x) = \int_0^t f(e(ru) \cdot x) \, dr = v_xt\]
where $v_x \in E^{c_m^\gamma \gamma}$ is the common value of $f$ at almost every point of $W^\alpha(x)$. Then $g$ is linear on a dense set of $W^\alpha$ leaves, and has derivative invariant under $\ker \beta$. Since $g$ is continuous, it follows that $g(t,x) = v_xt$, where $v_x$ is a derivative which depends only on the atom of $\mc F_\beta$. This concludes the case when $\sigma =1$.}}

% Since in each leaf, $\varphi$ is a continuous function whose derivative exists almost everywhere with respect to Lebesgue measure on the leaf.
 {  We return to the general case (either $\sigma = 1$ or $f \equiv 0$).} Let $v_x$ denote the common value for $f$ along $W^\alpha(x)$ and $\ker \beta \cdot x$. Fix $t_1 \in \R$ %and define $\tilde{\varphi}(t,x) = \varphi(t_1 + u,x)$. Then if
 such that $f(e(t_1u) \cdot x)$ exists (recall that the collection of such $t_1$ has full Lebesgue measure). Then by Lemma \ref{eq:cocycle-like-bb}

\begin{multline}
\left.\frac{d}{dt}\right|_{t=t_1}g(t,x) = \left. \dfrac{d}{dt}\right|_{t= 0}\varphi((t+t_1)u,x) \\ = \left.\dfrac{d}{dt}\right|_{t= 0} \Big(\varphi(t_1u,x) + \varphi(tu,e(t_1u)\cdot x) + p(t_1u,tu) \Big)
= f(e(t_1u) \cdot x) =  v_x + q(t_1)
\end{multline}

\noindent where $q$ is some polynomial independent of $x$. Since we know the initial condition $g(0,x) = 0$,  one may integrate $v_x + q$ to a get a polynomial form for $g$ at each $x$. Furthermore, since $v_x$ does not vary as $x$ moves along $\ker \beta$, the coefficients of the polynomial are constant along the atoms of $\mc F_\beta$. This shows Claim  12.12.2, and finishes the proof.
\end{proof}

\begin{proof}[Proof of the inductive step  \eqref{74out-induction1} and \eqref{74out-induction2}] 
Corollary \ref{cor:phi-polynomial} implies that for fixed $v$, the function $\hat{\rho}^{c_i^\alpha\alpha,c_j^\beta\beta}_{c_m^\gamma\gamma}(u,v,x)$ is a polynomial in $u$, whose coefficients are functions of the atoms of $\mc F_\beta$. It is not difficult to see that an analogous version of Lemma \ref{eq:cocycle-like-bb} holds when fixing $u$ and varying $v$. Indeed, if we define $\psi(v,x) = \hat{\rho}^{c_i^\alpha\alpha,c_j^\beta\beta}_{c_m^\gamma\gamma}(u,v,x)$, then a nearly identical proof shows that

 \begin{equation}
 \label{eq:psi-cocycle-like} \psi(v_1+v_2,x) = \psi(v_1,x) + \psi(v_2,e(v_1) * e(u) \cdot x) + p(u,v_1,v_2) 
 \end{equation}

\noindent for some polynomial $p$. {Furthermore, since we have already shown that $\hat{\rho}^{c_i^\alpha\alpha,c_j^\beta\beta}_{c_m^\gamma\gamma}(u,v,x)$ is a polynomial in $u$, it follows that the following equation for $\psi$ analogous to \eqref{eq:phi-pullout} for $\varphi$ holds for all $a \in \R^k$, not just $a \in \ker \alpha$:

\begin{equation}
\label{eq:psi-pullout} \psi(v,x) = e^{-\tau c_j^\beta 
\beta(a)}\psi(e^{c_i^\beta\beta(a)}v,a\cdot x).
\end{equation}

Thus, $\psi$ is invariant under $\ker \beta$, and hence constant on each atom of $\mc F_\beta$.
}

Since the leaves of $W^\alpha$ and $W^\beta$ are contained in the atoms of $\mc F_\beta$ by Lemma \ref{lem:keralpha-foliation}, the dependence of $\psi$ on $x$ does not affect \eqref{eq:psi-cocycle-like}. Therefore, on each atom of $\mc F_\beta$, the map $\psi$ satisfies the assumption of Lemma \ref{lem:stationary-cocycle-like} and on each atom of $\mc F_\beta$, the function $\hat{\rho}^{c_i^\alpha\alpha,c_j^\beta\beta}_{c_m^\gamma\gamma}(u,v,x)$ is a polynomial in $u$ and $v$, {  with coefficients depending only on the atom}.

We must show that the polynomials are independent of $x$. Notice that $\hat{\rho}^{c_i^\alpha\alpha,c_j^\beta\beta}_{c_m^\gamma\gamma}(u,v,x)$ is a family of polynomials whose coefficients depend on $x$, so \eqref{eq:rho-equivariance} implies that the polynomial is $\sigma$-homogeneous in $u$ and $\tau$-homogeneous in $v$, since otherwise the coefficients would grow to $\infty$ by applying contracting or expanding elements of $\R^k$ (so, in particular, $\sigma,\tau \in \Z$ and \eqref{74out-induction1} holds). Therefore, the polynomial is unchanged as $x$ moves along its $\R^k$-orbit  by \eqref{eq:rho-equivariance}. Since there is a dense $\R^k$-orbit and $\hat{\rho}^{c_i^\alpha\alpha,c_j^\beta\beta}_{c_m^\gamma\gamma}$ is continuous in all variables, it follows that it is independent of $x$. Hence \eqref{74out-induction2} holds. %It also follows from \eqref{eq:rho-equivariance} that $s,t \in \Z$, since it must be exactly the degree of the polynomial in $u$ and $v$, respectively.
\end{proof}

\section{Partial homogeneity implies homogeneity}
\label{sec:pairwise-sufficient}

The goal of this section is to produce homogeneous structures related to the partial homogeneous structures coming from the $\R^k$- and $N^\alpha$-actions. The arguments and approach expand those of \cite[Section 14]{Spatzier-Vinhage} in several ways. While the overall scheme is similar, several new obstacles appear due to the presence of the compact group $K$ as well the coarse Lyapunov foliations being parameterized by general nilpotent Lie groups (rather than copies of $\R$). In particular, we build group actions to describe homogeneous spaces and use Corollary \ref{cor:lc-lpc-structure} to produce Lie structures on them.

We carefully piece together such homogeneous structures to build one on $X$. In particular, we use specific, computable relations between the flows of negatively proportional weights provided by the classification in Lemma \ref{lem:Galpha}. We will see that such relations (which we call  \emph{symplectic}) yield a canonical presentation of paths, but only for a dense set of paths containing an open neighborhood of the identity in $\hat{\mc P}_{\set{\alpha,-\alpha}}$ \color{black}(recall the definition of $\hat{\mc P}_{\set{\alpha,-\alpha}}$  in Definition \ref{def:P-groups}).\color{black} % This procedure replaces a more complicated K-theoretic argument that has appeared in works on local rigidity \cite{DamjanovicKatok2011,MR2672298,vinhageJMD2015,Vinhage:2015aa}.

Our approach is to find canonical presentations for words in $\hat{\mc P}$ using only commutator relations and symplectic relations which we assume are constant and well-defined. By fixing a regular element, we will be able to use such relations to rearrange the terms in an open set of words to write them using only stable legs, then only unstable legs, then the action (Proposition \ref{stable unstable cycle decomposition}). This will imply that the quotient group of $\hat{\mc P}$ by the commutator relations and symplectic relations is locally compact, which allows us to use the structure of locally compact groups (Corollary \ref{cor:lc-lpc-structure}). The core of the approach is Lemma \ref{open dense commutation}, which gives the ability to commute stable and unstable paths. 

%{ 
%These arguments allow us to give a fairly simple argument for homogeneity if the action satisfies the stronger transitivity assumption that every codimension {\it two} subgroup has a dense orbit  (Section \ref{subsec:special-case}). However, this assumption is not guaranteed by the absence of rank one factors, and furthermore requires rank at least 3.  The general case is much more complicated, and requires the use of the Gleason-Yamabe Lie criterion using a no small subgroups condition in Section \ref{subsec:nss}.  
In the end we will have shown that if an ideal factor has constant $\rho$-functions, it is conjugate to a homogeneous action. In particular, this holds for any maximal factor, which sets up the induction on factoring out by a chain of ideals in Section \ref{sec:fibers}. Recall Definition \ref{def:canonical-order}. 
%}

%Fix a leafwise homogeneous,  topological Anosov action $\R^k \times M \curvearrowright X$, and two ideals $\Omega, \Theta \subset \Delta$ such that if $\Omega = \set{U_\alpha}$ and $\Theta = \set{V_\alpha}$, then $U_\alpha \subset V_\alpha$ for every $\alpha \in \Delta$. Let $X^\Omega$ be the ideal factor of $X$ associated with $\Omega$, and denote points of $X^\Omega$ by $\bar{x} = Y_\Omega(x)$, where $x \in X$. Then let $Y_\Theta^\Omega(\bar{x}) \subset X^\Omega$ be the set of points $\bar{y} = Y_\Omega(x) \in X^\Omega$. Finally, recall that if $\alpha \in \Delta$, $N_\alpha^\Omega = N_\alpha / U_\alpha$ is a group. Let $V_\alpha^\Omega = V_\alpha / U_\alpha$ be the subgroup which preserves the level sets $Y_\Theta^\Omega(\bar{x})$.

\begin{definition}
\label{def:const-pairwise}
 We say that a HAPHA has {\it constant pairwise cycle structure} if 

\begin{itemize}
%[label=(CPCS-\arabic*)]
\item [\mylabel{cpcs1}{(CPCS-1)}] for each pair of nonproportional $\alpha,\beta \in \Delta$, $\gamma \in \Sigma(\alpha,\beta)$ and fixed $u \in N^\alpha$ and $v \in N^\beta$, $\rho^{\alpha,\beta}_{\gamma}(u,v,y)$ (see Definition \ref{def:geo-comm}) is independent of $y \in X$, and 
\item [\mylabel{cpcs2}{(CPCS-2)}] for each $\alpha \in \Delta$ such that $-\alpha \in \Delta$, the action of $\mc P_{\set{\alpha,-\alpha}}$ factors through a Lie group action on $X$.%$\mc C_{\set{\alpha,-c\alpha}}(x)$ is independent of $x$ and the quotient $\mc P_{\set{\alpha,-c\alpha}} / \mc C_{\set{\alpha,-c\alpha}}(x)$ is a Lie group.
\end{itemize}
\end{definition}
\color{black}

In applications, the conditions above are deduced from the genuinely higher-rank \color{black} assumption \ref{ta2} and the SRB measure assumption \ref{ta-srb} \color{black} (indeed, they are consequences of Theorem \ref{thm:constant pairwise cycle structure} and Lemma \ref{lem:Galpha}, respectively). However, we do not need the genuinely higher-rank assumptions to conclude the main goal of this section:

\begin{theorem}
\label{thm:pairwise-to-homo}
If $\R^k \times K \curvearrowright X$ is a HAPHA with constant pairwise cycle structure, then the action is topologically conjugate to a translation action on a homogeneous space $G / \Gamma$, with $\Gamma$ discrete.
\end{theorem}

\subsection{Stable-unstable-neutral presentations}
\label{subsec:sun-presentations}

%We do this by first showing that the geometric brackets are constant, ie, that $\rho^{\alpha,\beta}(s,t,x)$ is independent of $x$, and applying arguments from generators and relations of Lie groups adapted to this dynamical setting to show that this is sufficient. Throughout this section, we assume that we work on the space $X^E$. We will sometimes emphasize this by referring to such factor actions as {\it maximal factor actions}.

We let $\rho^{\alpha,\beta}_{\gamma}(u,v)$ denote the common value of $\rho^{\alpha,\beta}_{\gamma}(u,v,x)$ for $\alpha,\beta \in \Delta$, $\gamma \in \Sigma(\alpha,\beta)$. This is guaranteed to be independent of $x$ since the action has constant pairwise cycle structure.
Let $\mc P \;(= \mc P_{\Delta})$ be the group freely generated by the groups   $N^\alpha$. \color{black} Let $\mc C'$ be the smallest closed normal subgroup containing all cycles of the form

\begin{itemize}
\item $[u,v] * \rho^{\alpha,\beta}(u,v)$ for $u \in N^\alpha$ and $v \in N^\beta$ as described in Definition \ref{def:geo-comm} and 
\item any element of $\mc P_{\set{\alpha,-\alpha}}$ which factors through the identity of the Lie group action provided by \ref{cpcs2} (ie, $\ker(\mc P_{\set{\alpha,-\alpha}} \to G_\alpha)$, where $G_\alpha$ is as in Lemma \ref{lem:Galpha}).% (ie, any path that ends at an element of $H_\alpha$ in the notation of Section \ref{sec:symplectic-ideal-prop}).
\end{itemize}

\noindent Since such cycles are cycles at every point by assumption, $\mc C' \subset \mc C(x) := \Stab_{\mc P}(x)$ and $\mc C'$ is normal. Consider the quotient group $\mc G = \mc P / \mc C'$.

%{\color{olive}
%Our goal will be to show that $\mc G$ is a Lie group, so that $X$ is a homogeneous space of a Lie group. Then since the Cartan action acts by automorphisms on the generating subgroups of $\mc G$, we will conclude that the action is affine. We will show that $\mc G$ is a Lie group by showing it is locally path-connected by showing that the relations in $\mc C'$ allow group elements of $\mc P$ to be put in a normal form. The normal form will give a Lie factor $H$ of $\mc P$ with the same structure in which the normal form is locally unique. Then we will conclude that $H \cong \mc G$, and $\mc G$ is Lie, as desired.

%}

Fix a regular element $a_0 \in \R^k$. The goal of this subsection is to show that any $\rho \in \mc G$ can be reduced (via the relations in $\mc C'$) to some $\rho_+ * \rho_- * \rho_0$, with $\rho_+$ having only terms from $N^\chi$ with $\chi \in \Delta^+(a_0)$, $\rho_-$ having only terms from $N^\chi$  with $\chi \in \Delta^-(a_0)$, and $\rho_0$ being a product of elements of $\R^k \times K$ generated by symplectic pairs (see Lemma \ref{lem:Galpha}). Rather, we will show this for the group obtained by taking the semidirect product of the $\R^k \times K$ with $\mc G$ see Proposition \ref{stable unstable cycle decomposition}. We begin by identifying well-behaved subgroups of $\mc G$. Given a subset $\Xi \subset \Delta$, let $\mc G_\Xi$ denote the subgroup of $\mc G$ generated by the subgroups $N^\chi$, as $\chi$ ranges over $\Xi$. %the subgroups of $\mc G$ corresponding to $\Xi$.
We say that $\Xi$ is {\it stable} if $\Xi \subset \Delta^-(a)$ for some $a \in \R^k$. {  We say that it is {\it closed} if for any $\alpha,\beta \in \Xi$, $\Sigma(\alpha,\beta) \subset \Xi$.

\begin{lemma}
\label{lem:stable-closed}
    If $\Xi$ is a stable, closed collection of coarse weights, then $\mc G_\Xi$ is a nilpotent Lie group.
\end{lemma}

\begin{proof}
    The lemma follows as in the proof of Lemma \ref{lem:polynomial-to-nil}. Indeed, the proof works verbatim after restricting the coarse weights $\chi \in \Xi$ to a generic $\R^2 \subset \R^k$. Then one uses the circular ordering to find unique presentations of $\mc G_\Xi$. See also \cite{Spatzier-Vinhage}, Section 5.2 and Lemma 17.5.
\end{proof}
}

Let $\chi \in \Delta$ be a coarse weight such that $-\chi \in \Delta$, and $\beta \in \Delta$ be any linearly independent coarse weight. Let $\Xi = \set{ t\beta + s \chi : t \ge 0, s \in \R} \cap \Delta$, and $\Xi' = \set{ t\beta + s\chi : t > 0, s\in \R} \cap \Delta = \Xi \setminus \set{\chi,-\chi}$.%We shall henceforth use the groups $N_\chi^\Omega$, noting that they may be replaced by $V_\chi^\Omega$ for the proof of Theorem \ref{thm:pairwise-to-homo-fiber}.% Notice that even though $\beta$ and $\chi$ appearing here

\begin{proposition}
\label{prop:semi-structure}
If $\Xi$ is as above and $\rho \in \mc G_\Xi$ is any element, then $\rho = \rho_\chi * \rho_{\Xi'}$, where $\rho_\chi \in \mc G_{\set{\chi,-\chi}}$, and $\rho_{\Xi'} \in \mc G_{\Xi'}$. Furthermore, such a decomposition is unique.
\end{proposition}

\begin{proof}
The proof technique is the same as that of Lemma \ref{lem:polynomial-to-nil}. Using constancy of commutator relations, we may push any elements of $N^{\pm \chi}$ to the left, accumulating elements of $\mc G_{\Xi'}$ as the commutator on the right.

To see uniquness, suppose that $\rho_\chi * \rho_{\Xi'} = \rho_\chi' * \rho_{\Xi'}'$. Then $(\rho_\chi')^{-1} * \rho_\chi = \rho_{\Xi'}'  * \rho_{\Xi'}^{-1}$. But $\mc G_{\Xi'}$ is a a subgroup of $\mc G_\Xi$ and it is clear that $\mc G_{\Xi'} \cap \mc G_{\set{\chi,-\chi}} = \set{e}$. Therefore, $\rho_\chi' = \rho_\chi$ and $\rho_{\Xi'} = \rho_{\Xi'}'$, and the decomposition is unique.
\end{proof}

\begin{corollary}
\label{cor:semi-stabilizer}
If $\Xi$ is as above, $\mc G_\Xi$ is a Lie group. Furthermore, $\mc G_\Xi$ has the semidirect product structure $\mc G_{\set{\chi,-\chi}} \ltimes \mc G_{\Xi'}$, % with $\mc G_{\set{\chi,-c\chi}}$ taking one of the forms described in Lemma \ref{lem:Galpha-cases}, and
 with $\mc G_{\Xi'}$ a nilpotent group.
\end{corollary}

\begin{proof}
Notice that in the proof of Proposition \ref{prop:semi-structure}, we get a unique expression by moving the elements of $\mc G_{\set{\chi,-\chi}}$ to the left, and doing so accumulates only the $\mc G_{\Xi'}$ terms. Therefore, the decomposition gives $\mc G_\Xi$ the structure of a semidirect product of $\mc G_{\set{\chi,-\chi}}$ and $\mc G_{\Xi'}$. These groups are Lie by \ref{cpcs2} and Lemma \ref{lem:stable-closed}, respectively. The action of $\mc G_{\set{\chi,-\chi}}$  on $\mc G_{\Xi'}$  is continuous since the action of its generating subgroups corresponding to $\chi$ and $-\chi$ are given by commutators, which are continuous. Therefore, $\mc G_\Xi$ is the semidirect product of the Lie group $\mc G_{\set{\chi,-\chi}}$ with the Lie group $\mc G_{\Xi'}$, with a continuous representation, and is hence a Lie group. %The possibilities for $\mc G_{\set{\chi,-\chi}}$  follow from Lemma \ref{lem:Galpha-cases} and the fact that $\mc G_{\set{\chi,-\chi}}$  is the quotient of $G_\alpha$ (resp. the subgroup of $G_\alpha$ generated by $V_{\pm \alpha}$) by $H_\alpha$, and the cases described in Lemma \ref{lem:Galpha-cases} are closed under taking quotients.
%We show that if $\sigma \in \mc G_\Omega$ fixes $x$, then $\sigma = \sigma_\chi * \sigma_{\Omega'}$, where $\sigma_\chi$ is a cycle in $\mc G_{\set{\chi,-c\chi}}$ and $\sigma_{\Omega'} \in \mc G_{\Omega'}$. By Proposition \ref{prop:semi-structure}, it can always be written as such with $\sigma_\chi$, $\sigma_{\Omega'}$ not necessarily being cycles, instead only paths. Assume for contradiction that $\sigma_\chi$ is not a cycle. Pick $a$ such that $a \in \ker \chi$ and $\beta(a) < 0$, so that $\lambda(a) < 0$ for all $\lambda \in \Omega'$. Then on the one hand, $a^n \sigma$ is a cycle at $a^n x$ for every $n$, but it becomes closer to $\sigma_\chi \cdot (a^n x)$ since $a \in \ker \chi$. By choosing a convergent subsquence we obtain that $\sigma_\chi$ is a cycle at some point and hence also at $x$ (since the cycles in $\set{\chi,-c\chi}$ are independent of $x$. Hence $\sigma_\chi$ is a cycle and so it $\sigma_{\Omega'}$. Therefore, if $\sigma$ fixes $x$, $\sigma = e \in \mc G_\Omega$. Then the map which associates some $\rho \in \mc G_\omega$ an element of the form $\rho_\chi * \rho_{\Omega'}$ shows that $\mc G_\Omega$ is a Lie group by Theorem \ref{lem:gleason-palais}. The semidirect product structure is clear from the proof of Proposition \ref{prop:semi-structure}.
\end{proof}

The crucial tool in showing that $\mc P / \mc C'$ is Lie is to show that it is locally Euclidean. To that end, the crucial result is Lemma \ref{open dense commutation}.
Fix a regular element $a_0 \in \R^k$, then define 
%\[ \Delta^+(a) = \set{\chi \in \Delta \setminus E : \chi(a_0) > 0}, \qquad \Delta^-(a_0) = \set{ \chi \in \Delta \setminus E : \chi(a_0) < 0},\] 
%\noindent and set 
$\mc G_+ = \mc G_{\Delta^+(a_0)}$ and $\mc G_- = \mc G_{\Delta^-(a_0)}$. {  Note that $\mc G_\pm$ are nilpotent Lie groups by Lemma \ref{lem:stable-closed}.}
%Consider a $\mc{G}_{\set{\chi, -\chi}}$ locally isomorphic to $\Heis$ or $SL(2, \R)$. 
{\color{black} Let $\mc{D}_{\chi}$ be the subgroup of $\mc{G}_{\{\chi, -\chi\}}$ which was denoted by $G_{0,\chi}$ in Lemma \ref{lem:Galpha}. %Notice that even in the setting of Theorem \ref{thm:pairwise-to-homo-fiber}, the group $\mc D_\chi$ is constant on all of $X$ since it uses only relations in \ref{cpcs2} to define. %which normalizes both $\eta ^{\chi}$ and 
%$\eta^{-c \chi}$. Thus, $\mc D_\chi$ corresponds to either the diagonal subgroup of $SL(2,\R)$ (together with any possible central elements), or the center of $\Heis$.
Let $\mc{D} \subset \mc{G}$ be the group freely generated by all such $\mc{D}_{\chi} $. Note that the action of each element of $\mc D$ commutes with the $\R^k$ action by construction. Furthermore, the action of $\mc D$ is not obviously faithful, and may fail to be, as is the case for the Weyl chamber flow on $SL(3,\R)$ where there are 3 symplectic pairs of weights, each generating one-parameter subgroups of $\operatorname{Diag} \cong \R^2$.

\begin{lemma}
\label{lem:CD-normal}
Suppose the $\R^k$ orbit of $x_0$ is dense.  Then 
if $d \in \mc{D}$ is a cycle at $x_0$, then $d$ is a cycle everywhere. 
\end {lemma} 

\begin{proof}
We know that the action of $\mc D$ commutes with the $\R^k$-action. Then if $d$ is a cycle at  $x_0$, $d$ is a cycle at any point in $\R ^k \cdot x_0$, hence everywhere as  the $\R^k$ orbit of $x_0$ is dense.  
%Let $g \in \mc{G}_{\{\chi, -c\chi\}}$ and $a \in \R^k$.  Then $a g = \phi _a (g)$ for some automorphism $\phi _a$ of $\mc{G}_{\{\chi, -c\chi\}}$ in the connected component of this automorphism group. Hence $\phi _a$ is an inner automorphism. Note that $\R^k$ and hence  $\phi _a$ normalize $\eta ^{\chi}$, and $\eta ^{-c \chi}$.  
% Hence $\phi _a (g) =d_0gd_0^{-1}$ for some $d_0 \in \mc{D}_{\chi}$.   As $\mc{D}_{\chi}$ is abelian,  $\phi _a (d) =d$ for $d \in \mc{D}_{\chi}$ and in fact for all $d \in \mc{D}$.  Hence  $a d = d a$.  Hence if $d$ is  a cycle at $x_0$, $d$ is a cycle at any point in $\R ^k (x_0)$, hence everywhere as $x_0$ is generic for the $\R ^k$ action.  
 \end{proof}
 
\begin{lemma}
\label{lem:coarse-normalize}
    If $u \in N^\beta \subset \mc G$ and $g \in \mc D \subset \mc G$, then $g ug^{-1} \in N^\beta \subset \mc G$.
\end{lemma}

\begin{proof}
    It suffices to show this when $g \in \mc D_\chi$, since elements of $\mc D$ are always products of such elements. Notice that $\mc D_\chi$ and $N^\beta$ are both subgroups of $G_\Xi$ as in Corollary \ref{cor:semi-stabilizer}. Furthermore, each $a \in \R^k$ induces an automorphism of $G_\Xi = G_{\set{\chi,-\chi}} \ltimes G_{\Xi'}$. If $X$ generates a one-parameter subgroup of $\mc D_\chi$ and $Y$ generates a one-parameter subgroup of $N^\beta$ tangent to a Lyapunov subspace, then $a_*X = X$ and $a_*Y = e^{c_i\beta(a)}Y$ for some $c_i \in \R$. Since $a_*$ is an automorphism $a_*[X,Y] = e^{c_i\beta(a)}[X,Y]$. Therefore, $[X,Y]$ belongs to the same eigenspace as $Y$ and is an element of $\Lie(N^\beta)$. Therefore, $\Lie(\mc D_\chi)$ normalizes $\Lie(N^\beta)$, and hence $D_\chi$ normalizes $N^\beta$, as claimed.
\end{proof}
 }
 Denote by $\mc{C}_D$ the group of cycles in $\mc{D}$ at a point $x_0$ with a dense $\R^k$-orbit (such a point exists by \ref{ta1}),  and set $G= \mc{G}/\mc{C}_D$ and $D= \mc{D}/\mc{C}_D$. Notice that $G$ and $D$ are topological groups since $\mc C_D$ is a closed normal subgroup by Lemma \ref{lem:CD-normal}. Furthermore, let $G_\pm$ denote the projections of the groups $\mc G_\pm$ to $G$. The following lemma is immediate from {\color{black} Lemma \ref{lem:coarse-normalize}:} %the fact that the action of $D$ coincides with that of the $\R^k \times M$ action and Corollary \ref{cor:semi-stabilizer}:
 
% \begin{lemma}
%$D$ is a subgroup of $\R ^k$.
%\end {lemma} 

%\begin{proof}
%For $d \in D$, $d x \in \R ^k (x)$.  
%\end{proof}

 \begin{lemma}
 \label{lem:D-commutators}
 Let $\rho ^{\pm} \in G _{\pm}$ (resp. $G_{\pm}(\bar{x})$) and $\rho ^0 \in D$.  Then 
$$\rho ^0 \rho ^{\pm} ( \rho ^0) ^{-1}\in G_{\pm} \mbox{ (resp. }G_{\pm}(\bar{x})).$$
 \end{lemma} 
 
 The following is the basic commutation argument. It is an adaptation of Lemma 14.11 of \cite{Spatzier-Vinhage}, with changes to account for the nilpotent groups being multidimensional.
 
  \begin{lemma} \label{open dense commutation}
 For an open set of elements  $\rho ^+ \in G _+$, $\rho ^-  \in G _{-}$ containing $\set{e} \times \set{e}$ there exist 
 $(\rho ^+) ' \in G_+$, $(\rho ^- )' \in G_-$
 and $\rho ^0 \in D$ such that 
 $$ \rho ^+ * \rho ^-  = (\rho ^-) ' * (\rho ^+ )'  * \rho ^0 .$$

\noindent Furthermore, $(\rho^+)'$, $(\rho^-)'$ and $\rho^0$ depend continuously on $\rho^+$ and $\rho^-$.
  \end{lemma}

 \begin{proof}
Order the coarse weights of $\Delta^+(a_0)$ and $\Delta^-(a_0)$ using a fixed circular ordering as $\Delta^+(a_0) = \set{\alpha_1,\dots,\alpha_n}$ and $\Delta^-(a_0) = \set{\beta_1,\dots,\beta_m}$. Since $\Delta^+(a_0)$ is a stable subset, $G_+$ is a nilpotent group by Lemma \ref{lem:stable-closed}. Therefore, we may write $\rho_+ = u_n * \dots * u_1$ for some $u_i \in N^{\alpha_i}$. 
We will inductively show that we may write the product $\rho^+ * \rho^-$ as $u_n * \dots * u_k * (\rho^-)' * v_{k-1} * \dots *v_1 * \rho^0$ for some $v_i \in N^{\alpha_i}$, $(\rho^-)' \in \mc G_-$ and $\rho^0 \in D$ (all of which depend on $k$, the index of the induction). Our given expression is the base case $k = 1$.
 
 Suppose we have this for $k$. If $-\alpha_k \in \Delta$, then it must be in {\color{black}$\Delta^-(a_0)$}. Let $l(k)$  denote the index for which $\beta_{l(k)} = - \alpha_k$  if $-\alpha_k$  is a coarse weight.  Otherwise, since there is no coarse weight negatively proportional to $\alpha_k$, we set $\beta_{l(k)} = -\alpha_k$ with $l(k)$ a half integer 
  so that $-\alpha_k$ appears between $\beta_{l(k) - 1/2}$ and $\beta_{l(k) + 1/2}$. Then decompose $\Delta$ into six (possibly empty) subsets: $\set{\alpha_k}$, $\set{-\alpha_k}$, $\Delta_1 = \set{ \alpha_l : l < k}$, $\Delta_2 = \set{\alpha_l : l > k}$, $\Delta_3 =  \set{ \beta_l : l < l(k)}$ and $\Delta_4 = \set{\beta_l : l > l(k)}$. See Figure \ref{fig:quadrants}.

\begin{figure}[!ht]
\begin{center}
\begin{tikzpicture}[scale=.75]
\draw [thick,<->] (-4,-2) --(4,2);
\draw [ultra thick,blue,dashed] (0,4) --(0,-4);
\node [right] at (4,2) {$\alpha_k$}; 
\node [left] at (-4,-2) {$-\alpha_k$};
\node [above right] at (0,4) {$\set{\chi : \chi(a_0) = 0}$};
\node at (2,-1) {$\Delta_1$};
\node at (2,2) {$\Delta_2$};
\node at (-2,1) {$\Delta_3$};
\node at (-2,-2) {$\Delta_4$};
\end{tikzpicture}
\caption{Decomposing $(\R^k)^*$ into quadrants}
\label{fig:quadrants}
\end{center}
\end{figure}
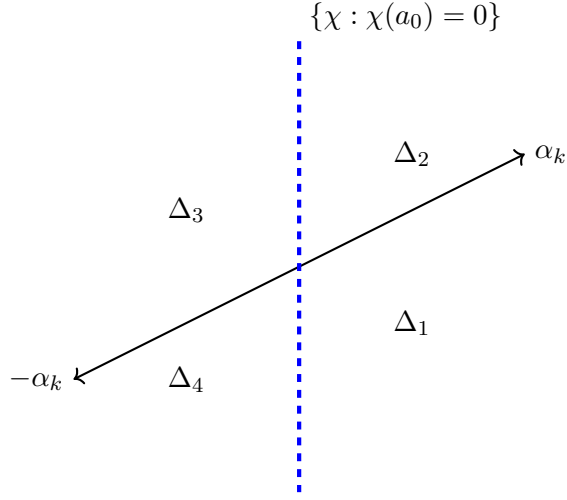

We let $G_{\Delta_i}$ denote the subgroup of $G$ generated by {\color{black}the groups $N^\beta$, $\beta \in \Delta_i$}. Notice that $\Delta_- = \Delta_3 \cup \set{-\alpha_k} \cup \Delta_4$ (with $\set{-\alpha_k}$ omitted if there is no coarse weight of this form) is stable, so again, since $G_-$ is nilpotent, $(\rho^-)' $ may be expressed uniquely as $q_3 * w * q_4$ with $q_3 \in G _{\Delta _3}$ and $q_4 \in G _{\Delta _4}$ and $w \in N^{-\alpha_k}$
    (if $-\alpha_k$ is not a coarse weight, we omit this term). Now, $\set{\alpha_k} \cup \Delta_2 \cup \Delta_3$ is a stable set whose associated group is nilpotent. So \color{black} $u_k * q_3 = q_2 * (q_3)' * v_k'$ for some $q_2 \in G_{\Delta_2}$, $(q_3)' \in G_{\Delta_3}$ and $v_k' \in N^{\alpha_k}$. Notice that by iterating some $a \in \R^k$ for which $\alpha_k(a) = 0$, and $\beta(a) < 0$ for all $\beta \in \Delta_2 \cup \Delta_3$, we actually know that $v_k' = u_k$. \color{black} %The term $s_k'$ is determined by a polynomial in $t_k$ and the lengths of the legs of $q_3$ by Lemma \ref{lem:group-integrality}(3). In particular, properties which hold for an open dense set of the $s_k'$ hold for an open dense set of the $t_k$. 
Thus, we have put our expression in the form:
 
\begin{eqnarray*}
u_n * \dots * u_k * (\rho^-)' * v_{k-1} * \dots * v_1 * \rho^0  & = & u_n * \dots * u_k  * (q_3 * w * q_4) * v_{k-1} * \dots * v_1 * \rho^0\\
 & = & u_n * \dots * (u_k  * q_3) * w * q_4 * v_{k-1} * \dots * v_1 * \rho^0\\
  &=&   u_n * \dots * u_{k+1} * (q_2 * (q_3)' * v_k') * w * q_4  \\
  && \qquad * \,v_{k-1} * \dots * v_1 * \rho^0
\end{eqnarray*}
 
 Now, there are two cases: $- \alpha_k \not\in \Delta$ in which case {\color{black}$w$ does not appear (we may take it to be $e$)}.  If $-\alpha_k \in\Delta$, then as long as $v_k'$ and $w$ are both sufficiently small, $v_k' * w = w' * v_k'' * g$ for some $w' \in N^{-\alpha_k}$, $v_k'' \in N^{\alpha_k}$ and $g \in D$ (since by Lemma \ref{lem:Galpha}, the corresponding subalgebras form a splitting of $\operatorname{Lie}(G_\alpha)$). {\color{black}In the first case when $w = e$, we set $w' = e$.} Furthermore, notice that $q_2$ and the terms appearing before $q_2$ all belong to $\Delta_2$, so we may combine them to reduce the expression to:
 
 \[u_n' * \dots * u_{k+1}' * (q_3)' * w' * v_k'' * g * q_4 * v_{k-1} * \dots * v_1 * \rho^0 \]
  for some collection of $u_i' \in N^{\alpha_i}$, and $g \in D$. But by Lemma \ref{lem:D-commutators}, $g$ may be pushed to the right preserving the form of the expression and being absorbed into $\rho^0$. We abusively do not change these terms and drop $g$ from the expression.
 
 Now, we do the final commutation by commuting $v_k''$ and $q_4$. Notice that $\set{\alpha_k} \cup \Delta_1 \cup \Delta_4$ is a stable subset. Therefore, we may write $v_k'' * q_4$ as $(q_4)' * v_k''' * q_1$ with $q_1 \in G_{\Delta _1}, (q_4)' \in G _{\Delta _4}$ and $v_k ''' \in N^{\alpha_k}$. Inserting this into the previous expression, we see that the $q_1$ term can be absorbed into the remaining product of the $v_i$ terms. This yields the desired form
 
  \[u_n' * \dots * u_{k+1}' * (q_3)' * w' * (q_4)' * v_k''' * (q_1 * v_{k-1} * \dots * v_1) * \rho^0, \]
 
  \noindent with the new $\rho^-$ equal to $(q_3)' * w' * (q_4)'$.
  \end{proof}

{\color{black} 
Recall that $D$ commutes with the $\R^k$-action. We wish to combine the Lie groups $D$ and $\R^k \times K$ to build a single Lie group to parameterize the neutral directions. We again consider the free product of groups $D * (\R^k \times K)$. Observe that $\R^k$ commutes with the elements of this group, and $\R^k$ has a dense orbit. Therefore, the cycles of this group at every point contain a co-finite dimensional normal subgroup (the cycle group at a point with a dense orbit). Let $\hat{D}$ denote the factor of $D * (\R^k \times K)$ by this group, and $\mc C_{\hat{D}}$ the cycles in the group $D *(\R^k \times K)$.
%Recall that the action of $D$ coincides with a subgroup the $\R^k \times M$ action, let $f : D \to \R^k \times M$ be the homomorphism which associates an element of $D$ with the corresponding element of the $\R^k \times M$ action. %If the groups depend on $\bar{x}$, recall the definition of $D^\Theta$ in Definition \ref{def:dalpha}. % $A(\bar{x})  = \set{a \in \R^k \times M : a \cdot Y_\Theta^\Omega(\bar{x})}^\circ$ denote the identity component of the subgroup of $\R^k \times M$ which preserves the fiber.
%Then if $d \in D$, then $d \in D^\Theta$, since $d$ is generated by the actions of $V_{\pm \alpha}^\Omega$, which by definition preserve the sets $Y_\Theta^\Omega(\bar{x})$. 

Let $\hat{G}$ be the quotient of $\hat{\mc P} := (\R^k \times K) \ltimes \mc P$ (see Definition \ref{def:P-groups}) by the {  smallest closed normal subgroup containing all elements from} $\ker(\mc P \to G)$ (ie, the relations already constructed in $G$) and elements of $\mc C_{\hat{D}}$.} %If there is no dependence on $\bar{x}$ (ie, we are in the setting of Theorem \ref{thm:pairwise-to-homo}), we use $\R^k \times M$ in place of $D^\Theta$.

{\color{black}
\begin{remark}
    Even in the case when the action satisfies \ref{ta9}\ref{(HA-7a)}, it is not necessarily true that the action of $\mathcal P$ is transitive without saturating by $K$-orbits. The situation is even worse for abstract Anosov actions satisfying \ref{ta9}\ref{(HA-7b)}, in which one must saturate by both $K$ and $\R^k$-orbits. Thus, to obtain a transitive action of a Lie group (and hence a homogeneous space structure), we must consider the groups $\hat{\mathcal P}$. Note also that including the subgroup $\mathbb{R}^k$ in the homogeneous space structure makes it clear that the $\R^k$ action is by translations in the homogenous coordinates.
\end{remark}
}

Recall that $\mc P$ is the free product of nilpotent Lie groups, and has a canonical CW-complex structure as described in Section \ref{sec:free-prods}. The cell structure can be seen by considering subcomplexes corresponding to sequences of coarse weights $\bar{\chi}= (\chi_1,\dots,\chi_n)$ and letting $C_{\bar{\chi}} = \set{ u_1 * \dots * u_n : u_i \in N^{\chi_i}} \cong N^{\chi_1} \times \dots \times N^{\chi_n}$. Then a neighborhood of the identity is a union of neighborhoods in each cell $C_{\bar{\chi}}$ containing 0.

%Let $D^\perp$ be a subspace of $\R^k$ such that $D^\perp \cap D = \emptyset$ and $D$ and $D^\perp$ generate the entire $\R^k$ action. Notice that since $D^\perp$ also normalizes each of the flows $\eta^\chi$, and also maps cycles to cycles, the group generated by the actions of $G$ and $D^\perp$ has the structure that any element can be written in the form $gd$ with $d \in D^\perp$. The main result of this section is the following  proposition. Let $\hat{G}$ be this group.

Let $\pi : \hat{\mc P} \to \hat{G}$ denote the canonical projection, and note that $\ker \pi$ is exactly the group generated by (conjugates of) commutator cycles, nontrivial symplectic cycles, cycles in $\mc C_{D}$, {\color{black}and cycles in $\mc C_{\hat{D}}$.} {\color{black}Let $\pi' : G_+ \times G_- \times \hat{D} \to \hat{G}$ denote the map $\pi'(g_1,g_2,d) = g_1g_2d$, where $G_+$, $G_-$ and $\hat{D}$ are all identified with their inclusions into $\hat{G}$. }%Furthermore, notice that putting $\Delta^+(a)$ and $\Delta^-(a)$ in a circular ordering makes $G_+ \times G_- \times {\color{olive}\hat{D}}$ a combinatorial cell in $\hat{\mc P}$.

\begin{proposition} \label{stable unstable cycle decomposition}
There exists an open neighborhood $U$ of $e \in \hat{\mc P}$ and a continuous, open map $\Phi : U \to G_+ \times G_- \times {\color{black}\hat{D}}$ such that $\pi \of \Phi = \pi'$.%if $\Phi(u) = \Phi(v)$, then $\pi(u) = \pi(v)$.%$u$ and $v$ represent the same element of $\hat{G}$. % (ie, any cycle $\rho$ with contractible homotopy type),% there exist $\rho _k \rightarrow \rho$ in $G$ such that 
%where $\rho ^{\pm} \in G _{\pm}$ and $ \rho^0 \in \R^k$ are unique.  Moreover,  $\rho  ^{\pm} $ and $\rho ^0 $ vary continuously on the open dense set and extend continuously to any $\rho \in \hat{G}$ such that $\eta (\rho) (x_0) = x_0$ on the universal cover . 
\end{proposition}

\begin{proof}
We describe the map $\Phi$, whose domain will become clear from the definition.  Let $\Delta^+(a_0) = \set{\alpha_1,\dots,\alpha_n}$ and $\Delta^-(a_0) = \set{\beta_1,\dots,\beta_m}$ be the coarse weights as described in the proof of Lemma \ref{open dense commutation}. Given a word $\rho = u_1 * \dots * u_n$, $u_i \in N^{\chi_i}$, $\chi_i \in \Delta$ for every $i$, % (resp. in $V_{\chi_i}$), 
we begin by taking all occurrences of $\alpha_n$ in $\rho$ and pushing them to the left, starting with the leftmost term.  When we commute it past another $\alpha_i$, we accumulate only other $\alpha_j$, $i+1 \le j \le n-1$, in $\rho^{\alpha_i,\alpha_n}$, which we may canonically present in increasing order on the right of the commutation. A similar statement holds for the commutation of $\alpha_n$ with $\beta_i$. We iterate this procedure as in the proof of Lemma \ref{open dense commutation} to obtain the desired presentation. Since the commutation operations involved are determined by the combinatorial type, the resulting presentation is continuous from the cell $C_{\bar{\chi}}$. 
{\color{black}
We now show that it descends to $\mc P$, which requires showing it respects the free product relations. If one of the terms  happens to be $e$, the procedure yields the same result whether it is considered there or not. Furthermore, if a coarse weight is repeated, then commuting past each will yield elements contained in a single stable which may be combined afterwards. Thus, it is a well-defined continuous map from a neighborhood of the identity in $\hat{\mc P}$ to $G_+ \times G_- \times {\color{black}\hat{D}}$.} It is continuous from $\hat{\mc P}$ because it is continuous from each $C_{\bar{\chi}}$.

Notice that in the application of Lemma \ref{open dense commutation}, we require that all terms are sufficiently small.  Thus, in each combinatorial pattern, since the algorithm is guaranteed to have a finite number of steps and swaps appearing, and each term appearing will depend continuously on the initial values of the terms, we know that for each $\bar{\chi}$, some neighborhood of 0 will be in the domain of $\Phi$, by the neighborhood structure described above.

Notice that the reduction of a word $u$ to a word of the form $u_+ * u_- * a \in G_+ \times G_- \times \hat{D}$ uses only relations defining $\hat{G}$. Therefore, if after the reductions, the same form is obtained, the original words must represent the same element of $\hat G$. That is, $\pi \of \Phi = \pi'$.
%The existence follows from Lemma \ref{open dense commutation}.  Uniqueness is argued by the dynamics: indeed,   if there was a second presentation with different  $\rho^-$ value, then moving them by $a _0 ^n$,  these would diverge while the other components contract or do not expand.  A similar argument gives uniqueness of $\rho^+$.  Uniqueness of the $\rho^0$ follows. Notice also that the map $(g_+,g_-,a) \mapsto g_+g_-a$ from $G_+ \times G_- \times \R^k$ locally surjective. In particular, by uniqueness, if a elements of the form $\rho^+ * \rho^- * \rho^0$ converge to a cycle, each individual term making up the product converges to a cycle.
 \end{proof}
 \begin{corollary}
\label{cor:hatG-Lie}
The group $\hat{G}$ is a Lie group.
\end{corollary}
\begin{proof}
Choose $U$ is as in Proposition \ref{stable unstable cycle decomposition}, and let $K \subset \Phi(U)$ be a compact neighborhood of $(e,e,e) \in G_+ \times G_- \times \hat{D}$. {  Note that such a neighborhood exists since $G_\pm$ and $\R^k \times K$ are all Lie groups using Lemma \ref{lem:stable-closed}.} Then $\Phi^{-1}(K)$ is a neighborhood of $e \in \hat{\mc P}$, and $\pi(\Phi^{-1}(K)) = \pi'(K)$ is a neighborhood of $e \in \hat{G}$. Therfore, $\hat{G}$ is locally compact. Furthermore, since $\hat{G}$ is the factor of the locally path-connected group $\hat{\mc P}$, $\hat{G}$ is locally path-connected. Hence $\hat{G}$ is a projective limit of Lie groups by Corollary \ref{cor:lc-lpc-structure}. Hence $\hat{G}$ has an associated sequence $G_n$ of connected Lie groups, factor maps $q_n : \hat{G} \to G_n$, and projections $p_n : G_n \to G_{n-1}$ such that $\ker p_n$ is compact, $q_n = p_{n+1} \of q_{n+1}$, and $\bigcap_{n=1}^\infty \ker q_n = \set{e}$ (see the diagram in Corollary \ref{cor:lc-lpc-structure}).

We first claim that there exists $N,d \in \N$ such that $\dim(G_n) = d$ for all $n \ge N$. Indeed, note that since each $p_n$ is surjective, $\dim(G_n) \ge \dim (G_{n-1})$, so it suffices to show that $\dim(G_n) \le d$ for all $n \in \N$ and some $d \in \N$. We have that $q_n(G_+)$, $q_n(G_-)$ and $q_n(\hat{D})$ are Lie subgroups of $G_n$, and since $q_n \of \pi \of \Phi : G_+ \times G_- \times \hat{D} \to G_n$ is an open map, we conclude that the map $(g_1,g_2,a) \mapsto g_1 \cdot g_2 \cdot a$ from $q_n(G_+) \times q_n(G_-) \times (\R^k \times K)$ is an open map. It follows that $\dim(G_n) \le \dim(q_n(G_+)) + \dim(q_n(G_-)) + \dim(q_n(\R^k \times K)) \le \dim(G_+) + \dim(G_-) + \dim(\hat{D})$, which is independent of $n$.

Since $\dim(G_n) = \dim(G_{n+1})$ for all $n \ge N$, $p_n$ is a local isomorphism for all $n \ge N+1$. It follows that the algebras $\Lie(G_n)$ are all isomorphic, and there exists a unique {\color{black}connected,} simply connected group $\tilde{G}$ such that $\Lie(\tilde{G}) \cong \Lie(G_n)$ for sufficiently large $n$. We may therefore construct local isomorphisms $f_n : \tilde{G} \to G_n$ inductively by defining $f_{n+1}$ to be the unique Lie group homomorphism with derivative $(dp_n)^{-1} \of df_n$. We therefore obtain the following commutative diagram:

% https://tikzcd.yichuanshen.de/#N4Igdg9gJgpgziAXAbVABwnAlgFyxMJZAJgBoAGAXVJADcBDAGwFcYkQAdDgC3p2ADiAXxBDS6TLnyEUZACzU6TVuy55GsQSLETseAkTLFFDFm0QgBAfUI6QGPdKIBmUsZqmVF68DABaAEZtcXtJfRlkOTcTZXNLK18-YmDdKQMUAOiPWPYfMABqINEQhzSI8iylM1UOAGMoCBwEO1LwogBWSs84rnrG5pKwpxQANi6c7ytyUUUYKABzeCJQADMAJwgAWyQKkBwIJE6qrxAARwSCoppGegAjGEYABSH0kDWsee4cYtWN7cRdvskGRjnFzoRrncHs9HK93p9vnZ1lsdjQgYhXKD2OdElcQDd7k8XjI3h8vj8QMj-oCDogoliLDj-MkKVSkJk9rSQd12CsbKy-uy0bSjjyLHzfIUUpTBYgOejMWLKRdAtK2XLhUh6UqJcy1bKxpzDkjZUd0SCCdDiexGDAVt9stULGgLlKBSjECCFY6Ti7bCF1Zj0dqJvYVUUTR76eiAOw+uIuxIsyPUzWIAAc8exU3d-3ltMzDOV0xTSDjRozWedVgCuaFFd29zAUDLHKVADEZkIgA
\begin{tikzcd}
                 &                              & \hat{G} \arrow[ldd, "q_{n+1}"'] \arrow[dd, "q_n"'] \arrow[rdd, "q_{n-1}"'] \arrow[rrdd, "q_{n-2}"] \arrow[rrrrdd, "q_0"]                                &                              &                              &                         &     \\
                 &                              &                                                                                                                                                         &                              &                              &                         &     \\
\cdots \arrow[r] & G_{n+1} \arrow[r, "p_{n+1}"] & G_n \arrow[r, "p_n"]                                                                                                                                    & G_{n-1} \arrow[r, "p_{n-1}"] & G_{n-2} \arrow[r, "p_{n-2}"] & \cdots \arrow[r, ] & G_0 \\
                 &                              &                                                                                                                                                         &                              &                              &                         &     \\
                 &                              & \tilde{G} \arrow[uu, "f_n"] \arrow[luu, "f_{n+1}"] \arrow[ruu, "f_{n-1}"] \arrow[rruu, "f_{n-2}"] \arrow[rrrruu, "f_0"] &                              &                              &                         &    
\end{tikzcd}

By the universal property of inverse limits, there exists a unique homomorphism $F : \tilde{G} \to \hat{G}$ such that $q_n \of F = f_n$. We claim that the image of $F$ is exactly the path component of $\hat{G}$. Indeed, if $\gamma : [0,1] \to \hat{G}$ is any path such that $\gamma(0) = e$, then $\gamma_n = q_n \of \gamma$ is a path in $G_n$, and $p_n \of \gamma_n = \gamma_{n-1}$. Since $f_n$ is a local isomorphism, there exists a unique $\tilde{\gamma}_n : [0,1] \to \tilde{G}$ such that $f_n \of \tilde{\gamma}_n = \gamma_n$. Since $p_n \of f_n = f_{n-1}$, the maps $\tilde{\gamma}_n$ all coincide, let $\tilde{\gamma} : [0,1] \to \tilde{G}$ denote the corresponding lift. {\color{black}Then by construction, 

\[ q_n \of F \of \tilde{\gamma} = f_n \of \tilde{\gamma} = \gamma_n.\]

Since $\gamma$ is determined by the family of paths $\gamma_n$, it follows that $F \of \tilde{\gamma} = \gamma$,} and the endpoint of $\gamma$ can be reached in the image of $F$.

Finally, since $\hat{G}$ is path connected, the path identity component is exactly $\hat{G}$, so $F$ is onto. Therefore, $\hat{G}$ is the continuous image of a Lie group, and therefore Lie.
%It follows that $\ker p_n$ is discrete and compact, hence finite for $n \ge N+1$. Therefore, $\ker q_n$ is a projective limit of finite groups, and is therefore totally disconnected for sufficiently large $n$. If the maps $p_n$ are not eventually isomorphisms, $\ker q_n$ is uncountable. Since the path component of $\hat{G}$ can only reach countably many elements of $\ker q_n$, it follows that $\ker q_n$ is countable, and $p_n$ is eventually an isomorphism.
%Therefore, $\hat{G} \cong G_{N_0}$ for some $N_0$, and $\hat{G}$ is a Lie group.
\end{proof}

\subsection{Description of the homogeneous spaces}

By Corollary \ref{cor:hatG-Lie}, $X$ is the homogeneous space of a Lie group $\hat{G}$, which is generated by subgroups which are images of  $N^\alpha$,  $\alpha \in \Delta$ and $\R^k \times K$. %Furthermore, by Proposition \ref{stable unstable cycle decomposition}, these subgroups provide coordinates, so $\Lie(\hat{G})$ (resp. $\Lie(\hat{G}(x))$) is the vector space factor of $V = \Lie(\R^k \times M) \oplus \bigoplus_{\alpha \in \Delta} \Lie(N_\alpha^\Omega)$ (replacing $N_\alpha^\Omega$ by $V_\alpha^\Omega$ when working on the fiber). Let $\phi : V \to \Lie(\hat{G})$ (resp. $\phi_x : V \to \Lie(\hat{G}(x))$) denote the linear map which is defined to be the unique extension of the following map on each component of $V$. If $v \in \Lie(N_\alpha^\Omega)$, $\exp(tv)$ is a one-parameter subgroup of $N_\alpha^\Omega$, which has a canonical inclusion in $\hat{G}$. Let $g_t$ denote the corresponding one-parameter subgroup of $\hat{G}$ and $\phi(v) = \frac{d}{dt}|_{t = 0} g_t \in \Lie(\hat{G})$. The definition on $\R^k \times M$ is identical and $\phi_x$ is defined analogously.
%If $a \in \R^k$, define $\psi_a : \Lie(\hat{G}_{x_0}) \to \Lie(\hat{G}_{x_0})$ defined by $\psi_a(v) = e^{\alpha(a)}v$ when $v$ is in the image of the Oseledets space with weight $\alpha$ under $\phi_{x_0}$, and $\psi_a$ on $\phi_{x_0}(\Lie(\R^k \times M))$ is the identity.
Furthermore, the group $\hat{G}$ is a factor of the group $\hat{\mc P} = (\R^k \times K) \ltimes \mc P$, where $\mc P$ is the free product of the groups $N^\alpha$, and let $\mc C$ denote the kernel of $\hat{\mc P} \to \hat{G}$. Therefore, $\mc C$ is the normal closure of the group generated by commutator relations $\rho^{\alpha,\beta}(u,v,x)$ (which do not depend $x$), symplectic relations, and identifications of the diagonal elements of $G_\alpha$ with the $\R^k$-action.

Let $S(x) = \Stab_{\hat{G}}(x)^\circ$, and notice that {$S(x)$} is the closed subgroup of a Lie group and therefore Lie.

%\begin{lemma}
%If $\bar{x}_0$ has a dense $\R^k$-orbit on $X$, $B(\bar{x}_0) \subset B(\bar{x})$ for all $\bar{x} \in X^\Omega$.
%\end{lemma}

%\begin{proof}
%Choose $a_n$ such that $a_nx_0 \to x$. If $\exp(tv)$ fixes $x_0$ for all $t \in \R$, $\exp(tv)$ fixes $a_nx_0$ for all $t \in \R$. Indeed:

%\[ \exp(tv)a_nx_0 =  a_n\exp((a_n)_*tv)x_0 = a_n\exp(tv)x_0 = a_nx_0. \]

%By continuity, it follows that $\exp(tv)x = x$.
%\end{proof}

%\begin{corollary}
%\label{cor:B-contained}
%The group $B(\bar{x}_0)$ does not depend on which point $x_0$ with dense orbit is chosen, and $B(\bar{x}_0)$ is a normal subgroup of $\R^k \times M$ and $\hat{G}$.
%\end{corollary}

%\begin{proof}
% If $g \in B(\bar{x}_0)$, and $h \in \hat{G}$, then $g \in B(h\bar{x}_0)$. Therefore, $h^{-1}gh \cdot x_0 = h^{-1}g \cdot hx_0 = h^{-1} \cdot hx_0 = x_0$, and $h^{-1}gh \in  B(\bar{x}_0)$.
%\end{proof}

%Let $G = \hat{G} / B(\bar{x}_0)$, and $S(\bar{x}) = S_1(\bar{x}) / B(\bar{x}_0)$.

\begin{lemma}
\label{lem:dimS-const}
$\dim(S(x))$ is independent of $x \in X$.
\end{lemma}

\begin{proof}
Since $\hat{G}$ acts transitively on $X$ by \ref{ta9}, it follows that $X$  is the homogeneous space of the Lie group $\hat{G}$. The result is now immediate since if $g \in \hat{G}$, $S(x)$ and $S(g \cdot x)$ are subgroups conjugated by $g$.
\end{proof}

Let $a_0 \in \R^k$ be a generic element (ie, an element such that $\alpha(a_0) \not= 0$ for all $\alpha \in \Delta$ and $c_i\alpha(a_0) \not= c_j\beta(a_0)$ for all functionals $c_i\alpha \in [\alpha]$, $c_j\beta \in [\beta]$.

\begin{lemma}
\label{lem:S-contains-oseledets}
If $x$ has a dense $\R^k$-orbit and is $a_0^{\pm 1}$-recurrent, and $S(x) \not= \set{e}$, then $\Lie(S(x))$ contains an element of an Oseledets subspace of $\Lie(N^\alpha)$ for some $\alpha \in \Delta$.
\end{lemma}

\begin{proof}
Suppose that $S(x) \not= \set{e}$, and consider $\Lie(S(x)) \subset \Lie(G)$. Notice that $x \mapsto \Lie(S(x))$ is semi-continuous in the following sense: if $x_n \to x$ and $\R v_n \subset \Lie(S(x_n))$, with $\norm{v_n} = 1$ and $v_n \to v$, then $\R v \subset \Lie(S(x))$. 

Now, simply notice that if $S(x) \not= \set{e}$ and ${a_0}^n \cdot x \to x$, then $S(x)$ contains its fastest Oseledets space (either forward or backward), or is contained in $(\R^k \times K)$. $S(x)$ is transverse to $\R^k \times K$ since the action of $\R^k \times K$ is locally free by the definition of a HAPHA. Therefore, $S(x)$ contains an Oseledets space if it is nontrivial.
\end{proof}

%\begin{lemma}
%If $S(\bar{x})$ contains an element of $\Lie(V_\beta^\Omega)$ for some $\beta \in \Delta$, then there exists a neighborhood $\mc U$ of $\bar{x}$ such that $S(\bar{y}) \not= \set{e}$ for all $\bar{y} \in \mc U$.
%\end{lemma}

%\begin{proof}
%It suffices to show that we can saturate $\bar{x}$ with coarse Lyapunov leaves and $\R^k \times M$-orbits and preserve the nontriviality of $S$. Motion along $\R^k \times M$ is clear since $a_*S(\bar{x}) = S(a\bar{x})$. The same is true for any element of $G$ (resp. $G(\bar{x})$) since the group will be conjugated. In the case of Theorem \ref{thm:pairwise-to-homo}, this finishes the proof. When working on the fiber, we also consider motion along the base (elements of $N_\alpha^\Omega$ which are not in $V_\alpha^\Omega$. In this case, if $v \in \exp(S(\bar{x}))$ belongs to some $V_\beta^\Omega$, and $u \in N_\alpha^\Omega$,

%\[u\bar{x} = uvu^{-1}u\bar{x} = \rho^{\alpha,\beta}(u,v,\bar{x}) v^{-1} u\bar{x} \]

%Now $v^{-1}$ and $\rho^{\alpha,\beta}(u,v,\bar{x})$ consist only of elements in $\set{V_\gamma}$, the $v^{-1}$ term cannot be cancelled. Therefore, since every such leg will converge to $e$ as $v \to e$, it follows that $\Stab_{\hat{G}(\bar{x})}(\bar{x})$ is not discrete and therefore $S(u \bar{x}) \not=\set{e}$.
%\end{proof}

%\begin{lemma}
%If  $\bar{x}$ is $a_0$-recurrent, and $S(\bar{x}) \not= \set{e}$, then $S(\bar{x})$ contains an element of $\Lie(V_\beta^\Omega)$ for some $\beta \in \Delta$.
%\end{lemma}

%\begin{proof}

%\end{proof}

%\begin{lemma}
%If $\bar{x}$ has a dense $\R^k$-orbit, then $S(\bar{x})= \set{e}$.
%\end{lemma}

\begin{corollary}
$S(x) = \set{e}$ for all $x \in X$.
\end{corollary}

\begin{proof}
If $S(x) \not= \set{e}$ for some point, it is nontrivial at every point by Lemma \ref{lem:dimS-const}. Since $a_0^{\pm 1}$-recurrence and dense $\R^k$-orbit are both residual properties, we may find some $x_1$ such that $x_1$ has a dense $\R^k$-orbit and for which $\Lie(S(x_1))$ contains an element of an Oseledets space by Lemma \ref{lem:S-contains-oseledets}. But since $a_*S(x_1) = S(a\cdot x_1)$ and Oseledets spaces are invariant under $a_*$, it follows that $\Lie(S(x))$ contains that element for all $x$. This contradicts \ref{ta4}, %since locally bi-Lipschitz \color{black}(No Lipchitz property now, should be replaced by 
the locally free property.%\color{black} implies locally free. \marginnote{\color{black}Explain the direction within Oseledets space is fixed by $a_\ast$.\color{black}}
\end{proof}

\begin{proof}[Proof of Theorem \ref{thm:pairwise-to-homo}]
We have just shown that  $S(x) = \set{e}$, $\Stab_G(x)$ is discrete. The group $\hat{G}$ is Lie by Corollary \ref{cor:hatG-Lie} and contains the $\R^k \times K$-action as a subgroup. The result follows.
\end{proof}

\subsection{Proof of Theorem \ref{thm:technical}} By Theorem \ref{thm:constant pairwise cycle structure} and Lemma \ref{lem:Galpha}, any action satisfying the assumptions of Theorem \ref{thm:technical} has pairwise constant cycle structure (recall Definition \ref{def:const-pairwise}). Therefore, Theorem \ref{thm:technical} is a consequence of Theorem \ref{thm:pairwise-to-homo}.

{ 

\section{Smooth Partially Hyperbolic  Actions}
\label{sec:smooth-top}

\color{black} In this section, we verify that for the smooth abelian actions of  $\R^k$  satsifying assumptions as in Theorem \ref{basic abelian} or \ref{abelian}, the principal bundle extensions constructed in Section \ref{fibration} are in fact genuinely higher-rank HAPHA actions with integral Lyapunov exponents and SRB measures. The bulk of the work needed for this is already done in Section \ref{fibration}, in particular  in Theorem \ref{thm:lifted-action}. Checking that these extensions have integral Lyapunov coefficients property (recall the paragraph following Definition \ref{def:rhohat}) is most involved and we do this in Section \ref{ilc}. Then Theorem \ref{thm:technical} gives a topological conjugacy to a homogeneous model. Verifying that the conjugating map is smooth is done in Section \ref{sec:regularity}.
\color{black}

%Given an action as  in Theorems \ref{basic abelian} or \ref{abelian} we first apply the construction in Section \ref{BPsubbundle} to obtain  the extended action on $\hat X$. Moreover, the construction in Section \ref{BPsubbundle} give the actions $N^\lambda$ which act continuously on $\hat X$, and have the same grading as the grading defined by Oseledets splitting for the base action, on the groups $...$. Therefore for the lifted action let $\Delta$ be the set of equivalence classes of linear functionals which are the coarse Lyapunov functionals for the base action.  

\begin{proposition}
\label{prop:smooth-is-top}
If $\R^k \curvearrowright X$ is a $C^\infty$ \color{black} action as in Theorems \ref{basic abelian} or \ref{abelian}\color{black}, then the \color{black}$\R^k\times K$\color{black}-action on $\hat X$ \color{black}(that is, the Brin-Pesin compact extension of the action constructed in Section \ref{BPsubbundle} %\ref{thm:lifted-action} 
where $K$ is the fiber group of the  extension) \color{black}is a genuinely higher-rank \color{black} \color{black} HAPHA with integral Lyapunov coefficients and SRB measures \color{black} satisfying \ref{ta9}(a) or \ref{ta9}(b), respectively. \color{black}
\end{proposition}

\begin{proof}

\color{black}Property \ref{ta1} follows from the construction of the Brin-Pesin bundle (Proposition \ref{prop:brin-pesin-construct}).

For the property \ref{ta4}, the continuity of the $N^\alpha$ actions follows from Theorem \ref{thm:lifted-action}, and the fact that $N^\alpha$ is $\Delta$-harnessed (recall Definition \ref{def:del-harnessed}) follows from Theorem \ref{thm:lifted-action} and Lemma \ref{lem:rich-automorphism}. 
The property \ref{ta5} is proved in  Theorem \ref{thm:lifted-action}.

%actions are   and \ref{ta5} are derived in Theorem \ref{thm:lifted-action} and Lemma \ref{lem:rich-automorphism}.

%The property \ref{ta10} follows from t
%The Property \ref{ta5}  comes from Theorem 10.12

Property \ref{ta8} follows from the total recurrence of the base $\R^k$-action (which is due to volume preservation) and Lemma \ref{lem:fiber-recurrence}. 

%(HA5) on the base is Proposition 5.4, then we need argument for the lifted action. 

%Property \ref{ta10} follows for the base action from Proposition \ref{prop:circ}.  %{\color{purple}(slightly more complicated, add short argument about projecting and lifting parameterizations).}

%The proof of Lemma 4.15 in \cite{Spatzier-Vinhage} is written for $\R^k$-actions, but works verbatim for $\R^k \times M$. Then, as in Lemma 4.17 of that paper, one may deduce density of periodic orbits when there is an invariant measure of full support. 
%Property \ref{ta10} follows from Proposition 

For property \ref{ta10}, since each $N^\alpha$ acts on $\hat{X}$ in a well-defined manner, the free product $\mc{P}$ also acts on $\hat{X}$ in a well-defined way. The injectivity of the restriction of the evaluation maps to $N^{\chi_1} \times \cdots \times N^{\chi_r}$ follows from the corresponding property %which we denote by $\hat W^\lambda$
on the base manifold $X$, as established in Proposition \ref{prop:circ}, together with the transversality of the (lifted) coarse Lyapunov foliations with the fibers. We denote by $\hat W^s_{a_1,\dots, a_m}(x)$  the image stated in \ref{ta10} passing through $x$.

%{\color{purple}(slightly more complicated, add short argument about projecting and lifting parameterizations).}

%USE LIFTING PROPERTY, APENDIX Property \ref{ta10} follows by essentially the same reason as the same property for totally Anosov actions in \cite[Lemma 5.11]{Spatzier-Vinhage}. For completeness we provide a sketch here: (to be completed, Lemma 5.2 in \cite{Spatzier-Vinhage} in our case is guaranteed by FA-1, and the same proof of Lemma 5.11 in  \cite{Spatzier-Vinhage} provide the injective parametrization of corresponding foliation on the base smooth manifold, then we apply the same proof of Theorem 9.15 to lift it to a parametrization of $W^s_{a_1,a_2,\dots, a_m}$. )

%and also about how to lift the stable holonomy. (It might be better to do it in the appendix?)

To verify Property \ref{ta3}, observe that any point $\hat{y} \in \hat W^s_{a_1,\dots,a_m}(\hat{x})$ can be connected to $\hat{x}$ through an $\{N^{\chi_1}, \dots, N^{\chi_r}\}$-path. By the definition of $\chi_i$ and the discussion of lifted foliations in Proposition \ref{prop:holder-holonomies}, each $a_j$ uniformly contracts all $N^{\chi_i}$-paths. Therefore, for any $\hat{y} \in W^s_{a_1,\dots,a_m}(\hat{x})$, we have $d(a_j^t \hat{x}, a_j^t \hat{y}) \to 0$ as $t \to \infty$. 

Conversely, let $\hat{x},\hat{y} \in \hat{X}$ cover $x,y \in X$, and assume that $d(a_j^t \hat{x}, a_j^t \hat{y}) \to 0$ for all $1 \leq j \leq m$. It follows immediately that $d(a_j^tx,a_j^ty) \to 0$, so $y \in W^{cs}(x)$. Since the $\R^k$-action is continuously Oseledets conformal (it satisfies property \ref{FP2} (c) in Section \ref{sec:intermission1}), it is isometric along $E^c$. In particular, $E^c$ is integrable (see \cite{B03}) and there exists a unique $z \in W^c(y) \cap W^s(x)$. But since the dynamics along $W^c$ is isometric and $d(a_j^tx,a_j^ty) \ge d(a_j^tz,a_j^ty) - d(a_j^tx,a_j^tz) \to d(z,y)$, it follows that $y = z$ and $y \in W^s(x)$.
Then by Proposition \ref{prop:holder-holonomies}, $\hat{y}$ lies in the lifted stable foliation of $X$. Since on the base manifold $W^s_{a_1,\dots,a_m}(\hat{x})$ is a product of $W^{\chi_i}$, the transversality of the lifted foliations implies that the lifted foliation of $W^s_{a_1,\dots,a_m}(\hat{x})$ is a product of $N^{\chi_i}$. This completes the verification of Property \ref{ta3}.

 Property \ref{ta9}(a) holds for actions as in Theorem \ref{abelian} as a direct consequence of the super accessibility assumption in Theorem \ref{abelian} and the construction of the compact extension in Section 10. Property \ref{ta9}(b) holds for actions in Theorem \ref{basic abelian} and it is a direct consequence of the local product structure of $s$, $u$, $\R^k\times K$ foliations for Anosov actions.

 \color{black} Property \ref{ta2} follows immediately from \ref{FP1}, which was assumed (resp. established in Section \ref{strong accessibility implies FA12}) for actions in Theorem \ref{basic abelian} (resp. Theorem \ref{abelian}).

 %Property \ref{ta2} for actions in Theorem \ref{abelian} follows from the strong accessibility assumption

For Property \ref{ta-srb}, we start by considering the $\R^k$-action invariant volume on $X$ (for both Theorem \ref{abelian} and \ref{basic abelian}). Absolute continuity of unstable (and stable) foliations for every partially hyperbolic element of our action is a classical result due to Brin and Pesin \cite{brin-pesin}. Then we apply exactly the same argument as in the proof of the \cite[Proposition 4.22]{Spatzier-Vinhage} and argue by induction on dimension, i.e. we show by induction that $W^\alpha$ is absolutely continuous in common unstable manifolds of
increasing dimension contained in $W^u_{a_0}(\supset W^\alpha)$, where $a_0$ is sufficiently close to $\ker \alpha$ such that $\alpha(a_0)>0$.

%For each coarse Lyapunov foliation $W^\alpha$ subfoliates a larger foliation  and it is the fast foliation inside the larger foliation $\mathcal F_1$ This improves on the regularity of $W^\alpha$ within the larger foliation and therefore implies  absolut continuity of $W^\alpha$ within $F_1$. is going to be the fast foliation inside the stable foliation $W_a^s$ of $a$. This improves on the regularity of $W^\alpha$ within $W_a^s$ and therefore implies that $W^\alpha$ is absolutely continuous with respect to the volume. 

\color{black} To construct $\mu$ that satisfies \ref{ta-srb}, for any Borel $A\subset \hat{X}$ take $$\mu(A):=\int_X\int_K \mathbf{1}_A(x,y) d\nu_x(y) d\mathrm{vol}(x),$$
where we denote by  $\nu_x$  the Haar measure on the fiber of the compact extension at $y$. The measure $\mu$ is obviously invariant under the lifted $\R^k$-action (defined in Theorem \ref{thm:lifted-action}) since the principal bundle structure is invariant.

Take an arbitrary measurable partition $\mathcal P_X$ of $\mathrm{vol}$ of $X$ which is subordinate to $W^\alpha$ (without loss of generality we can assume the diameter of each atom of $\mathcal P_X$ is much smaller then the size of local trivialization charts of the fiber bundle). It lifts naturally to a measurable partition $\mathcal P_{\hat X}$ of $\hat{X}$ which is subordinate to the lifted foliation $\hat W^\alpha$. For each atom $P$,  the local holonomy of $\hat{W}^\alpha$ induces a local product structure of $$\pi^{-1}(P)\cong P\times K.$$ We claim that for $\mathrm{vol}$-almost every atom $P$ of $\mathcal P_X$, the conditional measure of $\mu$ on $\pi^{-1}(P)$ has the product form with respect to this local product structure $$\mu|_{\pi^{-1}(P)}\cong \mathrm{vol}|_{P}\times \nu.$$ By the Avila-Viana invariance principle \cite[Corollary 4.3]{AV}, since the lifted $\R^k$-action %on the fiber
is isometric along fibers and therefore has  $0$-Lyapunov exponents everywhere along fibers, we know $\mu$ is $\hat W^\alpha$ holonomy-invariant. This implies the product structure: it is a consequece of a general fact that if a measure $\mu$ on a product space $X\times Y$ has the same conditional measure $d\mu_x(\cdot), \cdot \in Y$ for almost every $x\in X$, then $\mu$ is a product of $\mathrm{Proj}^X(\mu)$ with this invariant conditional measure $\mu_x$. In particular for $\mu|_{\pi^{-1}(P)}$-almost every local leaf $\hat W^\alpha(\hat x)$, 
\begin{equation}\label{eqn: sec 13 prod}
    \pi_\ast(\mu|_{\hat W^\alpha(\hat x)})=\mathrm{vol}|_{W^\alpha(x)} \text{ where } \pi(\hat x)=x.
\end{equation}
By our construction of $\hat{X}$ and $\hat W$ we know $\pi$ interwines the $N^\alpha$ action on each leaf $\hat W^\alpha(\hat x)$ with a $C^1$ $N^\alpha$ action along $W^\alpha(x)$, hence the absolute continuity of $\mathrm{vol}$ along $W^\alpha(x)$ combined with  \eqref{eqn: sec 13 prod}  implies \ref{ta-srb}.

%follows from the existence of SRB measures for Anosov actions and the smoothness of holonomies along fast foliations in an unstable. This is shown in \cite[Proposition 4.22]{Spatzier-Vinhage} for $\R^k$-actions, the proof works verbatim for $\R^k \times K$. (Gibbs?)

Since the proof of integral Lyapunov coefficients is more involved, we prove it in 
Lemma \ref{lem:brown} of the next subsection.
%{\color{olive} Should follow from Theorem \ref{thm:lifted-action} and Lemma 12.5 of \cite{Spatzier-Vinhage}.}
%\color{black}Complete the proof section
\color{black}
\end{proof}

\subsection{Integral Lyapunov coefficients} \label{ilc}
\color{black}

%Lifting and Projection of Actions by Nilpotent Groups

Recall that a {\it common stable foliation} is a foliation $W$ of $X$ whose leaves are given by $\bigcap_{k=1}^n W^s_{a_k}(x)$ for some collection $a_1,\dots,a_k$ of \color{black} partially hyperbolic elements of $\R^k$. Common stable foliations are H\"older foliations with smooth leaves (see \cite[Lemma 4.5]{Spatzier-Vinhage}).

Consider the compact extension $\hat{X}$ over $X$ constructed in Section~\ref{BPsubbundle}. For any common stable foliation $\bigcap_{k=1}^n W^s_{a_k}(x)$ of $X$, 
%the discussion in Appendix~\ref{app:brin-pesin} 
the proof of Proposition \ref{prop:smooth-is-top}
shows that it can be lifted to a topological foliation $\hat{W}^s_{(a_1,\dots,a_n)}$ on $\hat{X}$ satisfying~\ref{ta10} and \ref{ta3}. Moreover, the projection $\pi: \hat{X} \to X$, when restricted to any leaf $\hat{W}^s_{(a_1,\dots,a_n)}$ of $\hat{X}$, is a homeomorphism onto a common stable leaf $\bigcap_{k=1}^n W^s_{a_k}(x)$ of $X$. For any $x \in X$, choosing an arbitrary $\hat{x} \in \hat{X}$ such that $\pi(\hat{x}) = x$, there exists a unique leaf $\hat{W}^s_{(a_1,\dots,a_n)}(\hat{x})$ passing through $\hat{x}$ that projects to $\bigcap_{k=1}^n W^s_{a_k}(x)$ (see also Appendix \ref{app:brin-pesin}). 

By Proposition~\ref{prop:smooth-is-top}, for any $x \in X$ and any lift $\hat{x} \in \hat{X}$ of $x$, the group $N^\alpha$ acts freely on the leaf $\hat{W}^s_{(a_1,\dots,a_n)}(\hat{x})$ provided that $\alpha(a_k) < 0$ for all $k$. This action naturally induces (via $\pi$) a free action of $N^\alpha$ on $\bigcap_{k=1}^n W^s_{a_k}(x)$. However, since the choice of $\hat{x}$ is not unique, there is no canonical way to define an $N^\alpha$ action on $\bigcap_{k=1}^n W^s_{a_k}(x)$. In what follows, unless otherwise specified, we will implicitly choose an arbitrary $\hat{x}$ and define the corresponding $N^\alpha$ action on a common stable foliation (or the action by a subgroup of $N^\alpha$ on a subfoliation of a common stable foliation). Consequently, we can also discuss the geometric commutator of the $N^\alpha$ and $N^\beta$ actions on $\bigcap_{k=1}^n W^s_{a_k}(x)$, provided that $\alpha(a_k), \beta(a_k) < 0$ for all $k$. This commutator is well-defined once we fix an arbitrary lift $\hat{x}$.
\color{black}

\begin{lemma}
    \label{lem:local-transitivity}
    Consider a common stable foliation $ W$, which is the sum of coarse Lyapunov distributions $E_{\alpha_1},\dots,E_{\alpha_n}$. Let $(\beta_1,\dots,\beta_n)$ be an arbitrary ordering of $\{\alpha_1,\dots, \alpha_n\}$. 
    Then the map

    \[ \phi_x : \bigoplus_{i=1}^n E_{\beta_i} \to W(x)\subset X\]
 defined by
    \[ \phi_x(u_1,\dots,u_n) = \exp(u_1)\exp(u_2)\dots\exp(u_n)\cdot x\]
 is surjective from a neighborhood of 0 in $\bigoplus_{i=1}^n E_{\beta_i}$ onto a neighborhood of $x$ in $W(x)$. 
\end{lemma}
\color{black}Here, we identify $E_{\beta_i}$ with $\text{Lie}(N^{\beta_i})$, and $\exp$ denotes the exponential map from $\text{Lie}(N^{\beta_i})$ to $N^{\beta_i}$. Although the actions of the nilpotent groups on $W(x)$ depend on the choice of the lifting $\hat{x}$ of $x$, and the identification between $E_{\beta_i}$ and $\text{Lie}(N^{\beta_i})$ is also non-canonical, Lemma~\ref{lem:local-transitivity} holds for any such identification and any lifting. This result follows directly from \cite[Lemma 3.2]{Schmidt-thesis}, which builds on ideas from \cite{kryszewsi-plaskacz}, particularly Corollary 4.5 in \cite{kryszewsi-plaskacz}.
%Here we identify $E_{\beta_i}$ as $Lie(N^{\beta_i})$ $\exp$ is just the exponential map from $Lie(N^{\beta_i})$ to $N^{\beta_i}$. Although the action on $W(x)$ by $N^\alpha$ depends on the choice of the lifting $\hat x$ of $x$, and the identification between $E_{\beta_i}$ and $Lie(N^{\beta_i})$ also no canonical choice, but for any identification and any lifting Lemma \ref{lem:local-transitivity} always holds and follows from \cite[Lemma 3.2]{Schmidt-thesis}, which uses ideas from \cite{kryszewsi-plaskacz}, particularly Corollary 4.5. 
\color{black} We do not offer a complete proof, but summarize the strategy: note that if the actions of the nilpotent groups were smooth, this would be true from the inverse function theorem. One may then use smooth approximations and show that in the limit the original action is recovered, and the onto property persists by topological degree arguments.

\color{black}

% Denote by $\Omega$ the set of Lyapunov functionals (notice that $\Omega$ is not $\Delta$ here). Given $\alpha,\beta \in \Delta$, $\Sigma(\alpha,\beta)$ is the set of coarse Lyapunov exponents which can be written as positive linear combinations of $\alpha$ and $\beta$. Let $\Omega(\alpha,\beta) = \cup_{\gamma \in \Sigma(\alpha,\beta)} [\gamma]$. That is, $\Omega(\alpha,\beta)$ is a set of Lyapunov functionals between $\alpha$ and $\beta$. 
\begin{lemma}\label{lem: translation of vanishing comm to integrability}
    Let $\R^k \curvearrowright X$ be as in Theorem \ref{basic abelian} or \ref{abelian} and  \color{black}$\R^k\times K\curvearrowright \hat X$ be the action constructed in Section \ref{BPsubbundle} on the compact extension $\hat X$.
    %Let   $\Delta$ denote the set of coarse Lyapunov functionals. 
    If for any non-proportional $\alpha, \beta\in \Delta$, any $c_i^\alpha \alpha, c_j^\beta \beta\in \Omega$, we have that $$E_{|c_i^\alpha\alpha, c_j^\beta\beta|}:=\bigoplus_{c_k^\alpha\geq c_i^\alpha} E^{c_k^\alpha\alpha} \oplus \bigoplus_{c_l^\beta\geq c_j^\beta}E^{c_l^\beta\beta} \oplus \bigoplus_{\gamma\in \Sigma(\alpha,\beta)}\left(\bigoplus_{\substack{c_m^\gamma\gamma=\sigma c^\alpha_i\alpha+\tau c^\beta_j\beta \\ \sigma\geq  1\text{ or }\tau\geq 1}} E^{c_m^\gamma\gamma}\right)$$ on $TX$ is topologically integrable, then the action $\R^k\times K\curvearrowright \hat X$ has integral Lyapunov coefficients.
    
    %$\alpha,\beta \in \Delta$, and $\Sigma(\alpha,\beta) = \set{\gamma_1,\dots,\gamma_n}$, then the function $\hat{\rho}^{c_i^\alpha\alpha,c_j^\beta\beta}_{c_{k_0}^{\gamma_0}\gamma_0}$ of Definition \ref{def:rhohat} defined for $\R^k\times K\curvearrowright \hat X$ is vanishing if and only if the distribution \color{black}(Here, the notation $\ominus$ indicates that $E^{c_{k_0} \gamma_0}$ is excluded from the direct sum in the formulation of the distribution. WE NEED TO REPLACE $E_\alpha$ BY THE SUM OF DISTRIBUTION FASTER THAN $E^{c_{i-1}^\alpha\alpha}$)\color{black}
\end{lemma}
\begin{proof} For $\gamma\in \Sigma(\alpha,\beta)$ denote $A'_\gamma =\{c_m^\gamma\gamma : c_m^\gamma\gamma=\sigma c^\alpha_i\alpha+\tau c^\beta_j\beta \mbox{ with } \sigma\geq  1\text{ or }\tau\geq 1\}$. 

Recall that for any $x\in X, u\in E^{c_i^\alpha \alpha}(x), v\in E^{c_j^\beta\beta}(x)$ we could define the geometric commutator of $u,v$ within the common stable leaf $W^{|\alpha,\beta|}(x):=\cap_{\alpha(a),\beta(a)<0} W^s_{a}(x)$ in Definition \ref{def:geo-comm}, which may depend on the choice of the lift $\hat x$ of $x$. By the integrability of $E_{|c_i^\alpha\alpha, c_j^\beta\beta|}$ and Lemma \ref{lem:residual-recurrence}, its integral foliation $W^{|c_i^\alpha\alpha, c_j^\beta\beta|}$ can be lifted to a topological foliation $\hat W^{|c_i^\alpha\alpha, c_j^\beta\beta|}$ which is closed under the action by 
$$\hat N^{c_i^\alpha\alpha}:= \exp\left(\bigoplus_{c_k^\alpha\geq c_i^\alpha}E^{c_k^\alpha\alpha}\right),\quad  \hat N^{c_j^\beta\beta}:=\exp\left(\bigoplus_{c_l^\beta\geq c_j^\beta}E^{c_l^\beta\beta}\right)$$
and
$$\hat N^{\gamma, c_i^\alpha\alpha, c_i^\beta\beta }:=\exp\left(\bigoplus_{A'_\gamma} E^{c_m^\gamma\gamma}\right)$$
for $\gamma\in \Sigma(\alpha,\beta)$. Here we identify these distributions as subalgebras of $N^\alpha, N^\beta$ and $N^{(\alpha,\beta)}$ respectively. 

{\color{black}We claim that if $\Sigma(\alpha,\beta) = \set{\gamma_1,\dots,\gamma_\ell}$ is a circular ordering, it suffices to show that the map 

\[\phi_{\hat{x}} : \hat{N}^{c_j\beta} \times \prod_{i=\#\Sigma(\alpha,\beta)}^\ell \hat{N}^{\gamma_i,c_i\alpha,c_j\beta} \times \hat{N}^{c_i\alpha} \to W^{|c_i^\alpha\alpha,c_j^\beta\beta|}(\hat{x})\]

defined by $\phi(u,v_1,\dots,v_\ell,w) = \exp(u)\exp(v_1)\dots \exp(v_\ell)\exp(w)\hat{x}$ is a local homeomorphism at 0 for every $\hat{x}$. Indeed, if this is the case, then following Definitions \ref{def:geo-comm} and \ref{def:rhohat} for commutators between $\hat{N}^{c_i^\alpha\alpha} \subset N^\alpha$ and $\hat{N}^{c_j^\beta\beta} \subset N^\beta$, we observe that the terms $\rho^\gamma_{\alpha,\beta}(\exp(u),\exp(w),\hat{x})$ must take values in the groups $\hat{N}^{\gamma,c_i\alpha,c_j\beta}\subset N^\gamma$, since the geometric commutators are unique. That is, the action must have integral Lyapunov coefficients.

So we must show that $\phi$ is a (local) homeomorphism onto its image. This follows the same proof scheme as Proposition \ref{prop:circ}, where the arguments take place within $W^{|c_i^\alpha\alpha,c_j^\beta\beta|}(x)$ on the base manifold, and lifting the uniqueness of the path to the Brin-Pesin fibration.}\end{proof}

%Notice that the integrability of the distribution $$E_{(c_i^\alpha\alpha, c_j^\beta\beta)}:=\bigoplus_{\gamma\in \Sigma(\alpha,\beta)}\left(\bigoplus_{A'} E^{c_m^\gamma\gamma}\right)$$follows the integrability of $E_{|c_i^\alpha\alpha, c_j^\beta\beta|}$, since it is the intersection of two integrable distributions $E_{|c_i^\alpha\alpha, c_j^\beta\beta|}$ and $E_{(\alpha,\beta)}$. We denote by $W^{|c_i^\alpha\alpha, c_j^\beta\beta|}$ the integral foliation of $E_{|c_i^\alpha\alpha, c_j^\beta\beta|}$. As before, $W^{|c_i^\alpha\alpha, c_j^\beta\beta|}$ lifts to a topological foliation of $\hat W^{|c_i^\alpha\alpha, c_j^\beta\beta|}$ which is closed under the action by $\hat N^{\gamma, c_i^\alpha\alpha, c_j^\beta\beta }$.

%Therefore for any $u\in E^{c_i^\alpha\alpha}, v\in E^{c_j^\beta\beta},$ %\rho^{\alpha,\beta}_\gamma(u,v)\in\hat$, since $$\exp(u)\in \hat N^{c_i^\alpha\alpha} \mbox{ and } \exp (v)\in \hat N^{c_j^\beta\beta}, $$ by the definition of the function $\rho^{\alpha,\beta}_\gamma$ we have $\rho^{\alpha,\beta}_\gamma(u,v)$ is in $\hat N^{\gamma, c_i^\alpha\alpha, c_j^\beta\beta }$. Here we use the integrability of $\hat W^{|c_i^\alpha\alpha, c_j^\beta\beta|}$ and the property that it is closed under the action by $\hat N^{\gamma, c_i^\alpha\alpha, c_j^\beta\beta }$. Then the claim of the lemma
%Lemma \ref{lem: translation of vanishing comm to integrability} follows from the definition of the integral Lyapunov coefficients condition.

\begin{lemma}
\label{lem:brown}
    If \color{black} $\R^k \curvearrowright X$  is a $C^2$ action as in Theorem \ref{basic abelian} or Theorem \ref{abelian}, %then the compact extension  
    then \color{black}$\R^k\times K$-action on $\hat X$ constructed in Section \ref{BPsubbundle} has integral Lyapunov coefficients. \color{black}
\end{lemma}

\color{black}

\begin{proof}%Before proceeding with the proof, we recall notations developed in Section \ref{sec:top-anosov}. Each coarse Lyapunov exponent is an equivalence class of functionals $[\alpha]$ defined up to positive scalar multiple, and we choose $\alpha$ as the smallest representative.
\color{black}
%Recall from Definition \ref{def:rhohat} that if $c_m^\gamma\gamma = \sigma  c_i^\alpha \alpha + \tau c_j^\beta \beta$ for some $\sigma,\tau > 0$, we call $\sigma$ and $\tau$ the {\it Lyapunov coefficients} of $c_m^\gamma \gamma$ with respect to $c_i^\alpha \alpha$ and $c_j^\beta \beta$. An action has \textit {integral Lyapunov coefficients} if $\hat{\rho}^{c_i^\alpha\alpha,c_j^\beta\beta}_{c_m^\gamma\gamma} \equiv 0$ whenever both Lyapunov coefficients are less than 1. 
By Lemma \ref{lem: translation of vanishing comm to integrability} it suffices to show on $X$ that for any non-proportional  $\alpha, \beta\in \Delta$, and any $c_i^\alpha \alpha$ in the class of $\alpha$, any $c_j^\beta \beta$ in the class of $\beta$, 
the distrbution $E_{|c_i^\alpha\alpha, c_j^\beta\beta|}$ defined in Lemma \ref{lem: translation of vanishing comm to integrability} of $TX$ is topologically integrable.  

%If the compact extention (constructed in Section \ref{BPsubbundle}) of an action as in Theorem \ref{basic abelian} or Theorem \ref{abelian}, has a non-integral Lyapunov coefficient, this means that for some Lyapunov coefficients $\hat{\rho}^{c_i^\alpha\alpha,c_j^\beta\beta}_{c_m^\gamma\gamma}$ is not identically $0$. So there exists a non-trivial element of $N^\gamma$ which projects to the non-trivial element on the Oseledets space for the induced action on $\hat X$. This produces a non-integral Lyapunov coefficient for the base action. 
%\color{black}(This paragraph we may need to rewrite according to Lemma 14.3.)

%In the rest of the proof we show that the base action, i.e. an action as in Theorem \ref{basic abelian} or Theorem \ref{abelian} must have integral Lyapunov coefficients. 

%(this paragraph can be moved to some other section)
\color{black}

 Fix Lyapunov functionals $c_i^\alpha\alpha \in [\alpha]$ and $c_j^\beta\beta \in [\beta]$. Recall the definition \ref{def:  Omega, Omega_al be} of $\Omega(\alpha,\beta)$. Then we divide \color{black}$$%\Omega|\alpha,\beta|:=
 \Omega(\alpha,\beta)\cup \{c_k^\alpha\alpha, k=1,\dots\}\cup \{c_{\ell}^\beta\beta,\ell=1,\dots\} = A \cup B,$$\color{black} where $A = \set{ \sigma \,c_i^\alpha\alpha + \tau \, c_j^\beta \beta : \sigma \ge 1} \cap \Omega$ and $B = \set{ \sigma \, c_i^\alpha \alpha + \tau\, c_j^\beta\beta : \sigma < 1} \cap \Omega$. Let \[B' = \set{ \sigma\, c_i^\alpha\alpha + \tau \, c_j^\beta\beta :  \sigma< 1,\tau \ge 1} \cap \Omega \subset B.\] Give the coarse Lyapunov exponents in \color{black}$\Sigma(\alpha,\beta)\cup \{\alpha,\beta\}$ \color{black} a circular ordering  $\set{\alpha = \chi_1,\dots,\chi_n = \beta}$. We will show that if $c_p^{\chi_k}\chi_k,c_q^{\chi_\ell}\chi_\ell \in A \cup B'$, then $[c_p^{\chi_k}\chi_k,c_q^{\chi_\ell}\chi_\ell] \subset A \cup B'$ by induction on $\abs{k-\ell}$. Then the result follows since $\alpha = \chi_1$ and $\beta = \chi_n$.

The base case is trivial: if $\abs{k-\ell} = 1$, then $\chi_k$ and $\chi_\ell$ are adjacent in the circular ordering, and therefore $N^{\chi_k}$ and $N^{\chi_\ell}$ commute. We now try to commute $c_p^{\chi_k}\chi_k,c_q^{\chi_\ell}\chi_\ell \in A \cup B'$, $\abs{k-\ell} > 1$. We break into cases based on whether each weight lies in $A$ or $B'$. The easiest occurs  when both belong to $A$ or both belong to $B'$, so we cover these first.

Recall that $W= W^{\abs{\alpha,\beta}}$ is the foliation whose leaves are tangent to $\bigoplus_{k=1}^n E_{\chi_k}$, which is a common stable manifold. Notice that choosing $a \in \ker \beta$, and perturbing by a very small amount will yield an element $a'$ close to $a$ for which $\left(\bigoplus_{\chi \in A} E^\chi \right) \oplus \left(\bigoplus_{\chi \in B} E^\chi \right)$ is  a dominated splitting of $TW$. %(cf. slow foliations, Section \ref{sec:holonomy-action}).
Therefore, $\bigoplus_{c_p^{\chi_k}\chi_k \in A} E^{c_p^{\chi_k}\chi_k}$ is tangent to a foliation $W^A$. 

Assuming we have fixed $c_i^\alpha\alpha$ and $c_j^\beta \beta$, if $\chi \in \Sigma(\alpha,\beta)$, let $\hat{N}^\chi$ be the nilpotent group tangent to the subalgebra 

\[\hat{\mf n}_\chi = \bigoplus_{c_p^\chi\chi \in A} E^{c_p^\chi\chi}. \]

This is a (possibly trivial) subalgebra since each $a \in \R^k$ acts by an automorphism of $N^\chi$, so by standard Lie theory, $[E^{c_p^\chi\chi},E^{c_q^\chi\chi}] \subset E^{(c_p^\chi + c_q^\chi)\chi}$.
Notice that each leaf $ W^A$ has local $C^0$ surjections given by the maps from $\hat{N}^{\chi_1} \times \hat{N}^{\chi_2} \times \cdots\times \hat{N}^{\chi_n} \to  W^A(x)$ by

\begin{equation}
\label{eq:local-coords-A}(u_1,\dots,u_n) \mapsto u_1 \cdot u_2 \cdot \dots \cdot u_n\cdot x.
\end{equation}

The proof of this claim uses Lemma \ref{lem:local-transitivity}. % and \ref{lem:presentation-uniqueness}. Lemma \ref{lem:local-transitivity} shows surjectivity and the arguments of Lemma \ref{lem:presentation-uniqueness} injectivity of this map.
See also Lemma 5.11 of \cite{Spatzier-Vinhage}.  Essentially it follows from transversality of the distributions, the difficulty being that they are only H\"{o}lder.  Indeed, these show that \eqref{eq:local-coords-A} gives coordinates in a neighborhood near 0 and $x$, then use the expanding dynamics to show they are global coordinates.  In particular, if both $c_p^{\chi_k}\chi_k$ and $c_q^{\chi_\ell}\chi_\ell \in A$, $[c_p^{\chi_k}\chi_k,c_q^{\chi_\ell}\chi_\ell] \subset A$. %{\color{teal} a little bit confused, why we need injectivity for $[A,A]\subset A$?} 
Similarly, if both $c_p^{\chi_k}\chi_k$ and $c_q^{\chi_\ell}\chi_\ell$ belong to $B'$, then they both belong to $C = \set{u\alpha + v\beta : v \ge 1} \cap \Omega$. By identical arguments, choosing a perturbation of $b \in \ker \alpha$, gives $[B',B'] \subset [C,C] \subset C \subset A \cup B'$.

We now consider the case when $c_p^{\chi_k}\chi_k \in B'$ and $c_q^{\chi_\ell}\chi_\ell \in A$ with $\abs{k-\ell} > 1$  (the case when $c_p^{\chi_k}\chi_k \in A$ and $c_q^{\chi_\ell}\chi_\ell \in B'$ follows from a symmetric argument). Let $x \in X$, $u \in E^{c_p^{\chi_k}\chi_k} \subset \Lie(N^{\chi_k})$ and $v \in E^{c_q^{\chi_\ell}\chi_\ell} \subset \Lie(N^{\chi_\ell})$. We construct points related to a geometric commutator %(\color{black}FIXXX, say it also well-defined within stable manifolds of base manifolds\color{black}) 
in the following way

\[ \begin{array}{ccc}
y = e(u) \cdot x & \qquad & x' = e(v) \cdot x \\
y' = e(v) \cdot y & & w = e(u) \cdot x'
\end{array}\]

%$y = t^{(\chi_i)} \cdot x$, $x' = s^{(\chi_j)} \cdot x$, $y' = s^{(\chi_j)} \cdot y$ and $w = t^{(\chi_i)} \cdot x'$. 
Notice that $y' = [e(-v),e(-u)] \cdot w$ (see Figure \ref{fig:sub-res}).

Let $\Sigma_{k\ell} := \set{\chi_{k+1},\dots,\chi_{\ell-1}}$ be the set of coarse functionals strictly between $\chi_k$ and $\chi_\ell$. Each coarse functional of $\Sigma_{k\ell}$ splits into Lyapunov functional, let $\Omega_{k\ell}$ denote the set of Lyapunov functionals proportional to a coarse exponent of $\Sigma_{k\ell}$. Now write $\Omega_{k\ell}$ as 

\[ \Omega_{k\ell} = \set{\gamma_1,\dots,\gamma_{m_1},\delta_1,\dots,\delta_{m_2},\epsilon_1,\dots,\epsilon_{m_3}},\]
where $\set{\gamma_\bullet}$, $\set{\delta_\bullet}$ and $\set{\epsilon_\bullet}$ are the exponents of $A\cap \Omega_{k\ell}$, $(B \setminus B') \cap \Omega_{k\ell}$ and $B'\cap \Omega_{k\ell}$, respectively, with each subset listed in an order to be clarified later.%according to the following lexicographical order:

%\[ \delta_r \prec \delta_s \mbox{ if and only if } \delta_s = c \delta_r, c > 1 \mbox{ or } [\delta_r] \mbox{ precedes } [\delta_s] \mbox{ in the circular ordering on $\Sigma(\alpha,\beta)$.}\]
 We assume $y$ and $w$ are sufficiently close, to be determined later (if we show it for sufficiently small $u,v$, we may use the dynamics of $\R^k\times K \curvearrowright \hat X$ projects to $\R^k\curvearrowright X$ and Theorem \ref{thm:lifted-action} to conclude it for arbitrary $u,v$). %(\color{black}This sentence is not quite clear to readers.\color{black})

Notice that the distribution $\bigoplus_{s=k+1}^{\ell-1} E_{\chi_s}$ is uniquely integrable to a foliation $W^{k\ell}$ with $C^2$ leaves since it is the intersection of stable manifolds for the action. Since $y' = [e(-v),e(-u)] \cdot w$, $y' \in W^{k\ell}(w)$ by Lemma \ref{lem:geo-comm}. Therefore, by Lemma \ref{lem:local-transitivity} applied to the splitting $T\mc W^{k\ell}$ into the bundles $E^{\gamma_s}$, $E^{\delta_s}$ and $E^{\epsilon_s}$, there exists a path moving from $w$ to $y'$ which first moves along curves of the form $\exp(w)$, where \color{black}$w \in E^{\epsilon_s}$\color{black}, to arrive at a point $p$. Then similarly along exponential images of $w \in E^{\delta_s}$ to arrive at a point $q$ from $p$. Finally, we move along the exponential images of vectors in \color{black}$E^{\gamma_s}$\color{black}\,  to connect $q$ to $y'$. \color{black} In this way, $p$ is obtained from $w$ after moving along curves tangent to $B'$, and $q$ is obtained from $p$ by moving along curves tangent to $B \setminus B'$. Then $q$ is also connected to $y'$ via curves tangent to $A$. See Figure \ref{fig:sub-res}.

{\color{black}
\begin{figure}[!ht]
\begin{center}
\begin{tikzpicture}[scale=.75]
\node [below left] at (0,0) {$x$};
\draw [thick] (0,0) -- (12,0);
\node [below] at (6,0) {\tiny $e(u)$};
\node [below right] at (12,0) {$y$};
\draw [thick] (0,0) -- (0,8);
\node [left] at (0,4) {\tiny $e(v)$};
\node [above left] at (0,8) {$x'$};
\draw (12,0) to [out=90,in=240] (14,7);
\node [right] at (12.5,4) {\tiny $e(v)$};
\draw (0,8) to [out=0, in=220] (11,10);
\node [above] at (5.5,8) {\tiny $e(u)$};
\node [above left] at (11,10) {$w$}; 
\draw [red] (11,10) -- (13,10);
\node [above right] at (13,10) {$p$};
\node [below] at (12,10) {\tiny $B'$};
\draw [blue] (13,10) -- (14,9);
\node [right] at (14,9) {$q$};
\node [below right] at (14,7) {$y'$};
\draw [green] (14,9) -- (14,7);
\node [rotate=-45,below] at (13.5,9.5) {\tiny $B \setminus B'$};
\node [left] at (14,8) {\tiny $A$};
\draw[fill] (0,0) circle [radius=0.075];
\draw[fill] (12,0) circle [radius=0.075];
\draw[fill] (0,8) circle [radius=0.075];
\draw[fill] (14,9) circle [radius=0.075];
\draw[fill] (14,7) circle [radius=0.075];
\draw[fill] (11,10) circle [radius=0.075];
\draw[fill] (13,10) circle [radius=0.075];
\end{tikzpicture}
\caption{A geometric commutator}
\label{fig:sub-res}
\end{center}
\end{figure}
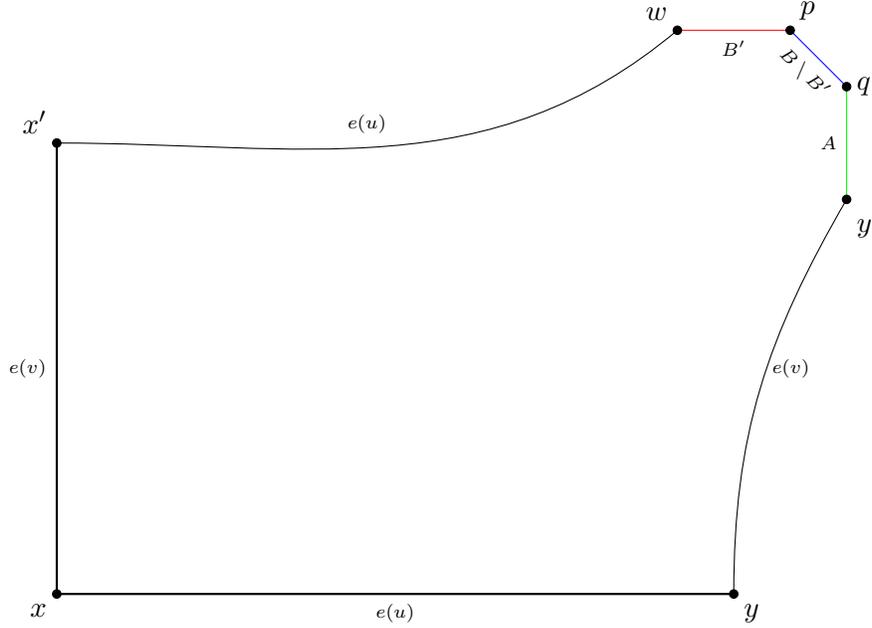
}

%Recall that $\abs{\alpha,\beta}$ denotes the set of coarse exponents between $\alpha$ and $\beta$, including $\alpha$ and $\beta$, and let \color{black} $ W^{\abs{\alpha,\beta}}$\color{black} \, denote the foliation whose leaves are tangent to the sum of these distributions (
Since $W^{\abs{\alpha,\beta}}$  is a common stable manifold, it is a H\"older foliation with $C^2$ leaves. Choose any pair of $C^{2}$ discs $D_1 \ni x,y$, $D_2 \ni x',q$ \color{black} along $W^B$\color{black}, of dimension $\sum_{\omega \in B} \dim(E^\omega)$ transverse to $W^A$ inside of $ W^{\abs{\alpha,\beta}}(x)$. This is possible since $x$ and $y$ are connected via $c_p^{\chi_k}\chi_k \in B'$ and $x'$ and $q$ are connected via only curves tangent to coarse Lyapunov foliations corresponding to exponents in $B$. \color{black} The next part of the argument crucially uses the uniform transversality of $W^B$ and $W^A$ within $W^{\abs{\alpha,\beta}}(x)$. \color{black} Therefore, $x'$ and $q$ are the images of $x$ and $y$ under the $W^A$ holonomy from $D_1$ to $D_2$. \color{black} In  \cite[Section 8.3, Lemma 8.3.1] {ledrappier_young1985} (similar results were obtained in \cite[Theorem 2.2]{Brown} and \cite[Appendix]{BPS}) under a uniform transversality condition on $D_1$ and $D_2$, %\color{black} ADD REMARK ABOUT $C^{1+}$ vs $C^2$....\color{black} 
there exist bi-Lipschitz coordinates for which the leaves of the foliation $\mc W^A$ are parallel Euclidean hyperspaces. In particular, the holonomy along $W^A$ is uniformly Lipschitz independent of the choice of $D_1$ and $D_2$ (due to the uniform transversality conditions),\color{black}\,  so $d(x,y)/d(x',q)$ is bounded above and below by a constant. \color{black}

We claim that $p = q$ (ie, that no weight of $(B \setminus B') \cap \Omega_{kl}$ appear). Roughly, the reason is that such weights contract too slowly. Indeed, pick an element $a'\in \ker \alpha$ such that $\beta(a') = -1$. We may perturb $a'$ to an element $a$ which is regular and such that $\alpha(a) < 0$, and such that if $\delta \in B\setminus B'$, $ \chi_k(a) < \delta(a) < 0$. This is possible because if $\delta = \sigma \, \alpha + \tau \, \beta \in B \setminus B'$, then $\tau< 1$, so $\delta(a') = -\tau > -1 = \beta(a') \ge c_p^{\chi_k}\chi_k(a')$, since when $c_p^{\chi_k}\chi_k \in B'$, the $\beta$ coefficient is at least 1. This is clearly an open condition for each $\delta$, so we may choose $a$ as indicated. We may also assume, {  by rescaling $a$ as necessary}, that $\beta(a) = -1$.

Since $y = \exp(u)x$, we can estimate distance between iterates of $x$ and $y$ using the intertwining property \ref{ta5}.
\color{black}
%the H\"older metric along the coarse Lyapunov leaf $W^{\chi_i}(x)$. 
Recall that since $c_p^{\chi_k}\chi_k \in B'$, $c_p^{\chi_k}\chi_k(a) < \beta(a) = -1$. Therefore $d_{W^{\chi_k}}(a^t\cdot x, a^t \cdot y) = e^{tc_p^{\chi_k}\chi_k(a)}d_{W^{\chi_k}}(x,y) < e^{-t}d_{W^{\chi_k}}(x,y)$ using the partially H\"older 
metric along the leaves of $W^{\chi_k}$. Now, suppose $p \not= q$. Recall that $p$ and $q$ are connected by legs in $B \setminus B' = \set{\delta_1,\dots,\delta_m}$, so that there exist $p = x_0,x_1,\dots,x_m = q$ such that $x_s$ $x_{s-1}$ are connected by a short curve everywhere tangent to $E^{\delta_s}$. Since the distributions $E^{\delta_s}$ are transverse, if $p \not= q$, there exists some $s$ for which $x_s \not= x_{s-1}$. Without loss of generality, we assume that $\set{\delta_s}$ are ordered such that $0 > \delta_1(a) > \delta_2(a) > \dots > \delta_m(a) > \chi_i(a)$. Then let $s_0$ be the minimal $s$ for which $x_s \not= x_{s-1}$, and $c_1 = \delta_{s_0}(a)$, $c_2 = \delta_{s_0+1}(a)$. Notice that $0 > c_1  > c_2 > -1$. By minimality, we get $x_{s_0-1} = p$.

Let $d$ denote the Riemannian distance on the manifold. Since for any $\chi \in \Delta$, the distance along each $W^{\chi}$ leaf is locally Lipschitz equivalent to the distance on the manifold, there exists $L > 0$ such that for all $\chi \in \Delta$ and sufficiently close points $z \in W^{\chi}(z')$, we have $L^{-1}d_{W^{\chi}}(z,z') \le d(z,z') \le Ld_{W^{\chi}}(z,z')$. Then after applying the triangle inequality, for sufficiently large $t$, we get:
\begin{eqnarray*} d(a^t \cdot x',a^t \cdot q) & \ge & d(a^t \cdot x_{s_0},a^t \cdot p) - d(a^t \cdot x',a^t \cdot p) - d(a^t\cdot x_{s_0},a^t q)  \\
 & \ge & L^{-1}d_{W^{[\delta_{s_0}]}}(a^t \cdot x_{s_0},a^t \cdot p) - d(a^t \cdot x',a^t \cdot p) - d(a^t\cdot x_{s_0},a^t q) \\
& \ge & L^{-2}e^{c_1t}d(x_{s_0},p) - L^2Ce^{c_2t} \ge C'e^{c_1t}
\end{eqnarray*}

\noindent since by construction, apply the triangle inequality to all legs connecting $a^t\cdot x_{s_0}$ and $a^t \cdot q$ and $a^t \cdot x'$ and $a^t \cdot p$, which contract faster than $e^{c_2t}$ and $e^{-t}$, respectively, since $c_1 > c_2 > -1$. % Note also that $a^t \cdot x'$ and $a^t \cdot 1$ contract at rate at least $e^{-t}$ since they are Lipschitz equivalent to $x$ and $y$, which are related by $c_p^{\chi_k}\chi_k \in A$.
We may construct new disks $D_{1,t}$ and $D_{2,t}$ \color{black} are tangent to $W^B$ with  the same uniform transversality conditions mentioned above \color{black} to $W^A$ connecting $a^t\cdot x$ and $a^t \cdot y$, and $a^t \cdot x'$ and $a^t \cdot q$ (note that we may not simply iterate the disks $D_1$ and $D_2$ forward, {  since $W^A$ is not the fast foliation for $a$}, and the transversality may degenerate). Therefore, since each of the foliations along weights of $B$ are uniformly transverse to those of $A$, and the holonomies are Lipschitz with a uniform Lipschitz constant on \color{black} uniformly \color{black} transverse discs, we arrive at a contradiction to the fact that $d(a^t\cdot x,a^t \cdot y) < C''e^{-t}$, so $p = q$.
\color{black}

Therefore, the connection between $w$ and $y'$ only involves weights of $A \cup B'$. By the induction hypothesis, the commutator of two weights in $  (A \cup B') \cap \Omega_{k\ell}$ produces only weights in $ (A \cup B') \cap \Omega_{k\ell}$.   Hence, for any $x$, 

\[ e(v) * e(u) \cdot x = \rho_1 * e(u) * e(v) \cdot x = e(u) * e(v) * \rho_2 \cdot x\] for some paths $\rho_1$ and $\rho_2$ which only involve $e(w)$ for some $w$ in the weight spaces $A \cup B'$, but may depend on $x$, $u$ and $v$. Using such relations, for any word in the weights $A \cup B'$, we may put it in a desired circular ordering without weights in $B \setminus B'$. This proves the inductive step, and hence the lemma, since $B \setminus B' = \set{\sigma \, \alpha + \tau \,\beta : \sigma,\tau < 1} \cap \Omega$. 
% Notice that choosing $a \in \ker \beta$, and perturbing to give a contracting element  gives that $\bigoplus_{\chi \in A} T\mc W^\chi$ is tangent to a foliation $\mc W^A$ (since $A$ will dominate $B$). Furthermore, by \cite{brown??}, the holonomy along $\mc W^A$ is $C^{1,\theta}$.
%Fix $x \in X$, and consider $y = (t^{(\alpha)} * s^{(\beta)}) \cdot x$, and $z = t^{(\alpha)} \cdot x$. By Lemma \ref{lem:local-transitivity}, there exists  $y' \in \mc W^A(y)$ and $z_i \in \mc W^{\chi_i}(z_{i-1})$, with $\chi_i \in B$ such that $z_n = y$, for $n = \abs{B}$.
%We show the claim if $u < 1$, the argument that $v < 1$ is completely symmetric. Choose any $a_0 \in \ker \alpha$. Then $\chi(a_0) < 0$ for all $\chi \in D(\alpha,\beta) \setminus \set{\alpha}$, so perturbing $a$ to a regular element $a$ such that $\alpha(a),\beta(a) < 0$ and $\abs{\alpha(a)} = \min\set{\abs{\chi(a)} : \chi \in D(\alpha,\beta)}$. Consider the foliation $\mc W$ with tangent distribution $\bigoplus_{\chi \in D(\alpha,\beta)} T\mc W^\chi$, which is H\"older with smooth leaves because it can be written as an intersection of stable manifolds. This contains the foliation $\mc W'$ with tangent distribution $\bigoplus_{\chi \in D(\alpha,\beta) \setminus\set{\alpha}} T\mc W^\chi$, which is invariant under $a$. Furthermore, the splitting into $T\mc W^\alpha \oplus T\mc W'$ is a dominated splitting, therefore, by \cite{brown????}, the holonomy along $\mc W'$ is $C^{1,\theta}$.
\end{proof}

\subsection{Regularity of conjugacies and proof of Theorems \ref{basic abelian}  and \ref{abelian}\color{black}}
\label{sec:regularity}

%In this section, we work in two regularities: $C^{2}$ and $C^\infty$. 
\color{black}
\begin{theorem}\label{regularity of conjugacies}
a) If \, $\R^k \curvearrowright X$  is a totally Anosov $C^\infty$ (resp. $C^2$) action 
satisfying assumptions of Theorem \ref{basic abelian}, then
%then there exists a bi-homogeneous space $K \backslash H / \Gamma$ such that
a finite cover of the action is $C^{\infty}$ (resp. $C^{1+}$) conjugate to a bi-homogeneous action on some bi-homogeneous space $K \backslash H / \Gamma$.
%diffeomomorphism 
%from a finite cover of $X$ to $K \backslash H / \Gamma$
%$h : X \to K \backslash H / \Gamma$ such that, setting $\tilde{a} = h \of a \of h^{-1}$ for any $a\in \color{black}\R^k\color{black}$, defines a bi-homogeneous action, where $r' = \infty$ or $(1,\theta)$ for \color{black} some \color{black}  $\theta \in (0,1)$, respectively. 
%Moreover, any $C^0$ conjugacy of the \color{black}$\R^k$\color{black}-action is $C^{r'}$. 

 b) If   $\R^k \curvearrowright X$ is a $C^\infty$ totally partially hyperbolic action  satisfying assumptions as in Theorem \ref{abelian}  then a finite cover of the action is $C^{\infty}$ conjugate to a bi-homogeneous action on some bi-homogeneous space $K \backslash H / \Gamma$.

%there exists a bi-homogeneous space $K \backslash H / \Gamma$ and a $C^{\infty}$ diffeomomorphism $h : X \to K \backslash H / \Gamma$ such that, setting $\tilde{a} = h \of a \of h^{-1}$ for any  $a\in \R^k$, defines a bi-homogeneous action. \color{black} 
%Moreover, any $C^0$ conjugacy of the \color{black} $\R^k$ \color{black}-action is $C^{r'}$. 
\end{theorem}

\begin{proof}%[Proof in $C^\infty$ case] 
By Section \ref{strong accessibility implies FA12} both actions in a) and b) satisfy \ref{FP1} and \ref{FP2}. As before we assume that leaves of  coarse Lyapunov foliations are orientable (otherwise if we pass to a finite cover) so that the constructions in Section \ref{part:build-bundle} apply.
\color{black}
By Proposition \ref{prop:smooth-is-top} and Theorem \ref{thm:technical}, it follows that there is a $C^0$-conjugacy between the canonical lift $\hat{X}$ produced in Theorem \ref{thm:lifted-action} and a translation action on a homogeneous space $H / \Gamma$.  The $C^0$ conjugacy between $\tilde{X}$ and $H/\Gamma$ induces a $C^0$ conjugacy $h$ between $X$ and $K\backslash H/\Gamma$. So it suffices to show that any $C^0$ conjugacy between $X$ and $K\backslash H/\Gamma$ is $C^{\infty}$ (resp. $C^{1+}$).

The coarse Lyapunov manifolds are parameterized by nilpotent group actions, which are subgroups of the full group $H$ in the homogeneous model (and hence act smoothly on the space $H / \Gamma$). They are conjugated to the actions $N^\alpha \curvearrowright X$ produced in Theorem \ref{thm:lifted-action}, which act by $C^{\infty}$ (resp. $C^{1+}$) diffeomorphisms on their  orbits by  Lemma \ref{lem:coarse-lyap-smooth} or \ref{lem:coarse-lyap-reglow}.

By standard Lie theoretic arguments (see, e.g., \cite{Montgomery-Zippin}, Section 5.1), since each element is a diffeomorphism, the group action is $C^1$ in both cases. In the $C^\infty$ setting, it also follows immediately that the group action is $C^\infty$, and therefore provides $C^\infty$-coordinates. Therefore, the conjugacy is $C^\infty$ restricted to each leaf of a coarse Lyapunov foliations.

\color{black}
In the $C^2$-case, we still have that $h$ is $C^1$ along leaves of $W^\alpha$. Notice that if $\norm{\cdot}_\alpha$ is the partial H\"older metric on $TW^\alpha$ constructed using \ref{FP2} in Section \ref{extension}, then it is invariant under the isometries in Proposition \ref{prop:isometry-construction}.  Then $h_*\norm{\cdot}_\alpha$ is a norm which is invariant under right translation on each leaf of $W^\alpha$ in $H / \Gamma$ (notice that while right translation is not defined on all of $H/ \Gamma$, it is well-defined on each $W^\alpha$-leaf). In particular, $h_*\norm{\cdot}_\alpha$ is  $C^\infty$ on each $W^\alpha$-leaf of $H/\Gamma$. Therefore, $h : W^\alpha(x) \to W^\alpha(h(x))$ is an isometry between a H\"older metric on the leaf in $\hat{X}$ and $C^\infty$ metric on the leaf in $H/ \Gamma$, and is therefore $C^{1+}$ by Theorem \ref{thm:taylor}.

% by standard Lie theoretic arguments (see, e.g., [Montgomery-Zippin, Section 5.1]).

Next we show that if $h$ is $C^{r}$ ($r=\infty$ or $r=1+$) along all the coarse Lyapunov foliations $W^\alpha$, then $h$ is $C^{r}$ along the foliations $W^s_a$ and $W^u_a$.  We show it for $W^s_a$, the proof for $W^u_a$ follows by considering $-a$. List the coarse exponents $\Delta^-(a) = \set{\alpha_1,\dots,\alpha_n}$ is a circular ordering, so that there are foliations $ W_i$ such that $TW_i = \bigoplus_{j=1}^i TW^{\alpha_j}$. 
%These foliations are obtained as intersections of stable manifolds from different Weyl chambers. 
We claim that $h$ is $C^{r}$ along $ W_i$ by induction on $i$.

We have already established the base case of $i = 1$. Assume that $h$ is $C^{r}$ along $W_i$. Then by construction, $ W_{i+1}$ is a foliation whose leaves have two transverse subfoliations: $ W_i$ and $W^{\alpha_{i+1}}$. Since we know that $h$ is $C^{r}$ along each by induction and since $W^{\alpha_{i+1}}$ is a coarse Lyapunov foliation, it follows that $h$ is $C^{r}$ along $W_{i+1}$ by Journ\' e's Theorem \ref{thm:journe}). This proves the inductive step. Since $W^s_a = W_n$, it follows that $h$ is $C^{r}$ along $W^s_a$. 

%and therefore, $h$ is $C^{r}$.

Now that we have that conjugacy has desired regularity along stable and unstable foliations, we split to two different arguments for Anosov and for partially hyperbolic case. 

{\it Proof of a).} 
The conjugacy $h$ is clearly $C^{\infty}$ (resp. $C^{1+}$) along the $\R^k$ \color{black}%\st{$\times K \times M$}\color{black}
-orbits. The smoothness of $h$ will now follow from iterated applications of Journ\'e's Theorem \ref{thm:journe}. Fix an $E^c$-partially hyperbolic element $a$.  We know $h$ is $C^{\infty}$ (resp. $C^{1+}$)  along the foliations $W^s_a$, $W^u_a$ and $\R^k$ \color{black}%\st{ $\times M \times K$}
\color{black}. Then by Theorem \ref{thm:journe}, since $\R^k$ %\st{ $\times M \times K$} 
 \color{black} and $W^s_a$ sub-foliate the center-stable manifold as transverse subfoliations, it follows that $h$ is $C^{\infty}$ (resp. $C^{1+}$) along the center-stable foliation. Then again, since the conjugacy is $C^{\infty}$ (resp. $C^{1+}$) along the center-stable foliation and the unstable foliation, another application Theorem \ref{thm:journe} shows that $h$ is $C^{\infty}$ (resp. $C^{1+}$). %{\color{teal}So here we are still talking about regularity of $h$ on $\tilde{X}$? (Since $\R^k\times K\times M$ only acting on $\tilde{X}$), but $\tilde X$ has no smooth structure so far, we may need to project it down to $X$ then apply Journ\'e.}

b) The inverse of the conjugacy 
 $h^{-1}$ is  uniformly $C^\infty$ along {\it algebraic} foliations $h(W^s_a), h( W^u_a)$ by the same argument as that we had above for $h$. 
Since the tangent bundles of $ h(W^s_a), h( W^u_a)$ with their Lie brackets  generate the whole tangent bundle of $K \backslash H / \Gamma$, by Theorem 2.1 \cite{KS-subelliptic estimates} which using subelliptic estimates gives global regularity of functions which are regular along smooth generating distributions,  we derive that $h^{-1}$ is $C^\infty$ on $K \backslash H / \Gamma$. Then  it suffices to show that $Dh^{-1}$ is non-degenerate everywhere. 

 Since $h$ is a conjugacy, by taking Jacobian (with respect to) invariant volumes) of the conjugacy equation   $$Jac( Dh^{-1}|_{a(x)})\cdot Jac(Da|_x)=Jac(Dh^{-1}|_x)\cdot Jac(Da|_{h^{-1}(x)}),$$
where we use the same notation $a$ to denote the action of $a$ on the algebraic and on the non-algebraic model, keeping in mind that $x$ is in the algebraic model and $h^{-1}(x)$ is in the non-algebraic one.  

Now by using that $\R^k\curvearrowright X$ is a volume preserving action, for any $a$ we have   $Jac(Dh^{-1}|_{a(x)})=Jac (Dh^{-1}|_x)$, hence $Jac (Dh^{-1})$ is $a$-invariant. Because $a$ is ergodic, $Jac (Dh^{-1})$ is constant on a full volume set.
%so for any  element $a\in \R^k$ and any $x$ such that $Dh^{-1}(x)$ is non-degenerate.
Since $h^{-1}$ is a homeomorphism,  $h^{-1}$ must have non-zero Jacobian at some positive volume set. This means $Jac (Dh^{-1})$ is a constant, non-zero function everywhere.
\color{black} 
\end{proof}
\color{teal}
%By Remark \ref{} we know any $\R^k$ action satisfying the assumptions in Theorem \ref{} is super accessible 

%\begin{theorem}\label{regularity of conjugacies: ph case}If $\R^k\curvearrowright X$ is a totally partially hyperbolic super-accessible $C^r$-action on a $C^\infty$-manifold $X$ for $r = \infty$ or $2$ satisfying assumptions \ref{FP1} and \ref{FP2}, then there exists a (bi)homogeneous space $K \backslash H / \Gamma$ and a $C^{r'}$ diffeomomorphism $h : X \to K \backslash H / \Gamma$ such that, setting $\tilde{a} = h \of a \of h^{-1}$ for any $a\in \R^k\times M$, defines a bi-homogeneous action, where $r' = \infty$ or $(1,\theta)$ for any  $\theta \in (0,1)$, respectively. Moreover, any $C^0$ conjugacy of the $\R^k$-action is $C^{r'}$. \end{theorem}
\color{black}

%\begin{proof}[{\it Proof of Theorem \ref{basic abelian}}]
 %By Theorem \ref{thm:lifted-action}, the action lifts to a continuous action on some cover. Then by Proposition \ref{prop:smooth-is-top}, this action is a leafwise homogeneous topological Anosov action, and by \ref{FP1}, it is genuinely higher rank. Hence, by Theorem \ref{thm}
 %\end{proof}

%\color{black}
%Recall that in the beginning of Section \ref{part:build-bundle} we assumed that we lifted the given abelian action to a finite cover to achieve orientability of leaves of coarse foliations. Then clearly, Theorem \ref{regularity of conjugacies} implies Theorem \ref{abelian}.

\color{black}

\part{Appendices}

\appendix

\section{Normal forms for contracting foliations}\label{app: nrml form}
% Move to appendix or just refer to another paper in Lemma 7.16.

We recall some aspects of normal forms theory, following \cite{KalSad16, Kal2020} which contains optimal results in the
 uniformly contracting setting. If $f : X \to X$ is a $C^\infty$ diffeomorphism of a Riemannian manifold with norm $\norm{\cdot}$ preserving a continuous foliation $\mc W$ with $C^\infty$ leaves, and $\chi  = (\chi_1,\dots,\chi_\ell)$ is an $\ell$-tuple of negative numbers such that $\chi_1 < \dots < \chi_\ell < 0$, we say that $f$ has $(\chi,\ve)$-spectrum on $\mc W$ if there is a splitting $T\mc W = \mc E^1 \oplus \dots \oplus \mc E^\ell$ into invariant subbundles such that for every $v \in \mc E^i$,
\[e^{\chi_i - \ve} \le \norm{df(v)} / \norm{v} \le e^{\chi_i + \ve}. \] 
By Remark 4.2 of \cite{Kal2020}, this is sufficient to obtain the usual narrow band condition on the Mather spectrum if $\ve$ is sufficiently small. Write a vector $v \in T_x\mc W$ in coordinates as  $v = (v_1,\dots,v_\ell)$, where $v_i \in \mc E^i_x$. A polynomial $q : \mc E_x \to \mc E_y$ is said to be $(s_1,\dots,s_\ell)$-homogeneous if 
\[q(\lambda_1v_1,\dots,\lambda_\ell v_\ell) = \lambda_1^{s_1}\dots\lambda_\ell^{s_\ell}q(v_1,\dots,v_\ell)).\] 
Then with a fixed $\ell$-tuple $(\chi_1,\dots,\chi_\ell)$, say that a polynomial $p : \mc E_x \to \mc E_y$ is of {\it subresonance type (with respect to $\chi$)} if 
\[ p(v_1,\dots,v_\ell) = (p_1(v_1,\dots,v_\ell), \dots ,p_\ell(v_1,\dots,v_\ell)), \]
and for $i = 1,\dots,\ell$, $p_i$ is a a sum of $(s_1,\dots,s_\ell)$-homogeneous polynomials such that $  \chi_i \le \sum s_j\chi_j$. It is of {\it resonance type (with respect to $\chi$)} if each $p_i$ is a sum of $(s_1,\dots,s_\ell)$-homogeneous polynomials such that  $\sum s_j \chi_j = \chi_i$.
The following is a consequence of Theorem 4.6 of \cite{Kal2020}. %where here we have much stronger assumptions, namely $C^\infty$. We also do not give precise estimates on how narrow the spectrum must be. 
\begin{theorem}
\label{thm:normal-forms}
Let $f : X \to X$ be a $C^\infty$ diffeomorphism of a $C^\infty$ manifold $X$ preserving a continuous foliation $\mc W$ with $C^\infty$ leaves. If $\chi = (\chi_1,\dots,\chi_\ell)$ is an $\ell$-tuple as above, then there exists a constant $\ve = \ve(\chi) > 0$ with the following property: if there exists a smooth Riemannian metric for which $df|_{T\mc W}$ has $(\chi,\ve)$-spectrum, then there exists a family of $C^\infty$ diffeomorphisms $\mc H_x : T_x\mc W \to \mc W(x)$ such that
\begin{itemize}
%[label={\rm (NF-\arabic*)}]
\item [\mylabel{nf:polynomial}{NF-1}] for every $x \in X$, $\mc H_{f(x)}^{-1} \of f \of \mc H_x$ is a subresonance polynomial,
\item [\mylabel{id}{NF-2}] for every $x \in X$, $d\mc H_x = \id$,
\item [\mylabel{subr}{NF-3}] if $\mc G_x$ is any other such family, then $\mc G_x = \mc H_x \of p_x$, for some family of subresonance polynomials $p_x$,
\item [\mylabel{nf:transfer}{NF-4}] if $y \in \mc W(x)$, then $\mc H_y = \mc H_x \of q_{x,y}$ for a composition of a translation with some subresonance polynomial $q_{x,y}$,
\item [\mylabel{nf:centralizer}{NF-5}] if $g : X \to X$ is a $C^\infty$ diffeomorphism which commutes with $f$, then $\mc H_{g(x)}^{-1} \of g \of \mc H_x$ is a subresonance polynomial.
\end{itemize}
\end{theorem}

\section{Polynomial functions along transverse foliations }

{%\color{cyan}
The following was shown by Margulis 
(\cite[Lemma 4]{MR0492072} or \cite[Lemma 17]{MR739627}), who proved it for rational functions in the measurable setting assuming Lebesgue almost everywhere properties (cf. also \cite[Theorem 3.4.4]{Z} and its proof.)} We provide a proof for polynomials in the continuous case which is much simpler and more straightforward.

\begin{lemma}
\label{lem:polynomial-restrictions}
Let $V$ be a vector space, and $f : V \to \R$ be a continuous function such that $t \mapsto f(v + tw)$ is a polynomial in $t$ for every $v,w \in V$. Then $f$ is a polynomial.
\end{lemma}

\begin{proof}
We prove that for any collection of linearly independent elements $v_1,\dots,v_n$, the map \[(t_1,\dots,t_n) \mapsto f(w+\sum t_iv_i)\]

\noindent is a polynomial in $(t_1,\dots,t_n)$ by induction on $n$. Notice that the base case of $n = 1$ is the assumption of the lemma, and the case of $n  = \dim(V)$ proves the conclusion. 

The degree of the polynomials must be uniformly bounded by some constant $N$ since the map $f$ is continuous. Assume we have shown it for subspaces of dimension $n - 1$, and consider a collection $v_1,\dots,v_n$. Then by assumption, fixing $t_2,\dots,t_n$ and letting $t_1$ vary yields a continuous family of polynomials in $t_1$, so
\begin{equation}
\label{eq:polynomial-fibration}
 f(t_1v_1 + \dots + t_nv_n) = a_0(t_2,\dots,t_n) + a_1(t_2,\dots,t_n)t_1 + \dots + a_N(t_2,\dots,t_n){t_1}^N 
\end{equation}
\noindent for some collection of continuous functions $a_i : \R^{n-1} \to \R$. By induction, the functions \[p_k(t_2,\dots,t_n) := f(kv_1 + t_2v_2  +\dots +t_nv_n)\] are polynomials. Applying \eqref{eq:polynomial-fibration} to the definition of the polynomials $p_k$  yields the following system 
\[ \begin{pmatrix}
1  & 1 & 1 & \dots & 1 \\
2 & 4 & 8 & \dots & 2^N \\
3 & 9 & 27 & \dots & 3^N \\
\vdots & & & \ddots & \\
N+1 & (N+1)^2 & (N+1)^3 & \dots & (N+1)^{N+1}
\end{pmatrix} \begin{pmatrix}
a_0 \\ a_2 \\ a_3 \\ \vdots \\ a_N
\end{pmatrix} = \begin{pmatrix}
p_1 \\ p_2 \\ p_3 \\ \vdots \\ p_{N+1}
\end{pmatrix} \]
The above is a Vandermonde matrix which is invertible, so each function $a_i$ is actually a linear combination of the polynomials $p_1,\dots, p_N$. Therefore, \eqref{eq:polynomial-fibration} shows that the function is a polynomial.
\end{proof}

\section{Generation of simple Lie groups and their Weyl groups by detected subgroups\color{black}}

\color{black}

Let $G$ be a simple real Lie group, $A \subset G$ be an $\R$-split Cartan subgroup, and $a \in A$ be a non-identity element. Furthermore, let $\Delta_G$ denote the set of restricted roots of $G$ relative to the Cartan subgroup.

\begin{definition}
    We call a functional $\lambda \in A^*$ {\it detected by $a$} if $\lambda(a) \not= 0$, {\it positive} for $a$ if $\lambda(a)>0$ and {\it negative} for $a$ if $\lambda(a) <0$. We let the sets $\Delta_a$, $\Delta_a^+$ and $\Delta_a^-$ denote the set of detected, positive and negative restricted roots for $a$ respectively.
\end{definition}

By definition, $\Delta_a = \Delta_a^+ \cup \Delta_a^-$.

\begin{lemma}
    For any simple real Lie group $G$ and non-identity element $a \in A$, $\mf u_a^\pm := \bigoplus_{\lambda\in\Delta_a^\pm} \mf g_\lambda$ are Lie subalgebras of $\Lie(G)$, which correspond to subgroups $U_a^\pm$. $G$ is generated by $U_a^+$ and $U_a^-$.
\end{lemma}

Recall that the {\it Weyl group} of $G$ with respect to $A$ is defined to be $W_G = N_G(A)/Z_G(A)$. The Weyl group is always finite and carries the information about the group $G$. The following comes from the standard theory of real Lie groups and restricted roots (for $\C$ groups, see \cite[Lemma 10.4B]{humphreys};  $\R$-groups are reduced to the $\C$-case via, e.g., \cite[Proposition 1.1.3.1]{Warner1}):

\begin{proposition}
\label{prop:weyl-standard}
    $W_G$ is generated by reflections of the form

    \begin{equation}
    \label{eq:reflection-formula}
        w_\beta(\lambda) = \lambda - \dfrac{\langle\beta,\lambda\rangle}{\langle \beta,\beta\rangle} \beta
    \end{equation}

    for some inner product $\langle\, , \, \rangle$ on $A^*$. Furthermore, for every $\beta, \lambda \in\Delta_G$,

    \begin{equation}
    \label{eq:weyl-conjugacy}
        w_\beta w_\lambda w_\beta^{-1} = w_{w_\beta(\lambda)}
    \end{equation}
    
{\color{black}    $G$ is simple if and only if $W_G$ leaves no proper subspace of $A^*$ invariant.}
\end{proposition}

Given some non-identity element $a \in A$, let $W_a \subset W_G$ denote the subgroup of $W_G$ generated by the reflections $w_\beta$, $\beta \in \Delta_a$, and $W_0 \subset W_G$ denote the subgroup of $W_G$ generated by reflections $w_\beta$, $\beta \not\in \Delta_a$ (undetected roots).

\begin{lemma}
\label{lem:undet-det}
    If $w \in W_G$ is an element, then we may write $w = w_1w_2$, where $w_1$ is a product of reflections across undetected roots, and $w_2$ is a product of reflections across detected roots.
\end{lemma}

\begin{proof}
    This follows from the normalization property $w_\beta w_\alpha w_\beta = w_{w_\beta(\alpha)}$, and the fact that if $\alpha$ is detected and $\beta$ is undetected, then $w_\beta(\alpha) = \alpha - c\beta$ for some $c \in \R$ is detected. Hence, we may use this commutation property to push the appearance of all undetected weights to the left.
\end{proof}

\begin{lemma}
\label{lem:undetected-flip}
    Let $G$ be simple, and $a$ be as above. For every undetected root $\beta$, there exists a detected root $\alpha$ such that $w_\alpha(\beta) \not= \beta$.
\end{lemma}

\begin{proof}
    Assume otherwise. Then there exists an undetected root $\beta$ such that for every detected root $\alpha$, $w_\alpha(\beta)= \beta$. Furthermore, for every undetected root $\gamma$, $w_\gamma(\beta)(a) = \beta(a)+c \gamma(a) = 0$. It follows by Lemma \ref{lem:undet-det} that for any $w \in W_G$, $w(\beta)$ is undetected (the detected roots of the decomposition fix $\beta$, and the undetected ones permute among themselves). Let $V \subset A^*$ denote the span of $\set{w(\beta) : w \in W_G}$. Then $V$ is invariant under $W_G$, and contained in the span of the undetected roots. This is a contradiction to the last claim of Proposition \ref{prop:weyl-standard}, so the lemma holds.
\end{proof}

\begin{corollary}
\label{cor:weyl-generation}
    If $G$ is simple, $W_a = W_G$ for every non-identity element $a$.
\end{corollary}

\begin{proof}
    It suffices to show that each reflection $w_\beta$ is generated by reflections in $W_a$. If $\beta$ is undetected, by Lemma \ref{lem:undetected-flip}, there exists a detected root $\alpha$ such that $w_\alpha(\beta) \not= \beta$. But since $\alpha$ is detected and $w_\alpha(\beta) = \beta - c\alpha$, $c\not= 0$, it follows that $w_\alpha(\beta)$ is detected. Therefore,
    \[w_\beta = w_{w_\alpha(w_\alpha(\beta))} = w_\alpha w_{w_\alpha(\beta)} w_\alpha\]
    is generated by reflections across detected roots.
\end{proof}

\begin{corollary}
\label{cor:detected-weyl}
    If $G$ is a simple Lie group and $a \in A$ is a non-identity element, then for every $\theta \in A^*$, there exists some $w \in W_a$ such that $w(\theta)$ is not proportional to $\theta$.
\end{corollary}

\begin{proof}
Note that if $w(\theta)$ is proportional to $\theta$ for every $w \in W_a$, then $\R\theta$ is an invariant subspace of $W_a$. The result follows immediately from the last statement of Proposition \ref{prop:weyl-standard} and Corollary \ref{cor:weyl-generation} and the fact that $\dim(A^*) \ge 2$.
 %   First, note that since the restricted roots of $G$ contain a dual basis to $A$, for any $a$, $W_a \not= \set{e}$. To see that $W_a = W_G$, we will show that it is normal. Indeed, if $\lambda$ and $\beta$ are both detected by $a$, it is clear that $w_\beta w_\lambda w_\beta^{-1} \in W_a$, since $W_a$ is the group {\it generated} by such reflections. So it suffices to show that if $\lambda \in \Delta_a$ but $\beta\not\in \Delta_a$, then $w_\beta w_\lambda w_\beta^{-1} \in \Delta_a$. But, by \eqref{eq:weyl-conjugacy} and \eqref{eq:reflection-formula}, we get $w_\beta w_\lambda w_\beta^{-1} = w_{w_\beta(\lambda)}$, and

%    \[ w_\beta(\lambda)(a) = \lambda(a) -\dfrac{\langle\beta,\lambda\rangle}{\langle \beta,\beta\rangle} \beta(a) = \lambda(a) \not= 0,\]

%    since $\lambda \in \Delta_a$ and $\beta\not\in \Delta_a$. Hence $W_a$ is normal.

%    The claim about hyperplanes follows immediately by contradiction, the fact that $W_a = W_G$, and the fact that no hyperplane is invariant under reflections across the kernels of a dual basis.
\end{proof}
\color{black}

\section{Centralizers of ergodic homogeneous actions}
\label{sec:centralizer-zhegib}

{\color{black}We prove some results related to the regularity of centralizers. For homogeneous flows, the centralizers have the best regularity: affine. It is used in Section \ref{sec:main-Gproof}.}

\begin{theorem}
\label{thm:affine-centralizer}
Let $G$ be a simply connected Lie group, $\phi_t : G / \Gamma \to  G / \Gamma$ be an ergodic homogeneous flow generated by an $\R$-semisimple element. Let $Z_{\Lip}$ denote the set of Lipschitz transformations commuting with $\phi_t$. Then $Z_{\Lip} = Z_\Aff$, the group of affine transformations commuting with $\phi_t$.
\end{theorem}

The principle tool in proving Theorem \ref{thm:affine-centralizer} is a result of Zeghib, which we summarize here. If $X$ is a $C^\infty$ manifold, a subset $N \subset X$ is called {\it rectifiable} if it is the Lipschitz image of a bounded subset of $\R^n$ for some $n$. If $N$ is a rectifiable subset, it has a well-defined Hausdorff measure, which we denote by $\mu_N$.

\begin{theorem}[Th\'{e}or\`{e}m A, \cite{Zeghib}]
\label{thm:zeghib}
Let $G$ be a simply connected Lie group, $\phi_t :  G / \Gamma \to  G / \Gamma$ be an ergodic homogeneous flow, generated by an $\R$-semisimple element. If $N \subset X$ is a rectifiable, $\phi_t$ invariant set, then $\mu_N$ is $\phi_t$-invariant, and the ergodic components of $\mu_N$ are the Haar measures on closed $H$-orbits in $G / \Gamma$, where $H$ is some fixed closed subgroup $H \subset G$. 
\end{theorem}

\begin{remark}
    In fact, Zeghib claims more, by working with bi-homogeneous flows, but we will only use the homogeneous flow version here.
\end{remark}

\begin{proof}[Proof of Theorem \ref{thm:affine-centralizer}]
Let $f \in Z_{\Lip}$, and consider the graph 

\[ N = \set{ (x,f(x)) : x \in  G / \Gamma} \subset (G \times G) / (\Gamma \times \Gamma). \]

By construction, since $f$ commutes with $\phi_t$, $N$ is $\phi_t \times \phi_t$-invariant. Furthermore, since $f$ is Lipschitz, $N$ is rectifiable. Therefore, by Theorem \ref{thm:zeghib}, $\mu_N$, the Hausdorff measure on $N$, is $\phi_t$-invariant. $\phi_t$ is ergodic on $ G / \Gamma$, $\phi_t \times \phi_t$ must be ergodic on $N$ as well.

Therefore, there exists a unique subgroup $H \subset G \times G$ such that $N = (g,h)\cdot H / (\Gamma \times \Gamma)$ for some $(g,h) \in G \times G$. Since $(e\Gamma, f(e\Gamma)) \in N$, we may without loss of generality assume $g = e$.

Finally, consider the Lipschitz transformation $\pi : N \to  G / \Gamma$ defined by $\pi(x,y) = x$. Since $N$ is the graph of a Lipschitz transformation, $\pi$ is a Lipschitz homeomorphism. This further implies that $\pi$ is a diffeomorphism, since $N$ is a coset of $H$. Hence, $\dim(H) = \dim(G)$, and for each $X \in \Lie(G)$, there exists a unique $\bar{F}(X) \in \Lie(G)$ such that $(X,\bar{F}(X)) \in \Lie(H)$. It follows immediately that $\bar{F}$ is a Lie algebra homomorphism since $H$ is a subgroup. Let $F : G \to G$ denote the lift of $\bar{F}$ to $G$.

It is immediate that $H = \set{(g,F(g)): g \in G}$. Therefore, from the definition of $N$, we get that $f(g\Gamma) = hF(g)\Gamma$, and $f$ is affine, as claimed.
\end{proof}

\section{Lifting $\G$-actions}

{\color{black} In the setting of $G$-actions, our main conclusion is to produce a bi-homogeneous action. The main technical step is to show that the action of an $\R$-split Cartan subgroup lifts to an action by translations on a fibration. {\color{black} Recall Definition \ref{def:gen-HR} and let $B \subset A \subset G$ be a genuinely higher-rank subgroup of $G$. We divide $B$ into $B$-Weyl chambers by considering the connected components of $B \setminus \bigcup_{\chi \in \Delta_G} \ker \chi|_B$, where $\Delta_G$ are the roots of $G$ for the Cartan subgroup $A$.}

To recover the action of $G$ on the fibration, we use the following in Section \ref{sec:main-Gproof}:}

%\color{black} CHECK this section; need it PH actions\color{black}
%Let $G$ is a semisimple Lie group, $A \subset G$ be a split Cartan subgroup of $G$ and $\set{ U_\beta : \beta \in \Delta_G}$ denote the set of (coarse) root subgroups of $G$ corresponding to $A$.
\begin{theorem}\label{lifting} Let $r = \infty$ or { $2$}, $G$ be a real semisimple group such that every simple factor has real rank at least 2, and $A$ be an $\R$-split Cartan subgroup of $G$. 
Let $G \curvearrowright X$ be a $C^r$ action  which is {\color{black}$B$-totally partially hyperbolic for some genuinely higher-rank subgroup $B \subset A$} with common central distribution $E^c$.

Assume that 

\begin{itemize}
\item $\pi : Y \to X$ is a $C^r$ fiber bundle over $X$,
\item the action of $B$ lifts to a $C^r$ action on $Y$,
\item for each $B$-Weyl chamber $\mathfrak C$ of $B \subset G$ there exists some distinguished $a \in \mathfrak C$ such that
 the lifted action of $a$ on $Y$ is partially hyperbolic with respect to $E_y = d\pi^{-1}(E^c(\pi(y)))$, the saturation of the common central bundle by the tangent bundle.
\end{itemize}

Then there exists some continuous action of $\tilde{G}$, the universal cover of $G$, on $Y$ which is a lift of the $G$-action in the sense the if $p : \tilde{G} \to G$ is the canonical projection, $\pi(g \cdot x) = p(g) \cdot \pi(x)$. % whose restriction to $A$ is the given lift. 
Furthermore, if there exists a continuous metric on $Y$ for which $db|_{T\pi^{-1}(\pi(y))}$ is isometric for every $b \in B$ and $y \in Y$, the $\tilde{G}$-action is $C^r$.
\end{theorem}

The proof of this theorem follows the scheme introduced by the first author and Katok in \cite{DamKatII}, extended by Zhenqi Wang in \cite{Wang1,Wang2} and made fully general by the third author in \cite{vinhageJMD2015}. We summarize some important definitions before proceeding with the proof.

Let $\Delta_G$ denote the set of roots $\chi$ of $G$ such that $\chi/2$ is not a root, and for each $\chi \in \Delta_G$, let $U_\chi$ denote corresponding coarse Lyapunov subgroup. {\color{black}We call $\chi \in \Delta_G$ {\it detected} if $\chi|_B\not\equiv 0$, and let $\Delta_{G,B}$ denote the set of detected roots.} As described in Section \ref{sec:free-prods}, we consider the free product $\mc P_B$ of the groups $U_\chi$, {\color{black}where $\chi$ ranges over $\Delta_{G,B}$} and $\mc P = \mc P_A$. Note that since the groups $U_\chi$ generate $G$, there exists a projection $\pi_B : \mc P_B \to \tilde{G}$ (with $\pi = \pi_A$) such that the kernel $\mc C_B$ (with $\mc C = \mc C_A$) is exactly the expressions in $\mc P_B$ which yield contractible cycles on $G$.

Since $G$ acts on $X$, any relations among the $U_\chi$ on $G$ hold on $X$ as well. Importantly, the commutator relations hold: if $\chi_1,\chi_2 \in \Delta_G$ are linearly independent, $u \in U_{\chi_1}$ and $v \in U_{\chi_2}$, then $[u,v] \in \displaystyle{\prod_{\substack{\chi = s\chi_1 + t\chi_2 \\ s,t \in \Z_+}}} U_\chi$. Hence, the commutator relations hold on $X$. The following lemma becomes of particular interest. It is based on Lemma 4.7 of \cite{vinhageJMD2015}:

\begin{lemma}
\label{lem:stable-dense}
Let $G$ be a semisimple Lie group such that every simple factor has real rank at least 2, $B \subset A$ be a genuinely higher-rank subgroup of an $\R$-split Cartan subgroup $A$, $\mc P_B$ denote the free product of the detected coarse Lyapunov subgroups of $G$, and $\mc C_{S,B} \subset \mc P$ denote the normal closure (in $\mc P_B$) of the group generated by the relations $v^{-1} * u^{-1} * v * u * \rho^{-1} \in \mc P_B$, where $u \in U_{\chi_1}$, $v \in U_{\chi_2}$ with $\chi_1,\chi_2 \in \Delta_{G.B}$ satisfying $\chi_1|_B \not\propto \chi_2|_B$, and $\rho$ is any presentation of $[u,v]$ in the group $\displaystyle \prod_{\substack{\chi = s\chi_1 + t\chi_2 \\ s,t \in \Z_+}} U_\chi$. Then $\mc C_{S,B}$ is a co-abelian subgroup of $\mc C_B$, and every continuous action of $\mc C_B / \mc C_{S,B}$ on a space of finite topological dimension is trivial.
\end{lemma}

{\color{black}
\begin{proof}
    We first consider the case when $B = A$, so that every root is detected. By \cite[Theorem 1.9]{Deodhar78}, the following short exact sequence is a perfect central extension of $\tilde{G}$, the universal cover of $G$ (in fact, it is the {\it universal central extension}):

    \[ 1 \to \mc C / \mc C_S \to \mc P / \mc C_S \to \tilde {G} \to 1.\]

    In particular, since $\mc C / \mc C_S$ is central in $\mc P / \mc C_S$, it is abelian. Furthermore, assume that $\mc C / \mc C_S \curvearrowright X$ is an action by homeomorphisms of a space of finite topological dimension. Since the group $\mc C/ \mc C_S$ is abelian, if $x \in X$, $H_x = \Stab_{\mc C / \mc C_S}(x)$ is a normal subgroup in $\mc C / \mc C_s$. In fact, since it is central in $\mc P / \mc C_S$, it is a closed normal subgroup there as well. It follows that we have a short exact sequence of topological groups

    \[ 1 \to (\mc C / \mc C_S) / H_x \to (\mc P / \mc C_S)/ H_x \to \tilde{G} \to 1\]

    We claim that $(\mc C / \mc C_S) / H_x$ is a Lie group. Indeed, it is locally path connected (since $\mc C$ is locally path connected; in every combinatorial cell, $\mc C$ is an algebraic variety), and the evaluation map $\sigma \mapsto \sigma \cdot x$ is an injective continuous map into $X$. By Theorem \ref{lem:gleason-palais}, $(\mc C / \mc C_S) / H_x$ is a Lie group. But since $\tilde{G}$ is a simply connected semisimple Lie group, it has no perfect Lie central extensions. Hence, $H_x = \mc C / \mc C_S$, and any such action is trivial.

    When not every root is detected, the result follows from \cite[Proposition 7.9]{vinhage-wang}, which shows that when $B$ is genuniely higher-rank, there is an isomorphism between $\mc P_B / \mc P_{S,B}$ and $\mc P / \mc P_S$ covering $\id$ on $\tilde{G}$.
\end{proof}
}

\begin{proof}[Proof of Theorem \ref{lifting}]
 Let $a \in B$ be one of the distinguished elements of the Weyl chambers which lift to partially hyperbolic maps. Consider $W^s_{a,X}$ and $W^s_{a,Y}$, the stable and unstable foliations of $a$ as it acts on $X$ and $Y$, respectively. Since the fibers of $\pi$ are contained in the center foliation, it follows that $\pi|_{W^s_{a,Y}(y)}$ is a $C^r$ diffeomorphism. For every $\chi$, we may therefore build an action $U_\chi \curvearrowright Y$ by choosing some $a \in \R^k$ such that $\chi(a) < 0$ and letting $u \cdot y$ denote the unique element of $W^s_{a,Y}(y)$ which projects to $u \cdot x \in W^s_{a,X}(\pi(y))$. Note that the action does not depend on the choice of $a$ or Weyl chamber which $a$ belongs to. By the universal property of (topological) free products, we may construct an action of $\mc P_B$, the free product of the groups $U_\chi$ with $\chi \in \Delta_{G,B}$, on $Y$ as well \color{black}(see Proposition \ref{sec:free-prods}). \color{black}
 %{\color{black} Should we explain this more???}

 Fix any pair of linearly independent roots $\chi_1,\chi_2 \in \Delta_{G,B}$. Then, by linear independence, there exists $a \in B$ such that $\chi_1(a),\chi_2(a) < 0$. If $u \in U_{\chi_1}$ and $v \in U_{\chi_2}$, then by Lemma \ref{lem:stable-dense}, there exists a $w$ written as a product of elements from the groups $U_\chi$, $\chi = s\chi_1 + t\chi_2$, $s,t \in \Z_+$ such that $[u,v]w^{-1} = e$, as elements in $G$. Since the relations hold on $G$, they act trivially on $X$. Furthermore, since the orbits of $U_{\chi_1}$, $U_{\chi_2}$ and every group appearing in the presentation of $w$ are contained in $W^s_{a,X}$, the relation holds when the groups act on $Y$ as well.

 Therefore, the restriction of the $\mc P$-action on $Y$ to $\mc C_{S,B}$ is trivial. Hence there exists a well-defined continuous action of $\mc C_B / \mc C_{S,B}$ on $Y$. This action must be trivial by Lemma \ref{lem:stable-dense}. Finally, since the $\mc C_B$ action is trivial, the action of $\mc P_B$ induces an action of $\tilde{G} = \mc P_B / \mc C_B$ on $Y$, which by construction must cover the action on $X$.

 To prove the regularity of the action, notice that the lift of the actions of the groups $U_\chi$ to $Y$ is exactly determined by stable holonomies. Since, restricted to $T\pi^{-1}(\pi(y))$, the action is isometric, the stable holonomies will have arbitrarily good pinching properties. By now-standard arguments in partial hyperbolicity (see, eg, \cite[Theorem 6.1]{PSW}), it follows that the holonomies, and hence action of each of the groups $U_\chi$ on $Y$ is $C^r$. Finally, since the $U_\chi$ generate $\tilde{G}$, every element of $\tilde{G}$ is a $C^r$ transformation of $Y$. By Theorem \ref{MZ}, the action $\tilde{G} \curvearrowright Y$ is a $C^r$ group action.
\end{proof}

\begin{remark}
The smoothness of the lift can be obtained from a weaker assumption, namely center bunching. To obtain a $C^\infty$ lift, one requires a more restrictive form of center bunching than that laid out in Section \ref{sec:PH-prelims}, which allows for no exponents on the fiber whatsoever (i.e., something imitating unipotent behavior). Since our application is to a compact group extension, and the action of a semisimple Lie group will always be conformal, we use this simpler version of the statement.
\end{remark}

%\begin{theorem}
%Assume that $X$ is a $C^r$ manifold, and that there exists a totally Anosov $C^r$ action of $A$ on $X$ such that 

%\begin{enumerate}
%\item for every $\beta \in \Delta_G$, there is an associated action of of $U_\beta$ on $X$, and
%\item for any Anosov $a \in A$, there is an action of the subgroup $V^s_a := \set{ g\in G : a^nga^{-n} \to e}$ whose restriction to the action of each $U_\beta$, $\beta \in \Delta_G$ coincides with the given action of $U_\beta$
%s $\set{ U_\beta : \beta \in \Delta_G, \beta(a) < 0}$ generate the action of a Lie group.
%\end{enumerate}

%Then there exists a $C^r$ action of the universal cover of $G$ on $X$ whose restrictions are the actions of $A$ and $U_\beta$ for every $\beta \in \Delta_G$.
%\end{theorem}

\section{Brin-Pesin theory in low regularity\color{black}}
\label{app:brin-pesin}

\color{black}

In this appendix, we develop analogs of theorems about principal bundles with compact structure groups in low regularity. Many of the statements and arguments are very similar to previous works of Brin and Pesin \cite{brin-pesin}, Wilkinson \cite{W} and Avila, Santamaria and Viana \cite{ASV}, but require special attention due to the presence of H\"older continuity only along the stable and unstable leaves.  

\subsection{Principal bundles in low regularity\color{black}}

We start by defining precisely regularity along a foliation. 

\begin{definition}
    \label{def:reg-along-fol}
    Assume that $\mc P$ is a continuous principal $K$-bundle over a $C^2$ manifold $X$, with a compact structure group $K$, and let $\mc W$ be a continuous $n$-dimensional foliation on $X$ with $C^{1}$ leaves. %If $\mc R$ is a regularity stronger than continuity but not stronger than $C^r$,
    We say that $\mc P$ is $\mc W$-H\"older  %along %{\it $\mc R$-regular along $\mc F$} 
    if 
    \begin{itemize}
        \item[(1)]  there exist continuous bundle charts $\tau_p : \R^n \times \R^{d-n} \times K \to \mc P$ such that the collection $\pi \of \tau_p|_{\R^n \times \R^{d-n} \times \set{e}}$ form a collection of continuous foliation chart of $\mc F$ which are $C^{1}$ along $\R^n \times \set{0} \times \set{e}$, and $\tau_p|_{\R^n \times \set{0} \times K}$ is uniformly H\"older continuous. We call such a chart $\tau_p$ a {\it foliation-bundle chart}. 
        \item[(2)] there is a family of identifications $\psi_{x,y}$ between fibers $\mc P_x$ and $\mc P_y$,  defined for all $x$ close to $y$ such that
\begin{itemize}
    \item $\psi_{x,y}$ is continuous in 
    $x,y$,
    \item for all $k\in K$, $k\psi_{x,y}(p) = \psi_{x,y}(kp)$,
    \item  $\psi_{x,y}$ varies uniformly H\"older continuously as $y \in \mc W(x)$ varies along the $\mc W$-leaf, and $\psi_{x,x} = \id $ for all $x\in X$. In particular, $d(\psi_{x,y}(p),p) < Cd(x,y)^\theta$ for some $C >0$ and $\theta > 0$ and all $y\in \mc W_{loc}(x)$ .
\end{itemize}
\end{itemize}
   
\end{definition}
\color{black}
\begin{remark} Using charts of foliation $\mc W$,  (1) is just equivalent to the existence of local continuous section of $\mc P$ which is H\"older along $\mc W$. We believe (2) can be induced from (1), however for simplicity we choose to state (2) separately. As Lemma \ref{lem: loc idfy} shows, (2) is not very restrictive.
\end{remark}

%The following lemma motivates the discussion of this section.

\color{black}
%\begin{remark}
%    In our application, $\mc R$ will often be H\"older continuity. In this case, we mean that there exists a common $C > 0$ and $\theta > 0$ for which the bundle-foliation charts of Definition \ref{def:reg-along-fol} are $(C,\theta)$-H\"older along the leaves of the foliation.
%\end{remark}

% \begin{lemma}
%     \label{lem:abstract-bundle}
%     Fix a $C^\infty$ manifold $X$ and Lie group $N$. Assume that $\Lie(N)$ admits a vector space decomposition $\Lie(N) = E_1 \oplus \dots \oplus E_n$. Let $H$ denote the set of Lie group automorphisms of $N$ preserving each subspace $E_i$.  Assume that for every $x \in X$, there exists a map $\tau : N \times X \to X$ such that $\tau(n_1n_2,x) = \tau(\varphi_{\tau(n_2,x),x}(n_1),\tau(n_2,x))$ for some family $\varphi_{e(n_2,x),x} \in H$ where $\varphi_{x,y}$ is defined when $d(x,y)$ is sufficiently small, and $\varphi_{x,y} \to \id$ as $x \to y$. We use the notation $\tau_x(n) = \tau(n,x)$ to denote the evaluation maps of the twisted action.
    
%     Then there exists a H\"older principle bundle with structure group $H$ consisting of pairs $(x,(\mf f_i)_{i=1}^n)$, where $\mf f_i$ is a framing of $(\tau_x)_*(E_i)$, and given any two $n$-tuples of framings $(\mf f_i)$ and $(\mf f_i')$ (possibly at different points $x$ and $x'$), then there exist some $h \in H$ such that $(\tau_x^{-1})_*(\mf f_i)=h((\tau_{x'}^{-1})_*(\mf f_i))$.
% \end{lemma}

% \begin{corollary}
%     With structures as in Lemma \ref{lem:abstract-bundle}, assume further that there is a metric on $\Lie(N)$ such that the spaces $E_i$ are orthogonal.
% \end{corollary}

\subsection{Stable holonomies\color{black}}\label{Appendix-holonomies}
Let $X, \mc P, \mc W, K$ as in Definition \ref{def:reg-along-fol}. We further assume that 
%Let $X$ be a smooth manifold, $\mc W$ be a H\"older foliation with $C^r$ leaves, $P$ be a continuous principal $K$-bundle which is H\"older along $\mc W$. 
$\R^k \curvearrowright \mc P$ be an action by continuous bundle automorphisms of $\mc P$ covering a $C^1$ $\R^k$-action on $X$ which uniformly exponentially contracts $\mc W$. 
 \color{black}Assume the $\R^k$-extension on $\mc P$ is $\mc W$-H\"older 
 %which means at a foliation-bundle chart defined in (1). of  Definition \ref{def:reg-along-fol}, the bundle automorphism is uniformly H\"older continuous along $\R^k\times \{0\}\times K$. 
 We denote both actions on $\mc P$ and $X$ by $a(\cdot)$ for $a\in \mathbb R^k$, since from the context it will be clear which one we refer to. %\color{cyan} If $x \in X$, let $F_x$ denote the fiber above $x$. \color{black}For our purpose we assume that 
 %This assumption (PUT THIS ASSUMPTION IN THE DEFINITION E.2.) is not very restrictive, it holds at least for $\tilde X^\lambda$ of Lemma \ref{lem:big-bundle} as well as the product bundle $\tilde X$.

%MAY NEED TO DEFINE $a$  ACTION AT HERE. AND WHAT IS THE REGULARITY OF $\R^k$ ACTION? IN WHICH PROPOSITION WE NEED TO ASSUME $\R^k$-BUNDLE AUTOMORHPISM ARE H\"OLDER ALONG COARSE?

\color{black}
If $y, x \in \mc W(x)$ are sufficiently close, $p$ covers $x$ and $q$ covers $y$, let $\kappa(p,q)$ denote the unique element of $K$ such that $\psi_{x,y}(p) = \kappa(p,q)q$. Note that $\kappa(k_1p,k_2q) = k_1\kappa(p,q)k_2^{-1}$.

\begin{lemma}
\label{lem:constructing-holonomy}
    If $K$ is a Lie group with a bi-invariant metric $d$, then for every $\ve > 0$, there exists a $T \ge 0$ such that if $t,s \ge T$, \color{black} $y \in \mc W_{loc}(x)$, $p$ covers $x$, $q$ covers $y$, \color{black} then $d(\kappa(a^tp,a^tq),\kappa(a^sp,a^sq)) < \ve$.
\end{lemma}

\begin{proof}
    For ease of notation, let $\kappa_t = \kappa(a^tp,a^tq)$, so $\psi_{a^tx,a^ty}(a^tp)  =  \kappa_ta^tq$. If $s > t$, \color{black}since $a^{s-t}$ is a principal bundle automorphism, we have \color{black} 
\begin{equation}
    \label{eq:kappat-pushed}
        a^{s-t}\psi_{a^tx,a^ty}(a^tp)  =  \kappa_ta^sq.
    \end{equation}
On the other hand, $\kappa_sa^sq = \psi_{a^sx,a^sy}(a^sp)$ and, so \begin{equation}\label{eq:comparing-kappas}\kappa_ta^sq = \kappa_t\kappa_s^{-1}\kappa_sa^sq = \kappa_t\kappa_s^{-1}\psi_{a^sx,a^sy}(a^sp), \end{equation} and combining \eqref{eq:kappat-pushed} and \eqref{eq:comparing-kappas}, we get
     \begin{equation}
        \label{eq:comparing-kappas2}
        \kappa_t\kappa_s^{-1}\psi_{a^sx,a^sy}(a^sp) = a^{s-t}\psi_{a^tx,a^ty}(a^tp).
    \end{equation}
Finally, observe that $d(\psi_{a^sx,a^sy}(a^sp),a^sp) < Cd(a^sx,a^sy)^\theta$, and $d(\psi_{a^tx,a^ty}(a^tp),a^tp) < Cd(a^tx,a^ty)^\theta$. Since $p$ and $q$ cover points of the same stable manifold and $s > t$, $a^{s-t}$ is nonexpanding on preimages of leaves of $\mc W$ in the bundle, and exponentially contracting on the base manifold. It follows that $d(\psi_{a^sx,a^sy}(a^sp),a^sp) < Cd(x,y)^\theta e^{-\theta\lambda s}$, and 
    \[ d(a^{s-t}\psi_{a^tx,a^ty}(a^tp),a^sp) \le d(\psi_{a^tx,a^ty}(a^tp),a^tp) < Cd(x,y)^\theta e^{-\theta \lambda t}.\] By the triangle inequality, $d(\psi_{a^sx,a^sy}(a^sp),a^{s-t}\psi_{a^tx,a^ty}(a^tp)) < 2Cd(x,y)^\theta e^{-\theta \lambda t}$, and by \eqref{eq:comparing-kappas2}, $\kappa_t\kappa_s^{-1}$ is exponentially small in $t$. The Cauchy property follows.
\end{proof}
Now we state the main result of Section \ref{Appendix-holonomies}.
\begin{proposition}
\label{prop:holder-holonomies}
Let $X,\mc P, \mc W, K $ be as in Definition \ref{def:reg-along-fol}. 
%a group with a bi-invariant metric, $X$ be a $C^\infty$ manifold, $\mc W$ be a continuous foliation with $C^r$-leaves, and $\mc P$ be a continuous principal $K$-bundle which is $\mc W$-H\"older continuous. Fix an action $\R^k \curvearrowright \mc P$ by a \color{black} continuous \color{cyan} bundle automorphisms such that $a \in \R^k$ is uniformly exponentially contracting along $\mc W$ of $X$ \color{black} and uniformly H\"older along $\mc W$. Assume that there exists a family of identification $\psi_{x,y}$ satisfies the starting assumption of Section \ref{Appendix-holonomies} \color{cyan}. Then 
There exists a family of maps $H_{x,y} : P_x \to P_y$ for $y\in \mc W(x)$ such that

    \begin{itemize}
        \item[(1)] $H_{x,x} = \id$ and $H_{y,z} \of H_{x,y} = H_{x,z}$ whenever $y,z \in \mc W(x)$,
        \item[(2)] $H_{x,y}(kp) = kH_{x,y}(p)$, 
        \item[(3)] $H_{bx,by} b = b H_{x,y}$ for all $b \in \R^k$, and
        \item[(4)] for all $p \in P_x$, $d(a^tp, a^tH_{x,y}(p))$  
        decays %exponentially 
        in $t$.
    \end{itemize}

    Furthermore, the maps $H_{x,y}$ are uniquely determined by these properties, and vary continuously in $x$ and $y$. The topological submanifolds $\widehat{\mc W}(p) := \set{ H_{x,y}(p) : y \in \mc W(x)}$ form an $\R^k$-invariant continuous foliation of $\mc P$ which covers $\mc W$ and is uniformly %exponentially 
    contracted by $a$.\color{black}
\end{proposition}

\begin{proof}
    We first show uniqueness. In fact, we show that property (4) determines the maps $H_{x,y}$. Suppose that $q$ and $q'$ are both points of the fiber above $y$ such that $d(a^tp,a^tq)$ and $d(a^tp,a^tq')$ decay. %exponentially. 
    Then $q' = kq$ for some $k \in K$, and $a^tq' = a^tkq= ka^tq$. Therefore, $ka^tq$ and $a^tq$ are %exponentially 
    close as $t \to \infty$. Since the group $K$ has a bi-invariant metric, it follows that $k = e$, and $q = q'$.
    
%{\color{black}    Note that property (3) implies property (4) by the H\"older regularity of the holonomies $H_{x,y}$, so it suffices to establish (1)-(3) and show that they determine the maps $H_{x,y}$ uniquely. We first establish uniqueness. Suppose that $H_{x,y}'$ is another family of maps satisfying (1)-(4) which vary H\"older continuously in $x$ and $y$. Since the fiber above $y$ is a single $K$-orbit, there exists a $\kappa_{x,y} : F_x \to K$ such that $H_{x,y}(p) = \kappa(p)H_{x,y}'(p)$. Using property (2), we see that $\kappa_{x,y}$ is actually a constant depending only on $x$ and $y$. Finally, using property (3), it follows that $\kappa_{x,y} = \kappa_{a^tx,a^ty}$. Finally, by property (1), $H_{x,x} = H_{x,x}' = \id$ for all $x \in X$, so $\kappa_{x,x} = e \in K$ for all $x \in X$. By H\"older continuous variation of the maps $H_{x,y}$ and $H_{x,y}'$, it follows that since $d(a^tx,a^ty) \to 0$, $\kappa_{x,y} = \lim \kappa_{a^tx,a^ty} = e \in K$. That is, $H_{x,y}(p) = H_{x,y}'(p)$ for all $p \in F_x$. 
%}
We now prove that such a family exists. We claim that it suffices to show %H\"older 
continuous variation in $x$ and $y$, the first half of condition (1), condition (2) and condition (3). That is, we can deduce condition (4) and the second half of (1) from these. To see (4), note that by (3),
\[ a^tH_{x,y}(p) = H_{a^tx,a^ty}(a^tp)\]
and by continuous variation of $H_{x,y}$ in $x$ and $y$ and the first half of (1), $d(H_{a^tx,a^ty}(a^tp),a^tp)$ decays. %exponentially fast. 
Thus (4) is satisfied. Finally, from (4) and the triangle inequality, it follows that both $a^t H_{y,z} \of H_{x,y}(p)$ and $a^t H_{x,z}(p)$ get  %exponentially 
close to $a^tp$. Since we showed that condition (4) determines $H_{x,y}$, it follows that $H_{y,z} \of H_{x,y} = H_{x,z}$.

We conclude the proof by exhibiting a family of functions $H_{x,y}$ which satisfy the first half of (1), condition (2) and condition (3).
    We use Lemma \ref{lem:constructing-holonomy}. Given $p$ and $q$ in the fibers above $x$ and $y$, respectively, let $\kappa_\infty(p,q) = \lim_{t\to\infty} \kappa(a^tp,a^tq)$, and note that $\kappa_\infty(k_1p,k_2q) = k_1\kappa_\infty(p,q)k_2^{-1}$. Therefore, given $p$, $x$ and $y$ as above the point
\[ H_{x,y}(p) := \kappa_\infty(p,q)q\]
is a well-defined map from $P_x$ to $P_y$. \color{black}Moreover by Lemma  \ref{lem:constructing-holonomy} we know $H_{x,y}$ is a uniform limit of a family of a continuous function $\kappa_t$ hence is continuous (in $x,y$ for $y\in \mathcal W_{loc}(x)$) as well. \color{black}

It is clear that $H_{x,x} = \id$. Furthermore, (2) is easily satisfied from the fact that $\kappa_\infty(kp,q) = k\kappa_\infty(p,q)$. So we need to verify condition (3). %, and that $H_{x,y}$ varies \color{black} H\"older continuously along coarse\color{cyan}. 
We verify all properties when the points $x$, $y$ and $z$ are sufficiently close, and extend to the global leaf $\mc W$ using (3) after showing them for close points. To see (3), we invoke the uniquness property established at the start of this proof, that condition (4) determines the image of $H_{bx,by}$. Then note that $q := H_{x,y}(p)$ is the point in the fiber of $y$ which converges  %exponentially 
toward $p$ under $a^t$. Then by commutativity, $bq$ converges %exponentially
 towards $bp$ under $a^t$. Hence, $bH_{x,y}(p) = H_{bx,by}(bp)$ for all $b \in \R^k$. %Note that since $\psi_{y,z} \of \psi_{x,y} = \psi_{x,z}$, it follows that $\kappa(q,r)\kappa(p,q) = \kappa(p,r)$. This extends to the versions with $\kappa_\infty$, and the second part of (1) follows.

\color{black}

To show that we get the foliation $\hat{\mc W}$, we build foliation charts explicitly. Choose a foliation chart for $\mc W$ on $X$, so that the horizontals of the chart are the leaves of $\mc W$, $\varphi : \R^\ell \times \R^{d-\ell} \to U$ locally defined and centered at $(0,0)$, $U \subset \mc P$. Let $\sigma : U \to \mc P$ be a local \color{black}  continuous section of the bundle. \color{black}  Then build the locally defined foliation chart

\[ \hat{\varphi} : \R^\ell \times \R^{d-\ell} \times K \to \mc P \qquad \hat{\varphi}(x,y,k) = H_{\varphi(0,y),\varphi(x,y)}(k\sigma(\varphi(0,y))) \]

It follows from (2) and the definition of a principal bundle that since $\varphi$ is locally a  homeomorphism onto its image, %\color{black} and H\"older along $\set{(x,\ast) : x \in \R^\ell}$ \color{black}
 so is $\hat{\varphi}$. Furthermore, by the construction the image of the horizontals $\set{(x,y_0,k_0) : x \in \R^\ell}$ are exactly the leaves of $\hat{\mc W}$. Finally, using properties (1) and (2), it follows that transition maps between any two such charts respect horizontals. That is, we have a foliation which is %\color{black} continuous and H\"older along coarses \color{black}. The foliation is 
 $\R^k$-invariant by (3) and uniformly contracting under $a$ by (4).
\end{proof}

 %such that $\lim_{t \to \infty} d(a^tp,a^tH_{x,y}(p)) =0$, such a point is unique, and $H_{x,y}$ varies H\"older continuously in $x$ and $y$.

 \subsection{Brin-Pesin constructions\color{black}}

Before proving the existence of a transitive subbundle for compact group extensions, we first prove some preliminary lemmas about partitioning spaces and recurrence.

{\color{black}

\begin{definition}
    If $X$ is a compact metric space and $\mc W = \set{\mc W(x) : x \in X}$ is a partition of $X$ (into not necessarily closed sets), we say that $\mc W$ is {\it continuously varying} if whenever $y \in \mc W(x)$ and $\ve > 0$, there exists some $\delta = \delta(\ve,x,y) > 0$ such that if $d(x,x') <\delta$, then there exists some $y' \in \mc W(x')$ such that $d(y,y') < \ve$.
\end{definition}

It is clear that if $\mc W$ is the partition into the leaves of a continuous foliation, then $\mc W$ is continuously varying. Furthermore, if one forms an accessibility equivalence relation out of a family of foliations $\mc F_1,\dots,\mc F_n$, namely that $y \in \mc W(x)$ if and only if there is a path with finitely many legs in the foliations $\mc F_i$ connecting $x$ and $y$, then $\mc W$ is also continuously varying.

\begin{lemma}
\label{lem:top-decomp}
    Let $X$ be a compact metric space, $K$ be a compact group, $\mc G$ be a Hausdorff topological group.  Assume that $K$ acts on $\mc G$ by automorphisms to form the semidirect product $K \ltimes \mc G$, and that there is a jointly continuous action $K \ltimes \mc G \curvearrowright X$. Furthermore, assume that $\mc W$ is a continuously varying partition of $X$ such that $h\mc W(x) = \mc W(hx)$ for all $h \in K \ltimes \mc G$.
    
    Assume that for some $x_0 \in X$, $K\mc W(x_0)$ is dense and the following property holds:
    
    \begin{itemize}
        \item[(*)] \label{topdecomp-star}for every $\ve > 0$, there exists a $\delta > 0$ and an $K$-invariant compact subset $B \subset \mc G$ such that $\displaystyle\sup_{x \in X}\operatorname{diam}(B\cdot x) < \ve$ and that for all $x,y \in X$ such that $d(x,y) < \delta$, there exists $y' \in \mc W(x)$, $g \in B$ and $k \in K$ such that $gky' = y$ and $d_K(e,k) < \ve$.
    \end{itemize}   %the restricted action of $M \ltimes \mc G$ has the following property: 
    
%\begin{quote}
%    There exists $x_0 \in X$ and $M$-invariant compact subsets $K_n \subset \mc H$ such that for every $x \in X$ and $h \in K_n$, $h \mc G\cdot x = \mc G h\cdot x$, $\diam(K_n\cdot x)$ converges to 0 uniformly in $n$, and for every $n$, $M \ltimes (K_n\mc G) \cdot x_0 = X$ (in particular, $x_0$ has a dense $M \ltimes \mc G$ orbit).
%\end{quote}    

    Then there is a closed subgroup $K' \subset K$ and closed sets $\mc F(\theta) \subset X$ indexed by $\theta \in K / K'$ such that

    \begin{itemize}
        \item[(1)] each $\mc F(\theta)$ is the closure of $\mc W(x)$ for some $x \in X$,
        \item[(2)] for every $\theta \in K/K'$, and $k \in K$, $k \mc F(\theta) = \mc F(k\theta)$, and
        \item[(3)] $X = \bigsqcup_{\theta \in K / K'} \mc F(\theta)$.
    \end{itemize}
\end{lemma}

\begin{proof}
    Throughout we assume that the metric on $X$ is $K$-invariant. This is not a loss of generality since given any metric, the metric $d_k(x,y) = d(kx,ky)$ is equivalent (by continuity of the $K$ action), and we may average to get $\tilde{d}(x,y) = \int_K d_k(x,y) \, d\mu(k)$, where $\mu$ is the Haar measure on $K$. Choose a point $x_0 \in X$ such that $K\mc W(x_0)$ is dense in $X$, and set $\mc F_0 = \overline{\mc W(x_0)}$. Define $K' := \set{k \in K : k\cdot x_0 \in \mc F_0}$.

    We claim that $K' \subset K$ is a compact subgroup. Indeed, it is clearly closed since $\mc F_0$ is closed, so it suffices to show that it is a semigroup since $K$ is compact. To see the semigroup property, assume that $K_1,k_2 \in K$. Then there exist sequences $x_n, y_\ell \in \mc W(x_0)$ such that $x_n \to k_1 \cdot x_0$ and $y_\ell \to k_2 \cdot x_0$. Since the $K$-action is continuous and intertwines the sets $\mc W(x)$, it follows that $x_n' := k_1^{-1}x_n \in \mc W(k_1^{-1}x_0)$ converges to $x_0$.

    Fix $\ve > 0$ and choose $\ell_0$ large enough so that $d(y_{\ell_0},m_2\cdot x_0) < \ve/2$. Since $\ell_0$ is fixed and $\mc W$ is continuously varying, there exists $\delta > 0$ such that if $d(x,x_0) < \delta$, then there exists $y_{\ell_0}'(x) \in \mc W(x)$ such that $d(y_{\ell_0}'(x),y_{\ell_0}) < \ve/2$. Finally, choose $n_0$ large enough so that $d(x_{n_0}',x_0) < \delta$. Then since $x_{n_0}' \in \mc W(k_1^{-1}x_0)$, and $y_{\ell_0}'(x_{n_0}') \in \mc W(x_{n_0}')$, $y_{\ell_0}' := y_{\ell_0}'(x_{n_0}') \in \mc W(k_1^{-1}(x_0))$, and

    \begin{equation} \label{eq:convergence}
    d(y_{\ell_0}',k_2x_0) \le d(y_{\ell_0}',y_{\ell_0}) + d(y_{\ell_0},k_2x_0) < \ve/2 + \ve/2 = \ve. \end{equation}

    Using $K$-invariance of the distance, and the intertwining property again, we get that $d(k_1y_{\ell_0}',k_1k_2x_0) < \ve$, and that $k_1y_{\ell_0}' \in \mc W(x_0)$. Since $\ve$ is arbitrary we get that \[k_1k_2\cdot x_0 \in \overline{\mc W(x_0)} = \mc F_0,\] so $K'$ is a subgroup.

    We claim that the set $\mc F_0$ is $K'$-invariant. Indeed, assume that $y \in \mc F_0$ and $k \in K'$. Then there exist sequences $x_n,y_\ell \in \mc W(x)$ such that $x_n \to kx_0$ and $y_\ell \to y$. As above we fix $ \ve >0$, and choose some $\ell_0$ such that $d(y_{\ell_0},y) < \ve/2$. Since $\ell_0$ is now fixed, we may choose $\delta = \delta(x_0,y_{\ell_0},\ve/2) > 0$ such that if $d(x,x_0) < \delta$, then there exists some $y_{\ell_0}'(x) \in \mc W(x)$ such that $d(y_{\ell_0}'(x),y_{\ell_0}) < \ve/2$.
    
    Finally, let $x_n' := k^{-1}x_n \in \mc W(k^{-1}x_0)$. Then leveraging that $K$ acts by isometries and intertwines $\mc \mc W$, we may choose $n_0$ such that  $d(x_{n_0}',x_0) < \delta$. A similar string of inequalities to \eqref{eq:convergence} shows that $d(y_{\ell_0}'(x_{n_0}'),y) < \ve$. Reapplying $k$ yields that $k\cdot y \in \mc F_0$ and $\mc F_0$ is $K'$-invariant.

    For $\theta \in K / K'$, we may now define $\mc F(\theta) = \theta \mc F_0$. Observe that this is well-defined since $\mc F_0$ is $K'$-invariant. The $K$-equivariance is true by definition. To see that the sets $\mc F(\theta)$ cover $X$, note that for any $y \in X$, since $K\mc W(x_0)$ is dense in $X$, there exists $x_n \in \mc \mc W(x_0)$ and $k_n \in K$ such that $k_nx_n \cdot x_0 \to y$. Since $K$ is compact, by passing to a subsequence, we may assume that the sequence $k_n$ converges to some $k_0 \in K$. But then

    \[ d(k_0x_n,y) < d(k_0x_n,k_nx_n) + d(k_nx_n,y).\]

    The first term goes to zero since the maps induced by $k_n$ converge to $k_0$ in the topology of uniform convergence (since the action is jointly continuous and $X$ is compact). The second goes to zero by assumption. Hence $k_0^{-1}y \in \mc F_0$, so $y \in k_0\mc F_0 = \mc F(k_0K')$.

    So it remains only to show that the sets $\mc F(\theta)$ are disjoint for different $\theta$-values. Given that $\mc F(\theta_1) \cap \mc F(\theta_2)\not= \emptyset$, we will show that $\theta_1 = \theta_2$. 

    Assume that $y \in \mc F(\theta_1) \cap \mc F(\theta_2)$, and $\theta_1=k_1K'$, $\theta_2 = k_2K'$. Then there exists some sequences $x_n,x_n' \in \mc W(x_0)$ such that $k_1x_n,k_2x_n' \to y$. When $n$ is large enough, by (*) we may choose some $h_n \in \mc G$, $x_n'' \in \mc W(x_n)$ and $k_n \in K$ such that $k_n \to e$, $d(x_n'',x_n) \to 0$, $h_n\cdot x \to x$ for all $x \in X$, and

    \[ h_nk_nk_1x_n'' = k_2x_n'.\]

    We will find a sequence of points $y_n \in \mc W(x_0)$ converging to $k_1^{-1}k_2x_0$, implying by definition that $k_1K' = k_2K'$. Indeed, observe that $x_n'' = k_1^{-1}h_n^{-1}k_n^{-1}k_2 x_n'$, and that since $K$ normalizes both itself and $\mc G$, $x_n'' = \hat{h}_n^{-1}\hat{k}_n^{-1}k_1^{-1}m_2x_n'$ for some $\hat{h}_n = k_1^{-1}h_nk_1 \in \mc G$ and $\hat{k}_n = k_1^{-1}k_nk_1 \in K$. Let $y_n = \hat{h}_n^{-1}\hat{k}_n^{-1}k_1^{-1}k_2x_0 = k_1^{-1}h_n^{-1}k_n^{-1}k_2x_0$, so that by the assumptions on $h_n$ and $k_n$, $y_n \to k_1^{-1}k_2x_0$. Furthermore, since the $K \ltimes \mc G$-action intertwines the partition $\mc W$, $y_n \in \mc W(x_n'')$. Finally, by assumption, $x_n'' \in \mc W(x_n)$ and $x_n \in \mc W(x_0)$, so $y_n \in \mc W(x_0)$. Hence the sets $\mc F(\theta)$ partition $X$.   \end{proof}
    
%    Then $m_1^{-1}m_2x_n' = \hat{h}_n\hat{m}_nx_n''$, where $\hat{h}_n$ and $\hat{m}_n$ are the conjugates of $h_n$ and $m_n$ by $m_1$. Notice that $x_n'' \in \mc W(x_n) = \mc W(x_0)$, and that $\hat{h}_n\hat{m}_nx_n'' \to m_1^{-1}m_2x_n'$. Fix $\ve > 0$, and choose $n$ such that $d(\hat{h}_n\hat{m_n}x_n'',m_1^{-1}m_2x_n') < \ve/2$. With such an $n$ fixed, since $x_n' \in \mc W(x_0)$, we may choose a $\delta > 0$ such that if $d(x,x_n') < \delta$, there exists $y \in \mc W(x)$ such that $d(y,x_0) < \ve/2$.  

%    Then $m_2^{-1}m_nm_1x_{n''} = \hat{h}_n^{-1}x_n'$, where $\hat{h}_n$ is the conjugation of $h_n$ by $m_2$. Since the $\mc G$-action intertwines the $\mc W$-sets and $x_n' \in \mc W(x_0)$, it follows that $m_2^{-1}m_nm_1x_n'' \in \mc W(\hat{h}_n^{-1}x_0)$. Fix $\ve > 0$ and choose $n_0$ large enough so that if $n \ge n_0$,  $d(\hat{h}_n^{-1}x_0,x_0) < \ve /3$ and $d(x_n'',x_n) < \ve/3$. 

%     In particular, this implies that $m_1^{-1}m_2 x_0 \in \mc F_0$, or that $m_1^{-1}m_2 \in M'$. Thus, if the sets intersect, the corresponding values of $\theta \in M / M'$ must coincide.

}

\begin{lemma}
\label{lem:fiber-recurrence}
    Let $\mc P$ be a {\color{black} continuous }%an $\mc R$-regular
    principal bundle with %$\mc R$  and 
    compact structure group $K$ over a base space $X$ with projection map $\pi$, and $F: \mc P \to \mc P$ be a {\color{black} continuous }%
    bundle automorphism covering a homeomorphism of the base, $f : X \to X$. Then $p \in \mc P$ is $F$-recurrent if and only if $\pi(p)$ is $f$-recurrent.
\end{lemma}

\begin{proof}
    First, note that if $F^{n_k}(p) \to p$, then by continuity of the projection, $f^{n_k}(x) \to x$, where $x = \pi(p)$. Hence, $F$-recurrence of $p$ implies $f$-recurrence for $\pi(p)$ immediately.

    Now assume that $x$ is recurrent, so that $f^{m_n}(x) \to x$ on $X$, and note that $F^{m_n} : \pi^{-1}(x) \to \pi^{-1}(f^{m_n})$ is a map intertwining the $K$-actions on the fibers. The family $F^{m_n}|_{\pi^{-1}(x)}$ is hence equicontinuous and sub-converges to a translation from $\pi^{-1}(x)$ to itself. Let $K'$ denote the set of translations obtained as $\lim_{n \to \infty} F^{m_n}|_{\pi^{-1}(x)}$ for a subsequence $m_n$ such that $m_n \to \infty$. We claim $K'$ is a closed sub-semigroup of $K$ (and hence a closed subgroup of $K$ \cite[6.15]{BQ}). Indeed, assume that $k_1,k_2 \in K$, so that $F^{m_n}(p) \to k_1 \cdot p$ and $F^{m_n'}(p) \to k_2\cdot p$. We will build a subsequence $F^{\ell_n}$ such that $F^{\ell_n}(p) \to k_1k_2 \cdot p$.

     Fix $\ve >0$, and choose $m_{n_0}$ such that $d(F^{m_{n_0}}(p),k_2\cdot p) < \ve/2$. For a fixed $m_{n_0}$, the transformation $F^{m_{n_0}}$ is uniformly continuous, so there exists a $\delta >0$ such that if $d(p,q) < \delta$, then $d(F^{m_{n_0}}(q),k_2\cdot p) < \ve$. Finally, choose $n_1$ such that $d(F^{m_{n_1}'} (p), k_1\cdot p) < \delta$. Then $d({k_1}^{-1}F^{m_{n_1}'}(p), p) < \delta$ and by choice of $\delta$, $d(F^{m_{n_1}'}{k_1}^{-1}F^{m_{n_0}}(p)),k_2\cdot p) < \ve$. Using the fact that $F$ commutes with $K$ and $K$ preserves the distance $d$ again, it follows that $d(F^{m_{n_0} + m_{n_1}'}p,k_1k_2 \cdot p) < \ve$. It follows that there is a subsequence in the $F$-orbit of $p$ which also converges to $k_1k_2 \cdot p$.

     Note that since $K'$ is a subgroup of $K$, it contains $e \in K$. This immediately implies that if $x$ is recurrent for $f$ on $X$, every point of the fiber $\pi^{-1}(x)$ is recurrent for $F$ on $\mc P$.
\end{proof}

\begin{lemma}
\label{lem:residual-recurrence}
     Fix an $\R^k$ action %$\alpha$ 
     by continuous automorphisms of a continuous principal subbundle $\mc P$ with compact structure group $K$ covering a partially hyperbolic,  totally recurrent $\R^k$-action %$\bar{\alpha}$ (\color{black}THE BAR NOTATION!\color{cyan}) 
     on a manifold $X$. Further assume that %$\mc P$ is $\mc W^s_a$- and $\mc W^u_a$-H\"older continuous 
     for every partially hyperbolic $a \in \R^k$ %\color{black} and that 
     the foliations $\mc W^s_a$ and $\mc W^u_a$ on $X$ lift to %partially H\"older 
     continuous foliations $\hat{\mc W}^s_a, \hat{\mc W}^u_a$ respectively on $\mc P$. \color{black} Then for any partially hyperbolic element $a\in \R^k$, there is a residual set of points $p \in \mc P$ such that $\overline{\R^k \cdot p}$ is saturated by $\hat{\mc W}^s_a, \hat{\mc W}^u_a$.
     %the stable and unstable manifolds of the action of  on $\mc P$.
\end{lemma}

\begin{proof}
    To show this, we follow the strategy as in Lemma \cite[Lemma 10.2]{Spatzier-Vinhage}. Fix a partially hyperbolic element $a \in \R^k$. % with a dense orbit on $X$. Such a point must exist because the action is cone transitive by \cite[Lemma 4.17]{spatzier-vinhage}. 
    %Note that by Proposition \ref{prop:holder-holonomies}, the stable and unstable foliations of $a$ on $X$ lift to \color{black} partially H\"older \color{cyan} topological foliations on $\mc P$.
    
     %By \cite[Lemma 10.2]{spatzier-vinhage},
     %We claim that there is a residual set of points $p_0 \in P$ such that $\overline{\R^k \cdot p_0}$ is saturated by the lifted stable and unstable foliations. 
     
     By Lemma \ref{lem:fiber-recurrence}, the set of $a$-recurrent points is a dense-$G_\delta$ subset of $\mc P$.  Hence the set of $a$-recurrent points is a dense $G_\delta$ saturated by $K$-orbits. Finally, we build a set of points such that $\overline{\R^k \cdot p}$ is saturated by stable and unstable foliations of the form $\hat{\mc W}^s_a$ and $\hat{\mc W}^u_a$.

     We recall a fact about topological foliations, the Kuratowski-Ulam theorem: if $R$ is a residual set and $\hat{\mc W}$ is a topological foliation, then there is a residual set $R'$ such that form every $p \in R'$, $R \cap \hat{\mc W}(p)$ is residual in the leaf topology of $\hat{\mc W}(p)$. Therefore, we may choose a residual set $R_0$ such that for every $p \in R_0$, the set of $a$- and $(-a)$-recurrent points in $\mc W^s_a(p)$ and $\mc W^u_a(p)$ are both dense in their respective leaf topologies. It now follows from standard arguments (the ``topological Hopf argument'') that the $\R^k$ orbit closure of any point in $R_0$ is saturated by $\hat{\mc W}^s_a(p)$ and $\hat{\mc W}^u_a(p)$. %the stable and unstable manifolds of $a$. 
     See, e.g, \cite[Lemma 10.2]{Spatzier-Vinhage}. \end{proof}

%     Finally, observe that this implies $\overline{\R^k \cdot p}$ is saturated by each {\color{black}coarse Lyapunov foliation}. Since points in the stable and unstable manifolds of any partially hyperbolic element can be reached by following paths in the coarse Lyapunov foliations, it follows that the orbit is closure is saturated by {\it every} stable and unstable manifold.

As a consequence of Lemma \ref{lem:residual-recurrence}, since there are only countably (in fact finitely) many choices of $a$ with distinct $\mc W^s_a$ and $\mc W^u_a$, by Baire Category theorem we have 
\begin{corollary}\label{coro of lem rec}Under assumptions of Lemma \ref{lem:residual-recurrence}, there is a residual set of points $p\in \mc P$ such that $\overline{\R^k\cdot p}$ is saturated by $\hat{\mc W}^s_a, \hat{\mc W}^u_a$ for any partially hyperbolic element $a\in \R^k$.
\end{corollary}
 \begin{proposition}
 \label{prop:brin-pesin-construct}
      Fix an $\R^k$ action by continuous automorphism of %$\alpha$ 
     a continuous principal subbundle $\mc P$ with compact structure group $K$ covering a partially hyperbolic, totally recurrent $\R^k$-action %$\bar{\alpha}$ 
      on a manifold $X$ such that for any $x,y \in X$, there exists $a \in \R^k$ such that $ax$ and $y$ are connected by a path along the $\mc W^s_{b}$
      %coarse Lyapunov foliations 
      with finitely many legs (for each leg $b$ could be different). Further assume for every partially hyperbolic $a \in \R^k$ %\color{black} and that 
     the foliations $\mc W^s_a$ and $\mc W^u_a$ lift to %partially H\"older 
     continuous foliations $\hat{\mc W}^s_a, \hat {\mc W}^u_a$ respectively on $\mc P$. 
      %Assume further that $\mc P$ is $W^s_a$-H\"older and $W^u_a$-H\"older for every $a \in \R^k$. \color{black}Assume that there exists a family of fiber identification $\psi_{x,y}$ (NOT $\psi$, just existence of holonomies.) satisfies the assumption of Proposition \ref{prop:holder-holonomies}. \color{cyan} 
      Then there exists a continuous subbundle $\mc P' \subset \mc P$ with structure group $K' \subset K$ such that 
     \begin{itemize}
         \item $\mc P'$ is $\R^k$-invariant.
         \item For any partially hyperbolic element $a$, $\mc P'$ is saturated by $\hat{\mc W}^s_a, \hat {\mc W}^u_a$ (as well as their finest intersections, the lift of coarse Lyapunov foliations).
         %the stable and unstable manifolds of the action 
         %of any partially hyperbolicelement $a$. %on $\mc P$
         %$\mc P'$ is invariant under the stable holonomies $H_{x,y}$ of Proposition \ref{prop:holder-holonomies}
         \item $\R^k$ has a dense orbit on $\mc P'$.
         \item If $\mc P''$ is another bundle satisfying these conditions with associated structure group $K''$, then there exists $k \in K$ such that $\mc P'' = k\mc P'$ and $K'' =kK'k^{-1}$.
         %\item {\color{black}There exists an anti-homomorphism $\beta : M \to N_K(K')/K'$ such that for all $m \in M$, $m\mc P' = \beta(m)\mc P'$. In particular, $\ker \beta$ leaves $\mc P'$ invariant.}
     \end{itemize}

 \end{proposition}

%{\color{black}
%\begin{corollary}
%    Fix an action $\alpha$ as in Proposition \ref{prop:brin-pesin-construct}. Then the map $\hat{\alpha} : (\R^k \times M) \times \mc P' \to \mc P'$ defined by

%    \[ \hat{\alpha}((a,m),p) = \alpha(a,m)\beta(m)^{-1}p\]

%    is a well-defined action on $\mc P'$. {INTERNAL NOTE: NOT BY BUNDLE AUTOMORPHISMS}
%\end{corollary}
%}
 We fix a choice of $\mc P'$ satisfying the conclusions of Proposition \ref{prop:brin-pesin-construct}, and call it the {\it ($\R^k$-)Brin-Pesin subbundle} (for the action). %$\alpha$. 
 \color{black}In particular for $\tilde X$ defined Section \ref{sec: Harn ort frm}, by Proposition \ref{prop: partial Hol prin bd} and Proposition \ref{prop:holder-holonomies}, 
% Lemma \ref{lem: loc idfy},
 its ($\R^k$-)Brin-Pesin subbundle is well-defined and we denote it by $\hat X$.\color{black}

\begin{proof}
    Define an equivalence relation $\mc W$ on $\mc P$ by $y \in \mc W(x)$ if and only if there is some $a \in \R^k$ and a path with finitely many legs along the some $\hat{\mc W}^s_{b}$ foliations ($b$ of each leg could be different) %of %$\alpha$ 
    connecting $ax$ and $y$, so $\mc W$ is actually the $\R^k$-saturated $\{\hat{\mc W}^s_{b_i}\}$-
    accessibility classes. Then let $\mc G = \set{e}$ and $K$ act trivially on $\mc G$, so that $K \ltimes \mc G = K$ acts on $\mc P$. Then all of the assumptions of Lemma \ref{lem:top-decomp} are satisfied, since $K$ preserves %the $\R^k$-saturated accessibility classes 
    the partition by $\mc W$, i.e. \begin{itemize}
        \item[(1)] for any $x,y\in \mc P$ such that $y\in \R^k\cdot x$ and any $k\in K$, $ky\in \R^k\cdot kx$ since $\R^k$ acting on $\mc P$ by automorphisms.
        \item[(2)] for any $x,y$ such that $y\in \hat {\mc W}^s_b(x)$ and any $k\in K$, $ky\in \hat {\mc W}^s_b(kx)$ since $y$ converges towards to $x$ under $b^t$ if and only if $ky$ converges towards to $kx$ under $b^t$.
    \end{itemize}
    
    %by Proposition \ref{prop:holder-holonomies}. 
    Furthermore, for any $x_0\in \mc P$, by assumption $K\mc W(x_0)$ is actually equal to all of $\mc P$ and property (*) holds since $\mc W(x)$ always covers $X$. Without loss of generality we could choose $x_0$ that as in Corollary \ref{coro of lem rec}. 
    Hence by Lemma \ref{lem:top-decomp} there exists a closed subgroup $K' \subset K$ %and a point $x_0 \in \mc P$ 
    such that the sets $\mc F(kK') = \overline{k\mc W(x_0)}$ partition $\mc P$. 

    Set $\mc P' = \mc F(K')$. Since $\mc P'=\overline{\mc W(x_0)}$,  it is clearly saturated by $\R^k$-orbits and saturated by $\hat{\mc W}^s_b, \hat {\mc W}^u_b$, for any partially hyperbolic element $b$. %$\mc P'$ is saturated by invariant under stable/unstable holonomies. 
    By the choice of $x_0$, we further know that $\overline{\R^k\cdot x_0} = \overline{\mc W(x_0)}$, so $\mc P'$ contains a dense $\R^k$-orbit.

    Now we show the uniqueness property. Let $\mc P'$ and $\mc P''$ be two such bundles with associated groups $K'$ and $K''$. Fix points $p'$ and $p''$ such that $\overline{\R^k \cdot p'} =  \mc P'$ and $\overline{\R^k \cdot p''} = \mc P''$. Since $\mc P'$ and $\mc P''$ both cover $X$, there exists some $q'' \in \mc P''$ such that $q'' = k_1p'$ for some $k_1 \in K$. Then $\overline{\R^k \cdot q''} \subset \mc P''$ and hence $k_1\mc P' \subset \mc P''$. It follows that $k_1K'k_1 ^{-1} \subset K''$. By a symmetric argument, there exists some $k_2 \in K$ such that $k_2 \mc P'' \subset \mc P'$, and hence $k_2K''k_2^{-1} \subset K'$. It follows that $\Ad(k_1)$ must take $\Lie(K')$ into $\Lie(K'')$ injectively, and $\Ad(k_2)$ takes $\Lie(K'')$ into $\Lie(K')$ injectively. Hence they have the same dimension and are isomorphic. Furthermore, the conjugations induce a map on the connected components of $K'$ and $K''$, which are finite sets, so we conclude that conjugation by $k_1$ is an isomorphism between $K'$ and $K''$. The result follows.%Then $k_2k_1K'(k_2k_1)^{-1}K' \subset K'$, so since $K'$ is a compact group, $k_2k_1$ normalizes $K'$  
\end{proof}

\color{black}

%\bibliography{bdsvx.bib}{}

\begin{thebibliography}{AAA}

\bibitem{AnBrZh} J. An, A. Brown and Z. Zhang, 2024. 
Zimmer's conjecture for non-split semisimple Lie groups.
arXiv 2411.13858.


\bibitem{ASV} A. Avila, J. Santamaria and M. Viana, 2013. Holonomy invariance: rough regularity and applications to Lyapunov exponents. \emph{Astérisque}, 358, pp.13-74.

\bibitem{AV} A. Avila, M. Viana, (2010). Extremal Lyapunov exponents: an invariance principle and applications. \emph{Inventiones mathematicae}, 181(1), 115-178.

\bibitem{BPS} \color{black} L. Barreira, Ya. Pesin, J. Schmeling. Dimension and product structure of hyperbolic measures. Annals of Mathematics. Volume: 149, Issue: 3, page 755-783. \color{black}


\bibitem{Bekka-Harpe} B. Bekka, P. de la Harpe, A. Valette. Kazhdan's property (T).  New Mathematical Monographs, 11. Cambridge University Press, Cambridge, 2008.


\bibitem{BFL}Y. Benoist, P. Foulon and F. Labourie, (1992). Flots d’Anosov à distributions stable et instable différentiables. Journal of the American Mathematical Society, 5(1), 33-74.

\color{black}
\bibitem{BQ} Y. Benoist and J.-F. Quint. Random walks on reductive groups. Ergebnisse der Mathematik und ihrer Grenzgebiete {62}, Springer 2016.

\bibitem{Benveniste-thesis}
Elie Jerome Benveniste, 
Rigidity and deformations of lattice actions preserving geometric structures.
Thesis (Ph.D.)–The University of Chicago. 1996. 40 pp.

\bibitem{Benveniste-Fisher}   
E. Jerome Benveniste and David Fisher, 
Nonexistence of invariant rigid structures and invariant almost rigid structures. Comm. Anal. Geom. 13 (2005), no. 1, 89–111.


\bibitem{brin-pesin}
M. I. Brin and Ja. B.   Pesin.
     Partially hyperbolic dynamical systems.
  Izv. Akad. Nauk SSSR Ser. Mat., {38}, 1974,
     170--212. 
\bibitem{B03}M. I. Brin, On dynamical coherence. Erg. Th. Dyn. Sys. (2003), 23, 395–401.

\bibitem{Brown}\color{black} A. Brown,  Smoothness of stable holonomies inside center-stable manifolds. Erg. Th. Dyn. Sys. 42(12):1-26
42(12):1--26. \color{black}
  
\bibitem{BRHW} A. Brown, F. Rodriguez Hertz, and Z. Wang. Smooth ergodic theory of $Z ^d$-actions.  Smooth ergodic theory of $\mathbb {Z}^d$-actions. Journal of Modern Dynamics 19 (2023).

\bibitem{BFH2020} A. Brown, D. Fisher and S. Hurtado. Zimmer's conjecture for actions of {${\mathrm SL}(m, \mathbb Z)$}, \emph{Inventiones mathematicae}, 221 (3), 1001--1060 (2020).

\color{black}
\bibitem{BFH2022} A. Brown, D. Fisher and S. Hurtado. Zimmer’s conjecture: Subexponential growth, measure rigidity, and strong property (T),  \emph{Annals of Mathematics} (2), 196 (3), 891--940, 2022. 
\color{black}

\bibitem{BFH2021} A. Brown, D. Fisher and S. Hurtado. Zimmer's conjecture for non-uniform lattices and escape of mass. arXiv:2105.14541.

\bibitem{BRHW2} A. Brown, F. Rodriguez Hertz, and Z. Wang. Global smooth and topological rigidity of hyperbolic lattice actions. \emph{Ann. of Math.}(2), 186(3):913–972, 2017.

\bibitem{BW}K. Burns and A. Wilkinson, 2010. On the ergodicity of partially hyperbolic systems. \emph{Annals of Mathematics}, pp. 451-489.

\bibitem{Butler}C. Butler. Rigidity of equality of Lyapunov exponents for geodesic flows.  \emph{Journal of Differential Geometry} 109, no. 1 (2018): 39-79.

\bibitem{BTV} GJ. Butler, JG. Timourian and C. Viger. The Rank Theorem for Locally Lipschitz Continuous Functions. \emph{Canadian Mathematical Bulletin}. 1988;31(2):217-226.

\bibitem{CM}P. Chernoff and J. Marsden (1970). On continuity and smoothness of group actions. Bulletin of the American Mathematical Society, 76(5), 1044-1049.


\bibitem{CP} S. Crovisier, R. Potrie, Introduction to Partially Hyperbolic Dynamics.

\bibitem{DWX} D. Damjanović, A. Wilkinson and D. Xu, D. (2021). Pathology and asymmetry: centralizer rigidity for partially hyperbolic diffeomorphisms. Duke Mathematical Journal, 170(17), 3815-3890.

\bibitem{DX1} D. Damjanovi\'c and D. Xu. On classification of higher rank Anosov actions on compact manifold,  \emph{Isr. J. Math.} 238, 745--806 (2020). 

\bibitem{DamKatII} D. Damjanovi\'c and A. Katok, Local Rigidity of Partially Hyperbolic Actions. II: The Geometric Method and Restrictions of Weyl Chamber Flows on $SL (n, \mathbb R)/\Gamma$. International Mathematics Research Notices. 2011; 2011(19):4405-30.

\bibitem{Deodhar78} V. Deodhar, On Central Extensions of Rational Points of Algebraic Groups. American Journal of Mathematics, Vol. 100, No. 2 (Apr., 1978), pp. 303-386.

\bibitem{Feres95} { R. Feres and F. Labourie.  Topological superrigidity and {A}nosov actions of lattices}, \emph{Ann. Sci. \'{E}cole Norm. Sup.} 31 (4), 1998,  599--629.

\bibitem{FH}D. Fisher, T. Hitchman, Harmonic maps with infinite dimensional targets and cocycle superrigidity, Int. Math. Res. Not. 2006, 72405, 1–19.

\bibitem{FKS} D. Fisher, B. Kalinin and R. Spatzier. Totally nonsymplectic Anosov actions on tori and nilmanifolds. Geom. Topol, 15(1), 2011, 191-216.

\bibitem{FM} D. Fisher and G. A. Margulis. Local rigidity for cocycles. In volume 8 of Surv. Differ. Geom., Int. Press, Somerville, MA, 2003, 191-234.

\bibitem{MR4186267} D. Fisher. Recent developments in the Zimmer program. \emph{Notices Amer. Math. Soc.} 67 (2020), no. 4, 492–499.

\bibitem{MR2807830} D. Fisher. Groups acting on manifolds: around the Zimmer program, Geometry, rigidity, and group actions, Chicago Lectures in Math., Univ. Chicago Press, Chicago, IL, 2011, pp. 72–157, DOI 10.7208/chicago/9780226237909.001.0001.

\bibitem{Goetze-Spatzier-Duke} {E. Goetze and R. J. Spatzier. On {L}iv\v{s}ic's theorem, superrigidity, and {A}nosov actions of semisimple {L}ie groups. \emph{Duke Mathematical Journal}, 88 (1), 1--27, 1997.}

\bibitem{gleason-palais}Andrew M. Gleason and   Richard S. Palais,  On a class of transformation groups. {\emph Amer. J. Math.} 79 (1957), 631–648.

\bibitem{Goetze-Spatzier}  E. Goetze and R. J. Spatzier. Smooth classification of Cartan actions of higher rank semisimple Lie
groups and their lattices. \emph{Ann. of Math.} (2), 150(3):743–773, 1999.

\bibitem{graev1950} M. I. Graev. On free products of topological groups. \emph{Izvestiya Akad. Nauk SSSR. Ser. Mat.}, 14:343–354, (1950).

\bibitem{gromov-rigid} M. Gromov.  Rigid transformations groups.  Géométrie différentielle (Paris, 1986), 65–139,
Travaux en Cours, 33, Hermann, Paris, 1988.

\bibitem{Hel}S. Helgason, Differential Geometry, Lie Groups, and Symmetric Spaces, 
Sigurdur Helgason, Graduate Studies in Mathematics, Volume: 34; 2001.

\bibitem{humphreys}James E. Humphreys,  Introduction to Lie algebras and representation theory. Second printing, revised. \emph{Graduate Texts in Mathematics}, 9. Springer-Verlag, New York-Berlin, 1978. 

\bibitem{Hurder}S. Hurder, 1992. Rigidity for Anosov actions of higher rank lattices. Annals of Mathematics, 135(2), pp.361-410.

\bibitem{Journe}J. L. Journé, (1988). A regularity lemma for functions of several variables. Revista Matematica Iberoamericana, 4(2), 187-193.

\bibitem{Kalinin} B. Kalinin, Livšic theorem for matrix cocycles. Annals of Mathematics (2011): 1025-1042.


\bibitem{Kal2020}B. Kalinin, Non-stationary normal forms for contracting extensions, to appear in Dynamics and related topics, 2020 vision and agenda.

\bibitem{KalSad16}B. Kalinin and V. Sadovskaya, 2016. Normal forms on contracting foliations: smoothness and homogeneous structure. Geometriae Dedicata, 183(1), pp.181-194.

\bibitem{KalSad} B. Kalinin and V. Sadovskaya. Cocycles with one exponent over partially hyperbolic systems. \emph{Geom Dedicata} 167, 167--188 (2013).

\bibitem{KalSad07}B. Kalinin and V. Sadovskaya, 2007. On classification of resonance-free Anosov $\mathbb Z^k$ actions. Michigan Mathematical Journal, 55(3), pp.651-670.


\bibitem{KSp}B. Kalinin and R. J. Spatzier. On the classification of Cartan actions. \emph{Geometric And Functional Analysis}, (2007), 17(2), 468--490.

\bibitem{KatKon} A. Katok and A. Kononenko, Cocycles’ stability for partially hyperbolic systems. Math.
Res. Lett. 3 (1996), no. 2, 191--210.

\bibitem{KLZ}A. Katok A, J. Lewis, R. Zimmer, Cocycle superrigidity and rigidity for lattice actions on tori. Topology. 1996 Jan 1;35(1):27-38.


\bibitem{KL91}A. Katok and J. Lewis. Local rigidity for certain groups of toral automorphisms. Israel J. Math., 75(2-3):203–241, 1991.

\bibitem{KL96}A. Katok and J. Lewis. Global rigidity results for lattice actions on tori and new examples of volume-preserving actions. Israel J. Math., 93:253–280, 1996.

\color{black}
\bibitem{KS-subelliptic estimates} A. Katok, R. Spatzier, Subelliptic Estimates of Polynomial Differential Operators and Applications to Rigidity of Abelian Actions, 
Mathematical Research Letters, vol 1, 1994, 193-202. 
\color{black}

\bibitem{KatSpat}A. Katok and R. J. Spatzier. Differential rigidity of Anosov actions of higher rank abelian groups and algebraic lattice actions. Tr. Mat. Inst. Steklova, 216(Din. Sist. i Smezhnye Vopr.):292–319, 1997.

\bibitem{kryszewsi-plaskacz} W. Kryszewski and S.  Plaskacz. Topological methods for the local controllability of nonlinear systems. SIAM J. Control Optim. 32 (1994), no. 1, 213–223. 

\bibitem{Ledrappier1986} F. Ledrappier. Positivity of the exponent for stationary sequences of matrices. In Lyapunov exponents (Bremen, 1984), volume 1186 of Lect. Notes Math.,  56–73. Springer, 1986.

\bibitem{ledrappier_young1985} F. Ledrappier and L. S. Young. The metric entropy of diffeomorphisms: part ii: relations between entropy, exponents and dimension.  \emph{Annals of Mathematics}, Vol. 12,(3),  p. 540-574, 1985).

{ \bibitem{Margulis-Problems}G. A. Margulis, 2000. Problems and conjectures in rigidity theory. Mathematics: frontiers and perspectives. Amer. Math. Soc., Providence, RI.}

\bibitem{M91}G. A. Margulis, 1991. Discrete subgroups of semisimple Lie groups (Vol. 17). Springer Science $\&$ Business Media.


\bibitem{MR0492072}  G. A. Margulis. Discrete groups of motions of manifolds of nonpositive curvature. Proceedings of the {I}nternational {C}ongress of {M}athematicians ({V}ancouver, {B}.{C}., 1974), {V}ol. 2, 21--34.

\bibitem{MR739627}  G. A. Margulis.   Arithmeticity of the irreducible lattices in the semisimple groups of rank greater than {$1$}. 
\emph{Inventiones Mathematicae}, 93--120, 76, (1984).

\bibitem{MontZip1} D. Montgomery and L. Zippin. Topological Transformation Groups. I. Annals of Mathematics, 41(4), 778–791, 1940.

\bibitem{Montgomery-Zippin} D. Montgomery and L. Zippin. Topological transformation groups. Courier Dover Publications, 2018.

\bibitem{Sag}\color{black} R. Saghin, On invariant holonomies between centers, Ergod. Th. $\&$ Dynam. Sys., (2025), 45, 274-293. \color{black}

\bibitem{PSW}C. Pugh, M. Shub and A. Wilkinson, H\"older foliations, Duke
Math. J., 86 (1997), no. 3, 517–546.

\bibitem{PSW04} C. Pugh, M. Shub, and A. Wilkinson. Partial differentiability of invariant splittings. \emph{J. Statist.
Phys}., 114(3-4):891–921, 2004.

\bibitem{ordman1974} E. T. Ordman. Free products of topological groups which are $k_{\omega}$-spaces. \emph{Trans. Amer. Math. Soc.}, 191:61–
73, 1974.

\bibitem{RSS} D. Repovs, E. Scepin. and A. Skopenkov, $C^1$-homogeneous compacta in $\mathbb R^n$ are $C^1$-submanifolds
of $\mathbb R^n$. Proc. Amer. Math. Soc., 124:1219--1226, 1996.



{\color{black}
\bibitem{RZ2024} F. Rodriguez Hertz and Z. Wang. Non-rigidity of partially hyperbolic abelian $C^1$-actions on tori, Ergodic Theory and Dynamical Systems,  doi:10.1017/etds.2024.18
}

\bibitem{Schmidt-thesis} B. Schmidt. Weakly hyperbolic actions of {K}azhdan groups on tori.
\emph{Geom. Funct. Anal.} 16 (5):1139--1156, 2006.
\color{black}

\bibitem{Schr}S. J. Schreiber, On growth rates of subadditive functions for semiflows, J. Differential Equations 148 (1998), 334–350. MR 1643183. Zbl 0940.37007. doi: 10.1006/jdeq.1998.3471.

\color{black}

\bibitem{Smale66} S. Smale. Differentiable dynamical systems. \emph{Bull. Amer. Math. Soc.}, 73:747–817, 1967.



\color{black}
\bibitem{Spatzier-Vinhage} R. J. Spatzier and K. Vinhage.  Cartan Actions of Higher Rank Abelian Groups and their Classification.  \emph{J. Amer. Math. Soc.} 37 (3), 731--859, 2024.
\color{black}

\bibitem{hilbert5-2014} T. Tao. Hilbert’s fifth problem and related topics, volume 153 of \emph{Graduate Studies in Mathematics. American
Mathematical Society}, Providence, RI, 2014.

\bibitem{Taylor} M. Taylor.  Existence and regularity of isometries.  \emph{Trans. Amer. Math. Soc.}, 358 (6), 2415--2423, 2006. 

\bibitem{vinhageJMD2015}  K. Vinhage. On the rigidity of Weyl chamber flows and Schur multipliers as topological groups. \emph{J. Mod. Dyn.},
9:25–49, 2015.

\bibitem{VinhageExample}  K. Vinhage. Instability for rank one factors of product actions. Preprint. 

\bibitem{vinhage-wang} K. Vinhage and Z. J. Wang. Local rigidity of higher rank homogeneous abelian actions: a complete solution via the geometric method. \emph{Geom Dedicata} 200, 385–439 (2019). 

\color{black}
\bibitem{Wang1} Z. J. Wang.  Local rigidity of partially hyperbolic actions. \emph{Journal of Modern Dynamics}, 4(2), p.271--327, 2010.

\bibitem{Wang2} Z. J. Wang. New cases of differentiable rigidity for partially hyperbolic actions: Symplectic groups and resonance directions. \emph{Journal of Modern Dynamics}, 4(4), 585--608, 2010.
\color{black}

\bibitem{Warner1}  G. Warner.
Harmonic analysis on semi-simple {L}ie groups. {I}.
Die Grundlehren der mathematischen Wissenschaften, Band 188,
Springer-Verlag, New York-Heidelberg, 1972,

\bibitem{W} A. Wilkinson. The cohomological equation for partially hyperbolic diffeomorphisms. \emph{Ast\'erisque}, 358 (2013).

\bibitem{Wilson} E. N. Wilson.
Isometry groups on homogeneous nilmanifolds.  \emph{Geometriae Dedicata},  volume 12,  pages 337--346 (1982).

\bibitem{Zeghib} A. Zeghib.
Ensembles invariants des flots géodésiques des variétés localement symétriques. 
\emph{Ergodic Theory Dynam. Systems} 15 (1995), no. 2, 379–412.

\bibitem{Z} R. J. Zimmer. Ergodic Theory of Semisimple Groups, \emph{Birkh\"{a}user}, Monographs in Mathematics (1984).

\bibitem{Z-ICM} R. J.  Zimmer.   Actions of semisimple groups and discrete subgroups. In \emph{Proceedings of the International Congress of Mathematicians}, Vol. 1, 2 (Berkeley, Calif., 1986), pages 1247–1258. Amer. Math. Soc., Providence, RI, 1987.

\bibitem{Z-1983}
R. J. Zimmer. Arithmetic groups acting on compact manifolds. Bull. Amer. Math. Soc. (N.S.), 8(1):90– 92, 1983.

\bibitem{Z-IHES}
R. J. Zimmer. Volume preserving actions of lattices in semisimple groups on compact manifolds. \emph{Inst. Hautes \'{E}tudes Sci. Publ. Math.}, (59):5–33, 1984.

\bibitem{Zimmer-notes}  R. J. Zimmer. Unpublished notes. 




\end{thebibliography}
%\bibliographystyle{alpha}

\end{document}